\documentclass[11pt]{article}
\usepackage{amsmath, amscd,amsbsy, amsthm, amsfonts, amssymb}
\usepackage{rotating}
\usepackage{dcolumn,multicol,url}

 \newtheorem{thm}{THEOREM}[section]
 
 \newtheorem{lem}[thm]{LEMMA}
 \newtheorem{prop}[thm]{Proposition}
 
 \theoremstyle{definition}
 \newtheorem{defn}[thm]{Definition}
 \newtheorem{exam}[thm]{Example}

 \newtheorem{rem}[thm]{Remark}

 \numberwithin{equation}{section}

 \newcommand{\A}{\mathcal{A}}

\def\C{{\mathbf C}}

\def\ZZ{{\mathbb Z}}

\def\cC{{\mathcal C}}

\def\cD{\mathcal D}
\def\cE{{\mathcal E}}

\def\cG{{\mathcal G}}

\def\cO{{\mathcal O}}
\def\cP{{\mathcal P}}

\def\cR{{\mathcal R}}
\def\cS{{\mathcal S}}
\def\cT{{\mathcal T}}

\def\cZ{{\mathcal Z}}

\def\FF{{\mathbb F}}
\def\FF{{\mathbb F}}
\def\ZZ{{\mathbb Z}}

\def\be{\begin{equation}}
\def\ee{\end{equation}}
\def\ba{\begin{eqneqnarray}}
\def\ea{\end{eqneqnarray}}
\def\tilde{\widetilde}

\def\e1{\epsilon}
\def\FFl{\mathcal{A}_{\lambda}}
\def\A0{\stackrel{\circ}{\FFl}}

\def\o1{\omega}
\def\01{\Omega}
\def\c1{\gamma}

\def\g1{\Sigma}

\def\l1{\Lambda}

\def\v1{\varphi}

\def\d1{\delta}

\def\f1{\frac}
\def\t1{\theta}
\def\b1{\beta}
\def\bar{\overline}
\def\bs{\begin{eqneqnarray*}}
\def\es{\end{eqneqnarray*}}
\def\m1{\Theta}
\def\w1{\wedge}

\def\ee{\epsilon}

\begin{document}

\title{ Explicit Acyclic Models and (Co)Chain Operations}

\date{ }
\author{Greg Brumfiel and John Morgan }
\maketitle

\tableofcontents

\setcounter{section}{-1}
 {. }\\

\newpage
\addcontentsline{toc}{section}{PART I: Contractions and Construction of Chain Maps}

\section*{ I: Contractions and Construction of Chain Maps}

\section{Introduction}  The purpose of this elementary and largely expository four  part project is to record some observations about acyclic model methods and their use for constructing chain maps between chain complexes, related to an operadic approach to cochain operations and cohomology operations.  More specifically, we were originally  interested in making more explicit the construction of Steenrod operations for all primes and establishing  their properties at the cochain level.  Classically the development of cohomology operations  was carried out in the 1940's and 1950's.  Steenrod began the study using direct cochain constructions, the $\cup_i$ operations leading to the Steenrod Squares, but subsequently homotopy theoretical methods proved more powerful.  More recently, in the 1990's and early 2000's, certain classical methods of Steenrod were extended to give a general treatment of multivariable cochain operations using the machinery of operads,  [3], [19].  Although the authors of those works understood that they were generalizing early constructions of Steenrod for the Steenrod Squares, and providing a framework for a more explicit cochain level discussion of odd prime Steenrod $p^{th}$ power operations than was carried out in the 1950's, their primary interest was not necessarily Steenrod operations, but rather other problems involving cochain algebras  and operations for which operad machinery was relevant.\\

In the papers [22] and [7], cochain level proofs of the Cartan formula and Adem relations for Steenrod Squares were presented, using the operad methods.  Substantial difficulties remained for giving analogous cochain level proofs of the Cartan formula and the Adem relations for  odd prime Steenrod operations.  In Part IV of this project we plan to give explicit cochain level proofs of  properties of the Steenrod operations, including  additivity,  the odd prime Cartan formula, $P^0 = Id$, stability under suspension, and the Adem relations, using the operad methods.\\

However, in the process, we realized that many explicit chain maps involved in the operadic discussion could be `explained' using a universal procedure, based on classical  ideas behind acyclic model methods.  Our basic procedure and many examples are detailed in Part I of our project. This includes some clarification of  the classical Alexander-Whitney and Eilenberg-Zilber chain maps.  Our procedure also eventually includes the operad structure maps of the operads known as the Barratt-Eccles operad  $\cE$ and the surjection operad $\cS$,   and certain  morphisms between them $\cE \leftrightarrows \cS \to \cZ$, where $\cZ$ is the Eilenberg-Zilber operad. We briefly explain this terminology.\\

The Eilenberg-Zilber operad  is the $CoEnd$ operad in the functor category for the functor $N_*(X)$ of simplicial sets. That is, an element of $\cZ(n)$ means  a collection of graded module maps $N_*(X) \to N_*(X)^{\otimes n}$ for all $X$ that satisfies the natural transformation requirement for maps between simplicial sets. Each $\cZ(n)$ is a chain complex, with degree of elements being the degree of the corresponding graded module maps. A simple duality yields the Eilenberg-Zilber $End$ operad. An element of this $End$ operad in arity $n$ is a natural transformation of graded module maps for simplicial sets $N^*(X)^{\otimes n} \to N^*(X)$, that is, a natural multilinear cochain operation.\\

We study three isomorphic versions of the surjection operad $\cS$ in Part II. The original due to McClure and Smith [19], [20], a somewhat simpler one due to Berger and Fresse [3], [4], and a third due to  Adamaszek and Jones [1].  McClure and Smith discovered the chain complexes $\cS_*(n) $ as a suboperad of the Eilenberg-Zilber operad, that is, with an operad morphism  {\it injection} $\cS(n) \hookrightarrow \cZ(n)$. Berger and Fresse defined an operad morphism {\it surjection} $\cE(n) \twoheadrightarrow \cS(n)$, where the components $\cE(n)$ of the  Barratt-Eccles operad are the chains $N_*(E\Sigma_n)$ on the  standard MacLane model of a contractible free $\Sigma_n$ simplicial set.\\

The operad results are covered in Part III of the project.  Part II should be viewed as a mini-course that organizes in a self-contained manner the three versions of the surjection complexes, and establishes their basic properties  relatively painlessly.  In the end, we found the observations in Parts II and III interesting by themselves, and more useful than our original goals of explicit cochain level proofs of the Cartan formula and the Adem relations.\\

This paper consists of Parts I, II, and III of our project, which accounts for its length.  Parts II and III are largely independent of Part I, once some basic ideas about contractions{\footnote{A contraction of a complex is a null homotopy of the complex extended by an augmentation. So the chain complex is homotopy equivalent to the chain complex  of a point. Details are in Section 4.}  $h$ of  chain complexes with $h^2 = 0$, and their use for constructing chain maps, are absorbed from Part I.  We contemplated submitting three separate articles.  Instead, we strongly encourage readers to view  this long paper  as three separate papers, skip around, discover what is accomplished in the separate parts. Some of the technical results will be  used in Part IV to establish properties of Steenrod operations at the cochain level.  Part IV of our project will be submitted separately.\\

What is the point of  explicit cochain level arguments concerning cohomology operations that we pursue in Part IV?  The actual cochain formulas for Steenrod operations and relations between operations turn out to be hopelessly large to be of any use.  The homological and homotopy theoretical methods that were used to develop the Steenrod algebra in the 1950's were profound and exquisite.  The classical applications and computations in examples don't use underlying cochain formulas.  However, it seems relevant to clarify, for  simplicial sets or  topological spaces, exactly how the Steenrod operations on cohomology groups can be defined {\it explicitly} at the cochain level, and their properties understood using {\it explicit cochain arguments.}  We present an answer in Part IV,  using results from Parts I, II, and III, that supplement the original constructions of the 1950's.\\

Originally we believed cochain level formulas for  Steenrod operations and  for Adem relations between compositions of operations would be useful for studying explicit cochain level simplicial set models of two and three stage Postnikov towers, as in our papers [9], [10] on low dimensional Spin bordism and [11] on Pin$^-$ structures.\\

But it became clear that  was not going to go very far.  It is already rather difficult  to find cocycle formulas $k_1$ and $k_2$ for the first two $k$-invariants of a three stage Postnikov tower $E$.  This is where explicit cochain formulas for cohomology operations and relations between operations are required to produce simplicial set models for three stage Postnikov towers.  Among other things, one needs {\it names} for cochain operations that go far beyond Steenrod's two variable $\cup_i$ operations.   Once one has enough names, maps $X \to E$ for simplicial sets $X$  are then described by triples of cochains $(w,p,a)$ on $X$ with $da = 0, dp = k_1 = k(a)$ and $dw = k_2 = k(p, a)$.  The simplicial sets $E$ that we define using explicit cochain formulas for $k$-invariants are minmal Kan complexes.\\

But it is not enough to simply enumerate the simplicial maps $X \to E$ in this way.  One also wants to describe with  formulas the homotopy equivalence relation on triples $(w, p, a)$.  This becomes a harder problem at the cochain level.  Even if it is known that $E$ is an $H$-space, it is an added level of difficulty to describe with formulas an explicit  simplicial set product map $E \times E \to E$, and prove that it is homotopy commutative and associative.  Essentially one wants some kind of explicit simplicial presentation of the abelian group of homotopy classes of maps $[X, E]$.  In the case of a loop space $E = \Omega \widehat{E}$ one also wants to understand the isomorphism $[\Sigma X, \widehat{E}] \simeq [X, E]$ in terms of formulas involving cochain operations and cochain suspension.  Our papers on the Pontrjagin duals of 3 and 4 dimensional Spin bordism provide   examples that confirm how difficult all this gets.\footnote{Our results on Spin bordism are closely related to explicit simplicial set models for the first three stages of the Postnikov towers of $S^3, S^4, S^5$, mostly based on the Adem relation $Sq^2Sq^2 +S^3Sq^1 =0$. How much simpler could it get?  Viewing $K(\ZZ/8, n)$ as a three stage explicit simplicial tower with building blocks $K(\ZZ/2, n)$, based on the relation $Sq^1 Sq^1 = 0$, sounds simpler but is also provocatively complicated even for small $n$.} For $k$-stage Postnikov towers with $k > 3$ these problems seem almost hopeless.\\

The operad methods produce much more than an alternate development of Steenrod operations and their properties.  The full structure of the normalized cochain algebra $N^*(X, \ZZ)$ as an algebra over the Barratt-Eccles operad $\cE$ or the surjection operad $\cS$  actually determines the homotopy type of a finite type simply connected simplicial set $X$. Only a very small part of this structure is needed to define Steenrod operations and establish their properties at the cochain level.  Versions of the full theorem, including deep results relating  homotopy categories of spaces and homotopy categories of operad algebras, are often referred to as Mandell's Theorem [16], [17].  Much work on such issues was also carried out by Justin R. Smith [27], [28], [29].  Earlier attempts were made by the Russian mathematician  V. A. Smirnov [26]. Several researchers have contributed to further refinements.\\

Before proceeding we will insert here some comments about base rings and coefficient modules for the chain complexes and homology and cohomology groups that we study.  Whenever a connection with Steenrod operations is mentioned, the base ring and coefficient modules will be understood to be a prime field $\FF_p$. Other times, for example in the discussion of operads in  categories of chain complexes,  we can work with chain complexes of  modules over any commutative ground ring, for example $\ZZ$.   But we also generally do not include notation for coefficients in our symbols $N_*(X)$ and $N^*(X)$ for chain  and cochain groups.  If important, we clarify.\\

So the functorial $\cE$-operad or $\cS$-operad algebra structure of $N^*(X, \ZZ)$ determines in some difficult to make precise manner not only the module structure of $H^*(X)$ over  Steenrod algebras, but also higher order cohomology operations, differentials in Adams spectral sequences, the homotopy groups of $X$, and Postnikov towers for $X$. It was always a mantra that  the functorial   structure of the $H^*(X)$  as algebras and modules over  Steenrod algebras was not enough to fully deal with  such questions, one needed to dig deeper into the chain and cochain level.  The operad results can be viewed as a modern take on the developement of (semi)simplicial methods for homotopy theory in the 1950's, including E. H. Brown's result that homotopy groups of finite simply connected complexes were algorithmically computable by simplicial  methods.  One cannot predict what the power of future computers might bring to the table in the study of direct simplicial methods, but the complexity of algorithms is  seriously exponential.\\

Perhaps our project of a cochain level study of Steenrod operations in Part IV should sort of be viewed as meeting a challenge, like climbers scaling some mountain by a new difficult route.  Of what use is that?  Not  to get to the top. You could maybe hike up steps, drive up a road, or take a helicopter.  But beyond the challenge, we found that some of the computations we made, and some of the connections between older and newer ideas related to cochain operations, acyclic models, and operads, were rather interesting to us, and might interest others.\\

Our paper is lengthy, because we include many details and examples and give essentially complete proofs of the major results, which takes many pages.  But it is conceptually relatively elementary compared to thousands of other papers on operads and homotopy theory  written by many hundreds of authors during the last 50 or more years.  We find versions of some of our results included in some of these recent and not so recent papers.    It seems clear that the operad algebra approach to the homotopy category and stable homotopy category deserves a permanent place in algebraic topology, of which our goal of a cochain level development of the Steenrod algebra and its action on cohomology rings is a very small piece.  Our paper can  be viewed as a reformulation of some historical results, an introductory work, or an insert  that should come before  deeper results, which likely have not reached a final form.\\

We will conclude this introductory section with a brief summary.  The final goal of these Parts I, II, III is a {\it new} development of the $E_\infty$ Barratt-Eccles operad $\cE $ and the $E_\infty$ surjection operad $\cS $ of McClure-Smith and Berger-Fresse, along with operad morphisms to the Eilenberg-Zilber $CoEnd$ operad  $\cZ$ for simplicial sets,\footnote{For chain complexes, $HOM(B_*, C_*)$ means the chain complex of graded linear maps between  graded  modules.  We will study these complexes in Section 3.} $$\cE(n) = N_*(E \Sigma_n) \twoheadrightarrow  S_*(n) = \cS(n) \hookrightarrow \cZ(n) = HOM_{func}(N_*( - ), N_*( - )^{\otimes n}).$$  
Any such operad morphisms yield functorial $E_\infty$ operad algebra structure extending the differential graded associative (DGA) algebra structure of cochains $N^*(X)$. We use our contraction based recursive procedure to rather easily define operad structure  maps for the $\cE(n)$ and $\cS(n)$. We can also use our contraction based procedure to easily define morphisms $\cE(n) \to \cS(n) \to \cZ(n)$.  We prove that these morphisms are operad morphisms, without using the fact that  the morphisms of Berger-Fresse and McClure-Smith are operad morphisms.  To obtain explicit formulas for our recursive procedure, we do compare with the formulas of Berger and Fresse, or equivalently with those of McClure and Smith, and proceed inductively.  But it is certainly plausible that we could more directly have found the explicit formulas. In any case, {\it recursively} producing  operad structure maps  and operad morphisms, then {\it comparing} with other formulas, gives {\it new proofs} that these other formulas do indeed define operads and operad morphisms.\\  

Neither the original proofs that the Berger-Fresse  maps are chain maps and operad structure maps nor the proofs that our easily defined chain maps fully agree with the Berger-Fresse maps are easy.  But replication of important results should have a role in mathematics. Also, our new proofs  introduce alternate methods and points of view.\\

The recursive contraction based method itself is far more elementary than the technical details of these final goals, which are found in Sections 16, 17, 19, and 20. The basic method  involves a very simple explicit version of classical acyclic model methods.  It made no sense to us to just focus on the final goals.  Instead, we wanted to advertise  the ubiquity of the method and exhibit systematically many results about it, both  specific examples and more theoretical principles. Several results are included in Parts I, II, III because they are relevant for our cochain level approach to Steenrod operations to appear in Part IV. \\

This long paper is  more like a textbook than a research announcement.  But readers who are interested in deciphering what we are doing are strongly advised to not try to initially follow all details in the order written.  Certainly the review of acyclic models in Section 1 and  the previews of the  sections of Parts I, II, and III in Section 2 is the way to start.  In the main text, begin with Part I but become familiar with what is done in Parts II and III before getting bogged down with technical results in Part I.  Look at paragraphs, remarks, statements of results, and examples with a first goal of just getting the gist of what they say.  Many results are included because a goal of our treatment is to be comprehensive, but often  results are not used again until later sections.   Keep turning pages and skip ahead to later sections.  One can always back up and focus on earlier details and proofs. There are literally no prerequisites other than a basic course in algebraic topology along with its underlying algebra, and some familiarity with simplicial sets\\

We have marked with asterisks various subsections, remarks, and examples that can be skimmed or skipped without any loss of local continuity.  Those paragraphs do contain about 50 pages of reasonably interesting results, and sometimes are referred to and used later in the paper, at which time they can be reviewed. But a reader who wants to get on with it in the paper can simply skip some or all of these asterisked pages.\\

We will continue our paper with two more introductory sections.  First we present a perspective on the acyclic model method in topology.  After that we give brief previews of each section of Parts I, II, III, along with some additional discussion.  The purpose of our three sections of introduction is to dilute the fact  that the full paper is indeed very long.  It would seem difficult to get a good understanding of what we are doing by just starting  at the beginning and trying to read the sections  one after the other. We believe the introductions provide a good overview of the entire paper and provide a head start on individual sections.

\newpage

\section{Review of Acyclic Models}
\subsection{Contraction Based Acyclic Model Methods}
Let us review the acyclic model method.  First, suppose given chain complexes $B_*$ and $C_*$  over some commutative ground ring $\FF$, not necessarily a field, with $B _*$ free and $C_*$ contractible.\footnote{An acyclic complex means zero homology, and a contractible complex means chain homotopy equivalent to the complex $\FF$ concentrated in degree 0. In this introductory section we will be somewhat casual with this distinction.}   Both $B_*$ and $C_*$ may have $G$-actions in which case we assume $B_*$ is free over $\FF[G]$.  Differentials in chain complexes are always $\FF[G]$ linear. If $B_*$ is graded in non-negative degrees then beginning with a suitable  map in degree 0  one can construct (equivariant) chain maps $\phi \colon B_* \to C_*$ that are  unique up to (equivariant) chain homotopy.  By freeness of the domain, one just needs to define $\phi(b)$ for a set of basis elements $\{b\} \subset B_*$.  One does this recursively on degree, using the fact that a cycle in $C_{n-1}$ is the boundary of an element of $C_n$.  Given a basis element $b \in B_n$, one chooses $\phi(b) \in C_n$   so that $d \phi(b) = \phi (db)$, this latter element being a cycle by induction.   If $C_*$ is merely contractible, not acyclic,  `suitability' of $\phi \colon B_0 \to C_0$ will include $\phi(dB_1) \subset dC_1$. One then extends $\phi$ in degree $n$ by linearity or equivariant linearity and moves on to the next degree.  In a similar recursive manner one can construct (equivariant) chain homotopies between two (equivariant) chain maps.\\

Now, there are three levels of explicitness that one can consider.  First, if one uses only the fact that $C_*$ is contractible, then the construction is not at all explicit.  But if one has an explicit contraction, roughly a chain homotopy   $h \colon C_* \to C_{*+1}$  with $dh + hd = Id$ in non-zero degrees and also on $ dC_1$, then one has an explicit recursive formula and an easy inductive proof that it is a chain map. For $\FF[G]$-basis generators $b$,  define inductively $\boxed{\phi(b) = h \phi (db).}$ Then $$\boxed{d \phi(b) = dh \phi(db) = \phi(db) - hd \phi(db) = \phi(db) - h \phi( d db) = \phi(db).}$$ 
Extend $\phi$ to an $\FF$ basis by $\phi(gb) = g\phi(b)$. It is routine to prove the further $\FF$-linear extension is fully equivariant,  $\phi(gx) = g\phi(x)$ for all $x \in B_*$.\\

In practice however, this method only gives an explicit recursive procedure for defining $\phi(b)$.  The third level of explicitness  then is that one might be lucky and find a `closed' formula for the recursively defined $\phi(b)$.  Once one has a candidate for a formula, an inductive proof can often be found.\footnote{Closed formulas are appealing.  However, for computer work when formulas get large, or for proving theorems, it isn't so clear that a closed formula is better than a recursive procedure.}  \\

The constructed map $\phi$ may depend on the choice of basis in $B_*$.  However, in practice, our complexes $B_*$  come with a preferred choice of basis.  Also, $\phi$ will certainly depend on the choice of the contraction chain homotopy $h$ of $C_*$.  Again, in practice, our complexes $C_*$ come with  preferred contractions.  In fact, they come with  preferred contractions that satisfy $h^2 = 0$, which turns out to be a remarkably useful assumption for establishing many results.\\

The above paragraphs are a precursor to the functorial acyclic model method that constructs natural transformations $$\phi_{func} \colon F_*(X_1, \ldots , X_k) \to K_*(X_1, \ldots , X_k)$$ between chain complex valued functors using  acyclic models.  For us, the $X_i$ will be simplicial sets,  and $F_*(X_1, \ldots , X_k)$ will be a functor  that is free over $\FF$ or $\FF[G]$.  A basis $\{u\}$ will have the form $\{u = \sigma^F_*(\bar{u})\}$, where the $\{\bar{u} \}$ are certain universal  elements  $\bar{u} \in F_*(\Delta^{n_{1}}, \ldots, \Delta^{n_{k}})$, and  $\sigma_{i} \colon \Delta^{n_{i}} \to X_i$ are `simplices' in the $X_i$,   {\it canonically determined by $u$}. `Universality' of the $\{\bar{u}\}$ is meant to express that the chosen sets of basis elements  of the $F_*(X_1, \ldots , X_k) $ are `functorial' for simplicial maps in the variables $X_i$.   Since simplicial maps can send non-degenerate simplices to degenerate simplices, it will actually be the sets $\{u, 0\}$ of basis elements together with 0 that are  functorial.   We will also have  preferred contractions $h_K$ of the complexes $K_*(\Delta^{n_1}, \ldots, \Delta^{n_k})$, and in the equivariant case a $G$-action on $K_*(X_1, \ldots, X_k)$. \\

Then $\phi_{func} \colon F_*(X_1, \ldots , X_k) \to K_*(X_1, \ldots , X_k)$ is constructed recursively as follows.  One always begins with some functorial map in degree 0, and assumes a functorial chain map has been constructed in degrees less than $n$.  Given a basis element $u = \sigma_*^F(\bar{u}) \in F_n(X_1, \ldots , X_k) $, one defines
$$\boxed{\phi_{func}(u) = \sigma^K _*\circ h_K\circ \phi_{func} (d\bar{u}) \in K_n(X_1, \ldots, X_k).}$$ As before, extend by linearity over $\FF$ or $\FF[G]$.  It is the assumption that the universal elements $\bar{u}$ and simplices $\sigma_i$ are canonically determined by $u$ that guarantees that $\phi$ is well-defined. A routine argument shows that this procedure defines a chain map that is equivariant and functorial in the $X_i$.\footnote{All that is important about this somewhat vague attempt to describe a multivariable functorial  acyclic model process in general is that it makes clear sense in our examples.}

\subsection{Preview of Examples}
Here are a couple of examples.  For a simplicial set $X$, let $N_*(X)$ denote the normalized simplicial set chain complex.  There are iconic functorial chain homotopy equivalences of Alexander-Whitney and Eilenberg-Zilber $$AW\colon N_*(X \times Y) \to N_*(X) \otimes N_*(Y)$$  $$EZ \colon  N_*(X ) \otimes N_*Y) \to N_*(X \times Y).$$  Formulas for these maps are well-known, but where do they come from really?\\

Preferred basis elements  of the normalized chain complex $N_*(X \times Y)$   in degree $k$ arise uniquely from a pair of simplices $\Delta^k \to X$ and $\Delta^k \to Y$, so that composition with the diagonal $ \Delta^k \to \Delta^k \times \Delta^k \to X \times Y $ is non-degenerate.  Preferred basis elements of degree $k$  in $N_*(X) \otimes N_*(Y)$ are uniquely written as tensors of pairs of non-degenerate simplices $\Delta^i \to X$ and $\Delta^j \to Y$ with $i+j = k$. If $\Delta, \Delta'$ are simplices, there are canonical contractions of  $N_*(\Delta) \otimes N_*(\Delta')$ and $N_*(\Delta \times\Delta')$.  Easy inductive arguments prove the  standard functorial recursive procedure using these bases and these contractions produce exactly the classical $AW$ and $EZ$ formulas.\footnote{The $EZ$ chain map can be `found', more or less inductively, by thinking about the geometry and combinatorics  of triangulating prisms. The paper [31] of Whitney describes in some detail the early history in the 1930's  of cup products and the dual $AW$ map.}\\

For three or more simplicial set variables, one also has functorial  maps  relating tensor products of chain complexes and chain complexes of products of simplicial sets, directly defined using our preferred contractions of multiple tensor products of normalized chains on simplices or normalized chains on multiple products of simplices. Uniqueness theorems for maps produced by our procedure trivially imply without any formulas that the multivariable maps agree with any way of associating and composing  two variable maps. Thus, the classical $AW$ and $EZ$ maps are fully associative. An obviously stated commutativity property also holds for the $EZ$ map, which again we prove without  any formulas.\footnote{In more sophisticated language, the associativity and commutativity  results we prove, without any formulas,  for our recursively defined maps exactly say that the assignment of $N_*(X)$ to simplicial sets is a symmetric monoidal functor from the category of simplicial sets to the category of chain complexes.}\\

Repeating somewhat, we point out that various texts write down formulas for $AW$ and $EZ$, followed by rather awkward proofs, at least in the $EZ$ case, that they are chain maps and are associative.  Our recursive procedure easily defines functorial chain maps, using explicit contractions of models, then observes without much work that these functorial chain maps are associative and are given by the classical $AW$ and $EZ$ formulas. This {\it proves} the $AW$ and $EZ$ formulas are associative chain maps.\\

Of course the recognition of the abstract acyclic model method in the early 1950's was a major conceptual advance in algebraic topology.  It explains rather quickly many things, such as  why the cohomology of a space is a (skew)-commutative graded ring. The ring structure  was initially regarded as sort of a mystery, dependent on properties of odd little formulas.  Acyclic models instantly gives diagonal approximations, unique up to chain homotopy, and the associativity and commutativity of the cohomology product becomes obvious.  On the other hand, the fact that a cochain complex  itself is a DGA algebra, that is, with good diagonal formulas multiplication is strictly associative and the coboundary is a derivation, was recognized as important.   {\it So some diagonals are much better than others.} \\

A more dramatic example than the $AW$ and $EZ$ formulas  is given by natural transformation $\Sigma_n$-equivariant chain maps $$\Phi^{(n)}(X) \colon N_*(E\Sigma_n) \otimes N_*(X) \to N_*(X)^{\otimes n}$$ constructed by our explicit functorial acyclic model procedure, also using canonical contractions of $n$-fold tensor products of normalized chains on simplices.  As mentioned previously,  $N_*(E\Sigma_n)$ denotes normalized chains on the MacLane model of a free contractible $\Sigma_n$ simplicial set, and will be specifically described in Example \ref{5.3}.  An $\FF[\Sigma_n]$-basis of $N_*(E\Sigma_n) \otimes N_*(X)$ in degree $k$ is given by tensors of $\FF[\Sigma_n]$ generators of $N_*(E\Sigma_n)$ of degree $i$  with non-degenerate simplices $\Delta^j \to X$, $i+j = k$.  The map $\Phi^{(n)}(X)$ is a $\Sigma_n$-equivariant extension of the Alexander-Whitney multidiagonal $AW^{(n)} \colon \{e\} \otimes N_*(X) \to N_*(X)^{\otimes n}$, where $e \in N_0(E\Sigma_n) = \FF[\Sigma_n]$ is the identity group element.   The map $AW^{(n)}$ is not fixed by  the $\Sigma_n$ action on the range and equivariant extensions defined on $N_*(E\Sigma_n) \otimes N_*(X)$, unique up to equivariant chain homotopy, were very important in the development of  Steenrod operations.  Arbitrary choices of equivariant functorial maps $\Phi^{(n)}$ are adequate for constructing well-defined operations at the cohomology level.  But   some equivariant extensions, like the maps $\Phi^{(n)}(X)$ we construct, have additional good properties.  We discuss this further.\\

As mentioned previously in Section 0,  chain complex components of the Barratt-Eccles operad $\cE$ are given by $\cE(n) = N_*(E\Sigma_n)$.  Chain complex components of the Eilenbeg-Zilber operad $\cZ(n)$ are given by natural transformations $HOM_{func}(N_*( - ), N_*( - )^{\otimes n})$, or dually $HOM_{func}( N^*( - )^{\otimes n}, N^*( - ))$.  Some of our main points are that the collection of adjoints $$\phi^{(n)} \colon  \cE(n) = N_*(E\Sigma_n) \to HOM_{func}(N_*( - ), N_*( - )^{\otimes n}) = \cZ(n)$$  of our functorial chain maps $\Phi^{(n)}(X)$  defined by the canonical recursive procedure  factor through the surjection operad components $\cS_*(n)$ and {\it coincide} with the operad morphisms $\cE(n) \twoheadrightarrow \cS(n) \hookrightarrow \cZ(n)$ constructed by hand by Berger and Fresse in [3], [4], expanding on work of McClure and Smith in [19],  [20].\\

In other words, the collection of adjoint maps $\phi^{(n)}$ produced by the functorial  contraction procedure satisfy serious strict associativity and strict equivariance properties, equivalent to the strictly commutative diagrams expressing  operad morphism axioms.  {\it So some equivariant maps $\Phi^{(n)}$ are much better than others.} This point is ignored in the classical acyclic model treatments of cohomology operations, which only use the existence of functorial equivariant maps $N_*(E\Sigma_n) \otimes N_*(X) \to N_*(X)^{\otimes n}$, well-defined up to equivariant chain homotopy.\\  

The situation is quite analogous to the discussion above about diagonal approximations.  Any diagonal approximation yields products in cohomology.  Good diagonals, like $AW$, which our contraction procedure produces, result in the cochain algebras $N^*(X)$ themselves being differential graded associative algebras.  Any equivariant maps $\Phi^{(n)}$ yield cohomology operations. Good such maps, which our contraction procedure also produces, result in the cochains $N^*(X)$ being algebras over  $E_\infty$ operads, like $\cE$ and $\cS$. This is fantastically more structure than just $DGA$ algebra structure.  In fact, it is enough structure to easily define Steenrod operations at the cocycle level and ultimately prove their properties.  Amazingly, this structure also determines the homotopy type of finite type simply connected $X$.\\

We do not reach the full explanation of these examples until the end of Part III.  Among other things, the full explanation includes the facts that the operad structure maps themselves for the Barratt-Eccles and surjection operads coincide with maps between contractible complexes constructed by a standard recursive contraction procedure. We discuss this in the following subsection. \\

Justin R. Smith [28] was  aware that an  operad morphism $\cE \to \cZ$ could be recursively constructed, without formulas,  by essentially  the same method as ours, including using the same contraction of  tensor products of normalized chains on a simplex.  But his primary interests were elsewhere, and it is not clear if he wrote down full details. \\

\subsection{A Contraction Based Approach to Some $E_\infty$ Operads}  In a long paper like this  it can be difficult to see the forest for the trees.  In an attempt to illuminate at least a part of the forest, we will insert here, perhaps prematurely,  some comments concerning $E_\infty$ operads, elaborating somewhat on the discussion in the previous subsection.  Roughly, the  definition of an $E_\infty$ operad in a category of chain complexes is a collection of contractible, augmented, based, free $\FF[\Sigma_n]$ chain complexes $\epsilon, \iota \colon B_*(n) \rightleftarrows \FF$, together with structure maps $$\cO_B  \colon B_*(r) \otimes B_*(s_1) \otimes \ldots \otimes B_*(s_r) \to B_*(s),\ s = \sum s_i,$$ that satisfy certain identity, equivariance, and associativity axioms.\\

We propose  {\it a new  starting point} for operads in some special cases, namely choices of contractions $h_n \colon B_*(n) \to B_{*+1}(n)$ with $hd + dh = Id - \iota \epsilon$, where $\epsilon \colon B_*(n) \to \FF$ is an augmentation and $\iota \colon \FF\to B_*(n)$ is a basepoint, with $\epsilon \circ \iota = Id$.  We also assume the contractions satisfy $h_n^2 = 0$ and $h_n  \iota = 0$, and that $\FF[\Sigma_n]$-bases of the $B_*(n)$ are chosen in $Im(h_n) = Ker(h_n)$, including $\iota(1) = e \in B_0(n) = \FF[\Sigma_n]$.  With these assumptions, a twisted variant of our basic recursive contraction procedure for constructing chain maps produces {\it candidates} for operad structure maps $\cO_B$ that satisfy everything except possibly the associativity axiom.\\

The associativity axiom holds if for all elements $c_0 \in h_r B_*(r)$ and $c_i \in h_{s_i} B_*(s_i)$ one has $\cO_B(c_0 \otimes c_1 \otimes \ldots \otimes c_r) \in h_s B_*(s)$.\\

{\it So operad structure  emerges from suitable  contraction structure.}\\

All these additional assumptions hold {\it trivially}  for the Barratt-Eccles chain complexes $B_*(n) = N_*(E\Sigma_n)$. The reason is,  focusing  on the already assumed enhanced  chain complex structure, one more assumption, namely that $h_n B_*(n) $ coincides with the $\FF$-span of an $\FF[\Sigma_n]$ basis, makes it easy to verify the condition for the associativity axiom.  Also interestingly, these combined properties characterize $(N_*(E\Sigma_n), h_n)$ uniquely among such complexes.\\

Thus we have a new recursive construction of {\it some} operad whose components are the $N_*(E\Sigma_n)$.  It is then {\it trivial} to prove using  our uniqueness theorems that the original  Barratt-Eccles operad structure maps, defined using the symmetric monoidal functor properties of $N_*(X)$, coincide with our recursively defined structure maps.\\
   
The basic extra structure, including the associativity condition,  holds  less trivially for the surjection complexes $B_*(n) = \cS_*(n)$.     Any chain complex $(B_*(n), h_n)$ with the basic extra structure is canonically a direct summand of $(N_*(E\Sigma_n), h_n)$.  If the associativity condition holds for the operad structure candidate maps  $\cO_B$ then the $B_*(n)$ operad is canonically a quotient of the Barratt-Eccles operad.  In this sense, the Barratt-Eccles operad is a kind of universal model for (contraction based) $E_\infty$ operads.\\

We will give many examples in the course of this work, some rather surprising,  of using the explicit contraction based recursive  procedures  to construct maps between chain complexes.  Many times it is quite easy to find closed formulas for the results of the recursive procedures.  Other times, as in the examples involving the operads $\cE, \cS, \cZ$, this is harder, so we are content to prove inductively that our recursive procedure leads to closed formulas already  found by others.   Also,  our proofs using uniqueness theorems that the chain maps $\cO_B$ we define recursively  are operad structure maps {\it reproves} that the  formulas of Berger and Fresse, and McClure and Smith, define operad structure chain maps.  The original proofs of these results amounted to  quite complicated computations.
\newpage

\section{Preview of Parts I, II, and III}

\subsection{Preview of Part I}
Section 3  is a digression before we really begin the paper that summarizes some basic properties of tensor products of chain complexes $C_* \otimes D_*$ and chain complexes of linear maps of graded modules $HOM(C_*, D_*)$. The latter include dual cochain complexes.  We state a number of adjoint relationships and duality relationships between such complexes.  We  discuss actions of symmetric groups on multitensor products of chain and cochain complexes. We also make a few basic remarks about chain homotopies. It seems reasonable to just set our conventions on signs and other aspects of chain and cochain complexes and tensor products early on.  Section 3 can be skimmed, or essentially ignored, before moving on. The results are used much later to explain the $CoEnd$ and $End$ operad structures associated to chain and cochain complexes.  In the final (asterisked) subsection of Section 3 we  preview how additive cohomology operations, including the Steenrod cyclic reduced power operations, arise from certain equivariant chain maps $B_* \otimes C_* \to C_*^{\otimes p}$.\\

In Section 4 we make precise the notion of a contraction of a base pointed, augmented chain complex $\FF\xrightarrow{\iota} C_* \xrightarrow {\epsilon} \FF$ with $\epsilon \iota = Id_\FF$. A contraction is a chain homotopy $h \colon C_* \to C_{*+1}$ from $\iota \epsilon$ to the identity.   So $dh+hd = Id_{C_*} - \iota \epsilon$ and  $C_*$ is chain homotopy equivalent to the complex $\FF$ concentrated in degree 0. We show that one can always find a contraction satisfying  $h^2 = 0$ and $h \iota = 0$, which is an extremely useful property that plays a major role throughout the paper.  The symmetrical structure consisting of a graded module and two self-maps $d, h$ of degrees $-1$ and $+1$ satisfying $d^2 = 0, h^2 = 0, dh+hd = Id$  seems provocative.\footnote{Symmetry is broken if one works primarily with complexes $C_*$ that are bounded below.   Also, with $G$-complexes, $d$ is equivariant but $h$ is not.}  Any contraction says more than  that a chain complex with differential $d$ has the same homology as a point.  It provides a specific somewhat sophisticated reason why that is true as part of the structure. It is more mysterious why the additional property $h^2 = 0$ is so useful. \\

In Section 5 we give many examples of contractions. Given contractions of $C_*$ and $D_*$, we construct preferred contractions of the tensor product complex and the $HOM$ complex. The contraction we choose of a tensor complex plays a key role in many results of our paper.  We study a standard contraction $h_G$ of MacLane models $N_*(EG)$,  similar to a standard contraction $h_{\Delta}$ of chains on a simplex $N_*(\Delta)$.    We also study a contraction of a minimal model $M_*$ of chains on a contractible free $C$-complex, where $C$ is a cyclic group of order $n$.  $M_*$ is much smaller than the MacLane model $N_*(EC)$.   In fact, each $M_{k} = \FF[C]\{y_k\} $ is free on one generator.  $M_*$ is a very useful complex when $n=p$ is prime in our study of Steenrod operations in Part IV. There is nothing new about $M_*$. It is the chain complex associated to a free action of $C$ on a regular cell complex structure on the infinite sphere $S^{\infty}$, with $n$ cells in each degree. This complex played  a major role in the history of  classifying spaces, group homology, and cohomology operations.\\

In the case of $\FF[G]$ complexes, we discuss in Section 5 contractible complexes as acyclic resolutions of the trivial $G$-module $\FF$, but also possibly of some other $G$-module structure on $\FF$.  This turns out to be useful in Part II for studying one of the surjection complexes and in Part IV for studying Steenrod operations acting on odd degree cocycles.\\

Section 6 begins what should be regarded as the main point of Part I of the project, namely, the use of explicit contractions to construct chain maps $B_* \to C_*$ by the methods outlined in Section 1.  We prove two very useful uniqueness theorems that provide conditions that guarantee that a map between complexes is necessarily the standard procedure chain map. Our basic uniqueness theorem says that if the contraction of the range satisfies $h^2 = 0$ then an equivariant map in degree 0 has a {\it unique} extension to an equivariant chain map with the property that basis generators of the domain map to elements in the image of the contraction of the range.  This result has useful extensions to functorial chain maps and to `twisted' equivariant chain maps that are of importance  in later sections.\\

Also of particular importance is the following uniqueness phenomenon.  Suppose $B_*$ and $C_*$ are  $\FF[G]$-complexes, both with contractions satisfying $h^2 = 0$ and with $B_*$ free.  Then with mild hypotheses, {\it any  $\FF[G]$-linear map $B_* \to C_*$} that commutes with the contractions is automatically a chain map and is identical to the chain map constructed by the recursive contraction procedure.  In particular, the necessary hypotheses always hold when $B_* = N_*(EG)$ is a MacLane model. Many of the maps $\phi \colon N_*(EG) \to C_*$ that have roles in our work when $C_*$ is contractible commute with contractions.  In fact, the $\FF[G]$-basis elements of $N_*(EG)$ span the image of $h_G$ over $\FF$ and such chain maps can  be defined   by $\phi h_{G}(x) = h_C \phi(x)$, for all $x$.\\

As examples of chain maps important in our paper, we mention  the standard contraction procedure equivariant chain map $M_* \to N_*(EC)$ for a cyclic group, and a retraction $N_*(EC) \to M_*$. The retraction commutes with contractions.   We also give in Section 6 the explanation that the Alexander-Whitney map $AW\colon N_*(EH \times EG) \to N_*(EH) \otimes N_*(EG)$ for MacLane models is a special case of our standard contraction constructions.  The Alexander-Whitney map commutes with contractions.  The Eilenberg-Zilber map $EZ \colon N_*(EH) \otimes N_*(EG) \to N_*(EH \times EG)$ for MacLane models is also described as a special case of our procedure.  The functorial versions of $AW$ and $EZ$ for simplicial sets are postponed until Section 8.\\

We also discuss in Section 6 some delicate issues about compositions of standard procedure chain maps.  These need not be standard procedure maps. But we establish a number of conditions that imply such compositions are standard procedure maps. The most useful of these condition is that the second map sends  the image of the contraction of the domain to  the image of the contraction of the range, again always assuming $h^2 = 0$.  This is weaker than the condition that the second map commutes with contractions. There are equivariant, functorial, and twisted equivariant versions of this composition result that become quite important in Parts II and III.\\

Section 7 continues the construction of chain maps, using the standard contraction procedure, to construct diagonal maps and equivariant diagonal maps $ C_* \to C_* \otimes C_*$ for the  complexes $C_* = M_*$ and $C_* = N_*(EG)$. We also extend the constructions to multidiagonals $ C_* \to C_*^{\otimes n}$.  Such multidiagonals and certain  equivariant enhancements $ N_*(E\Sigma_n) \otimes C_* \to C_*^{\otimes n}$, extending the seminal work of  Steenrod,  were absolutely crucial in the study of cohomology operations from the very beginning.  The final (asterisked) subsection of Section 7 initiates the study of these equivariant enhancements, with some results eventually useful for a cochain level proof of Adem relations in Part IV. But explicit formulas are not obtained until Section 17 of Part III, where we give our version of results of Berger and Fresse from [3], [4].\\ 

Section 8 deals with the construction of natural transformations between functors.  We use  our  explicit contractions for chains on simplicies, chains on products of simplicies, and tensor products of chains on simplices, to define functorial chain maps using minimal contractible carriers.  Of course this is just a reformulation of the classical acyclic model method, but made more explicit functorially at the chain level.  We give two uniqueness theorems in the functorial context that extend the uniqueness theorems from Section 6.  As two important examples, we discuss in  detail the functorial Alexander-Whitney map $AW \colon N_*(X \times Y) \to N_*(X) \otimes N_*(Y)$ and the functorial Eilenberg-Zilber map $EZ \colon N_*(X) \otimes N_*(Y) \to N_*(X \times Y)$ when $X$ and $Y$ are arbitrary simplicial sets.  Our methods easily imply that both $AW$ and $EZ$ are associative, $EZ$ is commutative, and $AW \circ EZ = Id$, even without knowing the formulas.  Our method also produces a canonical chain homotopy between $EZ \circ AW$ and $Id$.   The details of both the Alexander-Whitney map and the Eilenberg-Zilber map are used in fundamental ways in Parts II and III in our study of the Barratt-Eccles and surjection operads.\\

Finally in Section 9 we discuss a standard recursive procedure method for constructing (equivariant) chain homotopes between suitable  pairs of (equivariant) chain maps $B_* \rightrightarrows C_*$ We also review the method from [7] of using joins of chains on MacLane models to define equivariant chain homotopies between pairs of maps $B_* \rightrightarrows N_*(EG)$ for certain complexes $B_*$, and we discuss  how join homotopies are related to standard procedure recursive homotopies. We discuss in detail some examples that will link results in Part I to results on cohomology operations in Part IV.  Specifically, the examples involve  relations between the minimal model $M_*$ for the cyclic group $C_p$ of prime order $p$ and the MacLane models $N_*(EC_p)$ and $N_*(E\Sigma_p)$ for the cyclic group and the symmetric group.  The applications in Part IV will be to a cochain level proof of the Cartan formula for Steenrod operations, extending the arguments of Medina-Mardones [22] in the $p = 2$ case, and to an analysis of which cycles in $N_*(EC_p)$ map to explicit boundaries in $N_*(E\Sigma_p)$, and therefore determine zero cohomology operations.\\

Some of the examples and remarks in Part I, especially toward the end, are somewhat complicated, but worth following.  We chose examples for their historical context, to illustrate the ubiquity of the contraction procedure, and for later use in Parts II, III, and IV in our study of some operads and in our cochain development of Steenrod operations.

\subsection{Preview of Part II}
In Section 10 we make a few introductory remarks about common features of three versions of surjection chain complexes $\cS_*(n)$ that underlie the surjection operad.  These complexes are acyclic free $\FF[\Sigma_n]$ resolutions of $\FF$ with the trivial group action.\\

Sections 11 and 12 develop details of the version of the surjection complex that we call $\cS_*^{aj}(n)$ because it appears in the paper of Adamaszek and Jones [1]. The rough geometric  idea is that the ordinary normalized relative simplicial singular chain complex of a simplex modulo its boundary, shifted down in degrees,  can be viewed as an acyclic free $\FF[\Sigma_n]$ resolution $\widetilde{\cS}_*(n)$ of $\widetilde{\FF}$,  where $\widetilde{\FF}$ is a twisted module structure on $\FF$.  Generators of $\widetilde{\cS}_k(n)$ are  simplicial maps $\Delta^{n+k} \to \Delta^n$ that are surjectve and do not map two consecutive vertices in the domain to the same vertex in the range. A simple untwisting construction then produces the augmented surjection complex $\cS_*^{aj}(n)$, which is an acyclic free resolution of $\FF$.  The boundary operator is simpler and more geometrically motivated than that of other surjection complexes.  It is immediate from basic topology that our complex is contractible, but somewhat challenging to produce a contraction with $h^2 = 0$.\\

Section 13 develops properties of the surjection complex $\cS_*^{bf}(n)$ studied by Berger and Fresse [4], [5].  Their boundary operator $d$ is not so easy to motivate, and it is not completely trivial to even see why $d^2 = 0$.  Also, the Berger-Fresse complex was known to be contractible but we go beyond that and produce an explicit contraction, closely related to but somewhat simpler than our contraction of $\cS_*^{aj}(n)$.  \\

Section 14 develops properties of the surjection complex $\cS_*^{ms}(n)$ studied by McClure and Smith [19].  We focus from the outset on  clear geometric motivation for both the boundary operator and the $\FF[\Sigma_n]$ action, although this was certainly implicit in their original work.  We show that $\cS_*^{ms}(n)$ is in a very natural way the  chain complex associated to a geometric cell complex whose open cells are interiors of prisms.  A contraction is given by exactly the same formula as the contraction of $\cS_*^{bf}(n)$.\\

In Section 15 we establish isomorphisms between the three surjection complexes that preserve all the structure.  We find  all three surjection complexes interesting.   Also, it is sometimes more natural, and easier,  to first express and prove a result using a specific one of the surjection complexes.\\  

 The Sections 11-15 provide some details about the three surjection complexes that are not found in the original references.  In particular, we mention the contractions with $h^2 = 0$, the isomorphisms between the complexes commuting with contractions, and additional geometric motivation for the boundary operators and the symmetric group actions.  We found all these things  deserving of a simple unified  treatment.  But readers need not get bogged down with every detail, a light reading suffices for continuing with the rest of Part II and then Part III.\\
 
 Section 16 is one of the longest  sections of the paper.  We develop and study  in great detail the equivariant chain maps $N_*(E\Sigma_n) \leftrightarrows \cS_*^{aj}(n),  \cS_*^{bf}(n), \cS_*^{ms}(n)$ that arise from  the standard procedure constructions, using  our bases of the domains and contractions of the ranges of various complexes. We prove that our maps for the Berger-Fresse complex are the same as the maps they found and studied in [3], [4] and [5], but we provide additional motivation and insight using the explicit contraction procedure ideas.  In particular, we give independent self-contained proofs that the Berger-Fresse maps {\it are} equivariant chain maps. In the case of the newer surjection complex $\cS_*^{aj}(n)$, the geometric viewpoint that generators are given by actual simplical maps between simplices, and the resulting identification of the inverse image of the base barycenter with a geometric prism, allows us to motivate clearly the equivariant chain maps $\cS_*^{aj}(n) \to \cS_*^{ms}(n) \to N_*(E\Sigma_n)$ as expressing facts about prisms and the Eilenberg-Zilber map that triangulates prisms.  Berger and Fresse also described their map $\cS_*^{bf}(n) \to N_*(E\Sigma_n)$  in terms of prisms, which is where we learned it, but their discussion focused somewhat less on the geometry of prisms.  That said, at the end of the day the use of geometric simplices, simplicial maps, and triangulations of prisms is a mental crutch, everything  is  ultimately just algebra and combinatorics.  \\

In the last very technical (asterisked) subsection of Section 16 we treat in detail some interesting additional facts about the surjection complexes $\cS_*(n)$ and their relation with the MacLane complexes $N_*(E \Sigma_n)$ that were  briefly hinted at by Berger and Fresse in [5]. This  technical subsection is not needed to continue to Part III.
 
\subsection {Preview of Part III}

In Section 17, which is also long,  the equivariant  functorial map of Berger and Fresse $\Phi \colon \cS_*^{bf}(n) \otimes N_*(X) \to N_*(X)^{\otimes n}$ is proved to coincide with our standard procedure functorial construction.  This is a rather difficult result, although the argument is formally similar to the argument in Part I that the classic functorial $EZ$ map is produced by our standard procedure. It is also somewhat tricky to prove that the standard procedure functorial map  $N_*(E\Sigma_n) \otimes N_*(X) \to N_*(X)^{\otimes n}$ factors through $\Phi$.  These maps were studied by Berger and Fresse by seemingly ad hoc methods, partly based on a reorganization of the work of McClure and Smith.    We also describe as compositions the standard procedure maps $A_* \otimes N_*(X) \to N_*(X)^{\otimes n}$ for $A_* = \cS_*^{ms}(n), \cS_*^{aj}(n), N_*(EC_n), M_*$. All these maps give rise by  duality to operadic multilinear actions of the various complexes on tensors of cochains, which is the starting point for our study of Steenrod operations in Part IV.  \\

Towards the end of Section 17 we compute the Berger-Fresse map in the specific degrees $M_{q(p-1)} \otimes N_q(\Delta^q) \to (N_*(\Delta^q)^{\otimes p})_{qp}$, namely  $\Phi(y_{q(p-1)} \otimes \Delta^q) = c_{q,p}(\Delta^q) ^{\otimes p}$ for specific constants $c_{q,p}$.  In the classic reference [30], determination of these constants is also carried out by a long chain level computation.  It is one of the few direct chain level computations in that book.  The constants are needed in order to define the Steenrod cyclic reduced power oerations $P^j$ on cocycles in terms of the Berger-Fresse map,  to prove  $P^0 = Id$, and to prove the Cartan formula for the $P^j$.\\

It seems appropriate here to recall a little history, although we do not know it well.  The basic functorial map $\cS_*^{ms}(n) \otimes N_*(X) \to N_*(X)^{\otimes n}$ was first studied by McClure and Smith using their surjection complex. They also found the operad interpretation, which was a major advance.  Their method was to establish their complexes as a suboperad of the Eilenberg-Zilber operad of natural transformations of functors $\cZ(n) = HOM_{func}(N_*( - ), N_*( - )^{\otimes n})$. They did point out that there was earlier separate work in the 1990's and before by D. Benson, R.J. Milgram, and E. Getzler, using multilinear cochain formulas to generalize Steenrod's original two variable $\cup_i$ products, with the goal of studying odd prime Steenrod operations. As mentioned previously, Justin R. Smith also studied an operad action of the $N_*(E\Sigma_n)$ on tensors of cochains, although not with explicit formulas.  His work seems to use some of the same contractions that we use to recursively construct chain maps.  In particular, he realized that the operad structure maps on the collection of complexes $N_*(E\Sigma_n)$, including the verification of the operad  axioms and the operad map to the Eilenberg-Zilber operad, could be formulated recursively using explicit contractions and a uniqueness theorem for chain maps with certain properties. But his primary interest was elsewhere and we do not know if he wrote down full details.\\

A question, more philosophical than mathematical,  might be how many linear natural transformation chain operations $HOM_{func}(N_*( - ), N_*( - )^{\otimes n})$, or natural cochain operations $HOM_{func}(N^*( - )^{\otimes n}, N^*( - ))$, do we really need?  A reasonable answer might be {\it exactly} the suboperad given by  the surjection complexes $\cS_*(n)$.  This suboperad seems to contain all that is needed to fully and faithfully capture a lot of homotopy theory.  Also, the chain and cochain operations defined by the surjection complexes are very natural direct extensions of Steenrod's seminal discovery of the explicit two-variable $\cup_i$ operations and the Steenrod Squares, which, it could be argued, launched homotopy theory into the modern era.\\  

In Section 18, to prepare for our treatment of the Barratt-Eccles and surjection operads, we review some basic definitions concerning operads and we develop carefully the permutation group operad in the category of sets, with components the symmetric groups $\Sigma_n$.  We also review the $End$ and $CoEnd$ operads for chain complexes and the Eilenberg-Zilber operad   $\cZ$ of natural transformation chain maps  $HOM_{func}(N_*( - ), N_*( - )^{\otimes n})$.  We found some of the details, which are taken for granted in nearly all papers on operads, trickier than we expected.  These details are needed in our non-standard treatment of the Barratt-Eccles  and surjection operads.\\

Section 19 contains our treatment  of the Barratt-Eccles operad with components $N_*(E\Sigma_n)$. Our approach is to use a `twisted' variant of our equivariant standard contraction procedure, discussed briefly in Subsection 1.3, to define {\it candidates} for operad structure maps $$\cO_B \colon B_*(r) \otimes B_*(s_1) \otimes \ldots \otimes B_*(s_r) \to B_*(s),\  s = s_1 + \ldots + s_r,$$  for  certain  complexes $B_*(n)$. The candidates always satisfy the unit and equivariance axioms for operads.  A closer look is needed to establish the associativity axiom.  For that purpose, we extend previous uniqueness theorems to the twisted equivariant case. This provides a criterion involving images of contractions, mentioned in Subsection 1.3,  for the strict commutativity of various diagrams involving the $B_*(n)$.  The criterion is easy to check for the $N_*(E\Sigma_n)$. We also reconcile our treatment of the Barratt-Eccles operad with the treatment expressed in terms of  symmetric monoidal functors, found for example in [3] and many other places.\\

In Section 20, which is certainly the longest and most difficult section of the paper,   the complexes $\cS_*^{bf}(n)$ are proved to form an operad, again using our contraction based constructions and our twisted candidates for operad structure maps.  It is somewhat trickier to use our twisted uniqueness theorem to establish the criterion for associativity than it is  for the Barratt-Eccles operad.   Using essentially the same uniqueness theorem, we recover the result of Berger and Fresse that the surjection operad is a quotient of the Barratt-Eccles operad.  We also prove using the uniqueness theorem  that  the inclusion map studied in Section 17 from the surjection operad $\cS$ to the Eilenberg-Zilber operad $ \cZ$  is an operad morphism.\\

Many details in Section 17 through Section 20 of our paper are not easy.  But we believe the details in the original papers by Berger and Fresse [4], [5]  and McClure and Smith [19] are also not easy, in different ways than ours.  One might say ``pick your poison".  Replication of results should have a place in mathematics.  One can read and check the details in the original papers or read and check our details.  Or both.  All these chain maps, operad structure maps, and operad morphisms seemed originally  to depend on brute force analysis of complicated formulas, which was not always included.\footnote{The papers  of Adamaszec-Jones [1] and  of McClure-Smith [21] are more conceptual.} We believe we provide an alternate self-contained context for many of these important results.  Every single one of these maps is an example of our standard contraction based procedure for constructing equivariant chain maps.  One could argue that the actual formulas are secondary, although it is certainly good to know them.  In fact, we combine partial information about the formulas of others with our recursive method.  Many of the properties of the maps drop out relatively painlessly from general results about our contraction procedure that we develop in Part I.\\

\subsection{Common Features of Some Important Chain Maps}  Many of the chain maps $\phi \colon B_* \to C_*$ that we study  in Parts I, II, and III are rather simple. This includes the $AW$  maps, multidiagonal maps, tensor products of maps, and chain maps extending functions between groups. Sometimes we look at familiar maps that are known to be chain maps, and we observe directly that they coincide with our standard procedure maps, either by means of a little computation and an easy induction or as an application of a uniqueness theorem.\\

Other times formulas for maps are `known' but are much less familiar, seem rather ad hoc, and it is sometimes  rather difficult to check that the formulas even define  chain maps.  This is the case for many of  the maps related to the operads $\cE, \cS$ and $\cZ$.   But we can often prove by induction that these maps  coincide with our standard procedure chain maps, without too much difficult computation.  So this simultaneously proves certain complicated maps are chain maps, and establishes a natural conceptual framework for the origin of these maps.\\

In important cases this second, more dramatic, situation arises as follows. The standard procedure map applied to generators will be sums of terms of a certain form, along with $\pm$ signs.  That is, $\phi(x) = \sum \pm \cT x$.  In  standard equivariant situations the parameter set for the operators $\cT$ is the same for $x$ and $gx$.  The terms $\cT x$ are often easily described, and satisfy a crude equivariance property $g (\cT x) = \pm \cT(gx).$   Once one has a  candidate formula for the $\cT x$, an inductive proof that such a formula for the standard procedure map $\phi$ is correct can be given, without knowing the signs.\\

Namely, on basis generators $b$ of the domain the standard procedure map is given  recursively by  the formula $\phi(b) = h \phi (db)$, where $h$ is the contraction of the range.  Now $d$ and $h$ may each involve many  summands. But surprisingly often most of these are seen to contribute 0.  In fact, in important cases $h \phi(db)$ reduces to a sum of $h$ applied to a small number of boundary terms of $b$, sometimes just two. Moreover, on those boundary terms only a small number of summands of $h$ are non-zero, sometimes just one, and their evaluation is easy to understand.  The reason so many summands of $h \phi (db)$ are 0 is that $h^2 = 0$, so any terms in $\phi(db)$ that are already in $Im(h)$ will contribute 0.\\

Thus, if by some stroke of insight or theft one is given the candidates $\cT x$ in all degrees, an inductive proof that such a formula for $\phi(b)$ is correct can be fairly easy.  One just needs to match the formula for $\phi(b)$ with a sum of $h$ applied to a few similar looking terms from one degree lower.  Then for other generators $ x = gb$ $$\phi(x) = g\phi(b) = \sum \pm g(\cT b) = \sum \pm \pm \cT(gb)$$ also has the desired form.  As three examples we mention $$EZ \colon N_*(X) \otimes N_*(Y) \to N_*(X \times Y)$$
$$\Phi^{(n)} \colon S_*(n) \otimes N_*(X) \to N_*(X)^{\otimes n}$$
$$\cO_{\cS} \colon S_*(r) \otimes S_*(s_1) \otimes \ldots \otimes S_*(s_r) \to S_*(s).$$ 
The third example, a twisted equivariant standard procedure map, is somewhat more complicated, but the same basic strategy applies. \\

The formula $\phi(b) = h \phi (db)$ for basis elements always recursively forces the signs. The contractions $h$ often involve no signs and signs in $db$ terms are generally standard signs occurring in boundary formulas.  Thus if one also  has candidate formulas for the $\phi$ signs, one has the possibility of showing relatively painlessly by induction that these signs behave correctly in the verification of  $\phi(b) = h \phi (db)$  and $\phi(gx) = g\phi(x)$. This completes an inductive proof that the asserted full formula for $\phi(x)$ is the equivariant  standard procedure chain map, correct for all $x$.\\

It is interesting to compare our arguments with a direct computation that $d \phi = \phi d$. Direct computations can be awkward and seem to rely on large amounts of fortuitous cancellation of terms with opposite signs.  Our arguments require no such cancellation.\\  

It is true that if one checks or repeats or simply accepts previous arguments that $d \phi = \phi d$, and that $\phi$ is equivariant in a group action case, then the statement that $\phi$ coincides with the standard procedure map can be quite easy.  One just observes that for basis elements $b$ of the domain the summands $\cT b$ of $\phi(b)$ are in the image of the contraction of the range and therefore the basic uniqueness theorems apply.  This by itself is interesting.  But we believe our alternate proofs add value by simplifying some arguments, revealing hidden structure,  and putting results in a new context.\\

 Other issues involving complicated chain maps arise in the operad context.  One needs to understand why certain diagrams strictly commute.  These commutative diagrams correspond to operad axioms for a collection of maps, or to verification that certain collections of  chain maps define operad morphisms.  For this our uniqueness theorems seem very useful, since the compositions in question often are seen to be compositions of standard procedure chain maps.  Therefore  criteria that imply compositions of standard procedure chain maps are themselves standard procedure chain maps can be used to prove two different compositions around a diagram must coincide.\\
 
Of course the above few paragraphs are rather vague.  But having read them should make it easier to be motivated and follow details in specific examples.\\

\newpage

\section{Chain and Cochain Complexes}

First we make some comments about our conventions for  functions and permutations.  We write functions on the left of their arguments, $f(x)$.  Thus $fg$ means apply $g$ first, then $f$, so $f(g(x))$.  Permutations are functions, so we compose permutations $\sigma \tau$ by performing $\tau$ first, then $\sigma$, as in $\{1, \ldots, n\} \xrightarrow{\tau} \{1, \ldots, n\}  \xrightarrow{\sigma} \{1, \ldots, n\}.$ We often write permutations in $\Sigma_n$  as a sequence of values $(\sigma(1) ,\ldots, \sigma(n)) = (\sigma_1, \ldots, \sigma_n)$, although at times we use the  disjoint cycle notation without commas to name permutations. For example, $(12)(345) = (2,1,4,5,3)$.

\subsection{Tensor and $HOM$ Complexes}
We will be making extensive use of constructions such as tensor products of chain complexes, complexes of linear maps between chain complexes, including dual cochain complexes, and various sorts of maps between such constructions.  Among other things, there are sign conventions needed in all these constructions.   Historically, various sign conventions have been used in topology.  We feel that there is a  consensus of preferred conventions and we will clarify in this section  the conventions we will use throughout the paper. \\

Since our interest is topology we do not strive for great algebraic generality.  We work  with chain complexes $C_*$ over a commutative ground ring, which will usually be the integers  or a field.  In the presence of group actions, complexes will be  differential graded modules over a group ring.  In particular, the differential is equivariant. But the ground ring doesn't change.  We will be taking tensor products and modules of homomorphisms just over the original ground ring.\\

Our starting points are usually  positively graded\footnote{By a slight abuse of language we write `positively graded' to mean 0 in negative degrees and 'negatively graded' to mean 0 in positive degrees.} complexes that are  free over the ground ring in each degree.  However, then  dual chain complexes are negatively graded chain complexes and occasionally we  work with more general  complexes of homomorphisms that can have components in all integral degrees.  Neither of these will be free unless the ground ring is a field or additional finiteness assumptions hold \\

{\bf Tensor Complexes.} Given {\it arbitrary}  chain complexes $C_*$ and $D_*$ over the ground ring we have the graded tensor product module $C_* \otimes D_*$, which in degree $n$ is given by $\oplus_{i+j = n}C_i \otimes C_j$.  It  is free if $C_*$ and $D_*$ are free.  The boundary operator in a tensor complex  is
$$d(a \otimes b) = da \otimes b + (-1)^{|a|} a \otimes db.$$
where $|a|$ denotes the degree of $a$. It is a simple computation that $d^2 = 0$.\\

Suppose the ground ring is $\FF$, which can be any commutative ring, not necessarily a field. We regard $\FF$ as a chain complex concentrated in degree 0. The differential $d$ just defined on a tensor product  is the only choice that is natural with respect to chain maps and respects the obvious identifications $\FF\otimes C_* = C_* = C_* \otimes \FF$.  So $\FF$ is a unit object in the category of chain complexes over $\FF$. Naturality means the obvious linear map induced between tensor products by maps between the factors is in fact a chain map.  That is, the tensor product is a functor of two variable $\cC \times \cC \to \cC$, where $\cC$ is the category of chain complexes over $\FF$.\\

Geometrically, the  boundary  is related to a point being a unit object for spaces, $* \times X = X = X \times *$.  Also, for manifolds, one wants orientations to satisfy $\partial (M \times N) = \partial M \times N \cup (-1)^{dim(M)} M \times \partial N$, where boundaries are oriented by the ``outward  normal first" convention.  Products are oriented locally  by following an orientation tangent basis of the first factor by an orientation tangent basis of the second factor.\\

{\bf HOM Complexes.} There is  a complex of homomorphisms, $HOM(C_*, D_*)$, which in degree $n$ is given by $\prod _k Hom_{\FF}(C_k, D_{k+n}).$ So $HOM(C_*, D_*)  = Hom_{\cG\cR}(C_*, D_*)$ as graded modules, where $\cG\cR$ is the category of graded $\FF$ modules and graded module homomorphims. These complexes are more complicated because they can be non-zero in all degrees $n \in \ZZ$, even if $C_*$ and $D_*$ are positively graded.\\

To get a chain complex we will impose a differential on $HOM(C_*, D_*)$. Included in the discussion are dual cochain complexes $C^* = HOM(C_*, \FF)$, which are negatively graded when the $C_*$ are positively graded.\footnote{All complexes are chain complexes, that is, all differentials have degree $-1$. But we often refer to elements of cochain complexes as cochains, cocycles, or coboundaries. And we refer to the differential as the coboundary operator.}\\

As mentioned previously, we write functions on the left of their arguments, with the consequential understanding about compositions.  The  boundary operator in a $HOM$ complex   is $$du = d_D \circ u - (-1)^{|u|} u \circ d_C.$$  It is a simple computation that $d^2  = 0$\\

This is the only choice of a differential that makes evaluation of functions $HOM(C_*, D_*) \otimes C_* \to D_*$, $u \otimes x \mapsto u(x)$,  a chain map. Namely, since we must have $d (u \otimes x) = (du) \otimes x + (-1)^{|u|} u\otimes d_Cx$, evaluation being a chain map is equivalent to $ (du)(x) + (-1)^{|u|} u(d_C(x)) =  d_D (u(x)).$\\

As a special case, the coboundary operator in a cochain complex $C^* = HOM(C_*, \FF)$ is defined for $u \in C^*, \ x \in C_*$ by $$< du, x> = (-1)^{|x|} < u, dx >.$$

It is  an easy observation that  $HOM(C_*, D_*)$ is  a functor of both chain complex variables, contravariant in $C_*$ and covariant in $D_*$. Another easy observation is that cycles of degree 0,  $ u \in HOM_0(C_*, D_*)$ with  $0 = du = d_D\circ u - u \circ d_C $, are exactly the morphisms $Hom_{\cC} (C_*, D_*)$ in the category of chain complexes.  The degree 0 cycles also correspond to $Hom_{\cC}( \FF, HOM(C_*, D_*))$.\\

\subsection {Iterations of Tensor and HOM Constructions.} There are many  important chain maps between iterations of tensor and $HOM$ complexes.   We list some of them  here.  The first  few can be recognized as ingredients in establishing the category of chain complexes as a closed symmetric monoidal category, which we will discuss further below. The last few involve dual chain and cochain complexes and are relevant in topology when applied to chain and cochain complexes associated to simplicial sets or other topological space categories.\\

Each chain map in the statements is to be interpreted as a natural transformation of functors.  In all these statements the serious point is that the signs in the boundary formulas work out.  We first point out that in the category $\cG\cR$ of graded modules  there is a very elementary adjoint isomorphism $Hom_{\cG\cR}(B_* \otimes C_*, D_*) \simeq Hom_{\cG\cR}(B_*, Hom_{\cG\cR}(C_*, D_*))$, which is an easy extension of the same isomorphism for modules. The first bullet point below states that this $\cG\cR$ adjoint isomorphism is an isomorphism in the category of chain complexes. 

\begin{prop}\label{3.1}
\end{prop}
\begin{itemize}

\item{\bf 1}
There is a chain map isomorphism $$Ad \colon HOM(B_* \otimes C_*, D_*) \simeq HOM(B_*, HOM(C_*, D_*))$$   $$with\  u \leftrightarrow \mu \ \ where\  \mu(b) (c) = u(b \otimes c).$$

In particular, cycles in degree 0 coincide, which says the restricted map $Hom_{\cC}(B_* \otimes C_*, D_*) \simeq Hom_{\cC}(B_*, HOM(C_*, D_*))$ is an isomorphism of modules.  Thus $HOM( \bullet, D_*)$ is characterized as an adjoint of the functor $ B_* \otimes  \bullet$ in the category $\cC$ of chain complexes over the ground ring.

\item{\bf2:}
Tensor complexes satisfy the following associativity property:
$$ (B_* \otimes C_*) \otimes D_*  \simeq B_* \otimes (C_* \otimes D_* )\ \ \ (b \otimes c) \otimes d \leftrightarrow b \otimes (c \otimes d)$$  is  a   chain isomorphism.\\

\item{\bf3:}
Tensor complexes satisfy the following commutativity property:
$$ C_* \otimes D_*  \xrightarrow{\tau} D_* \otimes C_* \ \ \ c \otimes d \leftrightarrow \tau(c \otimes d) = (-1)^{|c||d|} d \otimes c$$  is  a   chain isomorphism.

\item{\bf4:}
The evaluation operation is a chain map, $$HOM(C_*, D_*) \otimes C_* \to D_* \ \ \ where\  u \otimes c \mapsto u(c).$$

Function composition $$HOM(C_*, D_*) \otimes HOM(B_*, C_*) \to HOM(B_*, D_*)\ \ \ u \otimes v \mapsto u \circ v$$ is a chain map. That is, $d(u \otimes v) \mapsto du \circ v + (-1)^{|u|} u \circ dv$.\\

Function composition is associative. Evaluation is the special case of composition with $B_* = \FF$.

\item{\bf5:}
If $f \in Hom_\cC(A_*, C_*)$ and $g \in Hom_\cC(B_*, D_*)$ are chain maps, then by functoriality of $\otimes$ the map $(f \otimes g)(a \otimes b) = f(a) \otimes g(b)$ is a chain map $A_* \otimes B_* \to C_* \otimes D_*$.\\

More generally, the following map is a chain map
$$HOM(A_*, C_*) \otimes HOM(B_*, D_*) \xrightarrow{\underline{\otimes}} HOM(A_* \otimes B_*, C_*\otimes D_*)\ \  f\otimes g \mapsto f \underline \otimes g $$ $$ where\ \ (f  \underline\otimes g)(a \otimes b)= (-1)^{|a||g|}f(a) \otimes g(b).$$ 

The product $\underline \otimes$ is associative for triple tensor products of $HOM$ complexes.\\

The map $f \underline{\otimes} g$ is also the composition $$(f \underline{\otimes} Id_D) \circ (Id_A \underline{\otimes} g) \colon A_* \otimes B_* \to A_* \otimes D_* \to C_* \otimes D_*.$$

\item{\bf6:}
The tensor operation $\underline \otimes$ and the composition operation $\circ$ for $HOM$ complexes are related as follows.  Suppose $\alpha_i \in HOM(A_{i*}, B_{i_*}), \beta_i \in HOM(B_{i*}, C_{i*}),\ i = 1,2.$  Then 
$$\beta_1 \circ \alpha_1\ \underline \otimes\ \beta_2 \circ \alpha_2 = (-1)^{|\alpha_1| |\beta_2| } \beta_1 \underline \otimes \beta_2\ \circ\ \alpha_1 \underline \otimes \alpha_2 \in HOM(A_{1*} \otimes A_{2*}, C_{1*} \otimes C_{2*}).$$ 
By induction, if $\alpha_i \in HOM(A_{i*}, B_{i*})$ and $\beta_i \in HOM(B_{i*}, C_{i*})$ for $1 \leq i \leq r$ then in $HOM( \otimes_iA_{i*}, \otimes_i C_{i*})$ one has 
$$\beta_1 \circ \alpha_1\ \underline \otimes\ \ldots \underline \otimes\ \beta_r \circ \alpha_r =  (-1)^k \beta_1 \underline \otimes \ldots \underline \otimes \beta_r\ \circ\ \alpha_1 \underline \otimes \ldots \underline \otimes \alpha_r,$$ 
where $(-1)^k $ is the Koszul sign that shuffles $\beta$'s across $\alpha$'s.

\item{\bf7:}
There is a chain map that is an isomorphism if  $C_*$ is finitely generated and free in each degree  $$B_* \otimes HOM(C_*, D_*) \to HOM(C_*, B_*\otimes D_*)\ where\  (b \otimes \gamma)(c) = b\ \otimes \gamma(c).$$

\item{\bf8:}
There are chain maps relating $HOM$ complexes and dual complexes.  Both of the maps below are chain maps that are isomorphisms if $C_*$ is finitely generated and free in each degree.
$$B_* \otimes C^* \to HOM(C_*, B_*), \ where \ (b \otimes \gamma) (c) = b\ \gamma(c)$$ and
$$C_* \to HOM(C^*, \FF) = C^{**}: \ c \mapsto c^*,\ where\ c^*(\gamma) = (-1)^{|c|} \gamma(c)$$

\item{\bf9:}
 The two maps below are chain maps, that are isomorphisms with finite free assumptions.  $$A^* \otimes B^* \to (A_* \otimes B_*)^*\ where\ <\alpha \otimes \beta, a \otimes b> = (-1)^{|a||\beta|}\alpha(a) \beta(b).$$ and 
$$HOM(C_*, D_*) \to HOM(D^*, C^*): \   u \mapsto u^*,\   u^*(\delta) (c) =   (-1)^{|u||\delta|} \delta(u(c)).$$ 
The second  map is related to composition by $(v \circ u)^* = (-1)^{|u| |v|} u^* \circ v^*$.

\end{itemize}

\begin{proof} One method of proof that all  the maps above are chain maps is to just unravel  the definitions and use the boundary formulas, keeping careful track of signs.\footnote{But it is very easy to make mistakes!} But another method of proof is to examine adjoints. By $\bullet 1$, a map is a chain map if its adjoint is a chain map, since chain complex isomorphisms induce bijections of 0 cycles.\\

$\bullet 2$ and $\bullet 3$. Both  are trivial.  By the Yoneda Lemma,  associativity of $\otimes$ and the second statement in $\bullet 1$ implies the first statement in $\bullet 1$.  Because associativity implies both $HOM$ complexes in the first statement of $\bullet 1$ represent the same functor $A_* \mapsto Hom_{\cC}(A_* \otimes B_* \otimes C_* , D_*)$.  \\

$\bullet 4$. The first statement  was already observed  when we defined the differential in $HOM(C_*, D_*).$\\

The adjoint of function composition is the composition of chain maps $$HOM(C_*, D_*) \otimes HOM(B_*, C_*) \otimes B_* \xrightarrow{Id \otimes ev}  HOM(C_*, D_*) \otimes C_* \xrightarrow{ev} D_*$$ which proves the second statement in $\bullet 4$. The remaining statements  are trivial.\footnote{If $C_* = D_*$, the differential $du$ is the bracket $[d, u]$ in the graded world.  The formula for $d(v \circ u)$ shows that  $HOM(C_*, C_*)$ becomes a differential graded algebra.  That is, the composition product $v \otimes u \mapsto v \circ u$ is a chain map $HOM \otimes HOM \to HOM$.  Also, $C_*$ becomes a differential graded left module over $HOM(C_*, C_*)$.}\\ 

$\bullet 5$. The first statement  follows from the fact that the $\cG\cR$ adjoint of the $ \underline \otimes $ map is the composition, which is a chain map, $$HOM(A_*, C_*) \otimes HOM(B_*, D_*) \otimes A_* \otimes B_* \to$$ $$  HOM(A_*, C_*) \otimes A_*  \otimes HOM(B_*, D_*)  \otimes B_* \to C_* \otimes D_*,$$  obtained by permuting the  two middle factors then tensoring two degree 0 evaluation chain maps.\\

The second statement, associativity of $\underline{\otimes}$, is easy to prove directly, but can also be proved by looking at adjoints and using associativity of composition.\\

$\bullet 6$. This result will be important for establishing properties of $End$ and $CoEnd$ operads in Section 18.  It seems the simplest proof for chain complexes is to just verify the adjoint statement by choosing $a_i \in A_i$ and applying both sides of the desired equation to $a_1 \otimes a_2$, keeping track of signs.  The key point is that the definition of the adjoint of $\underline{\otimes}$ in $\bullet 5$ is $(f  \underline\otimes g)(a \otimes b)= (-1)^{|a||g|}f(a) \otimes g(b)$.\\

One can replace this last formula containing a sign and two function evaluations by the more abstract formula for the adjoint of $\underline{\otimes}$ in $\bullet 5$, which reads $(f \otimes  g \otimes a \otimes b) \mapsto (ev \otimes ev) (f \otimes \tau(g \otimes a) \otimes b)$, where,  in the notation of $\bullet 5$, $$(f \otimes \tau(g \otimes a) \otimes b) \in HOM(A_*, C_*) \otimes A_* \otimes HOM(B_*, D_*) \otimes B_*.$$ Such an abstract approach makes sense in any closed symmetric monoidal category. Behavior of Koszul signs is replaced by the fact that  two morphisms between  tensor products with multiple factors  must coincide, if both are compositions of sequences of associativities and basic $\tau$ permutations.\\

$\bullet 7$. The $\cG\cR$ adjoint is $Id \otimes ev \colon B_* \otimes HOM(C_*, D_*) \otimes C_* \to B_* \otimes D_*$, which is an ordinary tensor product of chain maps.\\

 $\bullet  8$. The first map is the special case of $\bullet 7$ with $D_* = \FF$. The adjoint of the second map is $ev \ \tau \colon C_* \otimes HOM(C_*, \FF) \to HOM(C_*, \FF) \otimes C_* \to \FF$.\\
 
 $\bullet  9$. The first map  is the special case of  $\bullet 5$ with $C_* = D_* = \FF$. The adjoint of the second map is  $ \circ\ \tau \colon HOM(C_*, D_*) \otimes HOM(D_*, \FF) \to HOM(D_*, \FF) \otimes HOM(C_*, D_*) \to \FF.$ 
 \end{proof}

\begin{rem}\label{3.2}{*\bf Closed Symmetric Monoidal Categories*.} The collection of formulas in Proposition \ref{3.1} contain more than enough results to establish that  categories of chain complexes are closed symmetric monoidal categories, with $\otimes$ as the product and the $HOM$ complex adjoint to $\otimes$ as the internal hom functor. The isomorphism in $\bullet 2$ defines the associator.    The isomorphism $\tau$ in $\bullet 3$ that switches order of factors in $\otimes$ is the braiding. The formulas in $\bullet 2$ and $\bullet 3$ rather easily imply all ways of associating and braiding many factors in  tensor product diagrams give the same results.  This is a consequence of the fact that the Koszul sign in $\bullet 3$ can be viewed as the sign of a permutation in a big symmetric group, moving a block of objects across another block of objects. Then use the fact that the sign of a permutation is well-defind.\\

There are other symmetric monoidal category axioms involving a unit object, which for chain complexes is just the ground ring $\FF$ concentrated in degree 0.  Properties of the unit object  are completely trivial for chain complexes, so are not included in the proposition. Thus, it is rather easy to see that chain complexes form a symmetric monoidal category.  The adjoint functor statements of $\bullet 1$ then guarantee that chain complexes form a closed symmetric monoidal category.\\

In some sense, the signs involved in the various parts of Proposition \ref{3.1} are not so important, but rather it is the commutativity of various diagrams that is important.  However, some attention must be paid to basic signs to get a toehold on more complicated maps. It will be  useful in the construction of End and CoEnd operads that we take up in Section 18 to identify exactly which properties of closed symmetric monoidal categories are needed,  and that we have established those properties for chain complexes in Proposition 3.1. For $End$ and $CoEnd$ operads, the order of composition in $\bullet 4$  is quite important, as is the $\underline{\otimes}$ construction in $\bullet 5$ and the formula in $\bullet 6$.  $\qed$

 \end{rem}

{\bf *Comments on Cup Products.*} In a topological situation, the maps in bullet point 9 of Proposition \ref{3.1} will become the pairing between tensors of cochains and tensors of chains used to define cup products.  The sign in the map of bullet point 9 is required,  once we agree on the sign in the differential  in general cochain complexes, $<du, x> = (-1)^{|x|} <u, dx>$. \\

In topology we want the cochain algebras of spaces to be differential graded associative algebras. The sign convention for the coboundary  forces the cochain cup product formula in the topological case to include the ``Dold sign", in order that the coboundary is a derivation and that the cochains on a space form a $DGA$ algebra. Specifically, if $x = (x_0, x_1, \ldots, x_n)$ are vertices of an ordered  simplex and $\alpha$ and $\beta$ are cocycles whose degrees add to $-n$, the Dold sign convention for cup products using the $AW$ diagonal is the  definition\footnote{In the category of  simplicial sets a map $x \colon \Delta^n \to X$ is not determined by vertices. But the abuse of notation $x = (x_0, x_1, \ldots, x_n)$  makes it easy to refer to face operators.}  $$<\alpha \cup \beta, x> = \sum _{0
\leq i \leq n}(-1)^{i(n-i)} <\alpha, (x_0, \ldots,x_i)><\beta, (x_i, \ldots, x_n)>.$$

If the differential in a cochain complex $C^*$ is chosen to be the simple adjoint $\delta$, that is, $<\delta u, x> \ = \ <u, dx>$, then the cup product formula is  $$<\alpha \cup \beta, x>\  =\  \sum_{0 \leq i \leq n}  <\alpha, (x_0, \ldots,x_i)><\beta, (x_i, \ldots, x_n)>.$$

The conventions with no signs at all  might seem simpler, and is historically the way coboundaries and cup products of cochains were originally defined in topology.  But it is not the ``right" way to do things, because, among other reasons, one wants the coboundary to agree with the general choice for $HOM$ complexes.\\

We point out that there is an isomorphism of cochain complexes $(C^*, \delta) \simeq (C^*, d)$, defined by $\alpha \leftrightarrow (-1)^{|\alpha|(|\alpha|+1)/2} \alpha$.  In the topological situation, this isomorphism becomes an isomorphism between the two versions of the cup product $DGA$ cochain algebras of spaces with a fixed diagonal. $\qed$

\subsection{Group Actions on Tensor and $HOM$ Complexes}
In this subsection we discuss some aspects of group actions on chain complexes.  If a group $G$ acts on the left as chain maps of a chain complex $C_*$ and a group $H$ acts on the left on $D_*$  then $G \times H$ acts  on $C_* \otimes D_*$ by $(g, h) (a \otimes b) = ga \otimes hb$.  If $H = G$ then there is the associated diagonal $G$ action on $C_* \otimes D_*$, namely $g (a \otimes b) = ga \otimes gb$.\\

The group $G \times H$ also acts on the left of $HOM(C_*, D_*)$ by $((g, h) u) (c)  = h (u g^{-1}(c))$, which is the  composition of functions $h \circ u \circ g^{-1} \colon C_* \to C_* \to D_* \to D_*$.  
For example, with trivial group action on $\FF$ and a  left $G$ action on $C_*$, the left action on the  dual complex $HOM(C_*, \FF)$ is $(g\alpha)(x) = \alpha (g^{-1}(x))$, which is the composition of functions $g \alpha = \alpha \circ g^{-1} \colon C_* \to C_* \to \FF$.\\

Of course right group actions can always be converted to left group actions, and vice-versa,  by the definition $ g y = y g^{-1}$.   For example, the most natural way to view the action on a dual complex is a right action, $\alpha g^{-1} = \alpha \circ g^{-1}$, which converts to the left action $g \alpha = \alpha \circ g^{-1}$ defined above.\\
 
The natural action of $\Sigma_p$ on a tensor power $C_*^{\otimes p}$, as a group of chain complex automorphisms, is also a right action, $$(a_1  \otimes \ldots \otimes a_p)g = (-1)^k\  a_{g1} \otimes \ldots \otimes a_{gp},\ k = k(a,g).$$
To understand this, it is convenient to first view the data of a basic tensor $a_1 \otimes \ldots \otimes a_p$ as a function $a \colon \{1, \ldots, p\} \to C_*$.  Then $ag = (-1)^k\ (a \circ g)$ as functions. We need a sign because the automorphism of the tensor complex determined by the permutation $g$ must commute with the boundary operator.  The sign, from Proposition \ref{3.1}, is the Koszul sign that counts the parity of the number of odd degree pairs whose order is swapped by the permutation $g$ or $g^{-1}$.  That is, $(-1)^{k(a,g)} = \tau (g|_{|a_i| =  odd})$, where $\tau$ is the parity sign of the permutation $g$ interpreted as rearranging the order of the indicated  subset of the $\{ a_i\}$. The conversion of the right action on the tensor power to a left action is thus $g^{-1}(a_1 \otimes \ldots \otimes a_p) = (-1)^{k(a,g)}\ (a_{g1} \otimes \ldots \otimes  a_{gp})$. \\ 

Left action of a permutation $g^{-1}$ on tensor powers extends to chain maps $g^{-1} \colon (C_1)_* \otimes  \ldots \otimes (C_p)_* \to (C_{g1})_* \otimes  \ldots \otimes (C_{gp})_*$ between tensor products of different chain complexes.  One can also just view this as a left action of $\Sigma_p$ on a direct sum of complexes $\bigoplus_{h \in \Sigma_p}  (C_{h1})_* \otimes  \ldots \otimes (C_{hp})_*$.  But caution is required, $g^{-1}$ will map the indicated summands to $ (C_{hg1})_* \otimes  \ldots \otimes (C_{hgp})_*$ since permutations are applied to the {\it positions} of  tensor factors,  not the subscripts of the tensor factors. For example $(12) (a_{h1} \otimes a_{h2} \otimes \ldots) =  (a_{h2} \otimes a_{h1} \otimes \ldots)$, the first and second factors are switched.  In fact, already for powers of a fixed $C_*$ this observation applies to the fixed dimension components of $C_*^{\otimes p}$, which are direct sums of tensor factors of terms of different degrees.  \\

Since any permutation is a composition of transpositions of adjacent indices, the properties of symmetric group actions on tensor products can be viewed as iterations of  associativities and transpositions from bullet points 2 and 3 of Proposition \ref{3.1}. The inverse of one transposition is itself, but to obtain left actions one must pay attention to the order of compositions of transpositions. For example, with cycle notation and up to factors $\pm 1$,  $$(12)(23) (a_1 \otimes a_2 \otimes a_3 \otimes \ldots ) = (12) (a_1\otimes  a_3 \otimes a_2 \otimes \ldots ) =    (a_3\otimes  a_1 \otimes a_2 \otimes \ldots ).$$ Then the left action/inverse business works out, $g = (12)(23) = (231)$ and $g^{-1} = (132)$.  Koszul signs in these iterated permutation actions on general tensor products also take care of themselves.\\

If $B_*$ is an $H$-complex, $C_*$ is a $G$-complex, and $\iota \colon H \to G$ is a homomorphism, a chain map $\phi \colon B_* \to C_*$ is { \it $\iota$-equivariant} (or just {\it equivariant} if the groups and $\iota$ are understood) if $\phi (hx) = \iota(h) \phi(x)$ for all $h \in H, x \in B_*$. All the statements in the bullet points of Proposition \ref{3.1} have immediate equivariant extensions, given appropriate homomorphisms between groups that act on the complexes in the statements.

\subsection{Remarks on Chain Homotopies}
For homotopy theory in general, and specifically for our intended study of explicit chain level constructions involving contractions, chain homotopies play a central role. Here is a simple observation that is immediate from the definition of the differential $d = d_{HOM}$ in $HOM$ complexes.
\begin{lem}\label{3.3} (i). Suppose $H \colon A_* \to B_{*+1}$ is any linear map of degree 1.  Then $d_{HOM} (H) = d_B H + H d_A \in HOM_0(A_*, B_*)$.\\

(ii). Thus two chain maps $f, g \colon A_* \to B_*$, that is, cycles in $HOM_0(A_*, B_*)$, are chain homotopic exactly when $f - g $ is a boundary in the hom complex $HOM(A_*, B_*)$.   
\end{lem}

It is useful to record some other basic results  concerning compositions of chain homotopies $H \colon B_* \to C_{*+1}$  with chain maps and concerning suitably interpreted tensor products of chain homotopies with chain maps.
\begin{prop}\label{3.4} Suppose $f, g \colon B_* \to C_*$ are homotopic chain maps with $dH + Hd = f - g$.\\

(i). Suppose $\alpha \colon A_* \to B_* $ and $\gamma \colon C_* \to D_*$ are chain maps.  Then $H \circ \alpha$ is a chain homotopy between the chain maps $f \circ \alpha$ and $ g \circ \alpha \colon A_* \to C_*$.  Also  $\gamma \circ H$ is a chain homotopy between the chain maps $\gamma \circ f$ and $ \gamma \circ g \colon B_* \to D_*$.\\

(ii).  Suppose $\beta \colon X_* \to Y_*$ is a chain map. Then $H \underline{\otimes} \beta$ is a chain homotopy between chain maps $f \otimes \beta$ and $g \otimes \beta \colon B_* \otimes X_* \to C_* \otimes Y_*$. Also $\beta \underline{\otimes} H$ is a chain homotopy between chain maps $\beta \otimes f$ and $\beta \otimes g \colon X_* \otimes B_* \to Y_* \otimes C_*$.
 \end{prop}
\begin{proof} Part (i) is quite trivial since $$d_{HOM}(H \circ \alpha) = d(H\alpha) + (H \alpha) d = dH\alpha + Hd \alpha = (f-g) \circ \alpha$$  $$d_{HOM}(\gamma \circ H ) = d (\gamma H) + (\gamma H) d = \gamma dH + \gamma H d = \gamma \circ (f - g).$$ The point is, the chain maps $\alpha$ and $\gamma$ commute with boundaries.\\

Part (ii) is somewhat more subtle because of signs that enter computations involving the chain map $$\underline{\otimes} \colon HOM(B_*, C_*) \otimes HOM(X_*, Y_*) \to HOM(B_* \otimes X_*, C_* \otimes Y_*)$$ from bullet point 5 of Proposition \ref{3.1}.  From the chain map $\underline{\otimes}$ we obtain
$$d_{HOM}(H \underline{\otimes} \beta) = d_{HOM}( H) \underline{\otimes} \beta - H \underline{\otimes} d_{HOM}(\beta) = (dH + Hd) \otimes \beta = (f-g) \otimes \beta.$$
With the order of all tensor products reversed, we also obtain
$$ d_{HOM}(\beta \underline{\otimes }H) = d_{HOM}(\beta) \underline{\otimes} H + \beta \underline{\otimes} d_{HOM}(H) = \beta \otimes (dH+Hd) = \beta \otimes (f-g).$$
The points are that $H$ has degree 1 and $f, g, \beta$ and $d_{HOM}(H)$ are cycles of degree 0 in hom complexes. For tensor products of degree 0 chain maps the $\underline{\otimes} $ product  is just the ordinary functorial $\otimes$ product in the category of vector spaces.\\

The one-line proofs of the statements in part (ii) are painless, but only because we proved bullet point 5 of Proposition \ref{3.1} by an adjoint argument.  Here only very special cases are needed and it is an easier exercise to verify these special cases by evaluating the homomorphisms involved on tensors in $B_* \otimes X_*$ or $X_* \otimes B_*$. Specifically $H \underline{\otimes} \beta (b \otimes x) = H(b) \otimes \beta(x)$ while  $\beta \underline{\otimes} H (x \otimes b) = (-1)^{|x|} \beta(x) \otimes H(b)$. It still requires some effort to keep track of all the signs in the direct verifications of the two statements in part (ii).\\

Versions of Proposition \ref{3.4} for equivariant chain maps and chain homotopies are routine.
\end{proof}

\subsection{*Preview of Cochain Level Steenrod Operation*}In this asterisked subsection we preview the classic construction of the Steenrod cyclic reduced power cohomology operations beginning with certain chain maps involving tensor powers.  We discuss this rather abstractly here, but in the classic cases the coinvariant complexes $B_*/G$ to be introduced below are specifically chain complexes of classifying spaces $BG$ of groups $G$, and the complexes $C_*$ and $C^*$ are chain and cochain complexes associated to topological spaces, for example the normalized chains and cochains $N_*(X)$ and $N^*(X)$ of simplicial sets $X$.\\

The arguments sketched in this subsection are old arguments that can mostly be found  in the classic book of Steenrod and Epstein [31, and in many other places. We are including them here in the spirit that we regard our project as a textbook that starts from the beginning and develops Steenrod operations much more at the chain and cochain level rather than only at the cohomology level. The discussion here makes use of several of the basic constructions with chain complexes given in Section 3. The results sketched in this subsection will be covered in much greater detail later in Part IV.  As discussed in the introduction, the important results in Parts I, II, and III of this paper are not specifically about Steenrod operations but about operads. Steenrod operations will be the focus of the separate paper Part IV. What will be new in Part IV, compared to the classical arguments that we sketch here, is that we will work with explicitly defined chain maps and chain homotopies, rather than with maps whose existence and uniqueness up to chain homotopy is deduced by acyclic model methods.\\

We will work in the category of $\FF = \FF_p$ chain complexes, $p$ prime. Suppose $B_*$ is a  contractible\footnote{This means homotopy equivalent as $\FF$ complexes  to the unit complex $\FF$.}  free left $\FF[G]$ complex, where $G \subseteq \Sigma_p$ is a permutation group that contains a $p$-cycle $t$.     Equivalently $p$ divides  $|G|$.  Then $<t>\ \subseteq G$ is a Sylow subgroup.   If $p = 2$ then $G = \Sigma_2$, $\FF = \FF_2$, and there are no signs  in the discussion of this section, as well as other simplifications. We may as well assume $t = (23\ldots p1)$, the cyclic permutation\\

Suppose  $\phi \colon B_* \otimes C_* \to C_* ^ {\otimes p}$ is a $G$-equivariant chain map, where $G$ acts by left permutations on the target.\footnote{Such maps $\phi$ were discussed in the introductory Subsection 1.2.  Basic facts about $HOM$, tensor powers, and group actions were established here in Section 3. There is no reason to wait more than 200 pages to see how they can be used.} We can regard $\phi$ as a degree 0 cycle in $HOM(B_* \otimes C_*, C_*^{\otimes p})$.  Then the adjoint  map $Ad(\phi) \colon B_* \to HOM(C_*, C_* ^{\otimes p})$ from Proposition \ref{3.1} is also a 0 cycle, that is, a chain map, and is equivariant.  We can compose this equivariant adjoint map with the  duality maps in the last two bullet points  from Proposition \ref{3.1}  to get $$ B_* \xrightarrow{Ad(\phi)} HOM(C_*, C_*^{\otimes p}) \to HOM((C_*^{\otimes p})^*, C^*) \to HOM((C^*)^{\otimes p}, C^*).$$ All maps in this sequence are left equivariant chain maps, with the suitably formed left actions.  There are signs in the second two maps, given in Proposition \ref{3.1}. \\

This last composition  has an adjoint  $\Phi \colon B_* \otimes (C^* )^{\otimes p} \to C^* $,
which has the form $$\Phi(b \otimes (\alpha_1 \otimes \ldots \otimes \alpha_p))(c) =  (-1)^{|b| |a|}< \alpha_1 \otimes \ldots \otimes \alpha_p,\ \phi (b \otimes c)>, $$ where $|a| = | \alpha_1 \otimes \ldots \otimes \alpha_p|$.   In particular, if $|b|$ or $|a|$ is even the sign goes away.  The evaluation is 0 unless $-|c| = |b| + |\alpha_1 \otimes \ldots \otimes \alpha_p|$. The term $\phi(b \otimes c) \in C_*^{\otimes p}$  will be a sum of $p$-tensors $\sum c_1 \otimes  \ldots \otimes c_p$, each of total degree, $|b| + |c|$ and from Proposition \ref{3.1} the last evaluation is computed using $$< \alpha_1 \otimes \ldots \otimes \alpha_p,\  c_1 \otimes \ldots \otimes c_p>\  =\  (-1)^{\ell(\ell -1)/2} \prod <\alpha_j,\  c_j >$$ This evaluation on $p$-tensors is 0 unless $ |\alpha_j | = - |c_j|$.  The integer  $\ell$ is the number of $c_j$ of odd degree. If all $\alpha_i = \alpha$ and $| \alpha | = - |c_j|$  then the full $\Phi$ evaluation becomes $\Phi(b \otimes \alpha^{\otimes p})( c) = \sum (-1)^{|b| |\alpha | m} \prod <\alpha, c_j>$, where $m = (p-1)/2$.\\

The $\Sigma_p$ action on $C^*$ is trivial here, so the equivariance property of $\Phi$ is $$\Phi = \Phi \circ g  \colon B_* \otimes (C^*)^ {\otimes p} \to B_* \otimes (C^*)^ {\otimes p} \to C^*,\ \ g \in G.$$  
In other words, the map $\Phi$ factors through a map $\bar{\Phi}$ defined on the  coinvariant complex obtained by dividing by the diagonal action of $\Delta G \subset G \times G$, $$\Phi   \colon B_* \otimes (C^*)^ {\otimes p} \xrightarrow{/\Delta G}  B_* \otimes_{\Delta G} (C^*)^ {\otimes p} \xrightarrow{\bar{\Phi}} C^*.$$
Alternate formulations of the equivariance are $$\Phi (gb \otimes g (\alpha_1 \otimes \ldots \otimes \alpha_p))  = \Phi (b \otimes  \alpha_1 \otimes \ldots \otimes  \alpha_p)$$  $$\Phi (gb \otimes  \alpha_1 \otimes \ldots \otimes \alpha_p) = \Phi (b \otimes g^{-1} (\alpha_1 \otimes \ldots \otimes \alpha_p)) = \pm \Phi (b \otimes  (\alpha_{g1} \otimes \ldots \otimes \alpha_{gp})).$$
 Note there are Koszul signs in the evaluations of  any $g (\alpha_1 \otimes \ldots \otimes \alpha_p), g \in \Sigma_p$. If $t = (23 \ldots p1)$ then $t(\alpha_1 \otimes \alpha_2 \ldots, \alpha_p) = \pm (\alpha_p \otimes \alpha_1, \otimes  \ldots \otimes \alpha_{(p-1)}) $.\\

{\bf Discussion of Invariants and Coinvariants.} Denote $\bar{B}_* = B_* /G$, the coinvariants.  Suppose first that $G$ consists only of even permutations, for example $G =\ <t>$. Then  for any cocycle $\beta \in C^q$ the tensor $\beta^{\otimes p}$ is a $G$-invariant cocycle.  If $(\ \ )^G$ denotes the invariant subcomplex of a $G$-complex then there is an obvious well-defined chain map  inclusion $\bar{B}_* \otimes ((C^*)^{\otimes p})^G \subset B_* \otimes_{ \Delta G} (C^*)^{\otimes p}$ that maps $\bar{b} \otimes z$ to the $\Delta G$-orbit $b \otimes_{\Delta G} z$ of $b \otimes z$.  If $b \in B_*$ is such that its image $\bar{b} \in \bar{B}_*$ is a cycle (resp. boundary), then $\bar{b} \otimes \beta^{\otimes p}$ names a well-defined cycle (resp. boundary) in the coinvariant complex.  Hence $\bar{\Phi}(\bar{b} \otimes \beta^{\otimes p}) \in C^*$ is a cocycle (resp. coboundary). \\

If $G$ contains odd permutations and $deg(\beta) = q$ is odd then $\beta^{\otimes p}$ is not invariant, since $g \beta^{\otimes p} = \tau(g) \beta^ {\otimes p}$, where $\tau(g) \in \{\pm 1\}$ is the parity character.  For any $G$-complex, define a new action of $G = \widetilde{G} $ on the same complex by setting the new action of $\tilde{g}$ to be the original action of $\tau(g) g$.\footnote{Equivalently, tensor the original $G$-complex with $\widetilde{\FF}$, meaning the complex $\FF$ concentrated in degree 0 where $g$ acts on $\FF$ by $\tau(g) \in \{ \pm 1\}$.}  So the group $G$ and complex $B_*$ don't change, and the $\widetilde{G}$ notation simply indicates a different $G$ action. The diagonal actions of $\Delta G$ and $\Delta \widetilde{G}$ on  $B _* \otimes (C^*)^{\otimes p}$ are identical, since the two $\tau(g)$'s cancel, hence they have the same coinvariants and  the map $\bar{\Phi}$ can be viewed as being a map $\widetilde{\Phi}$ defined on $\Delta \widetilde{G}$-coinvariants.\\

Denote the new coinvariants of $B_*$ by $\widetilde{B}_* = B_* / \widetilde{G}$. We then have an inclusion $\widetilde{B}_* \otimes ((C^*)^{\otimes p})^{\widetilde{G}} \subset B_* \otimes_{\Delta \widetilde{G}} (C^*)^{\otimes p}$ that maps $\tilde{b} \otimes z$ to the $\Delta \widetilde{G}$-orbit $b \otimes_{\Delta \widetilde G} z$ of $b \otimes z$. If $\beta$ has odd degree $q$ then $\beta^{\otimes p} \in ((C^*)^{\otimes p})^{\widetilde{G}}$ is invariant.  Thus if $\tilde{b} \in \widetilde{B}_*$ is a cycle (resp. boundary) then $\widetilde{\Phi}(\tilde{b} \otimes \beta^{\otimes p}) \in C^*$ is a cocycle (resp. coboundary).\\

{\bf The Definition of Operations.} For simplicity in the proof of the following result, we will make use of additional assumptions about $B_*$.  We assume that in degree 0, $B_0 = \FF_p[G]$ and we assume that all homology classes of $B_* / G$ or $B_* / \widetilde{G}$ are represented by $b \in B_*$ for which the $t$-norm $Nb = \sum t^i b \in B_*$ is a boundary. Since $t$ is an even permutation this condition does not depend on which $G$ action is meant. These assumptions easily hold for $B_*$ a contractible free $\FF_p[G]$ complex, since they trivially hold if $G = \ <t>$ is cyclic, as we will see in Example \ref{5.4}. Then even for bigger groups we will only use homology classes that obviously come from the cyclic group. But regardless of that,  because $<t>$ is a Sylow subgroup of any bigger $G \subseteq \Sigma_p$,  a well known result is that the map of $\FF_p$ coefficient group homology of $<t>$ to group homology of $G$ is surjective.\\

We remind that if $C_*$ is positively graded then $C^*$ is negatively graded, so cocycles will have degrees $q \leq 0$.  We use supercripts to designate specific cochain and cohomology groups, $C^q$ and $H^q(C^*)$. 
 
\begin{prop}\label{3.5} The map $\bar{\Phi} \colon \bar{B}_* \otimes ((C^*)^{\otimes p})^G \to C^*$  induces a  linear homomorphism of degree $q(p-1) + k$ for $q$ even, $D \colon H_k(\bar{B}_*) \otimes H^q(C^*) \to H^{pq+k}(C^*)$, defined at the chain level by $D(\bar{b} \otimes \beta) = \bar{\Phi}(\bar{b} \otimes \beta^{\otimes p}) = \Phi(b \otimes \beta^{\otimes p})$.\\

The map $\widetilde{\Phi} \colon \widetilde{B}_* \otimes ((C^*)^{\otimes p})^{\widetilde{G}} \to C^*$  also induces a  linear homomorphism of degree $q(p-1) + k$ for $q$ odd, $D \colon H_k(\widetilde{B}_*) \otimes H^q(C^*) \to H^{pq+k}(C^*)$, defined at the chain level by $D(\tilde{b} \otimes \beta) = \widetilde{\Phi}(\tilde{b} \otimes \beta^{\otimes p}) = \Phi(b \otimes \beta^{\otimes p})$. 

\end{prop}
 \begin{proof} Of course if $G$ consists only of even permutations then $\bar{B}_*= \widetilde{B}_*$, and there is only one homomorphism $D(\bar{b} \otimes \beta) = \bar{\Phi}(\bar{b} \otimes \beta^{\otimes p})$, defined for all $q$. In this case the homomorphisms $D$ define cohomology operations on $H^*(C^*)$, one for each homology class $[\bar{b}] \in H_*(\bar{B}_*)$. The Steenrod operations arise from generators of the homology groups of  $\bar{B}_* = B_*/ <t>$,  which are  classifying spaces of the cyclic groups of order $p$.  It is well known that these homology groups are isomorphic to $\FF_p$ in each degree.\\
 
The construction of  the maps $D$ is provided in  paragraphs  above the proposition, along with the facts that the maps depend only on the homology classes of $\bar{b}$ and $\tilde{b}$.  Additivity in the first variable is trivial.  It remains to prove that if $\beta = d\alpha$ is a coboundary then $\Phi(b \otimes \beta^{\otimes p})$ is a coboundary and that the induced homology maps are linear. We will sketch proofs. We first consider what happens if $\beta = d \alpha$.  This seems surprisingly non-trivial compared to the observation that $\Phi(b \otimes \beta^{\otimes p})$ is a cocycle if $\beta$ is a cocycle..\\

We need to explain why in the coinvariant complexes the cycles $\bar{b} \otimes (d \alpha)^{\otimes p}$ or $\tilde{b} \otimes (d \alpha)^{\otimes p}$ are boundaries. We form the  small acyclic subcomplex $(\alpha, \beta)$ of $C^*$. Now, there are two actions of $G$ on $B_* \otimes (\alpha, \beta)^{\otimes p}$.  There is the diagonal action of $\Delta G$ and there is the action of $G$ just on the first factor $B_*$. It turns out that using the acyclicity of $(\alpha, \beta)^{\otimes p}$ we can use the methods we begin developing in Section 6 for using contractions to recursively construct chain maps to prove these two $G$ or $\widetilde{G}$-complexes are {\it equivariantly isomorphic}, extending the obvious isomorphism in the lowest degree. In that degree both complexes are $\FF_p[G] \otimes (\beta^{\otimes p})$, with the same actions. In fact, an equivariant isomorphism exists extending the identity on $B_* \otimes (\beta^{\otimes p})$. Therefore, the coinvariant complexes are isomorphic, $(B_* / G) \otimes (\alpha, \beta)^{\otimes p} \simeq B_* \otimes_{\Delta G} (\alpha, \beta)^{\otimes p}$, and similarly in the $\widetilde{G}$ case. But on the left side, $\bar{b} \otimes \beta^{\otimes p} = \pm d(\bar{b} \otimes (\alpha \otimes \beta^{\otimes (p-1)}))$  so the equivariant isomorphism produces a corresponding boundary formula for $\bar{b} \otimes \beta^{\otimes p}$ in the $\Delta G$-complex. The same argument applies in the $\widetilde{G}$ case.\\
   
Next we take up the linearity claim. The non-commutative expansion of $(\beta + \gamma)^{\otimes p}$ has the form $ \beta^{\otimes p} + \gamma^{\otimes p}+ \sum (terms) $, where $\sum (terms)$ consists $2^p - 2$ summands that are tensor products of  both $\gamma$'s and $\beta$'s. Note $\sum(terms)$ is $G$ or $\widetilde{G}$ invariant, depending on the parity of the degrees, since the three tensor power terms are invariant.\\

The terms fall into $(2^p - 2) / p$ blocks, where each block is characterized by some number $1 \leq k  < p$ of $\beta$'s and a cyclic placement of the $\gamma$'s and  $\beta$'s.   We are assuming that $G$ contains a $p$-cycle $t$. Then each block of terms has the form $Ny$, where $y$ is one of the tensor product terms in the block and $N = 1 + t + \ldots + t^{p-1}$ is the `$t$-norm'.\footnote{If  $\gamma$ and $\beta$ are linearly dependent, then $Ny = 0$.}   There are $\binom{p}{k}$ blocks with $k$ $\beta$'s. To be specific about the $y$'s, choose the tensor product term $y_{jk}$, $1 \leq j \leq \binom{p}{k}$, from each such cyclically  ordered block that comes first in lexicographic $\beta, \gamma$ order. Thus each $y_{jk}$ is a cocycle and $ Ny_{jk} = tNy_{jk}$ is a $t$-invariant cocycle.\\

Set $Y = \sum_{j,k} y_{jk}$. We have observed $ \sum(terms) = NY$ is $G$ or $\widetilde{G}$-equivariant.  In fact,  each $g t^iy_{jk}$ is uniquely a $ \pm t^{i'} y_{j'k}$, since $g$ permutes the set of tensors with $k$ $\beta$'s, with a Koszul sign when $deg(\beta) = deg(\gamma)$ is odd.  The Koszul sign is $\tau(g)$.\\ 

We are assuming $b$ can be chosen so that  $Nb \in B_*$ is a boundary.  Then each $Nb \otimes y_{jk}$ is a boundary. In the coinvariant complex $B_* \otimes_{\Delta G} (C^*)^{\otimes p}$ we have $tb \otimes y_{jk} \equiv b \otimes  t^{-1} y_{jk}$, hence $Nb \otimes y_{jk} \equiv b \otimes Ny_{jk}$ is a boundary. Then $\sum_{j,k} b \otimes Ny_{jk} = b \otimes NY $ is a boundary in the coinvariant complex. In the even degree case this boundary term is $\bar{b} \otimes NY$ and in the odd degree case it is $\tilde{b} \otimes NY$. Applying $\bar{\Phi}$ or $\widetilde{\Phi}$ implies $D$ is linear in the second variable.
\end{proof}

Explicit chain maps of form $\phi \colon B_* \otimes C_* \to C_*^{\otimes p}$ and $\Phi \colon B_* \otimes (C^*)^{\otimes p} \to C^*$ will turn up in Parts II, III, and IV when we discuss how chains $C_* = N_*(X)$ on simplicial sets are coalgebras over certain operads with components  $B_*$, and how cochains $C^* = N^*(X)$ are algebras over those same operads. The operad algebra structures on  cochains lead to explicit cochain level definitions of Steenrod operations and proofs of their properties in Part IV. It is the explicit form of various chain maps and chain homotopies that is a key point of our project.  For example, referring to the sketch of the proof above, we will ultimately have explicit formulas writing the $Nb$ as boundaries and we will have explicit formulas writing the $\bar{b} \otimes_{\Delta G} (d \alpha)^{\otimes p}$ as boundaries.\\

One might wonder that since the mod $p$ Steenrod operations on cocycles $\beta$ are defined in all degrees just using a classifying space for the cyclic group $<t>$   of order $p$, why are the $\widetilde{G}$ considerations needed for bigger groups? The answer is, many of the cohomology operations defined on $H^*(C^*)$ using $D$ as in Proposition \ref{3.5} turn out to be 0 precisely because the explicit maps $\Phi \colon B_* \otimes (C^*)^{\otimes p} \to C^*$ when $B_* /<t>$ is the chain complex of a classifying space model for the cyclic group, do extend  naturally to classifying space models  for the full symmetric group.  Certain cycles in  cyclic  group classifying spaces become explicit boundaries in symmetric group classifying spaces, as we will prove in Section 9 of Part I, hence these cycles determine trivial cohomology operations.\\

Also, with the $\bar{B}_*$ chosen as chains on classifying spaces for  symmetric groups $\Sigma_n$, we will use explicit operad algebra structure  maps $\Phi \colon B_* \otimes (C^*)^{\otimes n} \to C^*$ for $n = 2p$ and $n = p^2$ developed in Parts II and III to obtain  cochain level proofs of the Cartan formula and the Adem relations for the cochain Steenrod operations in Part IV.  Classical proofs of these formulas also made use of these classifying spaces but  took place only at the cohomology level, using acyclic model methods that produced maps only defined up to chain homotopy.

\newpage

\section{Contractions}

We will repeat some sentences from the beginning of the previous section. We work primarily with positively graded chain complexes $C_*$ over a ground ring $\FF$, which will usually be the integers or a field.   Positively graded means 0 in negative degrees. However,  cochain complexes are negatively graded chain complexes and ccasionally we  work with  $\ZZ$ graded $HOM$ complexes.  For completeness, we will define contractions for rather general $\ZZ$ graded chain complexes.
\subsection{Contractions of  Augmented Based Chain Complexes}
We regard the ground ring $\FF$ as a chain complex concentrated in degree 0.  Our chain complexes will have an augmentation chain map $\epsilon \colon C_* \to \FF$, that is $\epsilon d_1 = 0$,  and a base point chain map $\iota \colon \FF\to C_0$, that is $d_0 \iota = 0$, with $\epsilon \circ \iota = Id_{\FF}$.  A base point   is the same data as a cycle $c_0 = \iota(1) \in C_0$,  with $\epsilon (c_0) = 1$.  We set $\rho = \iota \circ \epsilon \colon C_* \to \FF\to C_0 \subset C_*$, and sometimes also refer to  the chain map $\rho$ as the base point. Of course if $C_*$ is positively graded then any $\iota(1) \in C_0$ is a cycle.\footnote{The category of augmented based chain complexes seems like a good thing. The unit object $\FF$ is both an initial object and a terminal object in that category.  Perhaps the closed symmetric monoidal category properties of chain complexes in Proposition \ref{3.1} should also be observed for augmented based chain complexes.}\\

{\bf DEFINITION 4.0.} A {\it contraction} of an augmented based chain complex  is a chain homotopy $h \colon C_* \to C_{*+1}$ with $dh + hd = Id - \rho$.  Thus $\epsilon$ and $\iota$ are chain  homotopy  inverses of each other.  If a contraction exists we call $C_*$ a {\it contractible complex}. If a contraction is chosen, we call $C_*$ a {\it contracted} complex.\\

One can also form an augmented complex $\widehat{C}_*$ by adding a summand $\FF$ in degree $-1$.  Set $\hat{d}_0 = d_0 \oplus \epsilon$ and set $\hat{d}_{-1} = d_{-1} \oplus 0$.  Extend the contraction $h$ to $\hat{h}$ by setting $\hat{h}_{-1} = h_{-1} \oplus \iota$ and $\hat{h}_{-2} = h_{-2} \oplus 0$. Then the degree 1 map $\hat{h}$ becomes a chain homotopy between  the identity and zero on the augmented complex.   That is, $\hat{d} \hat{h} + \hat{h} \hat{d} = Id$ in all degrees. Thus the augmented complex is acyclic, meaning 0 homology in all degrees.

\begin{rem}\label{4.1} For the positively graded complexes of interest to us, $C_0$ is not only free over $\FF$ but has a preferred basis.  Therefore there is an obvious choice of augmentation, sending basis elements to $1 \in \FF$. There are also obvious basepoints, namely basis elements. The augmented complex $\widehat{C}_*$ is especially simple. Namely, $\FF$ is added in degree $-1$, $\hat{d}_0 = \epsilon$ and $\hat{h}_{-1} = \iota$.
\end{rem}

\begin{rem} \label{4.2}{\bf *Twisted Augmentations and Contractions.*} Suppose $C_*$ is a $G$ complex.  If $G$ also acts on $\FF$, say by $\alpha \colon G \to Aut(\FF)$ what should be meant by an augmentation that includes the group action data?  Denote by $\widetilde{\FF}$ the  new $G$ module structure on $\FF$. We define a {\it twisted augmentation} to be a chain map $\tilde{\epsilon} \colon C_* \to \widetilde{\FF}$ with $\tilde{\epsilon}(gx) = \alpha(g) \tilde{\epsilon}(x)$.  A basepoint just means $\iota \colon \FF\to C_*$ with $\tilde{\epsilon} \circ \iota(1) = 1.$ We also call $\tilde{\rho} = \iota \tilde{\epsilon} $ a basepoint.  Then $\tilde{\rho} (gx) = \tau(g) \tilde{\rho}(x)$.\\

Define a {\it twisted contraction} to be $h \colon C_* \to C_{*+1}$ with $dh + hd = Id - \iota \tilde{\epsilon}$.  Then $C_*$ and $\widetilde{\FF}$ are chain homotopy equivalent.\\

Also, as before, adding an $\widetilde{\FF}$ in degree $-1$ yields an acyclic complex $\widehat{C}_*$ with a null-homotopy $\hat{h}$ with $d \hat{h} + \hat{h} d = Id.$  $\ \ \ \ \qed$
\end{rem}

\subsection{Contractions with $h^2= 0$}
Of course when contractions exist, they are usually highly non-unique.   It turns out that a very useful additional property of contractions is $h \iota = 0$ and $h^2 = 0$, which we will show in the next proposition can always be assumed.  This is the same as $\hat{h}^2 = 0$ for the associated null-homotopy of the augmented complex $\widehat{C}_*$.\\

In the $\ZZ$-graded based augmented case, we will first show that $h^2 = 0$  implies $\epsilon h = 0$. Since $\iota$ is injective, it is equivalent to show $\iota \epsilon h = 0$.  If $x \in C_{-1}$, then $(dh + hd)hx = hx - \iota \epsilon hx$.  But $h^2 = 0$ implies $(dh+hd)hx = hdhx = h(x- hdx) = hx$.  Of course $\epsilon h = 0$  is a vacuous statement in the positively graded case since there is no $h_{-1}$.\\

{\bf More General Contraction Data.} Contraction data $\FF\leftrightarrows C_* \xrightarrow {h} C_{*+1}$, with  $Id = \epsilon \iota  \colon \FF\to \FF$,  $dh + hd = Id - \iota \epsilon$,  and $h^2 = 0, hi = 0, \epsilon h = 0$, is a special case of what Rubio and Serveraert [25] call a {\it reduction} of a complex $C_*$ in their work on constructive algebraic topology.  Specifically, $D_* \leftrightarrows C_* \xrightarrow {h} C_{*+1}$ for some complex $D_*$, with the same composition  formulas as ours.  Such data also occurs in homological perturbation theory.  Lambe and Stasheff [14] refer to it as {\it strong deformation retract data} and  Real [24] uses the term {\it contraction} even in this more general case.\footnote{We thank Robert Lipshitz for pointing out to us this related work.}  We think it likely that many of the constructions in our paper can be extended somehow to this more general setting. The following result is a known example.
\begin{prop}\label{4.3} If a chain complex admits a contraction then it admits a contraction $h$ with $h^2 = 0$ and $h \iota = 0$.Then $Ker(h) = Im(h)$ in degrees $\not= 0$ and $Ker(h) = Im(h) \oplus Im(\iota)$ in degree 0. 
\end{prop}
\begin{proof}Any contraction of $C_*$  defines a null-homotopy $\hat{h}$ of the augmented complex $\widehat{C}_*$.  There are then split short exact sequences for the augmented complex,  
$$0 \to \widehat{Z}_n \to \widehat{C}_n \leftrightarrows \widehat{B}_{n-1} \to 0,$$
where the backwards arrow on the right is $\hat{h}$, and where the boundaries in $\widehat{C}_*$ coincide with the cycles in all degrees.  Given such splittings,  one has direct sum decompositions of $\widehat{C}_n$. In fact, $c_n = d\hat{h} c_n + \hat{h}d c_n = z_n + \hat{h}z_{n-1}$.  Then $h|_{\widehat{Z}_n} = \hat{h}$ and $h|_{\hat{h}\widehat{Z}_{n-1}} = 0$ also defines a null-homotopy of the augmented complex.\\

In fact, $h(c_n) = h(z_n) = \hat{h}(z_n) $, so $dh(c_n)= d \hat{h}(z_n) = z_n = d\hat{h}(c_n).$  Also $d(c_n)$ is a cycle, so $hd(c_n)  = \hat{h}d(c_n)$. Clearly $h^2 = 0$.  This produces a contraction $h$ of $C_*$ with $h^2 = 0 $ and $h \iota = 0$. The last statement of the proposition  follows easily from $dh(x) + hd(x) = x - \iota \epsilon (x)$ and $\epsilon h = 0$.
\end{proof}

\begin{rem} \label{4.4}{\bf *Change of Basepoint.*} We briefly discuss change of basepoint, even for $\ZZ$ graded $C_*$, with a fixed preferred augmentation $\epsilon.$  Suppose we have two basepoint cycles  in $C_0$, say $\iota(1) = c_0$ and $\iota'(1)  = c'_0$, with $\epsilon(c_0) = \epsilon(c'_0) = 1$. Suppose $h$ is a contraction for the basepoint $\iota$.  That is, $dh + hd = Id - \iota  \epsilon.$ Define $h' \colon \C_* \to C_{*+1}$ by $h'(c) = h(c) - \epsilon (c) h(c'_0)$.  So $h = h'$ except in degree 0. The following is just an exercise in all the definitions, but it is a little trickier than it might look.
\begin{prop}\label{4.5} $h'$ is a contraction for the basepoint $\iota'$.  That is, $dh' + h'd = Id - \iota' \epsilon.$  If $h^2 = 0$ and $h\iota = 0$ then $(h')^2 = 0$ and $h' \iota' = 0$.   
\end{prop}

In practice this remark could be  useful when $\iota'$ is a complicated basepoint that we really want to use, and $\iota$ is a much simpler basepoint, but one for which it is fairly easy to find a contraction. $\qed$
\end{rem}
{\it For the remainder of the paper we will assume all  contractions satisfy $h^2 = 0$ and $h \iota = 0$, which we  sometimes abbreviate as simply $h^2 = 0$.  When appropriate we will try to point out in various situations why this assumption holds or just how this assumption is being used.}
\newpage

\section{Examples of Contractions}
We often work with simplicial sets $X$.  We denote by $N_*(X)$ the normalized simplicial set chain complex of $X$ with $\FF$ coefficients,  obtained from the full simplicial  set chain complex by dividing by the subcomplex spanned by degenerate simplices.  If $N_*(X, \ZZ)$ is contractible  and $X$ is simply connected then the geometric realization $|X|$ is topologically contractible by the Whitehead Theorem.
\subsection{Simplices, Products, and MacLane Models}
\begin{exam}\label{5.1} \textbf{The Simplex.}  If $\Delta$ is an ordered simplex with vertices $\{0,1, \ldots, k, k+1, ...\}$ then there is a contraction $h(i_0,  \ldots, i_n) = (0, i_0, \ldots, i_n)$ of $N_*(\Delta)$. The base point is vertex $ 0 \in N_0(\Delta)$, and the augmentation sends all vertices to $1 \in \FF$. A routine calculation using the boundary formula $$d(i_0, \ldots, i_n) = \sum_{j = 0}^n (-1^j(i_o, \ldots, \widehat{i_j}, \ldots, i_n)$$  implies $h$ is a contraction.  Also $h^2 = 0$ and $h \iota = 0$ since $(0,0,x)$ is degenerate.\\

We point out that if $C_*(\Delta)$ is the {\it unnormalized} augmented, based,  simplicial set chain complex associated to the  ordered simplex then the same map $h$ is a contraction, but $h^2 \not= 0$.  The procedure of Proposition \ref{4.3}  yields a contraction with square zero, but it is a big mess. $\qed$
\end{exam}
\begin{exam}\label{5.2}\textbf{Products of Simplicial Sets.}  If $X$ and $Y$ are simplicial sets then $X \times Y$ is a simplicial set, with $(X \times Y)_n = X_n \times Y_n$.  The face and degeneracy operators act in the obvious diagonal way on pairs of simplices, and the full boundary operator is the usual alternating sum of the codimension one face operators, $d(\sigma, \tau) = \sum_j (-1)^j (d_j \sigma, d_j \tau)$.  \\

If $N_*(X)$ and $N_*(Y)$ are contractible  then $N_*(X \times Y)$ is contractible because it is free and  we know the homology by the Eilenberg-Zilber/Kunneth Theorem.  But it is not so easy to give an explicit contraction in general.  The tricky point is that a pair of $n$-simplices $(\sigma,\tau)$ can define a non-degenerate simplex in $X \times Y$ even if both $\sigma$ and $\tau$ are degenerate.  Thus a contraction for $N_*(X \times Y)$ must be defined at the level of all simplices in the separate factors, not just on the normalized chain complexes.\\

If $X$ and $Y$ are both ordered simplices, a contraction is given by $$h((i_0,  \ldots, i_n), (j_0, \ldots, j_n)) = ((0,  i_0, \ldots, i_n), ( 0, j_0, \ldots, j_n)).$$  The proof just amounts to the easy calculation of $dh + hd$, which is the same calculation as in  the case of one simplex.  Obviously $h^2 = 0$.\\

We occasionally  use the notation $I = \Delta^1$ for the 1-simplex.  There are fairly obvious  geometric contractions, in fact, simplicial set homotopies $ \Delta \times I \to \Delta$ and $  \Delta \times \Delta \times I\to \Delta \times \Delta$ between the identities and the constant base point maps. These homotopies factor through simplicial maps $CX \to X$, where $CX$ is the cone on $X = \Delta$ or $\Delta \times \Delta$. The maps from cones  induce the above contractions on normalized chain complexes of simplices and products of simplices via the  maps $N_*(X) \to N_{*+1}(CX) \to N_{*+1}(X)$.\footnote{The first map requires identifying the cone on $\Delta^k$  with $\Delta^{k+1} $ and saying precisely what is meant by $CX$. For us this seems to work better if the cone vertex is the first vertex and $CX = X \times I / X \times 0.$} $\qed$
 \end{exam}

\begin{exam}\label{5.3} \textbf{MacLane Models.} We look at $ N_*(EG)$, where $EG$ is a standard  construction of a contractible free simplicial $G$-set for a  group $G$.  Specifically $EG(n) = G^{n+1}$, the ``$n$ simplices'' of $EG$ are $n+1$ ordered tuples of group elements.  The simplicial set face operators delete a vertex and the degeneracy operators repeat a vertex. There is a canonical base point $e = 1 e \in N_0(EG) = \ZZ[G]$, where $e \in G$ is the identity element.  The augmentation sends all group elements to $1 \in \FF$.\\

A contraction is $h_G(x) = (e, x)$, where $x = (g_0, g_1, \ldots, g_n) \in N_n(EG)$ is an $n$-simplex generator.  Again $h_G^2 = 0$.  The boundary formula is  the same as that of a simplex, $dx = \sum (-1^j(g_0, \ldots, \widehat{g_j}, \ldots, g_n)$, as is the proof that $h_G$ is indeed a contraction. The contraction $h_G$ also arises from a  simplicial set homotopy $ EG \times I \to CEG \to EG$.\\

The boundary operator in $N_*(EG)$  is $G$-equivariant, where the left group action is given by $g(g_0, g_1, \ldots, g_n) =   (gg_0, gg_1, \ldots, gg_n)$.  Thus the augmented complex $N_*(EG) \to \ZZ $ is a free $\ZZ[G]$ resolution of $\ZZ$ regarded as a trivial $G$-module.  The coinvariant complex $ N_*(EG)/G = N_*(BG)$, where $BG =  EG/G$ is  a simplicial set classifying space for principal $G$ bundles.\\

{\bf Naming Simplices of $\bf BG$.} The non-degenerate $n$-simplex generators of $N_*(BG)$ are named by tuples $[f_1, f_2, \ldots, f_n]$ of group elements with $f_j \not= e$.  The quotient projection of $(g_0, g_1, \ldots, g_n)$ is given by $f_j = g_{j-1}^{-1} g_j,\ 1 \leq j$.  A natural section is given by $g_0 = e,\  g_j = \prod_{1 \leq i \leq j} f_i$.  With this choice of a projection to orbits and a section, the boundary formula in $N_*(BG)$ induced from that of $N_*(EG)$  is given by $$d[f_1, \ldots, f_n] = [f_2, \ldots, f_n] +  (-1)^n [f_1, \ldots, f_{n-1}]  + $$ $$     \sum_{1 \leq j < n} (-1)^j [f_1, \ldots, f_{j-1}, (f_j f_{j+1}), f_{j+2}, \ldots, f_n]. \ \ \ \ \   \qed$$

The contractions $h_G$ of $N_*(EG)$ are functorial in $G$.  In fact, {\it any} function $G \to G'$ will first induce a simplicial map $EG \to EG'$ and then a chain map $N_*(EG) \to N_*(EG')$. Functions that correspond identity elements  yield chain maps commuting with contractions. Of course these chain maps are not equivariant unless the function between groups is a homomorphism.  There will be some partial equivariance if the function $G \to G'$ is a homomorphism on certain subgroups of $G$.  This happens in the important case of certain  maps $$N_*(E(\Sigma_r \times \Sigma_{s_1} \times \ldots \times \Sigma_{s_r})) \to N_*(E\Sigma_{s_1+ \ldots + s_r})$$ related to the  operad structure maps for the Barratt-Eccles operad $\cE$ to be studied in Part III.\\

{\bf *Semidirect Products*.} If $H$ and $K$ are two groups, there is an obvious identification $E(H \times K) \simeq EH \times EK$.  The canonical contraction for the product group identifies with the product contraction of the factors, defined just as for the product of two simplices. \\

But somewhat more is true.  Suppose $G$ is a semidirect product of groups, a normal subgroup  $H \lhd G$ and  $K \subset G$ with $H \cap K = \{e\}$.  Then elements of $G = HK = KH$ can be uniquely written as products of subgroup elements, $ hk = k (k^{-1} h k) $, which can be identified with pairs in $H \times K$ or $K \times H$. The  product in $G$ in terms of pairs then becomes $(h, k) (h', k') = (h (kh'k^{-1}), kk')$ in $H \rtimes K$,  or      $(k',h')(k,h) = ( k'k, (k^{-1} h' k) h))$ in $K \ltimes H$, take your choice.  Thus there are still  obvious identifications\footnote{In the symbols $\ltimes$ and $\rtimes$, the little $\lhd$ points toward the normal subgroup. If $X$ and $Y$ are simplicial sets then there is an obvious isomorphism $X \times Y \simeq Y \times X$ with $(\sigma, \tau) \leftrightarrow (\tau, \sigma)$.}  $$E(H \rtimes K) \simeq EH\times EK \simeq EK \times EH \simeq E(K \ltimes H)$$ as simplicial sets, hence there is a canonical product contraction of the normalized chain complexes.  What is different is the action of the twisted group $H \rtimes K$  on $EH \times EK$, or $K \ltimes H$ on $EK \times EH$. $\qed$

\end{exam}
\subsection{The Minimal Model for Cyclic Groups}
\begin{exam}\label{5.4} \textbf{Minimal Models for Cyclic Groups.} Here is the first of a number of examples not directly related to simplicial sets.  Let $C_n = C =\  <T> $, the cyclic group of order $n$.  Our next augmented chain complex will be $M_* \to \ZZ$, the standard minimal free  $\ZZ[C]$ left module resolution of $\ZZ$, with $C$ acting trivially on $\ZZ$.  So $M_j \simeq \ZZ[C]$, with generator $y_j$.   The $C$-equivariant boundary operator is given on generators for $k \geq 1$ by $$ d y_{2k-1} = (T - 1) y_{2k-2}\ and\  d y_{2k} = N y_{2k} =   (T^{n-1} + \ldots + T^2 + T^1  + 1) y_{2k-1}.$$
For $p$ prime and field of coefficients $\FF_p$, we point out that the boundary operator in the coinvariant complex $\bar{M}_* = C_p \backslash M_*$ is 0 and there is one generator $\bar{y}_j \in \bar{M}_j$ in each degree. Also, since modulo $p$ the norm is $N = (T^p - 1) / (T -1) = (T-1)^{p-1}$, we see that for both even and odd $j$ we have $Ny_j \in M_j$ is a boundary.\footnote{The usefulness of this was mentioned in Subsection 3.5, previewing Steenrod operations.}\\  

Although $M_*$ is not the chain complex associated to a simplicial set, it is the cellular chain complex $C_*(W)$ associated to a free $C$ action on a regular cell complex $W$, specifically from a regular cell structure on the infinite sphere $S^\infty$, with $n$ cells in each degree.   Thus the coinvariant complex $ M_* / C$ is a model for the chains on a cellular classifying space $BC = W / C$, which is an infinite dimensional  lens space.\\

{\bf The Contraction of $M_*$.} The base point of $M_*$  is $\iota(1) = y_0$ and the augmentation is $\epsilon(T^iy_0) = 1$, so $\rho(T^i y_0) = y_0$. A contraction $h \colon M_* \to M_*$ with $dh+hd = Id - \rho$ is given by 
$$h(T^i y_{2k}) = \sum_{0 \leq j < i} T^j y_{2k+1}\hspace{.2in}  [So\  h(y_{2k}) = 0\ and\ h(Ty_{2k}) = y_{2k+1}]$$
$$h(T^i y_{2k+1}) = 0\ if\ i < n-1 \ \ \ \ \  and \ \ \ \ \ \   h(T^{n-1}y_{2k+1}) = y_{2k+2}.$$
The proof that $h$ is a contraction, including $h^2 = 0$, is a simple computation that divides into even degree and odd degree cases.  Also, the contraction $h$ arises from a  geometric cellular homotopy $W \times I \to CW \to W$.$\qed$

\end{exam}

\subsection{Tensor  Complexes}
\begin{exam}\label{5.5}   Here we construct contractions of tensor products of contractible complexes.  Tensor products of complexes were discussed in Section 3.  The augmentation of the tensor product of two augmented complexes  is the tensor product of the two separate augmentations, $ C_* \otimes D_* \to \FF\otimes \FF= \FF$. The base point of the tensor product of two based complexes is also the tensor product of the two base points, $\FF\otimes \FF \to C_0 \otimes D_0$.  We then have the composition of augmentations and base points, $\rho = \rho_C \otimes \rho_D \colon C_* \otimes D_*  \to \FF\otimes \FF\to C_0 \otimes D_0$.  If $C_*$ and $D_*$ are free, then so is $C_* \otimes D_*$.\\

Given contractions $h_C, h_D$ of $C_*$ and $D_*$, one can use general methods for dealing with compositions of chain equivalences and tensor products of chain equivalences to construct a contraction  $h$  of $ C_* \otimes D_*$.  Combining chain homotopy equivalences  $$C_* \otimes D_* \leftrightarrows \FF\otimes D_* \leftrightarrows \FF\otimes \FF= \FF,$$ one can derive a resulting contraction formula for  $C_* \otimes D_*$.\\

{\bf The Contraction of a Tensor Product.} The contraction formula is  $h_\otimes = h_C \otimes Id_D + \rho_C \otimes h_D$.  Explicitly,   $$h_\otimes(a \otimes b) = h_C(a) \otimes b + \rho_C(a) \otimes h_D(b). $$  One can also just begin with this formula and compute directly that it is a contraction of $C_* \otimes D_*$.  If the contractions of $C_*$ and $D_*$ satisfy the conditions $h^2 = 0, h \iota = 0$, then so does this contraction $h_\otimes$ of $C_* \otimes D_*$.\\

Symmetry is necessarily broken here, but this turns out to be the preferred choice for our purposes.
Using $C_* \otimes D_* \leftrightarrows C_* \otimes \FF\leftrightarrows \FF\otimes \FF$ instead, the resulting formula is $$h'_\otimes(a \otimes b) =  h_C(a) \otimes \rho_D(b) +  (-1)^{|a|}a \otimes h_D(b).$$  This second contraction can also be obtained by conjugating the first contraction method by the isomorphism $D_* \otimes C_* \simeq C_* \otimes D_*$ from Proposition \ref{3.1}, which includes Koszul signs.\\

The reason we prefer the first formula is that in topological situations it leads to the classical Alexander-Whitney diagonal map $N_*(X) \to N_*(X) \otimes N_*(X)$, as we will see in Sections 6 and 8.   The other formula leads to an alternate  diagonal, which perhaps has some merit.\\

{\bf Iterated Tensor Products.} If $C_* = D_*$ we have a contraction $h^{(2)} =   h \otimes Id + \rho \otimes h  $ of $C_* \otimes C_*.$  Then one can repeat the construction and define a contraction $h^{(3)}$ of $C_* \otimes C_* \otimes C_*$. The formula, viewing $C_* ^ {\otimes 3} = C_* \otimes C_* ^{\otimes 2}$,  is
$$h^{(3)} = h \otimes Id^{\otimes 2} + \rho \otimes h^{(2)}  =   h \otimes Id \otimes Id + \rho \otimes h \otimes Id + \rho \otimes \rho \otimes h.$$
The formula makes sense not just for tensor powers but for any triple tensor product of contractible complexes $C_*' \otimes C_*'' \otimes C_*'''$.  $\ \ \ \qed$\\

 This construction can be iterated to define preferred contractions  of $C_*^{\otimes {(k+1)}}$, or of any $(k+1)$-fold tensor product, $$ h^{(k+1)}  =  h \otimes Id^{\otimes k} + \rho \otimes h^{(k)}   =  \sum_{i \geq 0} \rho^{\otimes i} \otimes h \otimes Id^{ \otimes (k-i)} = \sum_{i \geq 0} h_i^{(k+1)}.$$

If all contractions of the tensor factors satisfy $h^2 = 0$ and $h \iota = 0$ then these last summands $h_i$ for any tensor product satisfy $h_i h_j = 0$, since $\rho h = 0$ and $h \rho = 0$ on each tensor factor.  Hence  $(h^{(k)})^2 = 0$ for all $k$, and  $h^{(k)} \iota ^{\otimes k} = 0$.  It follows that  $Ker(h^{(k)}) = Im(h^{(k)}) \oplus Im(\iota^{(k)})$ and  $ Im(h^{(k)}) = \sum Im(h_i^{(k)})$ for any $k$.\\

This sum decomposition of contractions of tensor products is remarkably similar to the structure of contractions of surjection complexes that we will study in Part II. Anticipating some later terminology, for $\ell \geq 1$ we will refer to  tensors in the image of  $h_{\ell-1} = \rho^{\otimes( \ell -1)} \otimes  h \otimes Id^{ \otimes (k - \ell)} $ as {\it clean tensors}, or more precisely {\it tensors clean at $\ell $}, where $\ell$  is the subscript naming the position of $h$ in the tensor $h_{\ell-1}$. The span of these clean tensors is always the image of the preferred contraction of a tensor product. In order to apply some of our uniqueness theorems later it is important to be able to recognize elements in the image of contractions. \\ 

 The preferred contractions of tensor products also satisfy an easily verified associativity property.  Namely, for $i+j = k$ and $h^{(1)} = h$:
$$ h^{(k+1)}  = \rho^{\otimes i} \otimes h^{(j)} + h^{(i)} \otimes Id^{\otimes j}.$$ 
These same contractions of tensor products, and more general versions,  can be found in J. R. Smith [28], Real [24], and no doubt other places. $\qed$

\end{exam}
\subsection{*$HOM$ Complexes*}
\begin{exam}\label{5.6}   Representing adjoints of tensor product functors $B_* \mapsto  B_* \otimes C_*$ are  complexes of homomorphisms $HOM(C_*, D_*)$, which in degree $n$ are given by $\prod_k Hom_\FF(C_k, D_{k+n})$. These were discussed in detail in Section 3. In particular, these complexes are more complicated because they can be non-zero in all degrees $n \in \ZZ$.  Included are the dual cochain complexes $C^* = HOM(C_*, \FF)$.  In general, even if $C_*$ and $D_*$ are free complexes, $HOM(C_*, D_*)$  is not free unless $C_*$ has finite type, or other assumptions are made about  $\FF$.  This will not be an issue because we will only map other complexes to $HOM$ complexes, and freeness of the target is not needed.\\

In the dual cochain complex of a based augmented chain complex, the  base point and the augmentation now become switched adjoints $\iota_{C^*} = \epsilon^* \colon \FF= \FF^* \to C^0$, and $\epsilon_{C^*} = \iota^* \colon C^0 \to \FF^* = \FF.$\\
 
In the general case of two based augmented complexes, the augmentation of   $HOM(C_*, D_*)$   is $\epsilon(u) = \epsilon_D \circ u \circ \iota_C$, which is a map  $HOM(C_*, D_*) \to HOM(\FF, \FF) = \FF$.  A base point is  $\iota(x) = \iota_D \circ x \circ \epsilon_C $, which is a map  $\FF= HOM(\FF, \FF) \to HOM(C_*, D_*)$.  In the form $ \rho = \iota \circ \epsilon \colon HOM(C_*, D_*) \to HOM (C_*, D_*)$, the basepoint is $\rho(u) = \rho_D \circ u \circ \rho_C$.\\ 

{\bf *The Contraction of a $HOM$ Complex.*} Suppose $h_C, h_D$ are contractions of $C_*, D_*$.  Then keeping track of the chain homotopies in the induced chain equivalences  $$HOM(C_*, D_*) \leftrightarrows HOM(\FF, D_*) \leftrightarrows HOM(\FF, \FF) = \FF,$$  a contraction of $HOM(C_*, D_*)$ with $dh + hd = Id - \rho$ is given by $$h_{HOM}(u) = h_D \circ u \circ \rho_C + (-1)^{|u|} u \circ h_C.$$  If the contractions of $C_*$ and $D_*$ satisfy the conditions $h^2 = 0, h \iota = 0$, then so does the contraction $h_{HOM}$ of $HOM(C_*, D_*)$. \\ 

If one replaces the middle term $HOM(\FF,  D_*)$  in the sequence above  by $HOM(C_*, \FF)$, the resulting contraction of $HOM(C_*, D_*)$  is  $$h_{HOM}'(u) = h_D \circ u + (-1)^{|u|} \rho_D \circ u \circ h_C.$$   
In the case of the dual $C^* = HOM(C_*, \FF)$ of a contractible complex, both contraction formulas yield the contraction $h^*(u) = (-1)^{|u|} u \circ h_C.$ 
\end{exam}
In the case of free complexes, the second contraction $h_{HOM}'$  of $HOM(C_*, D_*)$ can  be obtained from the inclusion  $HOM(C_*, D_*) \to HOM(D^*, C^*)$, using the dual complex contractions and restricting the first contraction method on the right.$\ \qed$\\

We include this treatment of $HOM$ complex contractions because of the reasonably pleasant parallels with tensor complex contractions.  We will not make much use  of $HOM$ complex contractions in our paper.  For one thing, in situations where we might use $HOM(C_*, D_*)$, the appropriate basepoint is not the simple one given here, but rather a more complicated cycle in $HOM_0(C_*, D_*)$.  For example, the Alexander-Whitney diagonal. We would  then need to use the change of basepoint method of Remark \ref{4.4}.  \\

Also, $HOM$ complexes are often not bounded below.  It would then be difficult to see how to get started constructing a chain map using our recursive contraction based procedure of the next section.\\

Regarding the adjoint isomorphism $$Ad\colon HOM(B_*\otimes C_*, D_*) =  HOM(B_*, HOM(C_*, D_*))$$ for three complexes with contractions, the desired relation $ h (Ad\ u)  = Ad (hu )  $ does hold, and similarly if all tensor and hom contractions are replaced by their $h'$ alternates.  However, it seems surprisingly  difficult to relate a single map $B_* \otimes C_* \to D_*$ constructed using a contraction of $D_*$ to  its adjoint map $B_* \to HOM(C_*, D_*)$, using contractions of $C_*$ and $D_*$. 

 \subsection{*Twisted Coefficients*}
 Suppose a based augmented complex $C_* \xrightarrow{\epsilon} \FF \to 0 $ is a free $\FF[G]$ acyclic resolution of the trivial $G$-module $\FF$.  Suppose $1 \not= -1\in \FF$.  A non-trivial homomorphism  $\tau\colon G \to \{ \pm 1\}$ defines another $G$ action on $\FF$, by $g*a  = \tau(g) a$.  With $\tau$ fixed, we refer to this twisted $G$-module as $\widetilde{\FF}$.\footnote{More generally one could twist by homomorphisms $G \to \FF^*$, the units in $\FF$, or even by $G \to Aut(\FF)$, the automorphisms of the abelian group $\FF$.  We could even study $C_* \otimes_\FF L_* \to \FF\otimes_\FF L_* = L_*$ for  $G$-complexes $L_*$.} As discussed in Remark \ref{4.2}, we can go from free acyclic resolutions of $\FF$ to free acyclic resolutions of $\widetilde{\FF}$, by the correspondence $C_* \leftrightarrow   C_* \otimes _\FF\widetilde{\FF}$, with diagonal $G$-action on the right,  $$g(x \otimes 1) = gx \otimes \tau(g) 1= \tau(g) gx \otimes 1 $$ for all $g \in G,\ x \in C_*$.\\ 

We regard  $\widetilde{C}_* = C_* \otimes _\FF\widetilde{\FF}$ as a tensor product of two based augmented $G$-complexes, with augmentation values in $\FF\otimes \widetilde{\FF}$.     The tensor product augmentation on $\widetilde{C}_*$ is  $\tilde{\epsilon}(x \otimes 1) = \epsilon(x) \otimes 1$, and the tensor product base point is  $\tilde{\iota}(1) = \iota(1) \otimes 1$. We see $\tilde{\epsilon}$ is indeed equivariant for the  $G$-action on $\widetilde{C}_*$ and the $G$ action on $\widetilde{\FF}$.  That is,  $$\tilde{\epsilon} (g (x \otimes 1)) =  \tilde{\epsilon}(gx \otimes \tau(g)) = \tau(g)(\epsilon(gx) \otimes  1) = \tau(g)(\epsilon(x) \otimes 1) = g \tilde{\epsilon}(x \otimes 1). $$
 
The two  chain complexes $C_*$ and $\widetilde{C_*}$ are canonically identified as chain complexes of $\FF$-modules, with $x \leftrightarrow x \otimes 1$,    so there is no real harm writing simply $\widetilde{C}_* = C_*$ as chain complexes of $\FF$-modules, with a new group action $g*x = \tau(g)gx$. It is convenient to refer to this second $G$ action on $C_*$ as the $\widetilde{G}$ action. The augmentation and base point remain the same, but now the computation just above shows that the augmentation is a $\widetilde{G}$ equivariant map from the $\widetilde{ G}$ action on $C_*$ to the action on the  twisted module $\widetilde{\FF}$.\\

One should be careful.  The coinvariant complexes $ C_*/ G$ and  $  C_*/  \widetilde{G} $ are certainly different.  For example, if $C_* = N_*(EG)$ then the homology of the coinvariant complex $ N_*(EG)/ \widetilde{G} = \widetilde{N}_*(BG)$ is the homology of $BG$ with twisted coefficients $\widetilde{\FF}$.\\

If $h$ is a contraction of $C_*$ then $h$ also defines a contraction of $\widetilde{C}_* = C_* \otimes \widetilde{\FF}$, namely $\tilde{h}( x) =  h(x) \otimes Id$.  This is a special case of a contraction of a tensor product. One will have $d \tilde{h} + \tilde{h}d = Id - \tilde{\rho}$. \\

Note as $G$-modules $\widetilde{\FF} \otimes \widetilde{\FF} = \FF$, so tensoring a $G$-resolution $\widetilde{C}_* \to \widetilde{\FF} \to 0$ with $\widetilde{\FF}$ yields a resolution $C_*$ of the trivial module $\FF$.  The two constructions are obviously  inverses of each other.\\

We will have occasion to use the twisting and untwisting construction described here in the case $G = \Sigma_n$ and $\tau$ the unique non-trivial $\{ \pm 1 \}$-valued ``parity" character.  A specific example,  given in Sections 11 and 12 of  Part II, leads to an alternate version of a surjection complex.  We regard the details of that surjection complex and its relations with other surjection complexes, as an interesting part of our work.\\

Also,  twisted complexes arising from the parity character are used in Part IV to explain quite naturally some subtle differences between cochain level Steenrod operations applied to even and odd degree cocycles.  In fact, in Subsection 3.5 we included a brief outline of a more general algebraic formulation of those constructions.

\newpage
\section{Using Contractions to Construct Chain Maps}
\subsection{The Basic Standard Procedure}

There is a standard recursive procedure that we will use in many different situations to construct chain maps $\phi \colon B_* \to C_*$.  The common ingredients are that $B_*$ and $C_*$ are positively graded, based, augmented complexes. Also,  $C_*$ has a given contraction and $B_*$ is free, either over a ground ring $\FF$ or a group ring $\FF[G],$ with a chosen basis.  In the equivariant case, $C_*$ will also be a $G$-complex and the constructed chain maps will be equivariant.\\

The constructions usually begin quite simply in degree 0, exploiting base points and augmentations in both domain and range.  In the equivariant case we assume the augmentations to be $G$-maps, usually with the trivial action on $\FF$, but possibly with the same $G$ action arising from some $\tau \colon G \to \{\pm 1\} $, as studied in Subsection \ref{5.5}, for both the $B_*$ and $C_*$ augmentations.  We always assume the basepoint $\iota_B(1)$ is a basis element in degree 0. \\

In degree 0, our preferred canonical maps will commute with basepoints and augmentations. That is, they will satisfy  $\phi_0 \iota_B = \iota_C$ and $\epsilon_C \phi_0 = \epsilon_B$. Then also $\phi_0 \rho_B = \iota_C \epsilon_B = \rho_C \phi_0$, where $\rho = \iota \epsilon$, so canonical maps commute with augmentations and basepoints, in  strong ways.\\

{\bf Comments on Maps in Degree 0.} Since dealing with chain maps can be a little tricky in low degrees, we point out here a couple of useful observations about canonical maps. The basepoint condition simply says $\phi_0(b_0) = c_0$, where $b_0 = \iota_B(1)$ and $ c_0 = \iota_C(1)$ are the basepoint  elements.  Note that by Proposition \ref{4.5} we are free to choose any $c_0 \in C_0$ with $\epsilon_C(c_0 ) = 1$ that we want as a basepoint by choosing an appropriate contraction of $C_*$.\\

If $\phi_0$ is $G$-equivariant and if the augmentation  condition $\epsilon_C \phi_0(b) = \epsilon_B(b)$ holds on basis elements $b \in B_0$ then it will hold on all elements of $B_0$  by a simple equivariance argument. Also, given the basepoint condition,  the  augmentation condition implies $\phi_0(d_1B_1) \subset d_1C_1$, since $Im(d_1) \subseteq Ker(\epsilon) $ in general, with equality in the contractible case.\\

In degree 0 the canonical map could be defined by $\phi_0(b) = \iota_C \circ \epsilon_B(b)$ on basis elements of the domain, and then extended equivariantly.  If $rank(B_0) =1$ this is the only choice that satisfies  our preference $\phi_0(b_0) = c_0$, and in general it is  the only such choice that also vanishes on $Ker(\epsilon_B)$. If $rank(B_0) > 1$  there are other choices that  satisfy our two preferences $\phi_0 \iota_B = \iota_C$ and  $\epsilon_C \phi_0 = \epsilon_B$ but are non-zero on $Ker(\epsilon_B)$. Remarkably, in every example of any importance in our paper, $rank(B_0) = 1$ always holds. $\ \ \ \ \qed$\\

We now formalize what we mean by the standard canonical procedure for extending equivariant maps in degree 0 that satisfy $\phi_0 \iota_B = \iota_C$ and $\epsilon_C \phi_0 = \epsilon_B$, given a basis of the domain and a contraction  of the range.\\

{\bf DEFINITION 6.0. The Recursive Standard Procedure.} For a basis element $b \in B_n$, with $n \geq 1$, the recursive  definition of the canonical chain map, or standard procedure chain map extending $\phi_0$,  is $$ \phi(b) = h_C( \phi (d_Bb)).$$  One extends $\FF$-linearly or $\FF[G]$-linearly.  By properties of $h_C$ and induction, the extension $\phi$ is automatically a chain map. Namely, by equivariance it suffices to check on basis elements, where we compute
$$d_C\phi(b) = d_Ch_C (\phi (d_B b)) = \phi(d_B b) - h_Cd_C (\phi(d_B b)) - \rho_C(\phi(d_B b)) $$ $$ = \phi(d_B b) - h_C(\phi(d_B d_B b)) - 0 =  \phi(d_B b). \  \qed$$

Note that in the language of based, augmented complexes from Section 4, with $B_{-1} = C_{-1} = \FF$, it is natural to begin the canonical map with $\phi_{-1} = Id$ in degree $-1$ and then $\phi_0 = h_C \circ \phi_{-1} \circ d_B =   h_C \circ d_B = \iota_C \circ \epsilon_B$ on basis elements of the domain.  But, as proved,  we can  actually begin with any  equivariant $\phi_0$ with $\phi_0 \iota_B = \iota_C$ and $\epsilon_C \phi_0 = \epsilon_B$  and extend using the same recursive definition.  So we do not insist $\phi_0 = \iota_C \epsilon_B$,  although that almost always holds.\footnote{If $B_0$ has rank 1, beginning {\it just} with  $\epsilon_C \phi_0 = \epsilon_B$, we can  declare the basepoint of $C_*$ to be $\iota_C  = \phi_0 \iota_B$ and then choose a contraction $h_C$ of $C_*$  for that basepoint, as described in Remark \ref{4.4}.  Then proceed with the canonical extension of $\phi_0$ in Definition 6.0 for those choices.}\\

Although  the recursive procedure does not easily yield a closed formula for $\phi$ in all degrees, it turns out that in many cases one can compute $\phi$ in low degrees, then see a pattern that suggests a closed formula, which can be proved by induction once it is written down. Or, maybe someone else already claimed a chain map formula and one can prove by induction that their formula is the result of the standard recursive procedure. Several examples will be given.

\begin{exam}\label{6.1} The first example is a bit silly.  If we use the contraction of the simplex given in Example \ref{5.1}, along with the standard $\ZZ$-basis,  the standard procedure produces a map $N_*(\Delta) \to N_*(\Delta)$.  In degree zero, this map sends all vertices to the base point vertex $(0)$, and is the zero map in all positive degrees.      So the construction just gives the map of chain complexes induced by the constant map from the simplex to the base point. \\

On the other hand, if for some group $G$ the complex $B_*$ is a free contractible $G$-complex with $B_0  = \FF[G]$ with the canonical augmentation $\epsilon(g) = 1$ and base point $\iota(1) = e$, and if the chosen basis of $B_*$ belongs to $Im(h) = Ker(h)$ in positive degrees, then  the canonical equivariant chain map $\phi \colon B_* \to B_*$ is the identity. Namely, $\phi_0 = Id$, and then by induction  if $b$ is a basis element of positive degree, $$\phi(b) = h \phi (db) = h(db) = b - d(hb) = b - d0 = b.$$ 
The set of complexes $B_*$ for which these assumptions hold is closed under tensor products and includes minimal models $M_*$, MacLane models $N_*(EG)$,  and  surjection complexes $S_*(n)$ to be studied in Part II. $\ \ \ \qed$
\end{exam}
\subsection{Fundamental Uniqueness Theorems and Compositions}
We will insert here some uniqueness results that seem very useful and interesting, and reinforce our emphasis that there are strongly preferred maps at the chain level that clarify various homological results. The following two uniqueness results are very important.
\begin{prop}\label{6.2} Suppose $C_*$ is a  $G$-complex with a contraction  $h_C$ with $h_C^2 = 0$ and suppose $B_*$ is  free over $\FF[G]$, with  basis elements $\{ b \}$.  Suppose $\psi \colon B_* \to C_*$ is an equivariant chain map with $\psi_0 \iota_B = \iota_C$, $\epsilon_C \psi_0 = \epsilon_B$, and with $\psi(b) \in Im(h_C) = Ker(h_C)$, for all basis elements $b$ in positive degrees.  Then $\psi$ is necessarily the standard procedure extension $\phi$ of $\psi_0$. 
\end{prop}
\begin{proof} Both $\phi$ and $\psi$ are equivariant, so it suffices to compute on basis elements. Since $\psi$ is a chain map, by induction the standard procedure extension of $\psi_0$ is defined recursively on basis elements $b$ of positive degree by $$\phi(b) = h_C \psi (d_B b ) = h_C d_C \psi (b) = \psi(b) - d_C h_C \psi(b) = \psi(b) - 0.$$  
Note the last step in this uniqueness result requires $h_C^2 = 0$, which we are universally  assuming for all contractions.\\

Here is an alternate formulation of the proof.  For $x \in C_*$ with $deg(x) > 0$, one has $x = 0$ if $d_Cx = 0$ and $h_Cx = 0$, since $x  = (dh + hd)(x)$.  For $b \in B_*$ a basis element, both $\phi(b)$ and $\psi(b)$ belong to $Im(h_C) = Ker(h_C)$.  By induction, $d_C(\phi(b)) = \phi(d_B(b)) = \psi(d_B(b)) = d_C\psi(b)$.  Thus $\phi(b) - \psi(b)$ lies in both $Ker(h_C)$ and $Ker(d_C)$, hence is 0.
\end{proof}

\begin{prop}\label{6.3}(i). Suppose $B_*$ is free over $\FF[G]$, with a contraction $h_B$, and suppose $C_*$ is an $\FF[G]$ complex with a contraction $h_C$.   Further suppose that $B_*$ has a chosen $\FF[G]$ basis consisting of elements $h_B(x)$, including in degree 0 with basis element $e = i_B(1) \in B_0 = \FF[G]$.  Suppose $\psi \colon B_* \to C_*$ is an equivariant  map of graded modules with $h_C \circ  \psi = \psi \circ  h_B$.  If either $h_B^2 = 0$ or $h_C^2 = 0$ then $\psi = \phi$, the canonical equivariant chain map constructed using $h_C$ and the chosen $\FF[G]$ basis of $B_*$.\\

(ii). In particular, there is a unique equivariant map $B_* \to C_*$ commuting with contractions (if one such exists) and it is automatically a standard procedure chain map independent of the choice of $\FF[G]$ basis of $B_*$ in the image of $h_B$.  If $C_*$ satisfies the same hypotheses as $B_*$ then $\psi = \phi$ is surjective.  If $C_* = B_*$ then $\psi = \phi = Id$.\\

(iii). Now just assume the standard procedure map $\phi \colon B_* \to C_*$ satisfies the weaker condition $h_C \phi h_B = 0$. Equivalently $\phi(Im(h_B)) \subset Im(h_C)$ and $\phi(Im (\iota_B)) \subset Im (\iota_C)$.  If $b = h_B(x) \in B_*$ is a basis element then $\phi(b) = \phi h_B(x) = h_C \phi(x)$. 
\end{prop}
\begin{proof}  The `commuting with contractions' hypothesis is meant to include that in degree 0, $\psi_0 \iota_B (1) = \iota_C(1)$.  Thus $\psi_0 = \phi_0$.  Now we proceed by induction on degree, assuming the result holds in degrees less than $n$. Again, by equivariance it suffices to compute on basis elements of degree $n$. We first use the hypothesis $h_B^2 = 0$. Consider such a basis element $b \in Im(h_B)$, so $h_Bb = 0$.  Then
$$\phi(b) = h_C \phi (d_Bb) = h_C \psi (d_Bb) = \psi h_B (d_Bb) = \psi(b - d_Bh_B b) = \psi(b) - 0.$$
Next we assume $h_C^2 = 0$, including $h_C \iota_C = 0$.  Write basis element $b = h_Bx$.   By induction, 
$$ \phi(b)  = h_C \phi (d_B h_B x) = h_C \psi (d_Bh_B x) = h_C \psi(x - h_B d_B x - \rho_B x)$$ $$ = h_C \psi(x) - h_Ch_C \psi(d_Bx) - h_C \rho_C \psi (x) = h_C \psi(x) = \psi (h_B x) = \psi(b).$$
Of course we generally assume both $h_B^2 = 0$ and $ h_C^2 = 0$, but the two independent arguments are amusing.\\

The statements in part (ii) of the proposition  follow by trivial induction on degree. One notes that by the hypotheses and induction, basis elements of $C_*$ can be written $h_C(y) \in C_*$, and these are all in the image of $\psi$.\\

For part (iii), since $b = h_B(db)$ we can write $x = db + h_B(y)$.  Then $\phi(x) = h_C \phi (db) + h_C \phi h_B(y) = h_C \phi (db) = \phi(b)$. Of course one can always discard summands of $db$ that are themselves basis elements, but being able to discard other summands in $Im(h_B)$ will prove to be very useful in Parts II and III. 
\end{proof}

{\bf Characterization of Contracted MacLane Models.} The contracted complex $(N_*(E_G), h_G) = (B_*, h_B)$, where $h_G(x) = (e, x)$, has the following obvious three properties.  First, $h_B^2 = 0$ and $h_B \iota = 0$.  Second, $B_*$ is free over $\FF[G]$, with a basis consisting of elements $b \in Im(h_B) = Ker(h_B)$, including in degree 0 with $B_0 = \FF[G]$ and basis element $\iota_B(1) = e$.  Third,  $Im(h_B)$ coincides with the $\FF$-span of the basis elements $\{b\}$. Note in positive degrees a basis element $b = h_B (d b)$, since $d h_B (b) = d 0 = 0$.

\begin{prop}\label{6.4}

(i). Assume the three properties above for a contracted complex $(B_*, h_B)$.  Let $C_*$ be a $G$-complex with a contraction $h_C$, of course with $h_C^2 = 0$ and $h_C \iota_C = 0$.   The standard procedure map $\phi \colon B_* \to C_*$, with $\phi_0(e) = \iota_C(1)$ and given on positive degree basis elements by $\phi(b) = h_C \phi (db)$ and then extended equivariantly, commutes with contractions.  That is, $\phi h_B(z) = h_C \phi(z)$ for all $z$.\\ 

(ii). If $C_*$ also has a basis in $Im(h_C)$ then one has a standard procedure splitting $ B_* \leftrightarrows C_*$, expressing $C_*$ as a canonical direct summand of $B_*$.\\
 
(iii). The three properties above uniquely characterize the pair $(B_*, h_B)$, up to canonical equivariant isomorphism. That is, $(B_*, h_B) \simeq (N_*(E_G), h_G)$.
\end{prop}
\begin{proof}
We first prove (i) by induction. From the hypothesis any element $h_B (z) =  \sum \epsilon_i b_i = h_B(\sum \epsilon_i db_i)$ for basis elements  $b_i$ and constants $\epsilon_i \in \FF$.  Since $Im( h_B) = Ker (h_B)$, we get $z = \sum \epsilon_i db_i + h_B(y)$ for some $y$.  Then  $$\phi h_B (z) = \phi(\sum \epsilon b_i) =   h_C \phi(\sum \epsilon_i db_i)  =  h_C \phi(z) - h_C \phi h_B(y) = h_C \phi(z), $$ since by induction $\phi h_B(y) = h_C \phi(y)$.\\

Parts (ii) and (iii) follow by combining part (i) with Proposition \ref{6.3}. For part (iii), one sees two standard procedure  maps $(B_*, h_B) \leftrightarrows (N_*(EG), h_G)$ with both compositions identity maps.
\end{proof}

{\bf Compositions of Standard Procedure Maps.} We insert here some easily checked statements that will turn out to be important later.  These are related to compositions of standard procedure chain maps. In fact, it is exactly  variants of parts (iii) and (iv) of the following proposition that allow us in Part III to prove relatively painlessly that the surjection complexes form an operad $\cS$,  that $\cS$ is a quotient of the Barratt-Eccles operad $\cE$, and that $\cS$  is a suboperad of the Eilenberg-Zilber operad $\cZ$.
\begin{prop}\label{6.5} Consider a composition of chain maps $A_* \xrightarrow{\alpha} B_* \xrightarrow{\beta} C_*$, perhaps in the  equivariant context.  Assume $B_*$ and $C_*$  have contractions satisfying $h^2 = 0$ and $h \iota = 0$. In part (i) assume $A_*$ has a contraction. \\

(i). If both maps commute with contractions then so does the composition. Thus if $A_*$ satisfies the domain hypotheses of Proposition 6.3, then both $\alpha$ and the composition $\beta \alpha$ are the standard procedure chain maps.\\

(ii).  If the second map commutes with contractions and the first map is a standard procedure chain map, then the composition is the standard procedure chain map. In particular, this always holds if $B_* = N_*(EG)$, since a standard procedure map with domain $N_*(EG)$  will always commute with contractions by Proposition \ref{6.4}.\\

(iii). If both maps are standard procedure chain maps and if for basis elements $a \in A_*$ one has $\alpha(a) = \sum \epsilon_i b_i$ for constants $\epsilon_i \in \FF$ and basis elements $b_i \in B_*$, then the composition is the  standard procedure chain map. (This includes cases other than $B_* = N_*(EG)).$\\

(iv).  If both maps are standard procedure chain maps and if for basis elements $a \in A_*$ it holds that $\beta \alpha(a) \in Im(h_C)$, then the composition is the standard procedure chain map. Note this hypothesis holds if $\beta(Im(h_B)) \subset Im(h_C)$, that is, if $h_C \beta h_B = 0$.
\end{prop}
\begin{proof} The first statement in part (i) is trivial. For part (ii),  if  $a \in A_*$ is a basis element, $\beta \alpha(a) =  \beta h_B(\alpha(da)) = h_C(\beta \alpha( da)).$
The point of part (iii) is that it is  easy to compute the second standard procedure chain map $\beta$ on images  $\alpha(a) = \sum \epsilon_i b_i$.  This doesn't work if the $\alpha(a)$ require non-trivial group action summands $g_ib_i \in B_*$.\\

Part (iv) is immediate from the  uniqueness result Proposition \ref{6.2}. Note part (iv) easily implies parts  (ii) and (iii).  Also, the uniqueness result Proposition \ref{6.3} implies the first map $\alpha$ in part (i) is a standard procedure map, so part (ii) implies the second statement in part (i). Thus we can view all the statements as consequences of the basic uniqueness theorems.  But sometimes one sees the hypotheses of (i), (ii), or (iii) directly, and the direct proofs from the definitions are easy, so it seems reasonable to isolate those statements.
\end{proof}
It is not true  in general  that  if the two maps are both standard procedure chain maps then the composition is a standard procedure chain map, even if the first map does commute with contractions.  For example, if $C_* = A_*$ and if $B_*$ is much smaller,  perhaps the only standard procedure chain map $A_* \to A_*$ extending the map in degree 0 is the identity, which couldn't factor through $B_*$.\\

For certain interesting compositions we will have to work pretty hard later to show they are indeed standard procedure chain maps. In other situations we will have to work pretty hard to deal with compositions of standard procedure maps  that are not standard procedure maps.\\

We turn to more examples.

\subsection{Alexander-Whitney Maps for MacLane Models.}  Given two groups, the standard procedure using the contraction of the range produces an $H \times G$ equivariant chain map $$AW \colon N_*(EH \times EG) \to N_*(EH) \otimes N_*(EG).$$  Since the domain is a MacLane model, $AW$ will commute with contractions. The construction begins in degree 0 with $AW(x ,y) = x \otimes y \in \FF[H] \otimes \FF[G]$, which defines the obvious isomorphism $\FF[H \times G] = \FF[H] \otimes \FF[G]$.  The $H \times G$-basis of the domain $N_*(EH \times EG)$ is given by non-degenerate pairs of sequences of group elements,  with both initial entries identity elements.  The standard procedure recursive formula for a basis element $b$ is $\phi(b) = h_\otimes \phi(db)$, where the contraction of the range is $h_\otimes = h^{(2)} =   h_H \otimes Id_G + \rho_H \otimes h_G $. 
Some low dimensional direct computations, using the contraction  of the range and equivariance,  pretty quickly leads to the following guess.
\begin{prop}\label{6.6}  For all pairs of simplices in the domain of the map $AW \colon N_*(EH \times EG) \to N_*(EH) \otimes N_*(EG)$  constructed by the standard procedure, one has $$AW(x_0, x_1, \ldots, x_n), (y_0, y_1, \ldots, y_n) =  \sum_{i = 0}^n (x_0, \ldots,  x_i) \otimes (y_i, \ldots, y_n).$$  
\end{prop}
\begin{proof} We want to illustrate a few of the above theoretical principles about standard procedure chain maps to explain this formula, so we will say much more than we need to.  First, one easily checks that the map given by this formula is equivariant and commutes with contractions, so Proposition \ref{6.3} applies. It is not necessary to verify that the formula defines a chain map. This brings out a `uniqueness' aspect of the iconic $AW$ formula.\\

We will also illustrate  more direct inductive arguments that prove  the standard recursive procedure yields the $AW$ formula.    Since the domain is a MacLane model, the standard procedure map $\phi$ commutes with the contractions, $h_\times$ and $h_\otimes$. The basis generators of the domain are given by $$b = (e_H, x_1, \ldots, x_n), (e_G, y_1, \ldots, y_n) = h_\times(\vec{x}, \vec{y}).$$  Observe that $db$ is a sum of many terms, but only $(x_1, \ldots, x_n), (y_1, \ldots, y_n)$ is {\it not} another basis generator.  Applying $h_\otimes \phi$ to the basis boundary terms gives 0.  Thus, we have two reasons why $\phi(b) = h_\otimes \phi (\vec{x}, \vec{y})$.  The formula  $h_\otimes = h_H \otimes Id_G + \rho_H \otimes h_G$ and induction gives $$h_\otimes(\vec{x}, \vec{y}) = \sum_{i=1}^n (e_H, x_1, \ldots, x_i) \otimes (y_i, \ldots y_n) + e_H \otimes (e_G, y_1, \ldots, y_n) = AW(b),$$ as desired.  The last summand is explained  by $\rho_H(x_1, \ldots, x_i) = 0$ unless $i = 1$, in which case $\rho_H(x_1) = e_H$ and $h_G(y_1, \ldots, y_n) = (e_G, y_1, \ldots, y_n)$.
\end{proof}
{\bf Iterations of $\bf {AW}$ Maps.} The inductive arguments  extend routinely to the MacLane models for products of three or more groups.  In particular, with three groups one can iterate the construction with $H \times (G \times K)$ or with $(H \times G) \times K$ and get the same result as given by using our canonical preferred contraction of triple tensor products.  But one  can also view this as an application of Proposition \ref{6.3}.  Just write down the iterated formula and observe that it is equivariant and commutes with contractions.\\

 Or, even easier,  use the uniqueness result Proposition \ref{6.2}.  Both ways around the associativity diagram 
$$ \begin{CD}
N_*(EH \times EG \times EK) & \xrightarrow{AW}\ \ \  \ \  N_*(EH \times EG) \otimes N_*(EK) \\ 
 \downarrow {AW} & \ \ \ \ \downarrow{AW \otimes Id} \\
 N_*(EH) \otimes N_*(EG \times EK) & \xrightarrow{Id \otimes AW}   N_*(EH) \otimes N_*(EG) \otimes N_*(EK) 
\end{CD}$$

are compositions of equivariant  chain maps, each of which commutes with contractions. Thus the compositions map basis elements  to elements in the image of contractions. Without even knowing the full $AW$ formula one sees  that  both compositions  around the diagram  agree with the standard procedure map. There are somewhat more general relevant results in Proposition \ref{6.11} below and Proposition \ref{6.5} above that clarify this, but they are not needed in this easy special case. $\qed$

\begin{rem}\label{6.7} Of course, for any simplicial sets $X, Y$ there are two-variable functorial Alexander-Whitney maps $AW \colon N_*(X \times Y) \to N_*(X) \otimes N_*(Y)$ given by exactly the same formula as in Proposition \ref{6.6}, which is the well-known sum of tensors of front faces and back faces.   There is  an explanation of the general $AW$ maps using  the contraction of tensor products of chains on simplices, along with the idea of minimal acyclic carriers discussed in Section 1.  We will explain this in Section 8. If $X = \Delta^m$ and $Y = \Delta^n$ are simplices, the map produced  by directly using the contraction of $N_*(\Delta^m) \otimes N_*(\Delta^n)$ definitely does not agree with the functorial $AW$ map.  In fact, these maps disagree already in degree 0. The issue is that the direct contraction method for a product of simplices ignores the functoriality requirement, already for vertices.   In the MacLane model case, equivariance replaces functoriality. $\qed$
\end{rem}
 \subsection{Eilenberg-Zilber Maps for MacLane Models.}   Suppose $X = EH$ and $Y = EG$ for groups $H$ and  $G$.  Then the standard procedure directly constructs an $H \times G$ equivariant map $$EZ \colon N_*(EH) \otimes N_*(EG) \to N_*(EH \times EG) = N_*(E(H \times G)),$$ using the obvious $\FF[H \times G]$ basis in the domain and the contraction  $h_{H \times G} = h_\times$ of the range.  An $\FF$-basis of $N_*(EH) \otimes N_*(EG)$ is given by   $\{\vec{x} \otimes \vec{y}\}$.  Here,  $\vec{x} = (x_0, \ldots, x_m)$ is a non-degenerate sequence of group elements in $H$ and  $\vec{y} = (y_0, \ldots, y_n)$ is a non-degenerate sequence of group elements in $G$. The preferred  $\FF[H \times G]$-basis is given by those tensors with $x_0 = e_H$ and $y_0 = e_G$.\\
 
  We can write simplex generators  in $N_{m+n}(EH \times EG)$ either as $$((h_0, h_1, \ldots, h_{m+n}), (g_0, g_1, \ldots, g_{m+n}))\ \ or\ \    ((h_0, g_0), (h_1, g_1), \ldots (h_{m+n}, g_{m+n})).$$
 The simplex is non-degenerate if $(h_i, g_i) \not= (h_{i+1}, g_{i+1})$ for all $i$, but the simplices $\vec{h}$ and $\vec{g}$ in the factors can certainly still be degenerate.  The preferred $\FF[G]$-basis is given by pairs with $(h_0, g_0) = (e,e)$.  The contraction  is $h_\times(\vec{h}, \vec{g}) = ((e, \vec{h}), (e, \vec{g}))$.\\

We will first state the Eilenberg-Zilber formula for MacLane models.  This will be followed by some discussion and then the proof.\\

\textbf{The Eilenberg-Zilber Formula.} The standard procedure equivariant chain map $EZ \colon N_*(EH) \otimes N_*(EG) \to N_*(EH \times EG)$  is given by $$EZ( \vec{x} \otimes \vec{y}) = \sum_{(I,J)} (-1)^{A(I, J)} ((h_0, g_0), (h_1, g_1), \ldots (h_{m+n}, g_{m+n})) \in N_*(EH \times EG).$$  

We clarify the meaning of all terms in this formula. The $h$'s are $x$'s and the $g$'s are $y$'s, with $(h_0, g_0) = (x_0, y_0)$ and $(h_{m+n}, g_{m+n}) = (x_m, y_n)$.    The $(I, J)$'s correspond to paths of length $m+n$ in the rectangle $[0, m] \times [0, n]$ from $(0,0)$ to $(m, n)$, increasing by one unit in either the horizontal or vertical direction at each step.  We can write $(I,J)$ either as  $$((i_0, \ldots, i_{m+n}), (j_0, \ldots, j_{m+n}))    \ \ or\ \ ((i_0,j_0), (i_1,j_1), \ldots (i_{m+n}, j_{m+n})).$$
An $(I, J)$ path determines a sequence of $h$'s and $g$'s, with the $h$ coordinate moving from one $x_i$ to the next if the path moves horizontally and the $g$ coordinate moving from one $y_j$ to the next if the path moves vertically.  One can also name an $(I,J)$ path by a sequence of $m$ $i$'s and $n$ $j$'s, where an $i$ or a $j$ in the sequence tells you whether to move in the horizontal or vertical direction in the rectangle.\\

The sign exponent $A(I,J)$ is the area between the edge path determined by $(I,J)$ and the path across the bottom of the rectangle and up the right side.  This last edge path is named  $(I,J) = ((0,0), (1,0), \ldots (m,0), (m,1), \ldots (m,n)$.  The parity of $A(I,J)$ is clearly the same as the parity of the number of swaps of an adjacent $i$ and $j$ moving the sequence $(i,i, \ldots, i, j, j, \ldots, j)$ to the sequence corresponding to the $(I,J)$ path.  This parity is the same as the parity of a shuffle permutation of $(i(1), \ldots, i(m), j(1), \ldots j(n))$, keeping the $i$'s and $j$'s in order.\\

In Section 8 we will explain the functorial version of the Eilenberg-Zilber map and we will reinterpret the notation here in terms of simplices in a triangulation of a prism. The details are almost identical to the details here, with simplices in $EH$ and $EG$ replaced by acyclic model simplices $\Delta^m$ and $\Delta^n$. The sign is explained in terms of orientations of the simplices and the prism.\\

{\bf Proof of the $EZ$ Formula.} We turn now to the proof of the Eilenberg-Zilber formula. The following discussion is an example of the situation discussed towards the end of  preview Section 2, for standard procedure chain maps having the form $\phi(x) = \sum \pm \cT x$.  More serious examples involving maps related to operads will be given in Part III.\\

The general $EZ $ formula is easily seen to be $H \times G$ equivariant. Therefore it suffices to prove the formula in   the special case of $H \times G$-basis tensors. We calculate by induction, assuming the $EZ$ formula in lower degrees,
$$EZ((e, \vec{x}) \otimes (e, \vec{y})) = h_\times \big[EZ( d(e, \vec{x}) \otimes (e, \vec{y}) + (-1)^m (e, \vec{x}) \otimes d(e, \vec{y}))\big]$$
$$ = h_\times \big[EZ(\vec{x} \otimes (e, \vec{y}) + (-1)^m (e, \vec{x}) \otimes \vec{y})\big].$$
We clarify the second equality.   There are far more than two boundary summands in the first line.  However, all but the two indicated in the second line have first entry of both tensor factors equal $e$, hence are basis generators in one lower degree.  Thus the value of $EZ$ on these generators is in $Im(h_\times) = Ker(h_\times)$.\\

This is the main point, and could be used as a starting point, along with some low degree calculation,  to perhaps {\it discover} the $EZ$ formula, rather than to just prove the stated formula is correct.\\

Ignoring signs, the desired formula for $EZ$ is very easily proved by induction.  The contraction $h_\times$ simply places identity element $(e,e)$ in front of each simplex tuple of the product.  Then observe that the terms in the conjectured sum beginning with $(e, x_1,\ldots), (e,e,\ldots)$ arise by applying $h_\times$ to the boundary term $EZ((x_1,\ldots, x_m) \otimes (e,\ldots, y_n))$, and the  terms $(e,e,\ldots), (e, y_1,\ldots)$ arise by applying $h_\times$ to the boundary term $EZ((e,\ldots, x_m) \otimes (y_1,\ldots, y_n))$.\\

Finally we deal with the signs $(-1)^{A(I,J)}$. The  term  $EZ((x_1, \ldots, x_m) \otimes (e, y_1, \ldots, y_n))$  corresponds to allowed paths in the rectangle $[1, m] \times [0, n]$.  The term  $EZ((-1)^m (e,x_1, \ldots, x_m) \otimes (y_1, \ldots, y_n))$ corresponds to allowed paths in the rectangle $[0, m] \times [1, n]$.  The main key now  is that the area in the full rectangle $[0, m] \times [0, n]$ cut off by a path beginning with   $(0,0), (1,0)$ and continuing in the smaller  rectangle $[1, m] \times [0, n]$ is  the same as the area cut off by the continued path. The area in the full rectangle  cut off by a path beginning with  $(0,0), (0,1)$ and continuing in the smaller rectangle $[0, m] \times [1, n]$ is $m$ plus the area cut off by the continued path in $[0, m] \times [1,n]$.  Done, the signs check by induction.     $\qed$ \\  

We point out that for basis tensors, $EZ( (e,\vec{x}) \otimes (e, \vec{y}))$ has initial entry of each summand tuple equal to the identity $(e,e)$.  These summands are basis generators of the range, up to sign.\\

{\bf Iterations of $\bf {EZ}$ Maps.} One can consider $EZ$ maps for products of three or more MacLane models.  Results such as associativity and commutativity can be proved immediately using the uniqueness result Proposition \ref{6.2}. The point is, chain maps defined by going different ways around various diagrams share the same equivariance properties, agree in degree 0,  and satisfy the property that images of basis elements are obviously in the image of contractions.  First we look at the associativity diagram.\\
$$ \begin{CD}
N_*(EH) \otimes N_*(EG) \otimes N_*(EF) &\  \xrightarrow{EZ \otimes Id}\ \ \    N_*(EH \times EG) \otimes N_*(EF) \\ 
 \downarrow {Id \otimes EZ} & \ \ \ \ \downarrow{EZ} \\
 N_*(EH) \otimes N_*(EG \times EF) & \xrightarrow{ EZ}   N_*(EH) \times N_*(EG) \times N_*(EF) 
\end{CD}$$
We have observed that all arrows map basis elements to sums of basis elements, up to signs.  Thus both compositions around the diagram result in the same standard procedure map. This is basically an example of Proposition \ref{6.5}(iii).\\

In the case of three or more groups the $EZ$ formula will include signs $(-1)^{A(I, J, \ldots, K)}$, where $A(I, J, \ldots, K)$ is the area of any collection of rectangles in a higher  dimensional box whose boundary is the union of two edge paths with the same endpoints.   This area is well defined modulo 2 because two such choices of rectangles add together to form a cycle in the box, which must be the boundary of some collection of cubes. But a cube has six rectangle faces, so the boundary of any collection of cubes has an even number of rectangular faces. The sign can also be interpreted as the sign of a shuffle permutation that moves a sequence of $i$'s, $j$'s, and $k$'s to $(i,i, ..., j,j,...,k,k,..)$ keeping identical letters in their original order.\\

Next we look at the commutativity diagram.
$$ \begin{CD}
N_*(EH) \otimes N_*(EG) &\  \xrightarrow{EZ}\      N_*(EH \times EG)  \\ 
 \downarrow {\tau} & \ \ \ \ \downarrow{\tau} \\
 N_*(EG) \otimes N_*(EH)  & \xrightarrow{ EZ}   N_*(EG \times EH) 
\end{CD}$$
The right vertical arrow take basis elements to basis elements with no sign.  The left vertical arrow takes basis elements to basis elements times a Koszul sign. The horizontal arrows take basis elements to sums of basis elements, up to signs in both cases.  Thus again both compositions agree with the standard procedure map.  Of course all these commutativities could be verified using the formula for $EZ$.  But our point is that general results can be used to prove equality of various maps without knowing formulas.\\

The $EZ$ map does not commute with contractions.  However,  the map $AW$  does commute with contractions, so Proposition \ref{6.5}(ii) implies that $AW \circ EZ$ is the standard procedure endomorphism of $N_*(EH) \otimes N_*(EG)$.  Then the second part of Example \ref{6.1} implies $AW \circ EZ = Id$ without using or even knowing the formulas for $AW$ and $EZ$.   $\qed$

\subsection{The Standard Procedure  Maps $  M_* \leftrightarrows N_*(EC)$.}
\begin{exam}\label{6.8} \textbf{A Chain Map $\bf \phi \colon M_* \to N_*(EC)$.} In this example and the next we want to relate the Examples \ref{5.3} and \ref{5.4} of the previous section.  With $C  = <T>$ the cyclic group of order $n$, we will  use the contraction of $N_*(EC)$ from  \ref{5.3} and freeness of $M_*$  from \ref{5.4} to construct a $C$-equivariant chain map $\phi \colon M_* \to N_*(EC)$.  The map $\phi$ does not commute with contractions because an attempt to define $\phi h_M(z) = h_C \phi(z)$ runs afoul of the equivariance requirement.  $Im(h_M)$ is not independent of the span of elements $g Im(h_M),\ g \not= 1 \in C$, unlike the situation of Proposition \ref{6.4}(i).  Also, if  such a $\phi$ did commute with contractions, Proposition \ref{6.3}(ii) would imply  it is surjective, which is clearly impossible. \\

We need to define $x_j = \phi(y_j) \in N_j(EC)$.  We begin with
$$x_0 = (1)\ and\   x_1 = (1, T).$$
The recursive formula for the $x_j$ is   $\phi(y_j) = h \phi (d y_j)$. Using the boundary formula in $M_*$ and the simple contraction formula $h(x) = (1, x) \in N_*(EC)$, this gives 
$$ x_{2k} = (1, (T^{n-1} + \ldots + T + 1)x_{2k-1})\  and\  x_{2k+1} = (1, (T-1) x_{2k}).$$
Then $\phi$ extended $C$-linearly is an equivariant  chain map.\\

Note that if $n = p$ is prime and $\FF = \FF_p$ then the images $\bar{x}_k \in N_*(BC, \FF_p)$ are cycles, since the $\bar{y}_k$ are cycles in $ M_*/ C$. \\

All summands of all $x_k$ begin $(1, \cdots)$, so the terms $(1, \pm x_{j-1})$ that occur in the definition of $x_j$ give degenerate simplices.    One gets rid of many degenerate terms by changing  the definition of $x_j$ by removing these terms.  Here are tidy formulas for the result.  Write $\bar{N} = T^{n-1} + \ldots + T$.  Then
$$ x_0 = (1),\ \ x_1 = (hT)x_0,\ \  x_{2k} = (h\bar{N}) x_{2k-1},\ \  x_{2k+1} = (hT) x_{2k}.$$
Here are the resulting closed formulas with that change, easily proved by induction:
\begin{prop}\label{6.9} The canonical equivariant chain map $\phi: M_* \to N_*(EC)$ produced by the standard procedure is given by $$ \phi(y_{2k}) = x_{2k} = \sum (1, T^{j_1}, T^{j_1+1}, T^{j_2}, T^{j_2+1}, \ldots, T^{j_k}, T^{j_k +1})$$
$$\phi(y_{2k+1}) = x_{2k+1} = \sum (1, T, T^{j_1}, T^{j_1+1}, T^{j_2}, T^{j_2+1}, \ldots, T^{j_k}, T^{j_k +1}).$$
Here the summation is over all $j_i \in \ZZ / n$.  To remove remaining degeneracies, simply add conditions $j_1 \not= 0$ for $x_{2k}$, $j_1 \not=1$ for $x_{2k+1}$, and also $ j_{i + 1} \not= j_i +1$ in both cases.
\end{prop}
This formula for a chain map $\phi$ was also known to Medina-Mardones [13] and probably others, perhaps going all the way back to Bensen.  We will give a geometric interpretation of the map $\phi$ in Subsection 16.3 of Part II.
\end{exam} 
\begin{rem}\label{6.10}
We have $x_k \in N_*(EC) \subset N_*(E\Sigma_n)$ , where $T \mapsto t = (2, 3, \ldots,  n, 1) \in \Sigma_n$, an $n$-cycle.  Since the inclusion $N_*(EC) \subset N_*(E\Sigma_n)$ commutes with the contractions $(z) \mapsto (1,z)$, the composed map $M_* \to N_*(EC) \to N_*(E\Sigma_n)$ is the map constructed by the standard procedure.  Of course this is trivial to see directly in this case.  Manipulation in MacLane models don't change just because everything belongs to a bigger group.\\

If $n = p$ is prime, in even degrees only  the homology classes of the cycles $\bar{x}_{2i(p-1)}$  are non-zero in  $H_*(B\Sigma_p; \FF_p)$, which we will prove in Section 9. These give rise to the Steenrod reduced $p^{th}$ power operations, at least on even degree cocycles.\footnote{The Steenrod operations on odd degree cocycles need more care.  The issue is that a homology group with twisted coefficients $H_*(B\Sigma_p; \widetilde{\FF})$ enters the full discussion.  Then in even degrees only the cycles $\bar{x}_{(2i+1)(p-1)}$ are non-zero in $H_*(B\Sigma_p; \widetilde{\FF})$, as we will also see in Section 9.  These cycles evaluate on tensor powers of odd degree cocycles to yield the Steenrod operations. Some of this was discussed in Subsection 3.5.}\\

In fact, a standard procedure functorial method  introduced  in Section 8 below gives rise to explicit functorial equivariant standard procedure chain maps $$ M_* \xrightarrow{\phi} N_*(E\Sigma_p) \xrightarrow{\Phi_{func}} HOM_{func}(N^*(X)^{\otimes p}, N^*(X))$$ and the image of $y_{2i(p-1)} \in M_*$ evaluated on a cocycle $\alpha^{\otimes p}$ of even degree $-2qp$ is a cocycle that up to a constant represents the Steenrod operation $P^{i-q}(\alpha) \in H^*(X)$ for $0 \leq i \leq q$.  But to establish key  properties of the Steenrod operations it is important that the map on $M_*$ factors through $N_*(E\Sigma_p)$.  We pursue all  this in Parts III and IV.  The details of a dual homology version of a functorial map equivalent to $\Phi_{func}$ are given in Section 17 of Part III.$\qed$
\end{rem}

\begin{exam}\label{6.11} \textbf{A Retraction $\bf \pi \colon N_*(EC) \to M_*$.}
 Again $C $ is the cyclic group of order $n$.  Retraction means $\pi \circ \phi (y_k) = y_k$ all $k$, where $\phi \colon M_* \to N_*(EC)$ is the map of Proposition \ref{6.9}. We will  use the retraction $\pi$ in a crucial way in Part IV in a cochain level proof of the Cartan formula for Steenrod operations and also for the Adem relations.\\

We first point out some easy consequences of previous results.  The standard procedure map $\pi \colon N_*(EC) \to M_*$ commutes with contractions by Proposition \ref{6.4} since the domain is a MacLane model.  Then by Proposition \ref{6.5} (ii), the composition $\pi \circ \phi \colon M_* \to N_*(EC) \to M_*$ is a standard procedure map.   This also follows from Proposition \ref{6.5} (i) and the formula for $\phi(y_i) \in N_*(EC)$ as a sum of basis generators.  Next, clearly both $\phi$ and $\pi$ are identity maps in degree 0.  The formulas for the contraction $h$  of $M_*$, which we will restate below, show $y_i \in Im(h)$.  Therefore, by the easy induction in the second paragraph of Example \ref{6.1}, $\phi \circ \pi =  Id$ must hold, without even knowing the formula for $\pi$.\\

{\bf The Procedure for Computing the Retraction $\bf \pi$.} Simplex generators of $N_*(EC)$ are named by sequences of powers of $T$, but it will be more efficient to just record the exponents and write a $k$-simplex  as $(i_0, i_1, \ldots, i_k)$.  The $i$'s can be arbitrary integers, interpreted modulo $n$.  The non-degenerate simplices with $i_0 = 0$ form a basis over the group ring $\FF[C]$.  Of course an equivariant chain map $\pi$ is determined by the standard recursive procedure from its values on these basis simplices.  But in the recursive determination and the proof that the formula we will give is correct, it is better to explain the value of $\pi$ on all simplices.   We have already pointed out that we know in advance that  the map produced by the standard procedure will be the (unique) equivariant map commuting with contractions.  That is, $$\pi(0, j_1, \ldots, j_k) = h \pi(j_1, \ldots, j_k), \ \   \pi(i_0, i_1, \ldots , i_k) =  T^{i_0} h \pi(i_1 - i_0, \ldots, i_n - i_0).$$

So, there are two approaches.  We could out of the blue write down a formula for $\pi$ and observe it is equivariant and commutes with contractions.  Or we can begin computing the standard procedure map in low degrees and observe how induction continues the procedure to produce a formula in all degrees. We will  carry out the latter.\\

Before  beginning,  we restate the formulas for the contraction $h$ of $M_*$.
$$h(T^i y_{2k}) = \sum_{0 \leq j < i} T^j y_{2k+1}\hspace{.5in}  [So\  h(y_{2k}) = 0\ and\ h(Ty_{2k}) = y_{2k+1}]$$
$$h(T^i y_{2k+1}) = 0\ if\ i < n-1\  and\  h(T^{n-1}y_{2k+1}) = y_{2k+2}.$$

We start the construction of the retraction  with $\pi(0) = y_0$.  Then $\pi(i) = T^i y_0$. Next, on basic 1-simplices $(0, i)$ the retraction $\pi$ is given by $$\pi(0, i) = h \pi (i) = hT^i y_0 = \sum_{j=0}^{i-1} T^j y_1.$$
Before going further, it will help to introduce some terminology. We view the powers $T^i$ on the unit circle in the complex plane, as powers of $e^{2\pi \sqrt{-1}/ n}$.  Thus these powers $T^i$, which we just name by the exponent,  inherit the positive cyclic  (counterclockwise) order.   We will write $\prec i', i''\succ$ for the cyclic {\it open} interval consisting of all powers $i$ with $i' \prec i \prec i''$ in the cyclic order.\footnote{Note $\prec$ is not a binary relation but $i' \prec i \prec i''$ makes sense.  Ordinary inequalities  will still be written $<$ and $\leq$.}  That is, start at $i'$ and continue counterclockwise around the circle to $i'' $.  The open interval consists of all $i$ that you pass. For example, $\prec i'-1, i' \succ$ is empty,  $\prec i', i' \succ$ consists of all $i \not= i'$, and if $i > 0$ then $\prec -1, i \succ = \{ 0 \leq j \leq i-1\}$.\\

We use this notation to describe the value of $\pi$ on 1-simplices $(i_0, i_1)$, where $0 \leq i_0, i_1 < n$.  $$\pi(i_0, i_1) = T^{i_0} \pi(0, i_1 - i_0) = \sum_{i\ \in \prec i_0-1, i_1\succ} T^i y_1.$$
The formula is correct even if $i_1 - i_0$ is negative or zero.  In the degenerate case $i_1 = i_0$, one has the  empty sum.  To see it in the non-degenerate case $i_1 < i_0$, add $n$ to $i_1 - i_0$. \\

It will turn out going forward  that for every simplex $\sigma \in N_*(EC)$ of even degree $2k$, we will have $\pi(\sigma) = 0$ or $\pi(\sigma) = T^i y_{2k}$ for some $i$.  For simplices $\sigma$ of odd degree $2k-1$, we will have $\pi(\sigma) = 0$ or $\pi(\sigma) = \displaystyle \sum_{i\ \in \prec i'-1, i'' \succ} T^i y_{2k-1}$ for some $0 \leq i', i'' < n$.  Note $T^{n-1}y_{2k-1}$ occurs in such a sum over an interval  if and only if $i''  < i'$.  This is important for the inductive arguments because $hT^{n-1} y_{2k-1} = y_{2k}$, but $hT^i y_{2k-1} = 0$ for all  $i < n-1$. Here are the general formulas.

\begin{prop}\label{6.12}  Write simplex generators in $N_*(EC)$ as sequences of exponents $0 \leq i < n$.  The equivariant chain map map $\pi \colon N_*(EC) \to M_*$ produced by the standard procedure is given on simplices of degree $2k$ by
$$\pi(i_0, i_1, \ldots, i_{2k} )= T^{i_0} y_{2k}$$ $$if\ k = 0\ or\ if\   i_{2j} \prec i_{2j-1} \prec i_{2j-2}\ for\ all\ 1 \leq j \leq k.$$
Otherwise $\pi(i_0, i_1 \ldots, i_{2k}) = 0.$\\

For simplices of degree $2k+1$,
$$\pi(i_0, i_1, \ldots, i_{2k+1} )= \sum_{i \in \prec i_0 -1, i_1 \succ} T^iy_{2k+1} $$ $$ if\ k = 0\ or\ if\ i_{2j+1} \prec i_{2j} \prec  i_{2j-1} \ \ for\ all\ 1 \leq j \leq k.$$
Otherwise $\pi(i_0, i_1, \ldots, i_{2k+1} ) = 0$.  In the non-zero cases here, $T^{n-1}y_{2k+1}$ occurs in the sum if and only if $i_1 < i_0$.\\

The map $\pi$ defined by the formulas is equivariant and commutes with the contractions of both domain and range.  Hence by Proposition \ref{6.3} it is indeed the equivariant chain map map produced by the standard construction.  The map $\pi$ is a retraction, that is, $\pi \circ \phi(y_k) = y_k.$
\end{prop}
Useful observations in an inductive proof are that, for any integers, the cyclic interval $\prec i' - i, i'' - i \succ$ is obtained by rotating $\prec i', i'' \succ$ clockwise $i$ units, and that the simplex $$(i_0, i_1, \ldots i_k) = T^{i_0}(0, i_1-i_0, \ldots, i_k - i_0).$$  Also, for a basis simplex one has $\pi (0, i_1, \ldots, i_k) = h \pi (i_1, \ldots, i_k)$.  $\qed$\\
\end{exam}
{\bf Examples Related to  $\bf {\pi \phi = Id}$}. As examples of the proposition, consider the terms that occur in the summation formulas for $x_{2k} = \phi(y_{2k})$ and $x_{2k+1} = \phi(y_{2k+1})$ given in Proposition \ref{6.9}.  In the even degree case, written in terms of exponents, the terms are $$(0, j_1, j_1+1, j_2, j_2+1, \ldots, j_k, j_k +1).$$
It is not hard to see that the only term not in the kernel of $\pi$ is the term $(0, n-1, 0, n-1,0, \ldots, n-1,0)$.  The reason is one cannot have $j_1+1  \prec j_1 \prec 0$ unless $j_1 = n-1$.  Then continue applying the cyclic order conditions to higher $j_k$ and conclude all $j_k = n-1$.  Then one concludes from the formula above involving the single remaining term that $\pi(x_{2k}) = y_{2k}$.\\

The odd degree case is almost the same.  The only summand in the formula for $x_{2k+1}$ not in the kernel of $\pi$ is the term $(0,1,0,1, \ldots, 0,1)$.  Then one concludes $\pi(x_{2k+1}) = y_{2k+1}$.\\

We have now verified directly  that $\pi \phi = Id$, so $\pi$ really is a retraction from $N_*(EC)$ back to $M_*$.    The earlier conceptual  argument that $\pi$ is a retraction would seem to give less information than the argument that all summands but one in $x_k = \phi(y_k)$ are in the kernel of $\pi$, and the remaining summand maps by $\pi$ back to $y_k$.  However, there are no signs in either $\phi$ or $\pi$,  hence no possible cancellations in the calculation of $\pi \phi(y_k)$. Thus it does follow from $\pi \phi = Id$ that for each $k$, exactly one of the simplicies in the sum formula for $\phi(y_k)$ is not in $kernel(\pi)$. \\

{\bf *The Retraction $ \bf {\bar{\pi} \colon N_*(BC) \to \bar{M}_*} =  M_*/C.$ *} The formulas in Proposition \ref{6.12}, especially the cyclic order conditions on triples of group elements, simplify somewhat for the induced map between coinvariant complexes.  We continue the notation that integers $f$ name cyclic group elements $T^f$, and also the notation from the end of Example \ref{5.3} that tuples $[f_1, \ldots, f_k]$ with $0 < f_j < n$ therefore name simplex generators of $N_*(BC)$.  To calculate $\bar{\pi}$ we make use of the section of the coinvariant map introduced in Example \ref{5.3}, $[f_1, \ldots, f_k] \mapsto (0, g_1, \ldots, g_k)$, where $g_j = \prod_{1 \leq i \leq j} f_i \in < T >$. The following is an exercise, translating conditions on the $g$'s to conditions on the $f$'s.

\begin{prop}\label{6.13}For generators of even degree,  $\bar{\pi}[f_1, \ldots, f_{2k}] = \bar{y}_{2k}\ or\ 0$.  The result is 0 unless $k = 0$ or if in the cyclic interval notation $f_{2j} \prec 0 \prec -f_{2j-1}$ for  $1 \leq j \leq k$.\\

For generators of odd degree,  $\bar{\pi}[f_1, \ldots, f_{2k+1}] = f_1 \bar{y}_{2k+1}\ or\ 0$.  The result is 0 unless $k = 0$ or if $f_{2j+1} \prec 0 \prec -f_{2j}$ for $ 1 \leq j \leq k$.   $\qed$
\end{prop}

In the odd degree case,  by slight duplication of notation, the $f_1$ on the right  just names the positive integer, not a power of $T$. The {\it integer} $f_1$ is the number of summands $T^iy_{2k+1}$ in the odd degree case of Proposition \ref{6.12}, namely the number of indices in the open cyclic interval $\prec -1, f_1 \succ $.\\

If $n$ is odd, another way to name simplices in $BC$ is as tuples $[e_1,e_2, \ldots, e_k]$ where $e_j \in \{\pm1, \pm 2, \ldots, \pm m \}$, where $m = (n-1)/2$.  The result on $\bar{\pi}$ then reads as follows.  For generators of even degree $2k$ the value of $\bar{\pi}$ is $\bar{y}_{2k}$ if each adjacent pair $(e_{2j}, e_{2j-1})$ has two negative entries or one positive and one negative entry with $e_{2j} + e_{2j-1} \geq 0$.  For generators of odd degree $2k+1$ the value of $\bar{\pi}$ is $e_1 \bar{y}_{2k+1}$ if each adjacent pair $(e_{2j+1}, e_{2j})$ satisfies those same condtions.  Otherwise the value of $\bar{\pi}$ in both even and odd degree cases  is 0. $\qed$

\subsection{A Standard Procedure Map Related to $\ell^{th}$ Powers}
\begin{exam}\label{6.14} \textbf{\bf $T^\ell$-Equivariant Chain Maps.}  
A group homomorphism $\iota \colon H \to G$ directly defines an $\iota$-equivariant map that we will also call $\iota \colon N_*(EH) \to N_*(EG)$.  One can also use freeness of the domain over $\FF[H]$ and the preferred contraction of the range to construct such an equivariant chain map, beginning with $\iota$ in degree 0.  These two maps are the same.\\

For an abelian group $G$, the $\ell$-th power map is a homomorphism.  Let us denote by $\iota_\ell$  the $\ell$-th power map for the cyclic group, $\iota_\ell(T) = T^\ell$.  Thus for the cyclic group $C_p $ one has a preferred $\iota_\ell$-equivariant self-map $\iota_\ell \colon  N_*(EC_p) \to N_*(EC_p)$.  It is pretty easy to see when $p$ is prime what the induced map in homology and cohomology with $\FF_p$ coefficients on the coinvariant complex $N_*(BC_p)$ must be.  It is the identity in degree 0 and multiplication by $\ell$ in degrees 1 and 2.  The structure of the cohomology ring then determines the cohomology map, and hence also the homology map.  In degrees $2k$ and $2k-1$, the homology map is multiplication by $\ell^k$.  However, it is not so clear exactly how the images of the classes $\bar{x}_{2k}$ and  $\bar{x}_{2k-1}$ under this map differ by  explicit boundaries from $\ell^k \bar{x}_{2k}$ and $\ell^k \bar{x}_{2k-1}$ in $N_*(BC_p)$.\\

To begin clarifying this, we  construct by the standard method an $\iota_\ell$-equivariant chain map $\lambda \colon M_* \to M_*$ extending $\lambda (y_0) =  y_0$. Then $\lambda T^iy_0 = T^{i\ell }y_0$.  We remind that the contraction of $M_*$ for the cyclic group $C_p$  is given by  the formulas:
$$h(T^i y_{2k}) = \sum_{0 \leq j < i} T^j y_{2k+1}\hspace{.5in}  [So\  h(y_{2k}) = 0\ and\ h(Ty_{2k}) = y_{2k+1}]$$
$$h(T^i y_{2k+1}) = 0\ if\ i < p-1\  and\  h(T^{p-1}y_{2k+1}) = y_{2k+2}.$$
Then $d y_1 = (T-1)y_0$, so $$\lambda (y_1) = h \lambda (d y_1) = h(T^\ell - 1) y_0 = \sum_{i=0}^{\ell-1} T^i y_1.$$
Next, using $d y_2 = N y_1$, where $N = \sum T^i$ is the norm, one gets from $NT^i = N$ that $\lambda(y_2) = \ell y_2$.  The inductive continuation is easy.\footnote{The $\iota_\ell$-equivariant chain map $\lambda$ was  exploited in Steenrod-Epstein [30] in the same way that we will use it.  They wrote down a formula and proved it was a chain map.  We knew the standard procedure produced a chain map and then we easily found its formula.}
\begin{prop}\label{6.15} The $\iota_\ell$-equivariant chain map $\lambda \colon M_* \to M_*$ is given by 
$$\lambda (y_{2k+1}) = \sum_{i=0}^{\ell-1} \ell^k T^i y_{2k+1} \ \ \ and \ \ \  \lambda(y_{2k}) = \ell^k y_{2k}.$$ 
\end{prop} 
We will postpone until Section 9 the use of Proposition \ref{6.14} to show that the images of the cycles $\bar{x}_{2k}, \bar{x}_{2k-1} \in N_*(BC_p)$ under the $\ell^{th}$ power map $\iota_\ell \colon N_*(BC_p) \to N_*(BC_p)$ differ by  explicit boundaries from $\ell^k \bar{x}_{2k}, \ell^k \bar{x}_{2k-1}$.  $\qed$
\end{exam}

\subsection{*Tensor Products of Standard Procedure Maps*}
\begin{exam}\label{6.16} \textbf{Tensor Products.}  Suppose we have constructed $H$ and $G$ equivariant chain maps  by the standard procedure  $\alpha \colon A_* \to C_*$ and $\beta \colon B_* \to D_*$.  Assume $A_*$ is free over $\FF[H]$, $B_*$ is free over $\FF[G]$, and suppose, as always,   contractions of $C_*$ and $D_*$ satisfy the assumptions $h^2 = 0, h \iota = 0$.

\begin{prop}\label{6.17}(i). With the assumptions above, $\alpha \otimes \beta \colon A_* \otimes B_* \to C_* \otimes D_*$, is the standard procedure $H \times G$ equivariant chain map  constructed using the tensor product $\FF[H \times G]$-basis of the domain and the tensor product contraction of the range.  If $\alpha$ and $\beta$ commute with contractions then so does $\alpha \otimes \beta$. By induction, the corresponding statements hold for any number of tensor product factors.\\

(ii). Assume basis elements of $A_*$ and $B_*$ are in $Im(h)$,  with $A_0 = \FF[H]$ and $B_0 = \FF[G]$.  Then the canonical  isomorphism $\tau \colon A_* \otimes B_* \to B_* \otimes A_*$, equivariant with respect to $H \times G \to G \times H$,  is the standard procedure chain map. The corresponding statement holds for a permutation isomorphism with any number of tensor factors.
\end{prop}

\begin{proof} These claims are exercises in the definitions, although they are a little tricky in low degrees where the basepoints of the complexes enter the computations.  A variant of this result will play a role towards the end of Part III of our paper when we investigate certain chain complex operads, so we give some details here.\\

(i). In degree 0, `standard procedure' means on basis elements $\alpha(a) = \iota_C \epsilon_A(a)$ and $\beta(b) = \iota_D \epsilon_B(b)$.  Thus $(\iota_C \otimes \iota_D) (\epsilon_A \otimes \epsilon _B )(a \otimes b) = \alpha(a) \otimes \beta(b)$.  Also, as pointed out at the beginning of this section,  $\rho_C \alpha(a) = \alpha (a)= \alpha \rho_A(a)$, both of which follow easily from $\epsilon \iota = Id$.\\
 
Now we could proceed by induction on basis elements, using $\alpha(a) = h_C \alpha (da)$ and $\beta(b) = h_D \beta(db)$ in positive degrees.   If $\phi$ denotes the standard procedure chain map on the tensor product, then $\phi(a \otimes b) = h_\otimes \phi(da \otimes b +(-1)^{|a|} a \otimes db)$, and we can apply induction to $\phi$ in one lower dimension.\\ 

But this proof is a little tricky in low dimensions, so we will leave that as an exercise.  Instead we will deduce the result more quickly from the uniqueness result Proposition \ref{6.2}.  Certainly $\alpha \otimes \beta$ is an equivariant chain map.  We need to show that for basis elements in positive degree, $\alpha(a) \otimes \beta(b) \in Im(h_{C_* \otimes D_*})$. If $deg(a) = 0$ then $$\alpha(a) \otimes \beta(b) = \rho_C\alpha(a) \otimes h_D \beta(db) = h_\otimes(\alpha(a) \otimes \beta(db)).$$  This does use $h_C(\alpha(a)) = h_C\iota_C\epsilon_A(a) = 0$. Similarly, if $deg(b) = 0$ then $$\alpha(a) \otimes \beta(b) = h_\otimes (\alpha(da) \otimes \beta(b)).$$  This uses $h_D(\beta(b)) = h_D \iota_D \epsilon_B(b) = 0$. Finally, if $deg(a), deg(b)$  are both positive, $$\alpha(a) \otimes \beta(b) = h_A \alpha(da) \otimes h_B \beta(db) = h_\otimes(\alpha(da) \otimes h_B \beta(db)).$$

It is routine to see that tensor products of maps that commute with contractions also commute with contractions.  This does use $\alpha \circ \rho_A = \rho_C \circ \alpha$.\\

(ii).  Again, it is fairly easy to see that the statement can be deduced directly from the definitions, with some care needed in degree 0 and some care with Koszul signs.  That proof will be left as an exercise.  But with the hypotheses about basis elements $a$ and $b$ belonging to $Im(h)$ in the two complexes,  it is quite easy to see that the equivariant chain map $\tau$ with $\tau(a \otimes b) = (-1)^{|a||b|} b \otimes a$ maps basis elements to elements in $Im(h_\otimes)$.  Namely, $hx \otimes hy = h_\otimes(x \otimes hy)$. Therefore Proposition \ref{6.2} applies again.   The result with more than two factors can be proved directly, or by composing isomorphisms given by swapping adjacent pairs.  Such compositions will be standard procedure maps by  Proposition  \ref{6.5}(iii), since up to sign  basis elements map to basis elements in each swap map.

\end{proof}
\end{exam}

\newpage

\section{Diagonal Chain Maps from Contractions}
\subsection{Diagonals for MacLane Models}
Given a contraction of a complex $C_*$, we constructed preferred contractions of the multiple tensor products $C_* ^{\otimes k}$ in Example \ref{5.5}.  If $C_*$ is free, we can use the standard recursive procedure to construct (equivariant) diagonal chain maps $C_* \to C_* ^{\otimes k}$, beginning with an obvious map in degree 0.  In this section we give two examples. The first example will be the MacLane model, from Example \ref{5.3}.

\begin{exam}\label{7.1} Since $N_*(EG)$ and $N_*(EG \times EG)$ are MacLane models, we can use the $G$-equivariant map $AW$ constructed in Example \ref{6.6}  to   construct a $G$-equivariant diagonal 
$$\delta = AW \circ \delta_\times \colon N_*(EG) \to N_*(EG \times EG) \to N_*(EG) \otimes N_*(EG),$$ where  $\delta_\times$ is the obvious diagonal map $N_*(EG) \to N_*(EG \times EG)$ induced by the map of cells of simplicial sets $\gamma \mapsto (\gamma, \gamma)$.  Both $AW$ and $\delta_x$ are equivariant and commute with contractions, hence the same is true of $\delta$.  Thus by Proposition \ref{6.5}(i) we get the Alexander-Whitney diagonal formula  $$ \delta(g_0, g_1, \ldots,  g_n) = \sum_{j=0}^{j=n} (g_0, \ldots, g_j) \otimes (g_j, \ldots, g_n)$$ as a standard procedure map.   Of course this is essentially just a special case of the two variable $AW$ map for MacLane models, and could have been discussed directly.\\

Next, using the recursive contractions  $ h^{(k+1)}  =   h \otimes Id^{(k)} + \rho \otimes h^{(k)} $ of $N_*(EG)^{\otimes (k+1)}$, one constructs $G$-equivariant  higher diagonals  $$\delta^{(k+1)} = AW^{(k+1)} \circ \delta^{(k+1)}_\times\colon N_*(EG) \to N_*(EG\times \ldots \times EG) \to N_*(EG)^{\otimes (k+1)}.$$ These also turn out to be the higher iterated Alexander-Whitney diagonals.    That is, $ \delta^{(k+1)} = ( Id \otimes \delta^{(k)})\circ \delta.$ The iterated diagonals are coassociative.  A simple proof is to observe that the iterated Alexander-Whitney diagonals are equivariant and commute with $h$ and $h^{(k+1)}$, so Propositions \ref{6.3} and \ref{6.5}(i) apply again. The formula for the iterated diagonal is, of course, $$\delta^{(k+1)}(g_0, g_1, \ldots, g_n) = \sum(g_0, \ldots, g_{i_1}) \otimes (g_{i_1}, \ldots, g_{i_2}) \otimes \ldots \otimes (g_{i_{k}}, \ldots, g_n),$$ where the sum is over all the indicated ordered overlapping splittings of $(g_0, \ldots, g_n)$ into $k+1$ tensor factors. $\qed$

\end{exam}
\subsection{Diagonals for Minimal Models}
\begin{exam}\label{7.2}
Next we study diagonals for the minimal complex $M_*$ for the cyclic group $C_n$ from Example \ref{5.4}.  We build a $C_n$-equivariant diagonal $\Delta \colon M_* \to M_* \otimes M_*$ using freeness of the domain and the chosen contraction of the range.  That contraction is $h^{(2)} =  h \otimes Id + \rho \otimes h  $, where $h$ is the contraction of $M_*$ defined previously, and given by the formulas:
$$h(T^i y_{2k}) = \sum_{0 \leq j < i} T^j y_{2k+1}\hspace{.5in}  [So\  h(y_{2k}) = 0\ and\ h(Ty_{2k}) = y_{2k+1}]$$
$$h(T^i y_{2k+1}) = 0\ if \  i < n-1\  and\  h(T^{n-1}y_{2k+1}) = y_{2k+2}.$$

We begin the diagonal with $\Delta y_0 = y_0 \otimes y_0$.  The recursive construction is $\Delta y_j = h^{(2)} \Delta (d y_j)$. We remind that the differential in $M_*$ is given by $dy_{2k+1} = (T-1)y_{2k}$ and $dy_{2k} = Ny_{2k-1}$, where $N = 1+T + \ldots T^{n-1}$.\\

Before writing down some formulas, let's make a prediction for $n = p$ prime and ground ring $\FF_p$ the field of order $p$.    Let $\bar{M}_* =  M_*/ C_p$ be the coinvariant complex.  So there is one generator $\bar{y}_j$ in each degree and the boundary operator is 0. \\

The prediction is that the diagonal of $M_*$ should cover the diagonal of $\bar{M}_*$ given by
$$ \Delta( \bar{y}_{2k}) = \sum_{i+j = k} \bar{y}_{2i} \otimes \bar{y}_{2j}\ \ \ and\ \ \  \Delta( \bar{y}_{2k+1}) = \sum_{i+j = 2k+1} \bar{y}_i \otimes \bar{y}_j.$$
The reason for this prediction is the known cohomology algebra of the cyclic group with $\FF_p$ coefficients.\\

Now let's explicitly compute the diagonal associated to the  contraction $h^{(2)} =  h \otimes Id + \rho \otimes h $ of $M_* \otimes M_*$.  
\begin{itemize}
\item$\Delta(y_0) = y_0 \otimes y_0,\ \ \ \  \Delta(y_1) = h^{(2)}\Delta(T-1)y_0 = y_0 \otimes y_1 + y_1 \otimes Ty_0$
\item$\Delta(y_2) = h^{(2)} \Delta N y_1 = y_0 \otimes y_2 +       y_2 \otimes y_0  +  \bigg( \displaystyle \sum_{i=1}^{n-1} h(T^iy_0) \otimes T^iy_1 \bigg)   $
\item$\Delta(y_3) = h^{(2)}\Delta (T-1)y_2 = y_0 \otimes y_3 + y_1 \otimes Ty_2 + y_2 \otimes y_1 + y_3 \otimes Ty_0$
\item$\Delta(y_4) = h^{(2)}\Delta N y_3 =  y_0 \otimes y_4 +    y_2 \otimes y_2   + y_4 \otimes y_0 \\ + \bigg( \displaystyle \sum_{i=1}^{n-1} h(T^iy_0) \otimes T^iy_3  \bigg)  + \bigg(  \displaystyle  \sum_{i=1}^{n-1} h(T^iy_2) \otimes T^iy_1 \bigg) $ 

\end{itemize}
The pattern for the odd $y_{2k+1}$ as $2k+1$ increases  is transparent.  The pattern for the even $y_{2k}$ contains the transparent terms with even subscripts, then groups of terms with odd subscripts, each group projecting to 0 in $\bar{M}_* \otimes \bar{M}_*$ when $n = p$ is prime.  In fact, each group turns out to be  a sum of $1+2+\ldots + p-1 = p(p-1)/2$ terms, all with the same projection to  $\bar{M}_* \otimes \bar{M}_*$.   The pattern of those groups of  odd subscript  terms is also transparent as $2k$ increases.   Specifically, for each $j+\ell = k,\ \ell > 0,$ there is a group of terms  $$\displaystyle \sum_{i=1}^{p-1} h(T^iy_{2j})\otimes T^iy_{2\ell-1}  = \sum_{i = 1}^{p-1} \bigg(y_{2j+1} + Ty_{2j+1} + \ldots + T^{i-1}y_{2j+1}\bigg)\otimes T^i y_{2\ell-1}.$$   So the prediction about the projection of the diagonal  holds true.\\

The diagonal calculation method above works for any $n$ and any ground ring $\FF$.  Once  discovered, the formula is easily proved by induction on degrees of the $M_*$ generators. There are really just two cases, the even degree generators and the odd degree generators. Here is the result.\footnote{This diagonal on the minimal model was also exploited in Steenrod-Epstein [30].  In that text they just wrote down the formula for $\Delta$, then proved it was a chain map.  We constructed a chain map by the standard recursive procedure, then found its formula.}
\begin{prop}\label{7.3} The standard contraction procedure diagonal $\Delta \colon M_* \to M_* \otimes M_*$ is given by $$\Delta(y_{2k+1}) = y_0 \otimes y_{2k+1} + y_1 \otimes Ty_{2k} + \cdots + y_{2k} \otimes y_1 +  y_{2k+1} \otimes Ty_0$$
$$\Delta(y_{2k}) = y_0 \otimes y_{2k} + y_2 \otimes y_{2k-2} + \cdots  + y_{2k} \otimes y_0
 + \bigg( \displaystyle  \sum_{i=1}^{n-1} h(T^iy_0) \otimes T^iy_{2k-1} \bigg)$$  $$ + \bigg( \displaystyle \sum_{i=1}^{n-1} h(T^iy_2) \otimes T^iy_{2k-3}  \bigg)   +    \cdots + \bigg(\displaystyle \sum_{i=1}^{n-1} h(T^iy_{2k-2}) \otimes T^iy_1 \bigg). \qed $$
\end{prop}
\end{exam}

\begin{exam}\label{7.4} {\bf Multidiagonals for Minimal Models.} Let us turn to multidiagonals for $M_*$.  In general, given a free complex $C_*$ with a contraction $h$, one also has the contraction $h^{(2)} =  h \otimes Id + \rho \otimes h $ of $C_*  \otimes C_*$, from which one constructs a diagonal $\Delta$ on $C_*$. But this diagonal is not necessarily coassociative.  Given $\Delta$, one can  build by iteration various higher diagonals  $C_* \to C_*^{\otimes k}$.  For example, one choice is $\Delta^{(k+1)} = (Id \otimes \Delta^{(k)}) \circ \Delta$.  But one also has preferred  contractions of the $C_* ^{\otimes(k+1)}$ given by $h^{(k+1)} =   h \otimes Id^{\otimes k} + \rho \otimes h^{(k)}  $, and these contractions, which do have a strong associativity property,  directly lead to higher diagonals.  Are these contraction diagonals related to higher diagonals given by simply iterating the original $\Delta$?\\

{\bf Iterated vs Standard Procedure Multidiagonals.} Some bad news is that the diagonal for $M_*$  itself is not coassociative. The two ways of associating to form triple diagonals already disagree on $y_2$.\footnote{Curiously, when $n = 3$ the other triple diagonal $(\Delta \otimes Id) \circ \Delta(y_2)$ is a sum of twelve terms, and exactly one of them, $Ty_0 \otimes Ty_1 \otimes T^2y_1$, is not in $Ker(h) = Im(h)$.}
 Of course the two equivariant triple diagonals are equivariantly chain homotopic, which is easy to see explicitly up to $y_2$ and perhaps in general for $M_*$.  But we definitely have a preferred contraction of $M_*^{\otimes (k+1)}$, so that associated diagonal is the one we always want.\\

Some good news is that in the case of the cyclic group $C$, it works out that the preferred triple diagonal of the minimal complex, $M_* \to M_*^{\otimes 3}$, agrees with the iterated diagonal $(Id \otimes \Delta) \circ \Delta$.  More generally, we have the following. Note statement (ii) is consistent with the structure of $H^*(BC)$.
\begin{prop}\label{7.5}(i). The diagonal associated to the contraction $h^{(k)}$ of $M_*^{\otimes k}$ is the iterated diagonal $\Delta^{(k)} = (Id \otimes \Delta^{(k-1)}) \circ \Delta$.\\

(ii). The induced iterated diagonal for the coinvariant complex $\bar{M}_*^{\otimes k}$ is 
$$\bar{\Delta}^{(k)} (\bar{y}_{2\ell}) = \sum_{\sum a_j = 2 \ell} \bar{y}_{a_1} \otimes \bar{y}_{a_2} \otimes \ldots \otimes \bar{y}_{a_k},\ where\ all\  a_j\ are\ even.$$
$$\bar{\Delta}^{(k)} (\bar{y}_{2\ell+1}) \sum_{\sum a_j = 2 \ell+1} \bar{y}_{a_1} \otimes \bar{y}_{a_2} \otimes \ldots \otimes \bar{y}_{a_k},\ where\ exactly\ one\ a_j\ is\ odd.$$
\end{prop}
 \begin{proof} The proof of (i) is an application of Proposition \ref{6.2}.  We first look at the triple diagonal. On basis elements $y_m$, the first tensor factor of any summand of $\Delta (y_m)$ is always in $Im(h) = Ker(h)$. Occasionally the first tensor factor is $y_0$, with $\rho(y_0) = y_0$, which we interpret as $h(1) = \iota(1)$.  Then one looks at the second tensor factor of summands of $\Delta(y_m)$ to compute $(Id \otimes \Delta) \circ \Delta(y_m)$.  One sees the image of the iterated diagonal on basis elements $y_m$ is a sum of clean tensors, that is, tensors in the image of the summands of $h^{(3)} = h \otimes Id \otimes Id + \rho \otimes h \otimes Id + \rho \otimes \rho \otimes h$.  The first tensor factor of most summands of $\Delta(y_m)$  is in $Im(h)$, but one also sees summands with first factor $y_0$ and second factor in $Im(h)$,  and one summand $y_0 \otimes y_0 \otimes y_m$.\\

Computationally it seems  easier to iterate the basic diagonal $\Delta\colon M_* \to M_* \otimes M_*$, rather than directly use the contraction of $M_*^{\otimes 3}$. \\

The argument for the triple diagonals of $M_*$ extends to higher diagonals.  That is, the diagonal associated to the preferred contraction $h^{(k)}$ of $M_*^{\otimes k}$ is the iterated diagonal $$\Delta^{(k)} = (Id \otimes \Delta^{(k-1)}) \circ \Delta = \cdots \circ(Id \otimes Id \otimes \Delta) \circ (Id \otimes \Delta) \circ \Delta.$$  The proof is an induction, but is essentially the same as the $k = 3$ case, also based on the uniqueness result Proposition \ref{6.2} and the formula for $h^{(k)}$.\\

The formulas in statement (ii) follow by an easy induction using statement (i) and the basic diagonal formula in Proposition \ref{7.3}.
   \end{proof}

Eventually we will need to use the $p$-fold tensor diagonal for the minimal model  $M_*$ in order to study the Steenrod $p^{th}$ power operations. The next subsection introduces some of the key ideas.
\end{exam}
\subsection{*Equivariant Steenrod Extensions of Multidiagonals*} 

The cyclic group $C = C_p = <T>$ acts in two ways on $M_*^{\otimes p}$.  First, there is the rotation action $T_R(u_1\otimes u_2 \otimes \ldots \otimes u_p) = (u_p \otimes u_1 \ldots \otimes u_{p-1})$.  This is the restriction of the $\Sigma_p$ action on $M_*^{\otimes p}$ discussed in Section 3 to the subgroup generated by the cyclic permutation $t = (12...p)$.  We remind that $t(u_1\otimes u_2 \otimes \ldots \otimes u_p) = (u_{t^{-1}(1)} \otimes u_{t^{-1}(2)} \otimes  \ldots \otimes u_{t^{-1}(p)})$.  Secondly, there is the diagonal action $T_\Delta$, for which the multidiagonal $\Delta^{(p)} \colon M_* \to M_*^{\otimes p}$ is equivariant. The rotation action and diagonal action commute, thus $C \times C = <T_R, T_\Delta>$ acts on $M_*^{\otimes p}$.\\

The multidiagonal  is not equivariant for the rotation action.  However $C \times C$ acts freely on $M_* \otimes M_*$ and we have a preferred contraction $h_\otimes$ of $M_*^{\otimes p}$, hence we have a standard procedure $C \times C$ equivariant chain map $\gamma \colon M_* \otimes M_* \to M_* ^{\otimes p}$ that begins with the multidiagonal $\Delta^{(p)} \colon \{y_0\} \otimes M_* \to M_*^{\otimes p}$. On $C \times C$ basis elements $b = y_i \otimes y_j$ the inductive formula is $\gamma(b) = h_\otimes \gamma (db)$, which is then extended equivariantly. In principle the map $\gamma$ is easy to recursively program, but for odd primes $p$ the calculations quickly get unmanageably large and no nice patterns emerge on which to base a guess of a formula.\\

There is also a $C \times C$ equivariant  standard procedure chain map $\gamma \colon N_*(EC) \otimes N_*(EC) \to N_*(EC)^{\otimes p}$ beginning with the multidiagonal on $\{e\} \otimes N_*(EC)$.  We have the standard procedure inclusion $\phi \colon M_* \to N_*(EC)$ from Example \ref{6.8} and Proposition \ref{6.9} and the standard procedure retraction $\pi \colon N_*(EC) \to M_*$ that commutes with contractions from Example \ref{6.11} and Proposition \ref{6.12}.

\begin{prop}\label{7.6} The diagram below commutes, where $\phi\colon M_* \to N_*(EC)$ is the standard procedure map from Proposition \ref{6.9}, the $\gamma$ are the contraction based standard procedure maps, and $\pi \colon N_*(EC) \to M_*$  is the standard procedure retraction from Proposition \ref{6.12}.

$$\begin{CD}  N_*(EC) \otimes N_*(EC) & \xrightarrow{\gamma_{EC}}&  N_*(EC)^{\otimes p}  \\ 
 \uparrow \phi \otimes \phi &  & \downarrow {\pi ^{\otimes p}}\\
 M_* \otimes M_* & \xrightarrow{\gamma}\ \   &  M_*^{\otimes p} 
\end{CD}$$
\end{prop}

\begin{proof} The vertical maps are identities in degree 0. The composition $\gamma_{EC} \circ (\phi \otimes \phi) $ obviously satisfies the condition of the uniqueness theorem Proposition \ref{6.2} by the formula for $\phi $ in Proposition \ref{6.9} along with Proposition \ref{6.5} (iii). That is, the maps $\phi$ and $\phi\otimes \phi$  take basis elements to sums of basis elements. Then we post-compose with $\pi^{\otimes p}$. The point then is $\pi^{\otimes p}$ commutes with contractions, hence by Proposition \ref{6.5} (ii) this second  composition is also a standard procedure map, hence must be $\gamma$.
\end{proof}

We can divide both complexes on the left side of the diagram by the second $C$ factor, which acts diagonally on both complexes on the right side. Then we can divide the right side by the product $C^p$  and induce a commutative diagram, where the horizontal maps are recursively defined, as quotients of the $N_*(EC)$ and $M_*$ diagram.
$$\begin{CD}  N_*(EC) \otimes N_*(BC) & \xrightarrow{\bar{\gamma}_{EC}}&  N_*(BC)^{\otimes p}  \\ 
 \uparrow \phi \otimes \bar{\phi} &  & \downarrow {\bar{\pi} ^{\otimes p}}\\
 M_* \otimes \bar{M}_* & \xrightarrow{\bar{\gamma}}\ \   &  \bar{M}_*^{\otimes p} 
\end{CD}$$

{\bf Preview of a Formula for $\bar{\gamma}$.}  Eventually in Section 17 of Part III of this work we do prove a formula for $\gamma_{EC}$, and thus also for $\bar{\gamma}_{EC}$ and $\bar{\gamma}$,  due to Berger and Fresse.  The details of that formula are crucial for our cochain study of Steenrod operation in Part IV, especially for the Adem relations. \\

Nonetheless, without any formulas and without any reference to $EC$ and $BC$ we can prove the following result about the chain map $\bar{\gamma}$ for all $p$ in the same manner that Adem proved it. Let  $T_R$ and $N_R$ denote the  rotation and the sum of rotations norm $N_R = \sum T_R^\ell$ acting on $\bar{M}_*^{\otimes p}.$\\

Basic tensors in $\bar{M}_*^{\otimes p}$ can have the symmetric form $\bar{Y}_\ell = (\bar{y}_\ell \otimes \bar{y}_\ell \times \ldots \otimes \bar{y}_\ell)$ if the total degree is $\ell p$.  Consider other basic non-symmetric tensors  of the form $\bar{Y}_I = (\bar{y}_{i1} \otimes \bar{y}_{i2} \otimes \ldots \otimes \bar{y}_{ip})$, where the tuple $I = (i_1, i_2, \ldots, i_p)$ appears first among all its cyclic rotations $T^j I$ in the lexocographic order with $0 < 1 < \cdots$.

 \begin{prop}\label{7.7} The map $\bar{\gamma} \colon M_* \otimes \bar{M}_* \to \bar{M}_*^{\otimes p}$ has the following form.  There are  constants $c_\ell, c_I, c'_\ell, c_{I,i} \in \FF_p$, which are functions of $k$ and $j$, with\\
 
  $\bar{\gamma}(y_{2k} \otimes \bar{y}_j) = \sum c_\ell \bar{Y}_\ell + \sum c_I N_R(\bar{Y}_I)$.\\
  
 $\bar{\gamma}(y_{2k+1} \otimes \bar{y}_j) =  \sum c'_\ell \bar{Y}_\ell + \sum c_{I, i}(T_R^i - Id) (\bar{Y}_I)$.\\
 
 Of course, symmetric tensors  only appear when the degree $2k+j$ or $2k+1+j$ is a multiple  $\ell p$. It is also understood that the degrees on the left side match the degrees of the $\bar{Y}_I$ on the right side.
 \end{prop}

\begin{proof} The main point is, since the differential is 0 in both $\bar{M}_*$ and $\bar{M}_*^{\otimes p}$,  and since $\bar{\gamma}$ is equivariant for the rotation action, we have two formulas depending on the parity of degrees of generators for the boundary operator in $M_*$,  $$0 = \bar{\gamma}(d (y_{2k+1} \otimes \bar{y}_j)) = (T_R-1)(\bar{\gamma} (y_{2k}  \otimes \bar{y}_j)) \in \bar{M}_*^{\otimes p}$$ $$0 = \bar{\gamma}(d (y_{2k} \otimes \bar{y}_j)) = N_R(\bar{\gamma} (y_{2k-1}  \otimes \bar{y}_j)) \in \bar{M}_*^{\otimes p}.$$ 
 
 Now, the graded $C$-module $\bar{M}_*^{\otimes p}$ decomposes as a direct sum of $C$-submodules with $\FF$-basis either a $T_R$-orbit $\{ T_R^i(\bar{y}_{i1}  \otimes  \bar{y}_{i2} \otimes \ldots  \otimes  \bar{y}_{ip}), 0 \leq i \leq p-1 \}$ or a $T_R$-fixed generator $(\bar{y}_\ell \otimes \bar{y}_\ell \otimes \ldots \otimes  \bar{y}_\ell)$. The $T_R$-action rotates tensor factors with Koszul signs. If $p$ is odd then $p-1$ is even, hence the symmetric generators are fixed by $T_R$ and are in the kernel of both $T_R - 1$ and $N_R$\\
 
 Since non-trivial orbit summands  are $C$-isomorphic to copies of $\FF_p[C]$, on such summands $Ker(T_R - 1) = Im(N_R)$ and $Ker(N_R) = Im(T_R - 1)$. In fact, in $\FF_p[C]$ elements in $ker(T - 1)$ are multiples of $N = \sum T^\ell$ and elements in $Ker(N)$ are elements $\sum a_iT^i$ with $\sum a_i = 0$, which are  uniquely sums of multiples of $T^{i+1} - T^i = (T-1)T^i$.  Another basis is given by the $T^i - Id$.
 \end{proof}

{\bf Discussion of the Unknown Constants.} When $k = 0$ one can just look at the simple formula for the diagonals $\bar{\Delta}^{\otimes p}(\bar{y_j})$ in Proposition \ref{7.5}(ii) and see these diagonals are  $T_R$-invariant, which are thus sums of non-trivial $T_R$ orbits, hence norms with coefficients 1, plus a symmetric term with coefficient 1 when $j = 2\ell p$ or $j = (2\ell +1)p$. What about the higher degree constants in Proposition \ref{7.7}?  First it is  necessary to develop a  {\it functorial} standard procedure $C$-equivariant chain map $\gamma_X \colon M_* \otimes N_*(X) \to N_*(X)^{\otimes p}$ for simplicial sets $X$, where  the cyclic group acts by rotations on the tensor power range. In Section 8 below we describe a general functorial recursive standard procedure, and  in Subsection 8.4 we specialize to maps that include the  $\gamma_X$.  We also prove there that the contraction based standard procedure map $\gamma_{EC}$ and its quotient $\bar{\gamma}_{EC}$ from the commutative diagrams above agree with the functorial standard procedure maps for the simplicial sets $EC$ and $BC$.\\

 The formula of Berger and Fresse [3], which we will prove in Section 17, is  a formula for all $\gamma_X$, hence includes the recursively defined  standard procedure map $ \bar{\gamma}_{EC} = \gamma_{BC} \colon M_* \otimes N_*(BC) \to N_*(BC)^{\otimes p}$. Therefore, one obtains a formula for $\bar{\gamma} \colon M_* \otimes \bar{M}_* \to \bar{M}^{\otimes p}$ from the commutative diagrams above.\\

The retraction  $\bar{\pi}$ is made quite precise at the end of Subsection 6.5. However, the problem with the Berger-Fresse formula is that it seems quite obscure how the many seemingly random basis summands that occur in the resulting expressions for $\bar{\gamma}( y_k \otimes \bar{y}_j)$ can be organized as sums of specific constants times symmetric tensors and specific constants times  $T_R^i -1$ or $N_R$ applied to non-symmetric tensors, which we have proved is always possible.\\

Constructively, one could just write out all the summands in the Berger-Fresse formula for $\gamma_{BC}$ and count multiples of the same basis tensors.  We would only need the counts modulo $p$.  Adem realized that the Cartan formula implies the  mod $p$ values of the constants for the symmetric terms, which is sufficient for the cohomological Adem relations.  But for the explicit coboundary form of Adem relations it would be nice to know the other coefficients.\\

When $p = 2$ the map $\gamma_X \colon M_* \otimes N_*(X) \to N_*(X)^{\otimes 2}$ is exactly Steenrod's  map from the 1940's  that defines the $\cup_i$ operations.   A simplifying feature is that for $p = 2$, $M_* = N_*(EC_2)$ the MacLane model, which is a simplicial set model.  For odd primes the MacLane model $N_*(EC)$ is much bigger than $M_*$, which is not a simplicial set model,  and the chain complex of the classifying space $N_*(BC)$ is also much bigger than the quotient $\bar{M}_* =  M_*/ C$.\\

In the paper [7] on $p = 2$ Adem relations we gave explicit combinatorial formulas for Steenrod's maps $\gamma_X \colon N_*(EC_2) \otimes N_*(X) \to N_*(X)^{\otimes 2}$ for $X = EC_2$ and $X = BC_2$.   Formulas for $\gamma_{EC}$ and $\gamma_{BC} = \bar{\gamma}_{EC}$  for odd $p$ are much more elusive. Nonetheless, determining the constants of Proposition \ref{7.7} from the Berger-Fresse formula is just a problem in combinatorics, so will most likely eventually be solved by somebody.

 \newpage

\section{ Functorial  Contraction Procedure Chain Maps}

In this section we discuss the method of  `minimal contractible carriers' that constructs natural transformations of functors when given preferred contractions of various complexes associated to minimal carriers.

\begin{defn}\label{8.1}{\bf  The Functorial Recursive Standard Procedure.} Suppose that we have a free $\FF[G]$ chain complex functor $F_*(X_1, \ldots, X_k)$, and an $\FF[G]$ chain complex functor $K_*(X_1, \ldots, X_k)$, both of the same number of simplicial set variables.  We suppose we have a preferred contraction of  $K_*$ when the  $X_i = \Delta_i$ are simplices, with  $h_K^2 = 0$.  We suppose that in degree 0  there is an obvious  equivariant map  $F_0 \to K_0$ when the $X_i$ are points, related to augmentations and base points for the complexes $F_*$ and $K_*$.  We suppose that there is a preferred  $\FF[G]$ basis of elements $u \in F_*(X_1, \ldots, X_k) = F_*(X_i)$ so that the sets $\{u, 0\}$ of basis elements together with 0 are functorial in the $X_i$.  Finally we assume there are universal elements $\bar{u} \in F_*(\Delta^{a_1(u)}, \ldots, \Delta^{a_k(u)}) = F_*(\Delta^{a_i(u)})$ for suitable simplices,  and minimal carrier maps $\sigma_i \colon  \Delta^{a_i(u)} \to X_i$ canonically associated to $u$, inducing chain maps  $\sigma_{\#} \colon F_*(\Delta^{a_i(u)}) \to F_*(X_i)$ with  $\sigma_{\#} (\bar{u}) = u$.  We then define an equivariant natural transformation $$\phi_{func} \colon  F_*(X_1, \ldots, X_k) \to K_*(X_1, \ldots, X_k)$$ inductively on degree  by extending equivariantly  the formula $$\phi_{func} (u) = \sigma_{\#} h_{K_*(\Delta^{a_i(u)})}(\phi_{func}(d \bar{u})). \ \ \ \ \qed$$ 
\end{defn} 
The above paragraph may be a bit vague, but we assume readers are familiar with the classical minimal acyclic carrier construction. The main point is that the $\bar{u}$ and the $\sigma_i$ are determined by $u$, so $\phi_{func}(u)$ is well-defined. We gave a similarly vague discussion in Section 1.  The important thing is that in all of our examples the details are quite simple and ambiguities disappear.\\

The  difference with the classical viewpoint is that here we don't make  new arbitrary choices of chains with given boundary for the minimal contractible carriers at each inductive step.   Instead our contractions $h_K$ of $K_*(\Delta^{a_i(u)})$ make the choices for us. The basic property of contractions, $dh_K + h_K d = Id$ in positive degrees and on $Im(d_1)$, and the usual  acyclic model arguments, imply that our procedure does define an equivariant  chain map $F_*(X_1, \ldots, X_k) \to K_*(X_1, \ldots, X_k)$ functorial in the $X_i$. 

\subsection{Uniqueness Theorems for  Functorial Chain Maps}
There is a variant of the uniqueness result Proposition \ref{6.2} in the functorial setting.
\begin{prop}\label{8.2} Suppose $\psi \colon F_*(X_i) \to K_*(X_i)$ is an equivariant functorial chain map of $\FF[G]$ complexes, agreeing with the standard functorial procedure map $\phi_{func}$ in degree 0.  Suppose also that the universal elements in minimal carriers $\bar{u} \in F_*(\Delta^{a_i(u)})$ satisfy $\psi(\bar{u}) \in Im(h_K)$. Then $\psi = \phi_{func}$.
\end{prop}
\begin{proof} The proof is an induction, essentially identical to the proof of Proposition \ref{6.2}, applied to the universal basis elements $\bar{u}$.
\end{proof}

There is also a  variant of the commuting with contractions  result Proposition \ref{6.3} in the functorial setting. 
\begin{prop}\label{8.3} Suppose $\psi \colon F_*(X_i) \to K_*(X_i)$ is an equivariant natural transformation of the underlying graded $\FF[G]$ modules that  agrees with $\phi_{func}$ in degree 0.  Assume $\psi$ commutes with contractions $h_F$ and $h_K$  and that either $h_F^2 = 0$ or $h_K^2 = 0$ when the $X_i =  \Delta^{n_i}$ are simplices.    Finally suppose  that for each degree the universal elements  $\bar{u} \in F_*(\Delta^{a_i(u)})$ consist of elements in the image of $h_F$.  Then $\psi = \phi_{func}$.  In particular, $\psi$ is automatically a chain map.
\end{prop}Again the proof is an induction, essentially the same as the proof of Proposition \ref{6.3}. $\qed$

\begin{rem}\label{8.4}  We do not seem to have interesting examples of Proposition \ref{8.3} when $G$ is non-trivial, although there might be some. The sticky point is the commuting with contractions for simplices hypothesis.\\

Regardless of whether $\phi_{func}$ commutes with contractions on all elements when the $X_i$ are simplices, it is sometimes possible to get $\phi_{func}(\bar{u}) = h_K \phi_{func} (\bar{v})$ where $h_F(\bar{v}) = \bar{u}$ and $\bar{v}$ is much simpler than $d\bar{u}$.  Specifically, this will hold if $d \bar{u} = \bar{v} + \bar{w}$ and $\phi_{func}(\bar{w}) \in Im(h_K)$.
 $\qed$
\end{rem}
\subsection{The Functorial $AW$ Standard Procedure Map}
We  turn to examples of Proposition \ref{8.3} and Proposition \ref{8.2} where there are no group actions.  For any simplicial sets $X, Y$ we have classical functorial Alexander-Whitney, Eilenberg-Zilber, and diagonal approximation  maps $$AW\colon N_*(X \times Y) \to N_*(X) \otimes N_*(Y)$$ $$EZ\colon N_*(X) \otimes N_*(Y) \to  N_*(X \times Y)$$  $$\delta = \delta_\otimes\colon N_*(X) \to N_*(X) \otimes N_*(X).$$  These complexes are not contractible in general, but there is still a common explanation of all three maps, using the preferred contraction of chains on simplices, chains on  products of simplices, and tensor products of chains on simplices, along with the idea of minimal contractible carriers outlined above. There is also an obvious diagonal map $$\delta_\times \colon N_*(X) \to N_*(X \times X),$$   given on simplices by $\delta_\times(\sigma) = (\sigma, \sigma)$.\\

We take up the $EZ$ map in the next subsection. We can use either Proposition \ref{8.2} or \ref{8.3} to prove that the recursive functorial procedure diagonals and general $AW$ maps are given by the well-known classical formulas.  In particular, if Proposition \ref{8.3} is used it is not necessary to prove the known formula is a chain map.

\begin{exam}\label{8.5} \textbf{The Map $\bf \delta_\times.$}  This map obviously commutes with the contractions of domain and range when $X$ is a simplex  Also, $N_n(\Delta^n)$ is generated by $(0, 1, \ldots, n) = h (1, \ldots n)$.  So Proposition \ref{8.3} applies and $\delta_\times = \phi_{func}$.    Of course this example is trivial to analyze directly. $\qed$
\end{exam} 

\begin{exam}\label{8.6} \textbf{The Map $ AW$.}  The classical formula for $AW(\sigma, \tau)$ is a sum of front faces of $\sigma$ tensored with back faces of $\tau$.  In degree 0, this begins with $AW(x_0, y_0) = x_0 \otimes y_0$.    The general formula is $$AW((x_0, \ldots, x_n), (y_0, \ldots, y_n)) = \sum_{j=0}^n (x_0 \ldots, x_j) \otimes (y_j, \ldots y_n).$$

It is quite easy to see that this functorial formula commutes with contractions when $X$ and $Y$ are simplices.  Also, the universal base simplex $(0,1, \ldots, n), (0,1, \ldots n)$ of $\Delta^n \times \Delta^n$ obviously belongs to $Im(h_\times)$.  Therefore Proposition \ref{8.3} applies, proving the classical $AW$ formula is the standard procedure functorial map.\\

If one first proves the classical $AW$ formula is a chain map then it is easy to see that  the universal elements map to elements in the image of contractions, hence Proposition \ref{8.2} also applies.\\

{\bf Warning.} Of course simplices in a general simplicial set $X$ are not determined by their vertices.  Our notation $(x_0, \ldots, x_n)$ really refers to a  map of simplicial sets $\sigma \colon\Delta^n \to X$.  But the notation makes it easy to denote face operators in $X$, which are restrictions of simplicial set maps $\sigma$  to faces of $\Delta^n$, hence {\it named} by subsets of $(x_0, x_1, \ldots, x_n)$.  But to interpret a subset as a simplex in $X$ it is necessary to interpret the subset as the result of a sequence of basic face operators $d_i$ in $X$ applied to $\sigma$.\\

It sometimes  helps to keep in mind the case that $X$ is the singular simplicial set of a space.  But in more abstract examples, the simplices and face operators in $X$ might have very little geometric meaning. $\qed$
\end{exam}

\begin{exam}\label{8.7} \textbf{The Map  $\delta_\otimes$.}  The classic formula for the Alexander-Whitney diagonal is $$ \delta_\otimes(x_0, x_1, \ldots, x_n) = \sum_{j = 0} ^ n (x_0, \ldots, x_j) \otimes (x_j, \ldots, x_n).$$  This map  is, of course, the composition $AW \circ \delta_\times$.  Since both $AW$ and $\delta_\times$ commute with contractions for appropriate simplex domains, so does the composition.  Thus $\delta_\otimes$ is  the standard procedure functorial map.\footnote{The relations between the classical $AW$ maps and maps constructed using our preferred contractions was  the main reason we chose the tensor product contraction that we did back in Example \ref{5.5}.  An alternate tensor product contraction would produce a different $AW$ map and a different  diagonal.}  $\qed$
\end{exam}

{\bf Comparison of  Two Constructions for MacLane Models.} We pointed out earlier in Remark \ref{6.8} that when $X = EH$ and $Y = EG$, the equivariant map $AW\colon N_*(EH \times EG) \to N_*(EH) \otimes N_*(EG)$ constructed directly, using the obvious $\FF[H \times G]$ basis in the domain and the preferred contraction of the range, produces  the same sum of front faces tensor back faces $AW$ formula as does the functorial minimal carrier construction.  In the MacLane model case, we actually carried out the inductive argument that led to the formula.\\

A direct inductive argument using simplices also quickly leads to the functorial $AW$ formula.  The computations with simplices are essentially identical to the MacLane model case because of the similarity of the two boundary formulas and the two contractions $h_G(g_0,g_1,\ldots,g_k) = (e, g_0, g_1, \ldots, g_k)$ and $h_\Delta(x_0, x_1, \ldots, x_k) = (0, x_0, x_1, \ldots, x_k)$. The equivariance requirement for MacLane models plays a similar role as the  functoriality requirement involving vertices of simplices. \\

In fact,  the group action maps $g \colon N_*(EG) \to N_*(EG)$ are simplicial maps, hence commute with maps constructed as natural transformations for all simplical sets.  Thus to show the two maps above coincide it suffices to show this for $G$-basis generators of the MacLane models, where the comparison with simplices is especially transparent. These comments apply both to the $AW$ map and the $EZ$ map to be discussed in the next subsection. $\qed$\\

{\bf *The Diagonal for the Simplicial Set $ BG = EG/ G$.*}  In Example \ref{5.3} we discussed classifying spaces $BG$ as simplicial sets.  The simplices are $G$-orbits of simplices of the MacLane model $EG$.  We name the orbit of $(g_0, g_1, \ldots, g_n)$ by the $n$-tuple of group elements $[f_1, f_2, \ldots, f_n] = [g_0^{-1}g_1, g_1^{-1}g_2, \ldots, g_{n-1}^{-1}g_n]$. The empty set $[\emptyset]$ names the only 0-simplex, the orbit of group elements $g_0 \in G$. A canonical lift of $[f_1, f_2, \ldots, f_n]$ is $(1, f_1, f_1f_2, \ldots, f_1f_2\cdots f _n)$.  The degenerate simplices of $BG$ are those with some $f_j = e \in G$.\\

The face operators in $BG$ are given as follows.  The first and last face operators $d_0$ and $d_n$ applied to $[f_1, \ldots, f_n]$ just delete $f_1$ or $f_n$.  Intermediate operators $d_i$ replace two adjacent entries in $[ \dots f_i, f_{i+1} \ldots]$ by their product $f_if_{i+1}$.  The boundary operator is $d = \sum_{j=0}^n\ (-1)^j d_j$.\\

Elements $f \in G$ name 1-simplices, that is, edges. One can keep all this straight in a picture of the standard $n$-simplex by labeling consecutive edges $(i-1, i)$ for $i \geq 1$ by $f_i$.  Other edges are named by products.  Faces in general correspond to subsets of the standard simplex vertices $(0, 1, \ldots, n)$ and are interpreted as compositions of basic face operators, hence correspond to various ordered sets of various products of the $f_i$. For example, the 2-simplex face of the 5-simplex $[f_1,f_2, f_3, f_4, f_5]$ corresponding to the simplex vertex subset $(1,2,5)$ is $[f_2, f_3 f_4f_5]$.\\

 The simplicial sets $BG$ are somewhat good examples of the warning in Example \ref{8.6} that face operators of simplicial sets in general may have very little geometric meaning. However, the picture of labeling edges of the standard simplex with group elements and products of group elements does serve as a good guide. The 2-simplex $(1,2,5)$ can be thought of as the composite of two edges $(1,2)$ and $(2,5) = (2,3)(3,4)(4,5)$, which separately are labeled $f_2$ and $f_3f_4f_5$.\\

A pleasant exercise is the formula for the diagonal for $N_*(BG)$, $$\bar{\delta} [f_1, f_2, \ldots, f_n] = \sum [f_1, \ldots, f_i] \otimes [f_{i+1}, \ldots, f_n],$$ where $i $ is allowed to be 0 or $n$, meaning an empty set for the first or last tensor factor, that is, the 0-cell. This can be deduced from the $G$-equivariant diagonal for $N_*(EG)$, or from the simplicial set functorial formula applied to $BG$.\\

Higher diagonals of $BG$ are given by $$\bar{\delta}^{(k)} [f_1,f_2, \ldots, f_n] = \sum [S_1]\otimes [S_2] \otimes  \cdots \otimes [S_k],$$ where the sum is over all non-overlapping ordered partitions of the ordered set $[f_1, f_2, \ldots, f_n]$.  Some $S_j$ can be empty, with $[S_j]$  the 0-cell. Otherwise, subscripts of elements of $S_i$ are strictly less than subscripts of elements of $S_{i+1}$. The data of a partition can be named  by inserting $k-1$ vertical dividers in the sequence $[f_1, f_2, \ldots, f_n]$. $\ \ \ \ \qed$

 \subsection{The Functorial $ EZ$ Standard Procedure Map}  The Eilenberg-Zilber map is more complicated than the Alexander-Whitney map.  We again want to use induction, functoriality,  the standard contraction of chains on products of simplices, and the tensor products of  simplices, which represent minimal carriers, to explain the functorial formula for $EZ \colon N_m(X) \otimes N_n(Y)\to N_{m+n}(X \times Y)$.  In degree 0 there is no problem, functoriality forces $EZ(x_0 \otimes y_0) = ( x_0, y_0)$.\\
 
We will explain in detail how the $EZ$ formula arises inductively using contractions of products of simplices. There is a great amount of overlap with the discussion of the $EZ$ map for MacLane models in Subsection 6.4.\\

{\bf Our Approach to the $\bf EZ$ Formula.} The formula for $EZ(\Delta^m \otimes \Delta^n) \in N_{n+m}(\Delta^m \times \Delta^n)$ is a sum, with orientation signs,  of all non-degenerate $m+n$ simplices in a triangulation of the prism.  Once that formula is found, naturality immediately gives the $EZ$ formula for all X, Y.  If one compares this formula for $EZ$ in terms of a triangulation with the formula for the oriented geometric boundary of $\Delta^m \times \Delta^n$, then $EZ$ is rather obviously a chain map. Also, inspection of the formula reveals easily that $EZ(\Delta^m \otimes \Delta^n)$ is in the image of the contraction of $N_{n+m}(\Delta^m \times \Delta^n)$.  Therefore Proposition \ref{8.2} applies, hence $EZ$ is the functorial standard procedure map.\\

But there is quite a bit involved in the paragraph above.  We want to indicate how low degree computations with the standard procedure  {\it lead} to the EZ formula.  We have the boundary formula $$d(\Delta^m \otimes \Delta^n) = (1, 2, \ldots, m) \otimes (0,1,\ldots, n) + (-1)^m (0,1 \ldots, m) \otimes (1,\ldots, n)$$  $$ +\  many\ other\ terms.$$ The other terms are tensors of faces with both first entries 0.  A map induced by face inclusions $N_*(\Delta^{m_i} \times  \Delta^{n_j}) \to N_*(\Delta^m \times \Delta^n)$ with $(0,0) \mapsto (0,0)$  commutes with contractions.  Thus functoriality and induction implies that $h_\times \circ EZ$ vanishes on these  terms in the boundary of a tensor of basic simplices.  This is an example of the simplification discussed in Remark \ref{8.4}. $\qed$\\

{\bf A Formula up to Signs.} Computations with  small values of $m, n$ lead  to the following `conjecture', which is an example of the discussion in Subsection 2.4 concerning standard procedure maps of the form $\phi(x) = \sum \pm \cT x$.
$$EZ(\Delta^m \otimes \Delta^n) = h_\times \circ EZ(d(\Delta^m \otimes \Delta^n)) = \sum_{(I,J)} \pm (\sigma_I, \tau_J),$$ where the sum is over all pairs $(\sigma_I, \tau_J)$ of simplices of $\Delta^m$ and $\Delta^n$, both  of degree $m + n$.  Of course each of these separate simplices are degenerate if $m,n > 0$.   In the conjecture we can remove pairs of  simplices that are degenerate in the product.  Degenerate means some adjacent pair of coordinates in the separate simplices both repeat at the same time.  A non-degenerate pair $(\sigma_I, \tau_J)$ must have first entries $(0,0)$ and last entries $(m,n)$, and corresponds to an edge path $(I, J)$ of length $m+n$ in the box $[0,m] \times [0,n]$ that increases one of the coordinates by 1 at each step.\\

The conjecture is very easily proved by induction.  It is only necessary to observe that the terms in the conjectured sum of form $(0,1,\ldots), (0,0,\ldots)$ arise by applying $h_\times$ to the boundary term $EZ((1,\ldots, m) \otimes (0,\ldots, n))$, and the  terms $(0,0,\ldots), (0,1,\ldots)$ arise by applying $h_\times$ to the boundary term $EZ((0,\ldots, m) \otimes (1,\ldots, n))$.  As mentioned above, applying $h_\times$ to $EZ$ of other boundary terms gives 0 by functoriality. $\ \  \qed$ \\

{\bf Determination of the Signs.} So the remaining issue is to determine the signs, which are of course determined recursively by the contraction procedure.   For that,  we bring in the inductive formula, which we have seen simplifies to
$$EZ (\Delta^m \otimes \Delta^n) =  h_\times \circ EZ (d (\Delta^m \otimes \Delta^n)) $$
  $$ = h_\times\ [EZ( (1, \ldots, m) \otimes (0, \dots,  n))] + (-1)^m h_\times\ [EZ((0,\ldots, m) \otimes (1, \ldots, n))].$$

The contraction $h_\times$  simply places $(0,0)$'s in front of pairs of simplices. We have observed that   $EZ((1, \ldots, m) \otimes (0, \ldots, n))$  corresponds to allowed paths in the rectangle $[1, m] \times [0, n]$.  The term  $EZ((-1)^m(0, \ldots, m) \otimes (1,2, \ldots, n))$ corresponds to allowed paths in the rectangle $[0, m] \times [1, n]$.  The key now, just as it was in the case of MacLane models in Subsection 6.4,  is the observation that the area $A(I, J)$ in the full rectangle $[0, m] \times [0, n]$ below a path $(I,J)$ beginning with   $(0,0), (1,0)$ and continuing in the smaller  rectangle $[1, m] \times [0, n]$ is  the same as the area  below the continued path in $[1,m] \times [0,n]$. The area $A(I,J)$  in the full rectangle below a path $(I, J)$ beginning with  $(0,0), (0,1)$ and continuing in the smaller rectangle $[0, m] \times [1, n]$ is $m$ plus the area below the continued path in $[0, m] \times [1,n]$. So the sign in front of $(\sigma_I, \sigma_J)$ is $(-1)^{A(I,J)}.$ $\ \ \ \ \qed$ \\

This proof of the functorial $EZ$ formula by the method of analyzing inductively a proposed formula $\phi(x) = \sum \pm \cT x$ is an easy example of the same strategy we will use  for the proofs of the two hardest results in our paper, namely Proposition \ref{17.3}, concerning an equivariant  functorial map $\phi \colon \cS_*(n) \otimes N_*(X) \to N_*(X)^{\otimes n}$, and Proposition \ref{20.3}, concerning a twisted equivariant  map for surjection complexes $\phi \colon \cS_*(r) \otimes \cS_*(s_1) \otimes \ldots \otimes \cS_*(s_r) \to \cS_*(s)$, where $s = \sum s_i$.  The proofs of those  two formulas by the $\phi(x) = \sum \pm \cT x$ method are much longer than the proof of the functorial $EZ$ formula. But we consider it  important that this one technique applies in many settings. We discussed all this in the introductory Subsection 2.4, as a hint of things to come. \\

{\bf *The Eilenberg-Zilber Formula as a Triangulation of Prisms.*} We have now proved that the complete  formula for the standard procedure Eilenberg-Zilber map on tensors of fundamental chains on simplices is  $$EZ(\Delta^m \otimes \Delta^n) =  \sum_{(I,J)} (-1)^{A(I,J)}(\sigma_I, \tau_J)  \in N_*(\Delta^m \times \Delta^n).$$

We will include here the explanation  that these simplices with signs  in the simplicial set product $\Delta^m \times \Delta^n$  correspond  to the non-degenerate simplices of maximal dimension  in the ordered simplical complex associated to the product of posets $(0<1,\ldots<m) \times (0<1<\ldots<n)$.  In other words,  simplices in a canonical triangulation of the prism. The signs arise from orientations.\\

The data $(I, J)$ for a non-degenerate $m+n$ simplex is also determined by a sequence consisting of $m$ $i$'s and $n$ $j$'s.  Namely, such a sequence records at each step whether the $i$ coordinate or the $j$ coordinate increases by 1.  Such a sequence can be viewed as  a shuffle permutation of $\{i,i,\ldots, i, j,j, \ldots, j\}$, generated by swaps of an adjacent $i$ and $j$.  Such a permutation has a sign $(-1)^{sh(I, J)}\in \{ \pm 1\}$ that records  the parity of the number $sh(I, J)$ of such swaps.  This sign is the parity character   of a shuffle permutation of  $m+n$ distinct objects $\{i(1), \ldots, i(m), j(1), \ldots, j(n)\}$ that keeps all the $i$'s and all the $j$'s in their given order.  It is clear that the parity of $sh(I, J)$ is the same as the parity of the area $A(I, J)$ below the edge path $(I, J)$ in $[0,m] \times [0, n]$.\\

The prism is canonically oriented as the product of two oriented manifolds.  Each non-degenerate $m+n$ simplex is a codimension 0 sub-manifold of the prism, hence is oriented as such, and  is also canonically oriented in its own right  as an ordered simplex.  There is thus a sign $(-1)^{o(I, J)} \in \{\pm 1\}$, where $o(I,J) \in \{0,1\}$ is defined by comparing the two orientations of the simplex.\\

The canonical orientation of the product manifold agrees with that of the simplex corresponding to the path sequence $(i, \ldots,i ,j, \ldots j)$ along the bottom then up the right side of the rectangle.  Each swap of an adjacent $i$ and $j$ changes the comparison sign $(-1)^o$  of the orientation of the corresponding simplices.  This is because the two ordered simplices corresponding to  two such path sequences have a common interior ordered $m+n-1$ face in the oriented product manifold, with the opposite vertices inserted in the same ordered position.  Therefore $(-1)^{sh(I,J)} = (-1)^{o(I, J)}$. \\

The  formula for the Eilenberg-Zilber map on tensors of fundamental chains on simplices, expressed as a sum of oriented simplices $(\sigma_I, \sigma_J)$ in the prism, can thus also be written $$EZ(\Delta^m \otimes \Delta^n) = \sum_{(I,J)}(-1)^{ o(I,J)} (\sigma_I, \sigma_J) = \sum_{(I,J)} (-1)^{sh(I,J)} (\sigma_I, \sigma_J).$$
When $m = 1$ the  $EZ$ formula becomes $$EZ((0,1) \otimes (0, \ldots, n)) = \sum_{j=0}^n\  (-1)^j ( (0, 0), \ldots, (0, j), (1, j), \ldots, (1, n)).$$ This is because the shuffle sign comparing $(0,1, \ldots, 1) $  and $(1, \ldots, 1, 0,1, \ldots , 1)$, when there are $j$ 1's before the 0, is $(-1)^j$. $\ \ \ \ \qed$\\

\begin{rem}\label{8.8}  {\bf Products With More Factors.} Our discussions of the functorial $AW$ and $EZ$ maps extend rather routinely to products of three or more spaces.  Both maps are associative and the $EZ$ map is commutative.  One can extend all the arguments above directly for multiple products, using the contractions of multiple products in the range. The signs $(-1)^{A(I,J, 
\ldots, K)}$ occurring in the $EZ$ formula for multiple products are also orientation signs, associated to maximal dimension simplices in a canonical triangulation of a multiple product of simplices.\\

But associative diagrams for both the $AW$ and $EZ$ maps  involving three or more spaces, and commutative diagrams for the $EZ$ map, are also easy consequences of the uniqueness result Proposition \ref{8.2}.  Namely, all maps in a diagram of standard procedure maps for simplices are chain maps.  In the $EZ$ case the diagrams are the same as the diagrams for MacLane models in Subsection 6.4, with simplices replacing the MacLane models. Compositions of the maps in relevant diagrams easily satisfy the criteria of the uniqueness result 8.2 that basis elements map to elements in the image of contractions.  Note that an attempt to prove a commutativity result for the $AW$ map fails, from our point of view, because a permutation of tensor factors in $N_*(\Delta^m) \otimes N_*(\Delta^n)$  does not preserve elements in the image of the contraction. $\qed$
\end{rem}

Our rather detailed discussion of triangulations of prisms and the $EZ$ map will play a prominent role in Parts II and III when we study  morphisms $\cE \leftrightarrows \cS$ between the Barratt-Eccles operad and the surjection operad from the viewpoint of constructing canonical chain maps using preferred contractions.  The $EZ$ map also plays an important role in the discussion of chain homotopies  in Section 9.

\subsection{Functorial Maps $ \gamma_X \colon B_* \otimes N_*(X) \to N_*(X)^{\otimes n}$}

We have in mind that $B_*$ is a free augmented based $\FF[H]$ complex for some subgroup $H \subset \Sigma_n$, with $B_0 = \FF[H]$. Then $H$ acts on $N_*(X)^{\otimes n}$ by permuting tensors, as discussed in Section 3.  A basis of $B_* \otimes N_*(X)$ is provided by $\{b \otimes \sigma \}$, where the $\{b\}$ form an $H$-basis of $B_*$ and the $\{\sigma \colon \Delta^m \to X\}$ form an $\FF$-basis of $N_*(X)$.  Then the standard functorial procedure constructs an $H$-map $\gamma_X \colon B_* \otimes N_*(X) \to N_*(X)^{\otimes n}$.\\

Restricted to $e \otimes N_*(X)$, the map $\gamma_X$ is the functorial $AW$ diagonal. In addition, suppose  a group $G$ acts freely and simplicially on the simplicial set $X$. Then $G$ also acts diagonally on the target $N_*(X)^{\otimes n}$, and this $G$ action commutes with the $H$ action. Thus $H \times G$ acts on the target  and also acts freely on the domain of $\gamma_X$.  By functoriality, $\gamma_X$ induces an $H$-equivariant map $\bar{\gamma}_X \colon  B_* \otimes N_*( X/ G)  \to N_*( X/G)^{\otimes n}$.\\

In particular, we can take $H$ to be $C_n$ or $\Sigma_n$, and $X = EG$ with  $ X/ G = BG$.  Then $N_*(EG)^{\otimes n}$ has a preferred contraction. There are many possibilities for $B_*$, including minimal models, MacLane models, and surjection complexes, that we will study in Section 17 of Part III. 

\begin{prop}\label{8.9}  The functorially defined  map $\gamma_{EG} \colon B_* \otimes N_*(EG) \to N_*(EG)^{\otimes n}$ is $H \times G$ equivariant and coincides with the $H \times G$ equivariant standard contraction procedure map $\gamma \colon B_* \otimes N_*(EG) \to N_*(EG)^{\otimes n}$.
\end{prop}
\begin{proof} The actions of $g \in G$ on $EG$ are simplicial maps. On the MacLane model  $e \otimes N_*(EG)$ the two $ G$-equivariant  maps $\gamma_{EG}$ and $\gamma$ are both the $G$-equivariant $AW$ multidiagonal, as observed in Subsection 8.2.  The proposition then follows from the fundamental uniqueness result Proposition \ref{6.2} if the functorial map $\gamma_{EG}$ maps basis generators $b \otimes (e, g_1,  \ldots,  g_m)$ to elements in the image of the contraction of $N_*(EG)^{\otimes n}$.\\

The functorial map on such a generator is induced by the simplicial map $\sigma \colon \Delta^m  \to EG$, with $\sigma (0,1, \ldots, m) = (e, g_1, \ldots, g_m)$, along with the induced map $N_*(\Delta^m)^{\otimes n} \to N_*(EG)^{\otimes n}$.    The recursive formula for both standard procedure maps is $\phi(b \otimes \sigma) = h_{\otimes}(\phi (db \otimes \sigma + (-1)^{|b|} b \otimes d \sigma))$. The key is that the chain maps induced by $\sigma$ commute with contractions, as well as boundaries. Thus $\gamma_{EC}(b \otimes (e, g_1, \ldots, g_m) ) $ is in the image of the contraction of $N_*(EG)^{\otimes n}$.\\

Instead of quoting Proposition \ref{6.2}, the observations about the simplicial map $\sigma$ and the two identical standard  recursive procedures, along with agreement of both maps with the $AW$ map on $e \otimes N_*(EG)$, rather immediately implies $\gamma_{EC} = \gamma$ by induction.
\end{proof}

The result Proposition \ref{8.9} supplements the discussion in Subsection 7.3 concerning applications of the chain maps  $\gamma$ and $\gamma_X$ to proofs of properties of Steenrod operations, especially the Adem relations. So Subsection 7.3 is important and could be re-read now.

\subsection{*Functorial Diagonals for Multisimplicial Sets*}

As another example of using contractions of models to construct functorial chain maps, we translate into our language some results of Medina-Mardones, Pizzi, and Salvatore [23] on diagonals for multisimplicial sets.  We will return to this example briefly in Section 14 in our discussion of the McClure-Smith surjection complex, but otherwise multisimplicial sets play no further role in our project.\\

{\bf Contractible Models for Multisimplicial Sets.} An {\it n-fold multisimplicial set} $X$ is a contravariant functor from the n-fold product of the simplex category with itself to sets, $\bold{\Delta}^n \to (Sets)$. The representing objects can be interpreted as prisms $ P = P(m_1, \ldots, m_n) = \Delta^{m_1} \times \cdots \times \Delta^{m_n}$, with  face and degeneracy operators in each variable.  A multisimplicial set has a geometric realization $|X|$, which is a cell complex whose open cells are interiors of such prisms.  The boundary operator is the prism boundary operator, which is the usual tensor product boundary operator in $N_*(\Delta^{m_1}) \otimes \cdots \otimes N_*(\Delta^{m_n})$.  This tensor product is the cellular chain complex $C_*(P)$ of the prism and has a preferred tensor product contraction.  For a multisimplicial set $X$, the cellular chain complexes of the prisms of $X$ fit together to give a (normalized) chain complex $C_*(X)$, which is the same as the ordinary cellular chain complex of  $|X|$.\\

We also have the preferred tensor product contractions of the $C_*(P) \otimes C_*(P)$, and therefore the functorial acyclic model recursive procedure produces functorial diagonal maps $\delta_X \colon C_*(X) \to C_*(X) \otimes C_*(X)$ for n-fold multisimplicial sets.  On the fundamental classes $P$ of the prisms, the recursive diagonal is $\delta(P) = h_\otimes \delta (dP)$.  In the spirit of comments following Proposition \ref{6.4}, only the boundary terms $\pm \Delta^{m_1} \otimes \ldots \otimes d_0\Delta^{m_i} \otimes \ldots \otimes  \Delta^{m_n}$ of $dP$ contribute non-zero terms to the diagonal.  But instead of working out the diagonal recursively, we can write down a functorial chain map diagonal and apply  the uniqueness result Proposition \ref{8.2}  to see that this coincides with the standard procedure diagonal.  The method is to exploit the functorial Alexander-Whitney diagonals $\delta_i$ that we already have in each simplex factor.

\begin{prop}\label{8.10} The composition $\delta_P = \tau \circ \otimes \delta_i$ $$ C_*(P) = \otimes_i N_*(\Delta^{m_i}) \xrightarrow{\otimes \delta_i} \otimes_i (N_*(\Delta^{m_i}) \otimes N_*(\Delta^{m_i})) \xrightarrow{\tau} C_*(P) \otimes C_*(P),$$ where $\tau$ is the permutation isomorphism between the indicated tensor products, is the standard procedure functorial diagonal  of the universal prism model $P$.
\end{prop}
\begin{proof} The argument is to just carefully show that Proposition \ref{8.2} applies, by checking that various formulas produce elements in the image of tensor product contractions.  Note that the permutation isomorphism $\tau$ introduces Koszul signs.  If we were to directly work with the standard recursive procedure, the signs arise from signs in  the relevant boundary terms in $dP$.  But since we know the composition $\tau \circ \otimes \delta_i$ is a functorial chain map, the uniqueness result Proposition \ref{8.2} saves us the trouble of working directly with the recursive procedure.  However, a direct discussion is not so different from a direct discussion of AW maps, as we explain next.
\end{proof}
{\bf A `Front Prism $\otimes$ Back Prism' Formula.} We can make quite explicit the diagonal formula $\delta_P$ as a sum with signs of `front subprisms tensor back subprisms', quite analogous to the AW diagonal formula.
Specifically, each integral vertex $I = (i_1, \ldots, i_n) \in [0, m_1], \times \ldots \times [0, m_n]$ defines a front prism face $P_I = \{(j_1, \ldots, j_n) | j_\ell \leq i_\ell\}$ and a back prism face $_I P = \{(j_1, \ldots, j_n) | j_\ell \geq i_\ell\}$.  Then 
\begin{prop}\label{8.11} $\delta_P = \sum_I \pm\    P_I\ \otimes\ _IP$.  The sign is the Koszul sign that moves the tensor of simplex faces $([0, i_1] \otimes [i_1, m_1]) \otimes \ldots \otimes ([0, i_n] \otimes [i_n, m_n])$ to the tensor $([0,i_1] \otimes \ldots \otimes [0, i_n]) \otimes ([i_1, m_1] \otimes \ldots \otimes [i_n, m_n]).$
\end{prop}
\begin{proof}The vertices $I = (0, \ldots, 0)$ and $(1,0, \ldots, 0)$ are special.  The corresponding diagonal terms arise from the contraction $h_\otimes = \rho_P \otimes h_P + h_P \otimes Id$ of $C_*(P) \otimes C_*(P)$ applied to the diagonal summand of $\delta(d_0\Delta^{m_1} \times \Delta^{m_2} \times \ldots \times \Delta^{m_n})$ corresponding to the vertex $(1,0, \ldots, 0)$ of that box. This is the only non-zero occurrence of the term $\rho_P \otimes h_P$ in the evaluation. Diagonal terms corresponding to other vertices $I = (i_1, \ldots, i_n )$ with $i_1 \geq 1$ or $I = (0, \ldots, 0, i_\ell, \ldots, i_n)$ with $i_\ell \geq 1$ arise as the unique non-zero term obtained by applying $h_P \otimes Id$  to the diagonal term of $\Delta^{m_1} \times \ldots d_0\Delta^{m_\ell} \ldots \times \Delta^{m_n}$  corresponding to the same vertex $I$ of that box. 
\end{proof}
In Parts II and III, following [23], we will look again at this diagonal in the case $C_*(X) = S_*^{ms}(n)$, the McClure-Smith surjection complex, which is indeed the normalized chain complex of an $n$-fold multisimplicial set $X$. 

\newpage

\section{Contractions and  Chain Homotopies}

\subsection{The Basic Procedure and a Uniqueness Theorem}
The standard recursive procedure for constructing equivariant chain maps between positively graded complexes also yields a standard procedure for constructing equivariant chain homotopies.  Suppose $\phi_0, \phi_ 1\colon B_* \rightrightarrows C_*$ are two equivariant chain maps with $\epsilon_C \phi_0 = \epsilon_C \phi_1$ in degree 0.  Also assume $B_*$ is free over $\FF[G]$ and $C_*$ has a contraction $h_C$ with $h_C^2 = 0$.  Then there is a preferred recursively defined equivariant  chain homotopy $H\colon B_* \to C_{*+1}$ from $\phi_0$ to $\phi_1$,  that is\footnote{The `from-to' language pays attention to the assymetry. We think of $\phi_0$ on the bottom of a cyclinder and $\phi_1$ on the top.},  $$dH + Hd = \phi_1 -  \phi_0.$$

As usual, we define $H(b)$ on $\FF[G]$ basis elements $\{b\}$ and extend linearly and equivariantly. In degree 0, $H(b_0) = h_C (\phi_1(b_0) - \phi_0(b_0))$.  For basis elements $b \in B_n$, $n \geq 1$, note by induction $\phi_1(b) - \phi_0(b) - H(db)$ is a cycle in $C_*$, hence a boundary. In fact, a specific boundary.  We define recursively $$\boxed {H(b) = h_C(\phi_1(b) -\phi_0(b) - H(db)).}$$

If there is no $G$ action the formula simplifies to $H(b) = h_C(\phi_1(b) - \phi_0(b))$ in all degrees since then $H(db) \in Im(h_C)$.  But if there is a $G$ action then $db = \sum g_i c_i$ and $H(db) = \sum g_i H(c_i)$ need not belong to $Im(h_C)$.\\
 
The method  extends routinely  to a functorial version of chain homotopies.  For example, the functorial standard procedure produces a preferred chain homotopy from $Id$ to  $EZ\circ AW$ that presumably agrees with a known such chain homotopy.  There is also a uniqueness theorem.\\

{\bf Proposition 9.0}  Any equivariant homotopy $H'$ that agrees with $H$ in degree 0 and maps basis elements to $Im(h_C) = Ker(h_C)$ coincides with $H$ in all degrees.
\begin{proof}  Here is an  argument analogous to the second given proof of the basic uniqueness theorem Proposition \ref{6.2}. Since $h_Cd_C + d_Ch_C = Id$ in positive degrees, it suffices to prove $h_C(H'b) = h_C(Hb) = 0$ and $d_C(H'b)) = d_C(Hb)$.  The first holds since both $H'b$ and $Hb$  belong to $ Im(h_C) = Ker(h_C)$.  The second holds  by induction since $$d_C (H'b) = \phi_1(b) - \phi_0(b) - H'(db) = \phi_1(b) - \phi_0(b) - H(db) = d_C(Hb).$$

\end{proof}
It turns out we will not make much direct use of the standard recursive construction of chain homotopies.  Instead, a rather explicit construction using joins will play a larger role.  In many cases we are able to use the uniqueness result  to relate the explict join homotopies to the recursive homotopies.\\

However, the general recursive construction of preferred chain homotopies is completely consistent with our view that preferred choices of contractions yield preferred choices of chain maps between chain complexes.  One might anticipate that the method extends to a standard construction of preferred higher homotopies between homotopies, and so on. Although in general closed formulas for recursive homotopies seem elusive, it is possible that the recursive construction of  homotopies has better theoretical properties than the explicit chain homotopies we introduce now.

\subsection{Joins and the Universal $ EZ$ Chain Homotopy}
One of the primary tools in our chain level study of Steenrod operations begun in [7] was the construction of chain homotopies between certain pairs of chain maps $ N_*(X) \to N_*(EG)$, using the join operation in the range.  The join operation just writes simplices next to each other, $(g_0, \ldots, g_n) * (g'_0, \ldots, g'_m) = (g_0, \ldots, g_n, g'_0, \ldots, g'_m)$, and extends multilinearly.  This is the same operation as the geometric join of an $n$-simplex and an $m$-simplex, resulting in an $(n+m+1)$-simplex, in terms of ordered vertices.  The domain should be a chain complex associated to a connected simplicial set, so that face operators are defined. The  boundary formula for joins is $$d(x * y) = dx * y - (-1)^{|x|} x * d y + \epsilon(x)y - (-1)^{|x|}x\epsilon(y),$$ where $|x| = deg(x)$, and $\epsilon \colon N_*(EG) \to \FF$ is the augmentation.\\ 

The join chain homotopies will be built from the universal functorial $EZ$ chain homotopy $EZ \colon N_*(I) \otimes N_*(X) \to N_*(I \times X)$, found towards the end of Subsection 8.3  to be given by
$$EZ((0, 1) \otimes (x_0, \ldots, x_n)) = \displaystyle \sum_{j=0}^n (-1)^j ((0, x_0), \ldots (0, x_j), (1, x_j), \ldots (1, x_n)).$$
The simplices on the right side can be viewed as non-degenerate $(n+1)$-simplices $(\ (0, \ldots 0, 1, \ldots 1), (x_0, \ldots x_j, x_j, \ldots x_n)\ )$ in the simplicial set $I \times X$. They are also joins of a bottom simplex  and a top simplex  in a triangulation of $I \times X$.
\begin{prop}\label{9.1}
Consider  two chain maps $\phi_0, \phi_1 \colon N_*(X) \to N_*(EG)$ that satisfy $\epsilon \phi_0(x) = \epsilon \phi_1(x) = 1 \in \FF$ for vertices $x \in  N_0(X) $.  A chain homotopy from $\phi_0$ to $\phi_1$, in terms of front and back faces of simplices of $X$,  is given by $$\boxed {J(x_0, x_1, \ldots, x_n) = \sum_{j = 0}^n (-1)^j\phi_0(x_0, \ldots, x_j)* \phi_1(x_j, \ldots, x_n).}$$ If a group $G'$ acts on $X$ and $\phi_0$ and $\phi_1$ are equivariant for some group homomorphism $G' \to G$, then $J$ is also equivariant. 
\end{prop}
\begin{proof}This can be proved by a direct computation, using the boundary formula for joins.  But a more conceptual proof arises from viewing the $EZ$ map in terms of triangulations of prisms and their boundaries.  Define a degree zero chain map $H \colon N_*(I \times X) \to N_*(EG)$ as follows.  On the top and bottom copies of $X$, define $$H ((1, \ldots, 1), (y_0, \ldots y_{n+1})) = \phi_1(y_0, \ldots,  y_{n+1})$$ and $$H ((0 \ldots, 0), (y_0, \ldots, y_{n+1})) = \phi_0(y_0, \ldots,  y_{n+1}).$$  On a simplex not on the top or bottom, define $$H ((0, \ldots 0, 1 \ldots 1), (y_0\, \ldots y_j,  y_{j+1}, \ldots y_{n+1})) = \phi_0(y_0, \ldots,  y_j) * \phi_1(y_{j+1}, \ldots, y_{n+1}).$$  The boundary formula for joins implies that $H$ is indeed a chain map.  It is here that the  hypothesis in degree 0 is used.\\

Define $EZ_h(x_0, x_1, \ldots, x_n) = EZ((0,1) \otimes (x_0, x_1, \ldots, x_n))$.  Since $EZ$ is a chain map, $$(dEZ_h + EZ_h d) (x_0, \ldots, x_n) = ((1, \ldots 1), (x_0, \ldots, x_n)) - ((0, \ldots, 0), (x_0, \ldots, x_n)).$$  Observe that $J(x_0, \ldots, x_n) = H \circ EZ_h (x_0, \ldots, x_n)$.  Then $$(dJ + Jd) ( x_0, \ldots, x_n) = H \circ (d EZ_h + EZ_h d) (x_0, \ldots, x_n)$$ $$ = \phi_1 (x_0, \ldots, x_n) - \phi_0(x_0, \ldots, x_n). $$
The statement concerning equivariance of $J$ is clear from the formula since the face operators in $X$ and the join operator in $N_*(EG)$ are equivariant.
\end{proof}

\begin{exam}\label{9.2}
 As an  example of an equivariant situation, elements $g \in G$ act by right multiplication on $G$.  This induces a left $G$-equvariant map $\boldsymbol{\cdot} g \colon N_*(EG) \to N_*(EG)$, defined by the right multiplication by $g$ on simplices $(g_0, g_1, \ldots, g_n) \boldsymbol{\cdot} g = (g_0g, g_1g, \ldots, g_ng) \in N_*(EG)$.  The identity is also left $G$-equivariant. Thus there is a canonical left $G$-equivariant join chain homotopy $K_g$ from $Id$ to $\boldsymbol{\cdot} g $ on $N_*(EG)$.\\

It is not necessary that either map $\phi_0$ or $\phi_1$ in the proposition be some kind of inclusion.  As an example, the basepoint $\rho = \iota \epsilon\colon N_*(EG) \to \FF\to N_*(EG)$ is chain homotopic to the identity, and the standard chain homotopy is just the contraction $J(g_0, \ldots,  g_n) = (e, g_0, \ldots, g_n)$, which can be seen to be a special case of the universal $EZ$ homotopy, the sum degenerating to a single term. $\qed$
\end{exam}

{\bf Join Homotopies Sometimes Determine Recursive Homotopies.} One can ask when does the join homotopy agree with the recursive standard procedure homotopy?  The answer is provided by the uniqueness theorem Proposition 9.0  for the standard procedure homotopy stated earlier in this section.

\begin{prop} \label{9.3} If $\phi : N_*(X) \to N_*(EG)$ is the standard procedure chain map constructed using a simplex basis of $N_*(X)$ and  the contraction of $N_*(EG)$, and if $\phi$ maps all front faces of basis simplices of $N_*(X)$ to sums of basis simplices of $N_*(EG)$,  then the join homotopy $J$ of Proposition \ref{9.1} from $\phi$ to $\phi_1$  is the standard procedure homotopy $H$.\\

More generally, the recursive homotopy $H$ between two chain maps $\phi_1$ and $\phi_2$  is the difference $H  = J_2 - J_1$, where $J_i$ is the join ($=$ recursive) homotopy from  $\phi$ to $\phi_i$, where $\phi$ is the standard procedure chain map.
\end{prop}
\begin{proof}The reason is, the joins of  basis simplices in $N_*(EG)$ with other simplices are basis simplices, hence in $Image(h_G)$.  Note the hypothesis holds if front faces of basis simplices of $X$ are themselves basis simplices, since $\phi$ is assumed to be a standard procedure map.\\

The second statement is obvious from the definition of the recursive homotopies at the beginning of this section.
\end{proof}
In general, to reverse the order of $\phi_0$ and $\phi_1$, one can either replace $H$ by $-H$, or $J$ by $-J$, or reverse the roles of the $\phi_i$ in the formula of 9.1.  These give different homotopies.  Since one of our themes is to single out preferred choices of chain maps, the same goal should apply to chain homotopies.  Therefore, the join method of Proposition \ref{9.1} is suspect unless one of $\phi_0, \phi_1$ is the preferred chain map.  In some of our examples above and below, it does seem to be the case that one of the chain maps $\phi_i$ is the standard procedure map.  But in any case, the second statement of the Proposition may apply. \\ 

Finally, we emphasize that some degree 0 hypothesis  is necessary for constructing chain homotopies.  If $\phi_0$ and $\phi_1$ do not induce the same map on homology in degree 0 they cannot be chain homotopic.  We will give several applications of Proposition \ref{9.1} and Proposition \ref{9.3} in the following subsections, but we will leave to  the reader  the detail of checking the degree 0 hypothesis and the hypothesis about front faces of basis simplices of $X$.

\subsection{A Chain Homotopy Useful for the Cartan Formula}

There are  other important situations where the  $EZ$ chain homotopies are very useful.
\begin{exam}\label{9.4}{\bf A Join Homotopy Between  Two Diagonals.}
For the cyclic group $C$ of prime order $p$ we have defined elements $x_k  = \phi(y_k) \in N_*(EC)$  that project to homology generators of $N_*(BC)$ with $\FF_p$ coefficients.  Here the equivariant chain map $\phi \colon M_* \to N_*(EC)$ was constructed in Example \ref{6.8}.  One of the problems encountered in proving the Cartan formula for Steenrod operations at the cochain level for odd primes $p$ is that, unlike the case for $p = 2$,   $\Delta_{AW}(x_k) \in N_*(EC) \otimes N_*(EC)$ is complicated, and even more complicated is its image $EZ \circ  \Delta_{AW}( x_k) \in N_*(EC \times EC)$.  We will overcome this difficulty in Part IV  by exploiting the  equivariant chain map retraction $\pi \colon N_*(EC) \to M_*$, constructed in Example \ref{6.11}.  Retraction means $\pi(x_k) = y_k$.  We consider the diagram of $C$-equivariant chain maps where $C$ acts diagonally on the right side
$$\begin{CD}  N_*(EC) & \xrightarrow{\delta_*}&  N_*(EC \times EC) \\ 
 \downarrow \pi  &  & \uparrow {EZ \circ (\phi \otimes \phi)}\\
 M_* & \xrightarrow{\Delta_M}  & M_* \otimes M_* 
\end{CD}$$
Here $\delta_*$ is induced by the simplicial set diagonal $\delta \colon EC \to EC \times EC$, $\sigma \mapsto (\sigma, \sigma)$. The diagram does not commute, but  it commutes up to equivariant chain homotopy.  Since we have joins in $N_*(EC \times EC)$, we have  an equivariant   join chain homotopy $J$ from $\delta_*$ to $EZ \circ \Delta$, where $\Delta = (\phi \otimes \phi)  \circ \Delta_M \circ  \pi$.  Here $J$ will also agree with the standard recursive homotopy, since $\delta_*$ is the standard procedure chain map, induced by the group diagonal $C \to C \times C$.  \\

Although the formula for join homotopies in Proposition \ref{9.1} is conceptually simple on individual simplices, at the end of the day it must be extended multilinearly.  For example, to apply the $J$ in the current discussion to the classes $x_k$, one must take front and back faces of each simplex in the summation formulas for the $x_k$, which depend on the parity of $k$.  For example, from Proposition \ref{6.9}, $x_{2k} = \sum (1, T^{a_1}, T^{a_1+1}, \ldots, T^{a_k}, T^{a_k+1})$.  For each separate simplex $(1, T^{a_1}, T^{a_1+1}, \ldots, T^{a_k}, T^{a_k+1})$ one must sum over successive front and back faces.  Applying $EZ \circ \Delta$ to each back face of these simplices  results in additional complicated summations.  The retraction $\pi$ applied to back faces  eliminates most terms, because of  Proposition \ref{6.12} and the example that followed, but it is rather awkward to express exactly which terms remain. The important point for us is that $J$ is an explicit chain homotopy from $EZ \circ \Delta$ to $\delta_*$.\\

The diagonal $EZ \circ \Delta$  is exactly what we will want, because $$(\phi \otimes \phi)  \Delta_M   \pi (x_k) = (\phi \otimes \phi)  \Delta_M ( y_k) \in N_*(EC) \otimes N_*(EC)$$ is a sum of known tensor products $T^u x_i \otimes T^v x_j$, computed in Proposition \ref{7.3}, so the $EZ$ image is a sum of terms  $ EZ(T^u x_i \otimes T^v x_j)$ that can be handled by the Barratt-Eccles operad mechanism that we will study in Part III.\\

In particular, projecting to $N_*(BC \times BC)$, and writing for simplicity $\bar{\Delta}$ for the diagonal on coinvariants induced by $EZ \circ \Delta $, one has extremely simple looking formulas for the chain homotopically improved diagonals of the classes $ \bar{x}_k = \phi(\bar{y}_k)$:\  $$EZ \circ \bar{\Delta}(\bar{x}_{2k}) = \sum_{i+j = k} \bar{x}_{2i} \times \bar{x}_{2j} \ \ \ \ and\ \ \ \ \ EZ\circ \bar{\Delta}(\bar{x}_{2k+1}) = \sum_{i+j = 2k+1} \bar{x}_i \times \bar{x}_j,$$ where $\bar{x}_m \times \bar{x}_n$ means  $EZ(\bar{x}_m \otimes \bar{x}_n)$.\\

One can iterate to form higher simpler diagonals, by repeatedly applying $EZ \circ \Delta$ and $EZ \circ \bar{\Delta}$ to the last variable. The iterations are awkward at the $EC$ level, but are extremely simple at the $BC$ level. $\qed$, 
\end{exam}
\subsection{Homotopies Implying Many $ [\bar{x}_n] \in H_*(B\Sigma_p)$ are Zero }
\begin{exam}\label{9.5}{\bf The Homology Classes $[\bar{x}_n] \in H_*(BC_p)$.}
We can also use the universal $EZ$ chain homotopy  to study the  images of the homology classes $[\bar{x}_{2k}], [\bar{x}_{2k-1}] \in H_*(BC_p) \to H_*(B\Sigma_p)$,  induced by the  inclusion of MacLane models $\iota \colon N_*(EC_p) \to N_*(E\Sigma_p)$, where  $T \mapsto t = (23 \ldots p 1)$, a $p$-cycle.  First, we return to the discussion of $\ell^{th}$ power maps $\iota_\ell$ begun in Example \ref{6.14}.  We will find explicit boundary formulas for the differences $\ell^k \bar{x}_{2k} -\iota_\ell(\bar{x}_{2k})$ and $\ell^k \bar{x}_{2k-1} -\iota_\ell(\bar{x}_{2k-1})$.  In Example \ref{6.14} we observed that these two differences are null-homologous.\\

We are denoting by $\iota_\ell$ both the $\ell^{th}$ power map on $C_p$ and the induced  map  $\iota_\ell \colon N_*(EC_p) \to N_*(EC_p)$, which coincides with the standard procedure map.   Then we have two $\iota_\ell$-equivariant chain maps, $\iota_\ell$ and  $\phi \lambda \pi \colon N_*(EC_p) \to M_* \to M_* \to N_*(EC_p)$, where $\pi $ is the retraction of Example \ref{6.11} and $\lambda$ is the $\iota_\ell$-equivariant chain map of Example \ref{6.14}.    Since we have the join operation in $N_*(EC_p)$, these two chain maps are $\iota_\ell$-equivariantly chain homotopic by a join chain homotopy $J$ from $\iota_\ell$ to $\phi  \lambda\ \pi$, which coincides with the recursively defined homotopy by Proposition \ref{9.3}.\\

From Proposition \ref{6.14}, $\lambda( y_{2k}) = \ell^k y_{2k}$ and $\lambda (y_{2k-1}) = \sum_{i=0}^{\ell-1} \ell^{k-1} T^i y_{2k-1}$.  Since $\pi(x_n) = y_n$ and $\phi(y_n) = x_n$ we then have $$\ell^k x_{2k} -\iota_\ell(x_{2k}) = (d J + J d) x_{2k}\ \ \ \ and$$ 
$$ \sum_{i=0}^{\ell-1} \ell^{k-1} T^i x_{2k-1} - \iota_\ell (x_{2k-1})   = (dJ + Jd) x_{2k-1}.$$  Since the chain homotopy $J$ covers a chain homotopy  $\bar{J}$ of $BC_p$, and since the  classes $\bar{x_j}$ are cycles, after  passing to coinvariants we get  $$\ell^k \bar{x}_{2k} -\iota_\ell(\bar{x}_{2k})  = d \bar{J}  (\bar{x}_{2k} )\in N_*(BC_p).$$  Since the $T^i x_{2k-1}$ all have coinvariant image $\bar{x}_{2k-1}$, we also have  $$\ell^k \bar{x}_{2k-1} - \iota_\ell(\bar{x}_{2k-1})   = d \bar{J} ( \bar{x}_{2k-1}) \in N_*(BC_p).$$  Thus we have explicit boundary formulas for the differences $\ell^k \bar{x}_{2k} -\iota_\ell(\bar{x}_{2k})$ and $\ell^k \bar{x}_{2k-1} -\iota_\ell(\bar{x}_{2k-1})$ in $N_*(BC_p)$, as promised in Example \ref{6.14}. $\qed$
\end{exam}
\begin{exam}\label{9.6} {\bf A Formula for Certain Multiples of $\bf [\bar{x}_n] \in H_*(B\Sigma_p)$.} Next let $\iota \colon N_*(EC_p) \to N_*(E\Sigma_p)$ be the map induced by the homomorphism sending $T$ to the $p$-cycle $t = ( 23 \ldots p1)$.  Let $\ell$ be a primitive $(p-1)^{th}$ root of unity in $\FF_p = \ZZ/p$.   The permutation $g(j) = j\ell \pmod p$ is a  $p-1$ cycle in $\Sigma_p$ regarded as permutations of $\ZZ/p = \{1,2. \ldots, p\}$.  In this representation, $t(j) = j+1$ and  $gt = t^\ell g$.  Thus $g\iota = (\boldsymbol{\cdot}g) \circ \iota \iota^\ell $. From Example \ref{9.2} the right multiplication map $\boldsymbol{\cdot} g$ on $N_*(E\Sigma_p)$ is $\Sigma_p$-equivariantly chain homotopic to the identity by a join chain homotopy $K \colon N_*(E\Sigma_p) \to N_{*+1}(E\Sigma_p)$.  Then, by an equivariant version of Proposition \ref{3.4}(i),  $$K \iota \iota_\ell \colon N_*(EC_p) \to N_* EC_p) \to N_*(E\Sigma_p) \to N_{*+1}(E\Sigma_p)$$ is an $\iota \iota_\ell$-equivariant chain homotopy from $ \iota \iota_\ell$ to $(\boldsymbol{\cdot}g) \circ  \iota \iota_\ell = g  \iota$.  Again, the join homotopies $K$ and $K \iota \iota_\ell$ coincide with the recursively defined homotopies.\\  

Combining the two $\iota \iota_\ell$-equivariant join chain homotopies $\iota J$ and $K\iota \iota_\ell$, by Proposition \ref{9.3} the difference $L = \iota J - K\iota \iota_\ell$ is a standard procedure homotopy. We get:
$$ \ell^k\iota (x_{2k})  - g \iota (x_{2k})    = (d L  + L d) x_{2k} \in N_*(E\Sigma_p)$$
$$ (\sum_{i = 0}^{\ell -1} \ell^{k-1}t^i) \iota (x_{2k-1})  - g \iota (x_{2k-1}) = (dL  + Ld) x_{2k-1} \in N_*(E\Sigma_p).$$
On coinvariants, where $g$ and $t$ act as the identity and the $\bar{x}_j$ are cycles, we get in $N_*(B\Sigma_p) = N_*(E\Sigma_p)/ \Sigma_p$ $$(\ell^k -1)\iota (\bar{x}_{2k}) = d \bar{L} ( \bar{x}_{2k})\ \ \ and \ \ \ (\ell^k - 1)\iota (\bar{x}_{2k-1} ) = d \bar{L} ( \bar{x}_{2k-1}).$$ 
We remark that since we have an explicit formula for any join homotopy from Proposition \ref{9.1} we  have explicit formulas for the two chain homotopies $\iota J$ and $K\iota \iota_\ell$ and hence also for  $L$. We do not write out these formulas here.  They are quite complicated, for reasons similar to those discussed in Example \ref{9.4}.  But their explicitness will be important in Part IV in applications to a cochain development of properties of Steenrod operations. $\qed$
\end{exam}
\begin{rem} \label{9.7}{\bf Vanishing of Many Steenrod Reduced Powers.}  We continue the discussion in the previous example. If  $\ell$ is a primitive $(p-1)^{th}$ root of unity mod $p$ and  if $k$ is not a multiple of $p-1$, we can multiply the last formulas in Example \ref{9.6} by a constant and obtain explicit formulas writing $\iota (\bar{x}_{2k})$ and $\iota (\bar{x}_{2k-1})$ as coboundaries in $N_*(B\Sigma_p)$.  In particular, this computation will lead to a  cochain level proof that the cyclic reduced $p^{th}$ power cohomology operations  determined by the classes $\iota (\bar{x}_{2k})$ and $\iota (\bar{x}_{2k-1})$ on {\it even degree} cocycles $\alpha \in N^*(X)$ are 0 unless $2k$ is an even multiple of $(p-1)$. The point will be that $\alpha^{\otimes p}$ is $\Sigma_p$-equivariant if $deg(\alpha)$ is even.\\

The study of the reduced powers on {\it odd degree} cocycles requires a somewhat different discussion because homology of the symmetric group with a non-trivial coefficient module enters the story.  Specifically, going back to Subsection \ref{5.4}, let $\widetilde{\FF}_p$ denote the $\Sigma_p$ module $\FF_p$ with twisted group action arising from the parity character $\tau \colon \Sigma_p \to \{ \pm 1 \}$.  We look at $\widetilde{N}_*(E\Sigma_p) = \widetilde{\FF}_p \otimes N_*(E\Sigma_p)$, which we view as the complex $N_*(E\Sigma_p)$, but with twisted group action $g * x = \tau(g) gx$.  In particular the augmentation $\widetilde{N}_*(E\Sigma_p) \to \widetilde{\FF}_p \to 0$ is a free acyclic resolution of the non-trivial module $\widetilde{\FF}_p$.\\

Since $p$-cycles are even permutations, the $C_p$-action on $ N_*(E\Sigma_p)$ is the same in the twisted and untwisted complexes.  Therefore the discussion of the $\iota \iota_\ell$-equivariant chain homotopies $\iota J$ and $K \iota \iota_\ell$ goes through exactly as above, until the very final step concerning coinvariants.  The point is, the $(p-1)$-cycle $g \in \Sigma_p$ that conjugates $t$ to $t^\ell$ is an {\it odd} permutation.  Therefore, in the twisted complex $\widetilde{N}_*(E\Sigma_p)$ we have $$g * \iota(x_{2k}) = -g\iota(x_{2k})\ \ and\ \ g * \iota(x_{2k-1}) = -g\iota(x_{2k-1}).$$  Thus in the coinvariants  $ N_*(B\Sigma_p; \widetilde{\FF}_p) = \Sigma_p \backslash \widetilde{N}_*(E \Sigma _p)$ we have $$(\ell^k + 1)\iota (\bar{x}_{2k}) = d \bar{L} ( \bar{x}_{2k})\ \  and  \ \ (\ell^k + 1) \iota(\bar{x}_{2k-1}) = d \bar{L} ( \bar{x}_{2k-1}).$$
Since $\ell$ is a primitive $(p-1)^{th}$ root of unity, we have $(\ell^k + 1) = 0 \pmod p$ if and only if $k$ is an odd multiple of $(p-1)/2$.  In particular, this computation will lead to a  cochain level proof that the cyclic reduced $p^{th}$ power cohomology operations  determined by the classes $\iota (\bar{x}_{2k})$ and $\iota (\bar{x}_{2k-1})$ on odd degree cocycles are zero unless $2k$ is an odd multiple of $(p-1)$. The point will be that $g \alpha^{\otimes p} = \tau(g) \alpha^{\otimes p}$ if $g \in \Sigma_p$ and $deg(\alpha)$ is odd. $\qed$
\end{rem}
\newpage
\addcontentsline{toc}{section}{PART II: The Surjection Complexes}

\section*{ II: The Surjection Complexes}

\section{Preview of the Surjection Complexes.}
We begin Part II of our paper with a brief discussion of  variants of  chain complexes  underlying a symmetric operad in the monoidal category of chain complexes, referred to by various authors as the surjection operad, the sequence operad, or the step operad.  We will postpone the discussion of operad structure until Part III, and concentrate in the first few sections  of Part II only on   chain complexes $\cS_*^{aj}(n),  \cS^{bf}_*(n), \cS^{ms}_*(n)$. In the last section of Part II we compare the surjection complexes to the MacLane complexes $N_*(E\Sigma_n)$.\\

The superscript initials of the latter two surjection complexes refer to complexes studied by Berger-Fresse [3], [5]  and McClure-Smith [19]. All three complexes are  free $\FF[\Sigma_n]$ resolutions of $\FF$, which is given the trivial $\Sigma_n$ action.   The  complex $\cS_*^{aj}(n)$ is newer but seems quite natural.  It  appears as part of the paper by Adamaszek and Jones [1], and perhaps elsewhere. \\

We want to emphasize from the outset that although our treatment of the three complexes is self-contained, it is entirely founded on the seminal  work of Berger-Fresse and McClure-Smith.  We believe we have `cleaned up' some details in the development of the surjection complexes, and their relations with other complexes.  Of course it has been nearly 25 years since the original works appeared, so it is not surprising some simplifications and reformulations could be found.  In fact, McClure and Smith have given more sophisticated versions of complexes related to their original construction in [21]. We believe the surjection complexes are important in algebraic topology, and we believe our reformulation of the details will serve a purpose making their development more accessible.\\

For all three surjection complexes, elements of a free basis over $\FF$ in each degree $k$ are named by non-degenerate surjections $$x \colon \{1,2, \ldots, n+k\} \to \{1,2, \ldots, n\}.$$  Degenerate means  $x(i) = x(i+1)$ for some $i$, and we set these equal 0.  We identify a function $x$ with the sequence $x = (x(1), x(2), \ldots, x(n+k))$, which we often write for notational reasons as $(x_1, x_2, \ldots, x_{n+k})$.   Functions that are not surjections are also declared to be 0 in the three complexes.  Note in degree $k = 0$ this gives $\FF[\Sigma_n]$ as an $\FF$-module.\\

In all three complexes the boundary operator has the form $$d (x) = d(x_1, x_2, \ldots, x_{n+k}) = \sum_j \gamma(j) (x_1, \ldots, \widehat{x_j}, \ldots, x_{n+k}),$$ where $\gamma(j) \in \{-1, 0, +1\}$ and $\widehat{x_j}$ means that term is deleted from the surjection $x$.\\ 

Throughout  Part II we will name the ordered vertices of the standard $N-1$ simplex $(1, 2, \ldots, N)$, rather than naming the first vertex 0.  The boundary of a simplex named by an ordered list of vertices  is then $$d(x_1, \ldots, x_N) = \sum (-1)^{j-1}(x_1, \ldots, \widehat{x_j}, \ldots, x_N).$$  The main reason for the change is to maintain consistency with the notation for surjection complex generators.  But there are other reasons this makes sense.  The most symmetrical view of the $N-1$ simplex is the convex hull of the unit basis vectors in affine space $R^N$, with its ordered basis. Also, the (ordered) join of such simplices in products of affine spaces is transparent as a convex hull. Tangent vectors are transparent as differences of unit basis vectors, which can also be expressed as an ordered pair (tail, head) of vertices.  The canonical orientation of $R^N$ induces, by the outward normal first convention, the standard orientation of the $N-1$ simplex viewed as part of the boundary of its join with the origin.  The standard orientation of the simplex corresponds to the ordered list of tangent vectors $\{12, 13, \ldots, 1N\}$.  An equivalent orientation is given by the list $\{12, 23, \ldots, (N-1)N\}$.  \\

It could be said that the only differences between the three surjection complexes are choices of sign conventions.  Technically this may be true, but we believe that the sign differences result from some initial conceptual differences that  lead to  sign choices.  For example, a surjection generator $x \colon \{1, 2, \ldots, n+k\} \to \{1, 2, \ldots , n\}$ can be interpreted geometrically  in different ways.  In the complex $\cS_*^{aj}(n)$, we interpret a generator as a simplicial map  $\Delta^{n-1+k} \to \Delta^{n-1}$, and then as an element of the normalized relative simplicial singular complex of the pair $(\Delta^{n-1}, \partial \Delta^{n-1})$.  In $\cS_*^{ms}(n)$ the data of a generator is interpreted as a prism  $\prod_{1 \leq \ell \leq n} \Delta^{k_\ell -1}$ with a total order on the combined set of vertices of the factor simplices.   The integer $k_\ell$ is the cardinality of $x^{-1}(\ell)$, the number of times $\ell$ appears in the sequence $x = (x_1, x_2, \ldots, x_{n+k})$. There is a non-degeneracy condition that no two vertices of the same simplex factor can be adjacent in the total order. In both $\cS_*^{aj}(n)$ and $\cS_*^{ms}(n)$ the boundary operator and the $\Sigma_n$ action arise from natural geometric considerations.  We don't really know how to think about $\cS_*^{bf}(n)$.  It mysteriously works, and in several ways has simpler properties than the other two complexes.\\

The base point in all three complexes will be $\iota(1) = 1e = e \in \FF[\Sigma_n] $, where $e \in \Sigma_n$ is the identity element.  The augmentations in the second two complexes send all permutations to $1 \in \FF$ and the permutation action in degree 0 is the obvious left regular representation action on $\FF[\Sigma_n]$. \\

In the complex $\cS_*^{aj}(n)$ the augmentation will send a permutation $g$ to $\tau(g) \in \{ \pm 1\}$, where $\tau$ is the parity character.  But in this case, the permutation action on $\FF[\Sigma_n]$ in degree 0 will be $g*g' = \tau(g) gg'$.  Note  this representation is $ \widetilde{\FF}[\Sigma_n] = \widetilde{\FF} \otimes \FF[\Sigma_n]   $, where $\widetilde{\FF}$ is the twisted $\Sigma_n$-module defined in Subsection 5.5 of Part I.   We think of $\Sigma_n$ as acting on coefficients as well as by multiplying group elements in sums $\sum a_i g_i, a_i \in \FF$.  Thus the $\tau$ augmentation is indeed an equivariant map $\widetilde{\FF}[\Sigma_n] \to\FF$, with the {\it trivial} module structure on $\FF$.\\

 We will define for all three complexes, $\cS_*^{aj}(n)$,  $\cS_*^{bf}(n)$, and $\cS_*^{ms}(n)$, a $\Sigma_n$ action, an equivariant differential, and a contraction.  Continuing our convention from Part I, contractions in Part II will always satisfy $h^2 = 0$ and $h \iota = 0$.  We will also establish isomorphisms between the three complexes, preserving all the structure.  The details in the first several sections to follow are somewhat lengthy, because there is a lot of structure data to be given, but they are not especially difficult.  Their importance will be brought out in Part III, where we explain how these complexes act on multitensors of cochains on simplicial sets $X$, in an operadic manner. This, of course, was part of the original motivation for introducing these complexes.\\

In the last section of Part II we present our approach to the equivariant chain maps $N_*(E\Sigma_n) \leftrightarrows \cS_*(n)$ studied by Berger and Fresse [3], [5] for their complex.  Our approach is based on the contraction and standard procedure methods of Part I.\\ 

 As always, we urge readers to not get stuck.  If you already know something, or if you find yourself somewhat bogged down, turn the page.
\newpage

\section{The Twisted Complex $\bf \widetilde{\cS}_*(n)$}
We actually begin with a twisted coefficient complex $\widetilde{S}_*(n)$ that arises very naturally.  Then we obtain an untwisted complex $\cS_*^{aj}(n)$ by tensoring with $\widetilde{\FF}$, as described in Subsection 5.4 of Part I.\\

The key observation is that the $\FF$-basis in degree $k$ given by certain surjections $x \colon \{1,2, \ldots, n+k\} \to \{1,2 \ldots , n\}$ coincides {\it exactly} with a basis of the normalized relative simplicial singular chain group $NSS_{n-1+k}( \Delta^{n-1}, \partial \Delta^{n-1})$ with $\FF$ coefficients. Generators of the simplicial singular chain group are given by the subset of the singular complex of the simplex consisting of  all affine linear maps $\Delta^{n-1+k} \to \Delta^{n-1}$ taking vertices to vertices. In the relative complex maps to the boundary of the simplex are declared 0. We then form the normalized graded $\FF$-module by dividing by the subcomplex of the singular complex spanned by degenerate simplices,  and shift degrees down by $n-1$. We call this shifted graded module $\widetilde{\cS}_*(n)$.  Precisely, a function $x$ applied to ordered vertices defines a map $\Delta^{n-1+k} \to \Delta^{n-1}$ between simplices. If two adjacent entries of a surjection coincide the resulting singular simplex is degenerate, hence zero in the normalized complex, and if the function $x$ is not surjective the image of the singular simplex lies in the boundary $\partial \Delta^{n-1}$.  We see that $\widetilde{\cS}_0(n) = \FF[\Sigma_n]$ as $\FF$-module.\\

We then take as the boundary operator in $\widetilde{\cS}_*(n)$  the standard simplical chain complex boundary map of the relative simplicial singular complex. That is,  $$dx = \sum_j (-1)^{j-1} (x_1, \ldots, \widehat{x_j}, \ldots x_{n+k}).$$  There are two ways a boundary summand could be 0, namely, the entry $x_j$ might occur only once in $x$ or one could have $x_{j-1} = x_{j+1}$.  Note that it is obvious to topologists that $d^2 = 0$.  This will also be the case for the differential in $\cS_*^{ms}(n)$, but not as much for $ \cS_*^{bf}(n)$.\\

For the left $\Sigma_n$ action on $\widetilde{\cS}_*(n)$ we simply post-compose simplicial maps $\Delta^{n-1+k} \to \Delta^{n-1}$ with permutations of $\{1,2, ..., n\}$, regarded as maps of the base $\Delta^{n-1}$ to itself.  Obviously the boundary operator is equivariant, and defines a free $\Sigma_n$ action in each degree. The relative homology of the simplex mod boundary is 0 in degrees other than $n-1$, and is $\FF$ in degree $n-1$. This last homology is shifted down to degree $0$.\\

What about the augmentation $\tilde \epsilon \colon \widetilde{\cS}_0(n) = \FF[\Sigma_n] \to \FF$?  The normalized long exact sequence of the pair comes with a connecting  ``boundary" map $$NSS_{n-1}( \Delta^{n-1}, \partial \Delta^{n-1}) \to H_{n-2}(\partial \Delta^{n-1}) = \FF,$$ which is easily  seen on an $\FF$-basis  to just record the degree of an automorphism of $\Delta^{n-1}$ given by a permutation of ordered vertices.  In other words, the ``boundary" map for the pair  identifies with the parity $\{\pm 1\}$-valued character $\tau$ of $\Sigma_n$.  Thus, we set $\tilde{\epsilon}(g) = \tau(g) \in \widetilde{\FF}$, where we write $\widetilde{\FF}$ to remind that  $\tilde{\epsilon}$ is equivariant for the twisted action on $\FF$.\\

OK.   $NSS_*(\Delta^{n-1}, \partial \Delta^{n-1})$, the normalized relative  singular simplicial chain complex of a simplex mod boundary, together with the augmenting connecting  homomorphism to the homology of the boundary, all shifted down in degrees, becomes an acyclic free $\FF[\Sigma_n]$ resolution  of the twisted module $\widetilde{\FF}$, and we call this complex $\widetilde{\cS}_*(n)$.  The base point is $\iota (1) = e$, and then in the form $\tilde {\rho} = \iota \tilde{\epsilon} \colon \widetilde{\cS_0}(n) \to \widetilde{\cS_0}(n)$ it is $\tilde{\rho}(g) = \tau(g) e$.
\newpage
\section{The Untwisted Complex $\bf \cS_*^{aj}(n)$}
From $\widetilde{\cS}_*(n)$ we  obtain $\cS_*^{aj}(n) = \widetilde{\cS}_*(n) \otimes_{\FF} \widetilde{\FF}$, an acyclic free resolution of the trivial module $\FF$, as in Subsection 5.4 of Part I.  Elements of the two complexes $\cS_*^{aj}(n)$ and $\widetilde{\cS}_*(n)$ have the `same' names and boundary operator, via the correspondence $x \leftrightarrow x \otimes 1$.  The only thing that changes from the shifted simplicial singular complex of the pair is the new $\Sigma_n$ action, which is now $g* x =  \tau(g) gx$ for $g \in \Sigma_n$.  The augmentation  $\epsilon \colon \cS_0^{af}(n) \to \FF$ is  $\epsilon(g) = (\tilde{\epsilon} \otimes Id)(g \otimes 1) = \tau(g) \in \widetilde{\FF} \otimes \widetilde{\FF} = \FF$.  This augmentation is indeed equivariant for the {\it twisted} group action on $\cS_*^{aj}(n)$ and the {\it trivial} action on $\FF$, as explained back in Subsection 5.4.  The basepoint is still $\iota(1) = e$. In the form $\rho = \iota \epsilon$, the base point is $\rho(g) \equiv (\iota \otimes Id)(\tilde{\epsilon} \otimes Id)(g \otimes 1) \equiv \tau(g) e.$ \\

\begin{rem}\label{12.1}{\bf The Contraction of $\bf \cS_*^{aj}(n).$} Topology tells us $\cS_*^{aj}(n)$ is contractible, since the augmented complex is free over $\FF$ and acyclic, but we want to give a contraction with $h^2 = 0$. We follow an idea from McClure-Smith [19], where they proved by induction that their complexes $\cS_*^{ms}(n)$ were contractible without actually giving a contraction.\footnote{The formula for the contraction  of the McClure-Smith complex that we present in Section 14 can also be found in the paper [13] of R. Kaufmann and A. Medina-Mardones.}  The operation on surjections $s(x) = (1, x)$ is not quite a contraction but satisfies a formula $ds + sd = Id - i\hat{r}$, where $\hat{r}$ and $i$ are maps of degree 0,  involving $\cS_*(n-1)$.\\

Specifically, $\hat{r} \colon \cS_*^{aj}(n) \to \cS_*^{aj}(n-1)$ is $0$ on any surjection $x$ that takes the value $1$ more than once.  If $x = (x_1, x_2, \ldots, x_{n+k})$ does contain a single $1$, say $x_j = 1$, then $$\hat{r}(x) = (-1)^{j-1}(x_1 -1, \ldots,  x_{j-1} - 1, x_{j+1}-1, \ldots x_{n+k}-1).$$  In words, remove the $x_j =1$ from $x$, subtract 1 from each remaining entry, and put in a sign.\\

The map $i \colon \cS_*^{aj}(n-1) \to \cS_*^{aj}(n)$ is defined by $$i(y_1, \ldots, y_{n-1+k} ) = (1, y_1 + 1, \ldots, y_{n-1+k} +1).$$  In words, add 1 to each entry of $y$, and then put a 1 in front.  Thus, if $x$ contains a single 1, then $i\hat{r}(x)$ just moves the 1 in $x$ to the front, and puts in a sign $(-1)^{j-1}$ depending on where the single 1 occurred in $x$.  The exponent $j-1$ in the sign counts the number of entries in the surjection $x$ that $x_j = 1$ moves past to get to the front of the new surjection $i\hat{r}(x)$.
\begin{prop}\label{12.2}
(i). The maps $\hat{r}$ and $i$  satisfy $d\hat{r} = -\hat{r}d$ and $di = -id$.\\  

(ii). It holds that $ds + sd = Id - i\hat{r}.$\\

(iii). A contracting homotopy $\hat{h} =  \colon \cS_*^{aj}(n) \to \cS_{*+1}^{af}(n)$ is given by $$\hat{h} = s - i s \hat{r} + i^2 s \hat{r}^2 - .... \pm i^{n-2} s \hat{r}^{n-2}.$$
(iv). It also holds that  $\hat{h} \iota = 0$ and $$\hat{h}^2 = (s - i s \hat{r} + i^2 s \hat{r}^2 - ....)( s - i s \hat{r} + i^2 s \hat{r}^2 - ....) = 0.$$ 
\end{prop}
\begin{proof}For verification of (i), one just needs to look  at the simplicial boundary formula with alternating signs. The $i$ case is quite simple.\\

In the $\hat{r}$ case, one looks separately at surjections $x$ that contain one, two, or more than two 1's.   If $x$ contains more than two 1's then $d\hat{r}(x) = 0 = \hat{r}d(x)$.  If $x$ contains two 1's, then $d\hat{r}(x) = 0$.  Suppose $x_p = x_q = 1$, with $p < q$.  Then $dx$ contains two terms with a single 1, with signs $(-1)^{p-1}$ and $(-1)^{q-1}$.  Then $\hat{r}d(x)$ will consist of two identical terms  but with opposite signs $(-1)^{q-2}(-1)^{p-1}$ and $(-1)^{p-1}(-1)^{q-1}$.  So $\hat{r}d(x) = 0$.   If $x$ contains a single 1, then $\hat{r}d(x)$ and $d\hat{r}(x)$ have the same number of terms which pair up with opposite signs. \\

A similar argument proves (ii), again considering separately the cases where $x$ contains one or more than one 1's.  In the more than one 1 case, $i\hat{r}(x) = 0$ and the proof that $ds(x) + sd(x) = x$ is  the same as verifying the contraction formula for a simplex in Example 5.1 of Part I.  This works because the alternating sign boundary formulas and the contraction formulas  are the `same' in the two cases.    In the case of a single 1, write it out and find $ds(x) + sd(x) + i\hat{r}(x) = x.$  The difference in the two cases is that $sd(x)$ is `missing' a term $(-1)^{j-1}(1, x_1, \ldots, \widehat{x_j}, \ldots , x_N) = i\hat{r}(x)$ corresponding to the degenerate term in $dx$ where the $x_j = 1$ gets deleted. So you need to add back that term to get $x$. \\

Statement (iii), that $\hat{h}$ is a contracting homotopy, follows from a telescoping formula
$$(d\hat{h} + \hat{h} d) (x) = [x - i\hat{r}(x)] + [i\hat{r}(x) - i^2\hat{r}^2(x)] + ..... = x - i^{n-1} \hat{r}^ {n-1} (x),$$
which is easily proved using the relations between $ s, \hat{r}, i$ and $d$ in (i) and (ii). For example,
$$[d (is\hat{r}) + (is\hat{r}) d] (x) = - i [ ds + sd] (\hat{r}x) = -i (\hat{r}x - i\hat{r}\hat{r}x) = - i\hat{r}x + i^2\hat{r}^2x.$$
We also need to prove $ i^{n-1} \hat{r}^ {n-1} (x) = \rho(x),$  where $\rho \colon \cS_*^{aj}(n) \to \cS_0^{aj}(n)$ is the base point.  It is easy to describe in words the maps $i^m\hat{r}^m$ and $i^m s \hat{r}^m$.  Note $\hat{r}^m(x) \not=0$ only if $x$ contains a single $1,2, \ldots, m$.  Then $i^m \hat{r}^m(x)$ removes those singletons from $x$ and inserts $1,2, \ldots, m$ at the front, along with a sign,  and $i^m s \hat{r}^m$ removes those singletons from $x$ and inserts $1,2, \ldots, m,  m+1$ at the front with the same sign.\\

It is easy to see that if $deg(x) > 0$ then $i^{n-1}\hat{r}^{n-1}(x) = 0$, since this sequence must end with repeated entries $n$.  If $deg(x) = 0$, we must look at the sign.

Now $x$ is just a permutation in $\Sigma_n$ and the sign associated to the compositions of $\hat{r}$'s in $i^{n-1} \hat{r}^{n-1}(x)$ is the parity of the number of moves it takes to move the 1 in $x$ to the front, then move the 2 to follow the 1, and finally move the $n-1$.   Thus $i^{n-1}\hat{r}^{n-1}(x) =  \tau(x)(1,2, \ldots, n) =  \tau(x) e = \rho(x)$.\\

For statement (iv),  $\hat{h}^2 = 0$  follows easily from the obvious identities $s^2 = 0,\ si = 0,\ \hat{r}s = 0,\ \hat{r}i = Id$.  Also $\hat{h} \iota(1) = \hat{h}(1,2, \ldots, n) = 0$.
\end{proof}
\end{rem}
\newpage

\section{The Berger-Fresse Complex $\bf \cS_*^{bf}(n)$}
 We now turn to the second surjection complex $\cS_*^{bf}(n) $.  The $\Sigma_n$ action $gx$ will simply post-compose surjections $x$ with permutations $g$.  The augmentation sends all permutations $g \in \cS_0^{bf}(n) = \FF[\Sigma_n]$ to $ \epsilon(g) = 1 \in \FF$. The base point is $\iota(1) = e \in  \FF[\Sigma_n]$. The fact that there are no signs in the $\Sigma_n$ action is the reason the Berger-Fresse complex has some good properties.

\begin{rem}\label{13.1}{\bf The Boundary Operator in $\bf \cS_*^{bf}(n)$.} The $\Sigma_n$-equivariant boundary operator in $\cS_*^{bf}(n)$ also  has the form $$d(x_1, x_2, \ldots, x_{n+k}) = \sum_{j=1}^{n+k} \gamma(j) (x_1, \ldots, \widehat{x_j}, \ldots x_{n+k}),$$ but the signs $\gamma(j) \in \{-1, 0, +1\}$ are tricky.\footnote{With no signs in the $\Sigma_n$ action, something exotic must occur in the boundary operator in order that the augmentation is an equivariant map to the {\it trivial} module $\FF$.} The signs are explained clearly in Berger-Fresse [3], in terms of signs attached to the entries of what they call the `table form' of a surjection.  We will explain these signs next.  The $\Sigma_n$ equivariance of $d$ will be obvious from the sign algorithm.  \\ 

Divide the entries of a surjection $x$ into rows, by inserting a separator  $|$ after each entry $x_j$ that is {\it not} the final occurrence of the value $x_j$ in the surjection.  The signs $\gamma(j) \in \{\pm 1\}$ in the boundary formula associated to these entries alternate, beginning with $+1$.  These entries $x_j$ are called the {\it caesuras} of the surjection $x$.\\

The remaining entries in the surjection, those not immediately preceding a separator $|$,  represent the final occurrences of values $x_i$.  Attach to such an $x_i$ the sign $\gamma(i)$ {\it opposite} to the sign already assigned to the entry $x_j$ with $x_j = x_i$ and $x_j$ the {\it immediately preceding} caesura occurrence of the value $x_i$ in the surjection.  If a value $x_i$ occurs only once in the surjection, a sign $\gamma(i)$ is irrelevant, since the function with the term $x_i$ deleted is not a surjection.  One just removes these terms from the boundary formula, or sets $\gamma(j) = 0$.  Other boundary terms may become degenerate, hence 0, even though they have an attached sign,  if deleting an entry results in two equal adjacent entries.\\

{\bf Example.}  Consider the surjection $x = (2,1,2,3,4,2,3,1,5,4, 1,2)$.  First insert the $|$'s after the caesuras, $x = (2|, 1|, 2|, 3|, 4|, 2|, 3, 1|, 5,4,1,2)$.   Next place alternating signs next to the caesuras, yielding $$(+2|, -1|, +2|, -3|, +4|, -2|, 3, +1|, 5,4,1,2).$$ This leaves the final occurrences of $3,5,4,1,2$ without signs.    The $5$ is a singleton and receives  no sign, or sign  0.  The signs assigned to the final occurrences of $3,4,1,2$ are opposites of the signs of the immediately preceding occurrences of $3,4,1,2$.  So, the final result is the surjection $x$  with a sign $+$ or $-$, equivalent to $\gamma(j)$, written before each entry $x_j$ that is not a singleton, $$x = (+2|, -1|, +2|, -3|, +4|,  -2|, +3, +1|, 5, -4, -1, +2).$$ From this the boundary $dx$ is easily read off.  Note deleting the first 1 entry results in a degenerate surjection, and deleting the 5 results in a function that is not a surjection, so these terms do not appear in the boundary.  $\qed$\\

We point out that $d^2 = 0$ is true, but somewhat awkward to verify.  The most elementary proof is to group the non-degenerate  terms of $d^2(x)$ in pairs, identical except for the signs.  The pairs correspond to pairs of the indices of the entries of $x = (x_1, x_2, \dots, x_{n+k})$.  Then consider various cases and argue the signs are opposite for each pair.  The cases are a pair of caesuras of $x$, a pair of non-caesuras of $x$, and pairs consisting of one caesura and one non-caesura.  The latter two cases have subcases.\\

It turns out the formula for $d$ is forced inductively by $d^2 = 0$ and two other properties. First, in all degrees the coefficient of the boundary term obtained by deleting the first caesura should be +1.  Second, in degree 1, the image of $d$ must be in the kernel of the augmentation, which means the coefficient of the term obtained by deleting the only non-caesura should be $-1$. $\qed$
\end{rem}
\begin{rem}\label{13.2}{\bf The Contraction of $\bf \cS_*^{bf}(n)$.} A contraction $h$ of $\cS_*^{bf}(n)$ is quite similar to the contraction $h$ of $\cS_*^{aj}(n)$ in the preceding section.  To begin, $s(x) = (1, x)$ will again satisfy a relation $ds + sd = Id - ir$.  Here the map $r \colon \cS_*^{bf}(n )\to \cS_*^{bf}(n-1)$ is 0 on any surjection $x$ that contains the entry 1 more than once.  Otherwise, $r(x)$ deletes the 1 and reduces all other entries by 1.   So, $r$ is the same as the previous $\hat{r}$ for the complex $\cS_*^{aj}(n)$, but without a sign.  The map $i \colon \cS_*^{bf}(n-1) \to \cS_*^{bf}(n)$ is the same as the previous map $i$, namely increase each entry of a surjection $y$ by 1 and put a 1 in front.  
\begin{prop}\label{13.3}
(i). The maps $r$ and $i$ are chain maps, that is,  $dr = rd$ and $di = id$.\\
  
(ii). It holds that $ds + sd = Id - ir.$\\

(iii). A contracting homotopy $h  \colon \cS_*^{bf}(n) \to \cS_{*+1}^{bf}(n)$ is given by $$h = s + isr + i^2 s r^2 + ... + i^{n-2} s r^{n-2}.$$
(iv). It also holds that  $h \iota = 0$ and $$h^2 = (s + isr + i^2 s r^2 + ... ) (s + isr + i^2 s r^2 + ... ) = 0.$$ 
\end{prop}
\begin{proof}
That  $dr= rd$ and $di = id$  is somewhat harder to verify than the corresponding results for $\cS_*^{aj}(n)$, again  because the boundary formula in $\cS_*^{bf}(n)$  is more complicated. One must pay  attention to the relation between the signs in the boundary $d$ and the caesura/non-caesura structure of $x$. In the $r$ case one considers cases where $x$ contains one, two, or more than two 1's.  \\

Just as in the previous section, the proof that $ds + sd = Id - ir$ proceeds by considering cases where $x$ contains one or more than one $1$'s.  Again one needs to be careful with the signs in the boundary map $d$. Note $sx = (1,x)$ is either degenerate or places a new caesura entry 1 in front of $x$.   \\

Pretty much the same telescoping sum argument as in Proposition \ref{12.1} of the previous section shows that $$h = s + isr + i^2 s r^2 + ... + i^{n-2} s r\ ^{n-2}$$ is a contraction of $\cS_*^{bf}(n)$, with $h^2 = 0$ and $h \iota = 0$.  The argument that $(dh + hd)(x) = x - \rho(x)$ is slightly simpler than the corresponding argument for $\cS_*^{aj}(n)$ because there are no signs to worry about and the base point $\rho$ is simply $\rho(g) = e$ for $g \in \FF[\Sigma_n] = \cS_0^{bf}(n).$ The main points, similar to the argument for $S_*^{aj}(n)$, are first  the telescoping argument that shows $(dh +hd)(x) = x - i^{n-1}r^{n-1}(x)$, and then $i^{\ell}  r ^\ell(x)  = 0$ unless $x$ contains $1,2 \ldots, \ell$ as singletons, in which case those entries are removed from $x$ and  $12 \ldots \ell$  is  placed in front, with no sign. \\
 
 We also  point out that $i^{\ell} s r\ ^\ell(x)  = 0$ unless $x$ contains $1,2 \ldots, \ell$ as singletons, in which case those entries are removed from $x$ and  $12 \ldots \ell(\ell+1)$ is placed in front,  with no sign.
\end{proof}
\end{rem}
\newpage
\section {The McClure-Smith Complex $\bf \cS_*^{ms}(n)$}
 Finally we turn to the third complex $\cS_*^{ms}(n)$, which historically appeared first [19].  McClure and Smith developed their surjection complexes from the outset as forming a suboperad of the Eilenberg-Zilber operad  $\cZ$ of natural multivariable cochain operations.  So the boundary operator and the $\Sigma_n$ action and the action on tensors of cochains as $n$ varies were all blended together.  This was a very natural approach, as the surjection generators acting on tensors of cochains provided a natural generalization of Steenrod's $\cup_i$ products.  In addition, they developed a filtration of their surjection operad, with terms related to the little cubes operads.  All in all, a very impressive piece of work, but not so easy to digest all at once.  We are less ambitious and develop at first only the surjection complex $\cS_*^{ms}(n)$ itself.\\
 
 The boundary operator and group action at first look complicated, but when viewed correctly they are really pretty simple, and geometrically motivated.  The $\Sigma_n$ action on surjection generators  will be postponed until after we define the boundary operator.  The augmentation sends all permutations $g \in \cS_0^{ms}(n) = \FF[\Sigma_n]$ to $ 1 \in \FF$. The base point is $\iota(1) = e \in  \FF[\Sigma_n]$.\\

Recall that we interpret surjection generators $x = (x_1, x_2, \ldots, x_{n+k})$ of $\cS_*^{ms}(n)$  as prisms (abstract for now), $Prism(x) = \prod_{1 \leq \ell \leq n} \Delta^{k_\ell -1}$, with a total order on the combined set of vertices of the factor simplices.    There is the non-degeneracy condition that adjacent vertices in the total order cannot belong to the same simplex factor.  One can interpret  the corresponding simplex factor of the prism as having the  $x^{-1}(\ell)$ as vertices, where  $x \colon \Delta^{n-1+k} \to \Delta^{n-1}$.   The total order on the combined set of factor vertices is inherited from the domain simplex. The integer $k_\ell$ is the cardinality of $x^{-1}(\ell)$, the number of times $\ell$ appears in the sequence $x$.  Then $\Delta^{n-1+k}$  is the ordered join of these subsimplices.

\begin{rem}\label{14.1}{\bf The Boundary Operator in $\bf \cS_*^{ms}(n)$.} The prism has an oriented geometric boundary, as the boundary of an ordered product of oriented manifolds, $$\partial Prism(x) = \bigsqcup_{1 \leq \ell \leq n} \Delta^{k_1-1} \times \ldots \times \partial \Delta^{k_\ell-1} \times \ldots \times \Delta^{k_n-1}.$$ Each boundary factor $\partial \Delta^{k_\ell-1}$ is written as an alternating signed sum of oriented simplex faces.  In front of that sum is another sign, namely $(-1)^{p_\ell}$, where $p_\ell$ is the dimension of the product of the preceding simplex factors.  We are using the `outward normal first' convention twice in this description of the boundary.  Thus the total boundary of the prism  is  a signed sum of codimension one oriented  prism faces, obtained by deleting single vertices from each prism factor of positive dimension.  The boundary $\partial \Delta^0$ of a $\Delta^0$ factor of $Prism(x)$ is empty.  The set of all vertices of  the factors of the non-empty boundary prism faces of course remain totally ordered, so represent other surjection generators of one less degree. Some of these become degenerate if deleting a vertex results in two other  vertices from that simplex factor becoming adjacent in the total order.\\
 
 Thus the  boundary operator in $\cS_*^{ms}(n)$ will  again  have the form $$d(x_1, x_2, \ldots, x_{n+k}) = \sum_{j=1}^{n+k} \gamma(j) (x_1, \ldots, \widehat{x_j}, \ldots, x_{n+k}),$$ with signs $\gamma(j) \in \{-1, +1\}$.   To understand the signs $\gamma(j)$, instead of removing entries from $x$ one at a time, beginning on the left, one can first remove the 1's, then the 2's, etc, from left to right.  The signs within each block alternate, with the first sign for an $i+1$ removal coinciding with the last sign for an $i$ removal.  The sign for the first 1 removal is +1.  This exactly describes the boundary of the associated oriented prism as a union of codimension one oriented prism faces.  Resulting degenerate terms or non-surjective terms in the boundary become 0, even though signs are attached initially to guide the full sign process.  It is obvious from topology that $d^2 = 0$.\\
 
{\bf Example.} If $x = (2,1,2 ,4,2,3,1 ,4, 1,2)$ then $$Prism(x) = (2,7,9) \times (1,3,5,10) \times (6) \times (4,8).$$ We can write signs $+, -$ equivalent to $\gamma(j)$  in front of each vertex to clarify the sign associated to the  corresponding codimension one prism boundary face obtained by deleting that vertex.
$$Prism(x) = (+2,-7,+9) \times (+1,-3,+5,-10) \times (-6) \times (-4,+8).$$  

We can then write these signs $\gamma(j)$ in front of each entry of $$x = (+2, +1, -2, -4, +2, -3, -1, +4 +1, -2),$$ and  read off the boundary.  In the boundary $dx$ there are only six non-zero terms, since deleting the first 1, either 4, or the 3, result in degenerate sequences. $\qed$
\end{rem}

\begin{rem}\label{14.2}{\bf The Permutation Action on $\bf \cS_*^{ms}(n)$.} We will next describe the $\Sigma_n$ action on $\cS_*^{ms}(n)$.  It has the form $$g*x = (-1)^{\mu(g, x)} gx,$$ where $gx$ is the post-composition of a surjection $x$ with a permutation $g$.  We will give two descriptions of the sign.  First, let $k_\ell = \# \{x^{-1}(\ell)\}$.  Then $$\mu(g, x) = \displaystyle \sum_{certain\ \ell, \ell'}(k_\ell -1)(k_{\ell'} -1),$$ where the sum is taken over all pairs $1 \leq \ell < \ell' \leq n$ with $g(\ell) > g(\ell')$.  Such a pair contributes to the sign $(-1)^{\mu(g, x)}$ only if both $k_\ell, k_{\ell'}$ are even.  Equivalently, the dimensions of the corresponding prism factors are odd.\\

We turn to the second description of the sign, which is more geometric in nature, and relates the prisms of the McClure-Smith complex $\cS_*^{ms}(n)$ to the complex $\cS_*^{aj}(n)$, where we regarded surjection $x$ as a simplicial map $x \colon \Delta^{n-1+k} \to \Delta^{n-1}$.  Then $k_\ell$ records the number of vertices in the domain that map by $x$  to the $\ell^{th}$ vertex of the range.  These domain vertices span  simplices $\Delta^{k_\ell -1}$, and the domain can be viewed as a join of $n$ simplices $$ \Delta^{n-1+k} =  \Delta^{k_1 - 1} * \Delta^{k_2 - 1} * \cdots *  \Delta^{k_n -1}.$$  Points of the join can be identified with convex liner combinations in some big affine space, $\sum t_\ell a_\ell$, $a_\ell \in \Delta^{k_\ell- 1}$, with $0 \leq t_\ell \leq 1, \sum t_\ell = 1.$  The base is a join of $n$ vertices $1,2 \ldots, n$, and the join $\Delta^{n -1+k}$ maps to the base $\Delta^{n-1}$ by $\sum t_\ell a_\ell \mapsto \sum t_\ell \ell.$   This map is $x \colon \Delta^{n-1+k} \to \Delta^{n-1}$, viewed as a map from the join of $n$ simplices to the join of $n$ points.\\

The inverse image $x^{-1}(b)$  under the map $x$ of the barycenter $b \in \Delta^{n-1}$  of the base   identifies with $Prism(x) = \Delta^{k_1 - 1} \times \Delta^{k_2 - 1} \times \cdots \times \Delta^{k_n -1} \subset \Delta^{n-1+k}$, under the embedding $(a_1, \ldots, a_n) \mapsto \sum \frac{1}{n}a_\ell$. Each vertex of the large simplex occurs as a vertex of one factor simplex of the prism.  The vertices of the abstract prism itself are $n$-tuples of factor vertices $(v_1, v_2, \ldots, v_n)$ with $x(v_\ell) = \ell$.  Each such prism vertex identifies with the barycenter of a face of the domain $\Delta^{n-1+k}$ that maps by $x$ isomorphically to the base $\Delta^{n-1}$.\\

We now bring in the permutation $g \in \Sigma_n$, which we will first view as a permutation isomorphism of the base $\Delta^{n-1}$ with $g(b) = b$. Given $x \colon \Delta^{n-1+k} \to \Delta^{n-1}$,   note the two prisms $Prism(x) = x^{-1}(b) \simeq \Delta_1 \times \ldots \times \Delta_n,$ and $Prism(gx) = x^{-1}g^{-1}(b) \simeq \Delta_{g^{-1}1} \times \ldots \times \Delta_{g^{-1}n}$ actually coincide as a submanifold of $\Delta^{n-1+k}$.  The set of vertices of the factors of the  two prisms are identical, hence both are  totally ordered.  But $Prism(x)$ is organized as $\prod \Delta_i$ while $Prism(gx)$ is organized as $\prod \Delta_{g^{-1}(i)}$.  There is an obvious geometric isomorphism $g_x$ of one prism to another, which permutes simplex factors. This isomorphism has a degree, since both prisms are oriented, and one has $$(-1)^{\mu(g, x)} = deg(g_x).$$  This clam is easy to prove since $(-1)^{\mu(g, x)}$ is just the Koszul sign associated to the permutation of factors of a tensor product.\\  

Another  interpretation of this  sign  is the following.  Let $j_1-1, \ldots, j_m-1$ denote the  odd integers among the factor simplex dimensions $k_1-1, \ldots, k_n-1$.  Then the permutation isomorphism $g$  rearranges these odd dimension factors, which occur in  specific orders in the two prisms.  Then the sign is the parity sign of the associated permutation in $\Sigma_m$ of these factors, since this parity counts mod 2 how the odd dimensional simplex factors are passed across each other by the isomorphism $g_x$.\\

{\bf Example.}  Suppose $n = 5$, $x = (2,4,1,3,4,2,5,1,5,2)$, and $g = (3,5,2,1,4)$.  Then the odd dimension factors of $Prism(x)$ are $\Delta_1 = \Delta^1, \Delta_4 = \Delta^1$, and $ \Delta_5 = \Delta^1$.  The order of the 1 and 4 factors are reversed by $g$, to 3 and 1.  The order of the other pairs 1,5 and 4,5 are preserved.  So $deg(g_x) = -1$.\\

If $g' = (5,3,1,2,4)$, then the order of the pairs 1,4 and 1,5 are reversed, and the order of 4,5 is preserved.  So $deg(g'_x) = +1$. $\qed$
\end{rem}

\begin{prop} \label{14.3} The definition $g*x = deg(g_x) gx$ defines a left group action of $\Sigma_n$ on $\cS_*^{ms}(n)$.  The boundary operator $d$ of $\cS_*^{ms}(n)$ is equivariant for  this $\Sigma_n$ action.
\end{prop}
\begin{proof} From our topological viewpoint this is more or less obvious. The first statement is clear because compositions of permutation isomorphisms of joins and prism subsets of joins behave in an obvious way.\\

For the second statement, our geometric bijections between oriented prisms with boundary commute with both the geometric and algebraic boundary maps. The boundary operator in $\cS_*^{ms}$ is specifically defined in terms of the prism boundary faces.  The key is that the degree of a permutation isomorphism between our oriented prisms coincides with the degree of the same permutation isomorphism restricted to each non-degenerate prism boundary face, oriented by the outward normal first convention.
\end{proof}

\begin{rem}\label{14.4} {\bf $\bf \cS_*^{ms}(n)$ as a Prismatic Cell Complex.} We insert here a very nice picture of the chain complex $\cS_*^{ms}(n)$.  It is the chain complex of a geometric cell complex whose open cells are the interiors of the various prisms $Prism(x)$.  On the boundaries there is collapsing of some prism faces to codimension two prisms, whenever deleting an entry of $x$ results in two adjacent entries with the same value.  A face obtained by deleting a singleton from $x$ is an empty face, so plays no role in the boundary, algebraically or geometrically.\\

 In order to understand this structure, form $|SS(\Delta^{n-1})|$, the geometric realization of the simplicial singular complex of a simplex.  There is an $n-1+k$ simplex for each simplicial map $x \colon \Delta^{n-1+k} \to \Delta^{n-1}$, surjective or not.  But the geometric realization collapses a simplex whenever two adjacent vertices in the domain have the same image.  There is the tautological map $\chi \colon |SS(\Delta^{n-1})| \to \Delta^{n-1}$. Then the cell complex underlying $\cS_*^{ms}(n)$ is $\chi^{-1}(b)$, where $b$ is the barycenter of the base.  Note singular  simplices that map to the boundary of $ \Delta^{n-1}$ are disjoint from $\chi^{-1}(b)$. \\
 
 It is more or less clear from this picture that there is a chain complex isomorphism $\phi \colon \cS_*^{aj}(n) \to \cS_*^{ms}(n)$, of the form $\phi(x) = p(x)x$, with $p(x) \in \{\pm 1\}$.  In fact,  the map $\phi$ is a kind of chain level Thom isomorphism.  The sign $p(x)$ is determined by thinking about orientations.  $Prism(x)$ is oriented as an ordered product of oriented simplices.  The domain and base of $x$ are canonically oriented simplices.  Then the sign $p(x) \in \{\pm 1\}$ is determined by the orientation equation $$p(x)\ o(Prism(x))\ o(Base(x)) = o(Domain(x)).$$ It is useful to observe that the normal bundle of $Prism(x) \subset Domain(x)$ is oriented as the pull back of the tangent space of $Base(x)$ at the barycenter. We will look at this chain complex isomorphism $\phi$ much more closely in the next section, including the $\Sigma_n$ equivariance of $\phi$. $\qed$
\end{rem}

\begin{rem}\label{14.5}{\bf The Contraction of $\bf \cS_*^{ms}(n)$.} Finally we turn to the contraction of $\cS_*^{ms}(n)$.  The contraction $h \colon \cS_*^{ms}(n) \to \cS_{*+1}^{ms}(n)$ is given by exactly the same formula as the contraction of the complex $\cS_*^{bf}(n)$ given  in the previous section. That is, the contraction is  again $$h = s + isr + \ldots + i^{n-2} s r^{n-2}.$$

In fact, the full Proposition \ref{13.2} remains true  for $\cS_*^{ms}(n)$.    Again,  $rx =  0$ unless $x$ contains a unique 1, in which case $r$ removes the 1 and reduces all other entries by 1, and $ix$ increases all entries of $x$ by 1 and puts a 1 in front.  There are no signs. Both $r$ and $i$ are chain maps, $dr = rd$ and $di = id$.  For a surjection $x$, set $s(x) = (1, x)$ and get $(ds+sd)(x) = x - ir(x)$.  Of course all these facts need to be checked using the description of the boundary $d$.  One can either use the  prisms and their geometric boundary in the definition of $d$ or the algorithm described above for placing signs $\gamma(j) $ in front of boundary summation terms obtained by deleting entries $x_j$ of $x$.  Note the number of 1's in any $x$ shows up in the first simplex prism factor $\Delta^{k_1 -1}$ of $Prism(x)$.  So $rx = 0$ unless that first prism factor is $\Delta^0$.  Also, for $sx$ the dimension of the initial prism factor is  increased  by one unless the initial entry of $x$ is 1 in which case $sx = 0$. \\

The telescoping argument and analysis of $i^{n-1}r^{n-1}$ go through as in the $\cS_*^{bf}(n)$ case to show $dh + hd = Id - \rho$ and $h^2 = 0$. $\qed$
\end{rem}

\begin{rem}\label{14.6}{\bf *Diagonals for Surjection Complexes.*} Since all the surjection complexes $\cS_*(n)$ are free over $\FF[\Sigma_n]$ and have preferred contractions, there are preferred recursively defined equivariant diagonals $\cS_*(n) \to \cS_*(n) \otimes \cS_*(n)$.  In fact, there are  preferred multidiagonals $\cS_*(n) \to \cS_*(n)^{\otimes k}$.  Exploiting the equivariant isomorphisms $ \cS_*^{bf}(n) \simeq \cS_*^{ms}(n) \simeq \cS_*^{aj}(n)$ of the next section, which commute with contractions, these diagonals become the `same'.  Also, these diagonals are coassociative.  All this can be seen by applying the basic uniqueness result Proposition \ref{6.2} and other results of Section 6 to various diagrams. \\

The McClure-Smith complex $\cS_*^{ms}(n) = C_*(X)$ is the cellular chain complex of the geometric realization of an $n$-fold multisimplicial set $X$, as indicated in Remark \ref{14.4} above.   To form $X(m_1-1, \ldots, m_n-1)$ one first looks at {\it all} surjections $x\colon \{1,\ldots, n+k\} \to \{1, \ldots, n\}$ with $m_i = x^{-1}(i)$ and $\sum m_i = n+k$, including the degenerate surjections with repeated adjacent values. Face operators delete an $i$-value and degeneracy operators repeat an $i$-value.  Thus on $\cS_*^{ms}(n)$  there is also an explicit  multisimplicial functorial diagonal studied in [23]. We treated the general multisimplicial functorial diagonal of [23] in Subsection 8.5.   The two diagonals for $\cS_*^{ms}(n)$ coincide, which can be seen by applying Proposition \ref{6.2} to the explicit sum of front subprisms tensor back subprisms formula of Proposition \ref{8.11}. In particular, one checks in this case that the formula of 8.11 is indeed equivariant. $\qed$
\end{rem}
\newpage
\section {The Isomorphisms $ \cS_*^{bf}(n) \simeq \cS_*^{ms}(n) \simeq \cS_*^{aj}(n)$}

\subsection{Basis Generators and Clean Generators.} We will want to construct equivariant maps with domains the surjection complexes.  For this we need a $\Sigma_n$-basis.  A very convenient choice in all three surjection complexes $\cS_*(n)$ consists of all surjections $b$ so that the initial occurrences of $1,2, \dots, n$ occur in that order.\\

 Basis generators are special cases of what we will call {\it clean}  surjection generators, by which we mean $c = (123\ldots\ell \ldots \ell \ldots)$ where $\ell$ is the first caesura of $c$.  In degree 0 permutations have no caesuras, so the only clean permutation is $e = (12\ldots n)$.   The notion of caesura was emphasized for the Berger-Fresse complex, but that notion and  the notion of clean generator  has the same meaning in any of the complexes.  We are going to prove that  the $\FF$-span of the clean generators is exactly the image of the contractions $h$ in the surjection complexes. This will important in Part III for using variants of the uniqueness theorems of Section 6 to identify certain maps with target a surjection complex as a standard procedure map.  $\qed$\\
 
There are  signs in the maps $ \hat{r}, \hat{h}$ for the surjection complex $\cS_*^{af}(n)$ from Section 12, but the statements in the proposition below hold for all the surjection  complexes if $r, h$ are interpreted as $\hat{r}, \hat{h }$ in that case. 

\begin{prop}\label{15.1}(i). For clean surjections $c \in \cS_*(n)$, $i ^j s r^j (c) = 0$ for all $j$. Thus  $hc = 0$ for clean $c$.   Therefore $c = hdc + dhc = hdc$ belongs to $Im(h)$.\\

(ii). For any $x$ and any $\ell$, $i^{\ell-1} s r^{\ell-1} (x)  = 0$ unless $1,2 \ldots ,\ell-1$ are singletons in $x$ and $\ell-1$ is not immediately followed by $\ell$.  In the non-zero case, $ \pm i^{\ell-1} s r^{\ell-1} (x) =  (12...\ell\ldots \ell ...)$ is clean.  Hence $Im(h) = Ker(h)$ is the $\FF$-span of all clean generators.\\

(iii). Basis generators $b$  have clean form $b = (12\ldots \ell (\ell+1) \ldots \ell \ldots)$, where $\ell$ is the first caesura, and initial entries of $1,2 \ldots, n$ occur in that order.  Any clean generator can  be uniquely  written $c = \pm gb$, where $b$ is a basis generator and $g \in \Sigma_n$ fixes $1,2,\ldots \ell$. \\

(iv). For basis generators $b$, in the `composition  of two sums' formula for $hdb$, it holds that $ i^{\ell - 1} s r^{\ell -1} ( d_jb ) = 0$ unless  $ j = \ell$ and $b_j = \ell$ is the first caesura of $b = (12 \ldots \ell(\ell+1)  ...\ell \ldots) $. Thus only one composed  term is non-zero in the double summation formula for $hdb = b$. \\

(v).  If $b \in \cS_*(n)$ is a basis generator of positive degree and $c \in \cS_*(q)$ is a clean generator, let $c[n]$ be the result of adding $n$ to each entry of $c$.  Then the concatenation $bc[n] \in \cS_*(n+q)$ is a clean generator and the conclusion of (iv) holds for $bc[n]$.  
\end{prop}
\begin{proof} This is a long exercise in reviewing the definitions of the operators $i, s, r$.   Parts (i) and (ii) give the result about clean generators and  the image of contractions mentioned above.  Part (iii) is obvious.  The sign arises because in two of the surjection complexes the permutation action on surjection generators involves a sign. The result  (iv) will be used in this section, and  (iv) and (v) will be used in the final section of Part III, where we study the operad structure associated to the complexes $\cS_*$.\\

Regarding (v),  for clean generators $c$ in general the double sum $hdc$ can have summands that cancel out in the formula $hdc = c$.  For example,  $c = (1,2,3,2,5,4)$. Then in the Berger-Fresse complex $dc = (1,3,2,5,4) - (1,2,3,5,4)$ and $hdc = ((1,2,3,2,5,4) + (1,2,3,4,5,4)) - (1,2,3,4,5,4)$.
\end{proof}
 \begin{rem}\label{15.2}{\bf Isomorphisms Between  Surjection Complexes.}  In Section 16 we will study standard procedure chain maps $N_*(E\Sigma_n) \to \cS_*(n)$. But already from Propositions \ref{6.3} and \ref{6.4} we know these chain maps are the unique equivariant maps of graded modules that commute with contractions, and that they are surjective.  No formulas are needed. In Lemmas \ref{15.3} and \ref{15.6} below we directly construct equivariant isomorphisms of graded modules $\cS_*^{bf}(n) \simeq \cS_*^{ms}(n) \simeq \cS_*^{aj}(n)$. The isomorphisms all have the form $\phi(x) = \pm x$ for generators $x$.  The signs are interesting geometrically. Since the surjection complexes are quotient complexes of $N_*(E\Sigma_n)$,  it follows easily from Propositions \ref{6.3} and \ref{6.4} that {\it any} composition in the sequence of equivariant graded module maps $N_*(E\Sigma_n) \to \cS_*^{bf} \simeq \cS_*^{ms} \simeq \cS_*^{aj}(n)$ is a standard procedure chain map, which moreover commutes with contractions.\\
 
 However, some interesting details concerning the isomorphisms between the surjection complexes are hidden in this quick proof. Therefore, in the next two subsections we will include more details. The two Lemmas \ref{15.3} and \ref{15.6}  and Proposition \ref{15.9} give explicit formulas for the isomorphisms between the surjection complexes, but other than that the details in these two subsections can be skipped. Some of the details use properties of basis generators from Proposition \ref{15.1}. $\qed$
 \end{rem}
  
\subsection{*The Isomorphism $ \phi \colon \cS_*^{bf}(n) \to \cS_*^{ms}(n)$*}
We will construct an equivariant chain map isomorphism between the two surjection complexes $\phi \colon \cS_*^{bf}(n) \to \cS_*^{ms}(n)$.  For surjections $x$   the map will have the form $\phi(x) = c(x) x$, where $c(x) \in \{\pm 1\}$.\\

Denote the caesuras of $x = (x_1, x_2, \ldots, x_{n+k})$  by $(c_1, c_2, \ldots, c_k)$ and let $sh(C_x)$ denote the number\footnote{Strictly speaking, only the parity of $sh(C_x)$ is well-defined.} of transpositions of adjacent entries that rewrites the caesuras in order $(1,.., 1,2, .. ,2, \ldots,n, .., n)$, keeping the caesuras of the same value in their original order.  Of course not all values $\ell$ will occur as caesura values if there are singletons in $x$.  The number of $\ell$'s, in either ordered form of the caesuras, is the dimension of the simplex factor $\Delta_\ell = \Delta^{k_\ell -1}$ of $Prism(x)$. Then we define $c(x) =  (-1)^{sh(C_x)}$ to be the parity sign of the shuffle permutation that rewrites the caesuras in the non-decreasing  order.\\
 
If $g \in \Sigma_n$, we have the map of prisms $g_x \colon Prism(x) \to Prism(gx) $ that permutes the simplex factors of the prisms.  The $\Sigma_n$ action on $\cS_*^{ms}(n)$ is given by $g*x = deg(g_x) gx$, as explained in  remark \ref{14.2}.
\begin{lem}\label{15.3}The map $\phi(x) = c(x)x$  is a $\Sigma_n$-equivariant isomorphism of graded modules $\phi \colon \cS_*^{bf}(n) \simeq \cS_*^{ms}(n)$.
\end{lem}
 To prove the lemma  it suffices to show $c(gx) = deg(g_x) c(x)$. Because then $$g*\phi(x) = g* c(x)x = c(x) deg(g_x) gx = c(gx) gx = \phi(gx).$$ 
If $x = (x_1, x_2, \ldots, x_{n+k})$ has caesuras $(c_1, c_2, \ldots, c_k)$ then the caesuras of $gx$ are $(gc_1, gc_2, \ldots, gc_k)$  We put these in non-decreasing order in two steps.  First in $sh(C_x)$ steps, we reorder the caesuras of $gx$ to  $(g1, \ldots, g1, g2, \ldots, gn)$.    Next we put these caesuras of $gx$ in non-decreasing order by permuting the blocks of $g\ell$'s. For computing signs, only the blocks of odd size are relevant. For these, one counts the pairs $1 \leq i < j \leq n$ with $gi > gj$.  From Remark \ref{14.2}, one sees the sign associated to the second step is $deg(g_x)$.  The lemma is thus proved. $\qed$

\begin{prop} \label{15.4}The map $\phi \colon  \cS_*^{bf}(n) \to \cS_*^{ms}(n)$ defined by $$\phi(x) = (-1)^{sh(C_x)} x = c(x) x$$ is  the equivariant chain map given by the standard procedure, using the chosen basis of the domain and the contraction of the range.  Moreover, $\phi $ commutes with the contractions $h$  in the two complexes.
\end{prop}
\begin{proof} 
We will first prove  that $\phi$ is the standard procedure chain map by using Proposition \ref{15.1}(iv).  On basis generators we want $ \phi(b) =  c(b)b = h\phi (db) $. From 15.1(iv) and the definition $\phi(d_\ell b) = c(d_\ell b) d_\ell b$, we see that  the double sum expansion of $h \phi(db)$ reduces to a single term, $h\phi(db)  = c(d_\ell b) b$, where $d_\ell b$ is the boundary term  given by deleting the first caesura of $b$.  But obviously $c(d_\ell b) = c(b)$, since deleting the first caesura of a basis generator doesn't affect the process of putting the caesuras in the relevant order.\\

For an alternate proof that $\phi$ is the standard procedure chain map, it suffices by Proposition 6.3 of Part I to just directly show that $\phi$ commutes with contractions, that is,
$$\phi \big(\sum_{1 \leq \ell \leq n1} i^{\ell-1} s r^{\ell-1} x \big) = \sum_ {1 \leq \ell \leq n-1} i^{\ell-1} s r^{\ell-1} \phi(x).$$
There is actually a form of commutativity with each term of the contractions.
\begin{lem}\label{15.5} For surjection generators $z, y, x \in \cS_*^{bf}(n)$ the following commutativities hold:  $$\phi(iz) = i \phi(z)$$
$$\phi(sy) = s \phi(y)$$
$$ \phi(rx) =  r \phi(x).$$
\end{lem}
Assuming this lemma, Proposition \ref{15.3} is proved by repeatedly applying the different parts of the lemma to calculate $\phi(h x)$.  First, $\phi(sx) = s \phi(x)$.  Next, 
$$\phi(i s r x) =  i\ \phi(srx) =  is\ \phi(rx) = isr\ \phi(x).$$
Next,  an induction gives  $\phi (i^m s r^m x) =  i^m s r ^m \phi(x)$.\\

The lemma itself is immediate from the definitions of $i, s, r$ and the observations $c(iz) = c(z)$, $c(sy) = c(y)$, and $c(rx) = c(x).$ These last equalities are consequences of the facts that the operators $i, s, r$ don't affect the positions of relevant caesuras that get permuted.
\end{proof}

\subsection{*The Isomorphism $ \phi \colon \cS_*^{aj}(n) \to \cS_*^{ms}(n)$*}
We will  next construct an equivariant chain map isomorphism between the two surjection complexes $\phi \colon \cS_*^{aj}(n) \simeq \cS_*^{ms}(n)$.  For a surjection $$x \colon \{1,2, \ldots, n+k\} \to \{1, 2, \ldots, n\},$$ the map (in both directions) will again  have the form $\phi(x) = p(x) x$, where $p(x) \in \{\pm 1\}$.   In degree 0, that is when $k = 0$, then $x = g$ is a permutation in $\Sigma_n$ and we define $p(g) = \tau(g)$,  the parity character of $g$.  In this degree, $ \tau(g)$ is the degree of the isomorphism of a simplex $\Delta^{n-1}$ induced by permutation $g$. This degree is just the sign of the comparison of two orientations of the simplex. In degree 0, $Prism(x)$ is a point, $deg(g_x) = 1$, and $\phi$ is equivariant.\\

To define $p(x)$ in general, we refer to the prism $$Prism(x) = \Delta^{k_1-1} \times \ldots \times \Delta^{k_n-1} \subset 
\Delta^{n-1+k},$$ identified with the inverse image of the barycenter under $x \colon \Delta^{n-1+k} \to \Delta^{n-1}$.  We compare two orientations of the domain simplex of $x$.  First, it has its standard orientation as an ordered simplex.  The prism submanifold also has a standard orientation as an ordered product of oriented simplices.  The normal bundle of the prism  identifies with the pull back of the tangent bundle of the base simplex of $x$ at the barycenter, which is oriented.  As discussed in Remark \ref{14.4}, we can amalgamate the prism and base orientations to define a second orientation of the domain and define the sign $p(x) \in \{\pm 1\}$ by the orientation equation $$p(x)\ o(Prism(x))\ o(Base(x)) = o(Domain(x)).$$

\begin{lem}\label{15.6}  The map $\phi(x) = p(x) x $ is a $\Sigma_n$-equivariant isomorphism of graded modules $\cS_*^{aj}(n) \to \cS_*^{ms}(n)$.
\end{lem}
For the equivariance we need $\phi(g*x) = g* \phi(x)$, which says $$\tau(g) p(gx) gx = \deg(g_x) p(x) gx.$$  The prisms of $x$ and $gx$ in the domain $\Delta^{n-1+k}$ are identical submanifolds, but  their simplex factors are permuted by the isomorphism $g_x$.  This introduces a sign $deg(g_x)$ comparing the two  prism orientations.  But also the vertices of the base $\Delta^{n-1}$ are permuted by $g$, which introduces a sign $\tau(g)$ comparing the two base orientations.  These two signs yield the desired equivariance equation. $\qed$

\begin{prop}\label{15.7}(i).  The map $\phi \colon \cS_*^{af}(n) \to \cS_*^{ms}$  coincides with the standard procedure chain map constructed using the contraction $h$ of $\cS_*^{ms}(n)$ and the basis of $\cS_*^{aj}(n)$.\\

(ii). The map $\phi$ commutes with the contractions $\hat{h}, h$ in the two complexes.  
\end{prop}
\begin{proof}To  prove (i), we use Proposition \ref{15.1}(iv) and follow the argument of Proposition \ref{15.4}.  For a basis generator $b \in \cS_*^{aj}(n)$ we need $p(b) b = \phi(b) = h \phi (db)$.  But from 15.1(iv), there is only one term in the double sum on the right side, namely $h \phi(db) = p(d_\ell b) b$, where $b_\ell$ is the first caesura of $b$.  But $p(d_\ell b) = p(b)$.  Because  deleting the first caesura names the first non-trivial boundary face of $Prism(b)$, which has coefficient $+1$ in the $\cS_*^{ms}(n)$ boundary formula and coefficient $(-1)^{\ell -1}$ in the $\cS_*^{aj}(n)$ boundary formula.  But in the domain $\Delta^{n-1+k}$, the face opposite vertex $x_\ell$ has a change of orientation sign $(-1)^{\ell - 1}$ in the simplex boundary formula.  Thus the two signs cancel when comparing the orientation equations for $p(b)$ and $p(d_\ell b)$.\\

The proof that $\phi$ commutes with contractions $\hat{h} = s -is\hat{r} + i^2 s \hat{r}^2 - \cdots \pm i^{n-2}s\hat{r}^{n-2}$ and $h = s + isr+ \cdots + i^{n-2} s r\ ^{n-2}$  is quite a bit more subtle than the corresponding situation in Proposition 15.4 for the isomorphism $\cS_*^{bf} \simeq \cS_*^{ms}.$\\

As in Lemma \ref{15.5} there is  a form of commutativity with each term of the contractions.
\begin{lem} \label{15.8} For surjection generators $z, y, x \in \cS_*^{aj}(n)$, the following commutativities hold, with $|z|$ and $|x|$ denoting the degrees of $z$ and $x$:
$$\phi(iz) = (-1)^{|z|} i \phi(z)$$
$$\phi(sy) = s \phi(y)$$
$$ \phi(\hat{r}x) = (-1)^{|x|} r \phi(x).$$
\end{lem}
Assuming the lemma, Proposition \ref{15.7} is proved by repeatedly applying the different parts of the lemma to calculate $\phi(h x)$.  First, $\phi(sx) = s \phi(x)$.  Next, since $|rx| = |x|$ and $|srx| = |x| + 1$,
$$\phi(-i s \hat{r} x) = (-1)(-1)^{|x|+ 1}\ i\ \phi(s \hat{r} x) = (-1)^{|x|}\ is\ \phi(\hat{r}x)  =  isr\ \phi(x).$$
Then by an  induction  $\phi ((-1) ^m i^m s \hat{r}^m x) = i ^m s r^m \phi(x)$.\\

The lemma itself follows from the definitions of $i,s, \hat{r}$ and $  r$, and the prism sign relations $$p(iz) = (-1)^{|z|} p(z),\ \  p(sy) = p(y),\   {\rm and}\ \  p( rx) = (-1)^{\ell - 1}(-1)^{|x|} p(x)\ {\rm if}\ rx \not= 0.$$  We interpret the third of these prism sign relations as follows. We can regard $r$ as an operator with no sign  on surjection generators of any of the surjection complexes. Then $rx = 0$ unless $x$ has a singleton 1 entry, $x_\ell = 1$. This is the $\ell$ appearing in the third relation.\\

The three prism sign formulas follow from close scrutiny of the geometry behind various prism sign orientation equations. Alternatively, they can be deduced relatively easily from Proposition \ref{15.9} below.\\

 Using these prism sign relations, the only tricky deduction for the lemma is the third line.  Recall $\hat{r}x$ and $rx$ are either both 0 or differ  by a sign $(-1)^{\ell -1}$ if $x$ has a singleton 1 entry, $x_\ell = 1$. Then either both $\phi(\hat{r} x) $ and $r \phi(x)$ are 0 or $$\phi(\hat{r} x) = (-1)^{\ell -1} \phi(rx) =   (-1)^{\ell-1} p(rx) rx = (-1)^{|x|}  p(x) rx = (-1)^{|x|} r\phi(x).$$
\end{proof}

{\bf Computation of $\bf{p(x)}$.} The prism  sign $p(x)$ has a nice geometric definition, but it is not expressed directly in terms of the surjection $x$.  To calculate  $p(x)$ we refer to the organization of the prism associated to $x = (x_1, \ldots, x_{n+k})$,
$$Prism(x) = (x_{11}, \ldots, x_{1k_1}) \times \ldots \times (x_{n1}, \ldots, x_{nk_n}).$$
All the integers between 1 and $n+k$ occur exactly once as subscripts of these $x$ entries. The entries in each simplex factor are increasing and the last entries of each are the non-caesura entries of $x$.  In the interest of slightly cleaner notation, write  $x_{\ell k_{\ell -1}} = x_{\ell k_\ell'}$.  Of course if $k_\ell = 1$ there are no caesura vertices $x_{\ell j}$ and the corresponding prism factor is $\Delta^0$.\\

Let $p_x \in \Sigma_{n+k}$ denote the permutation  read off from the {\it subscripts} of the $x$ entries written in the order $$p_x \equiv (x_{11}, \ldots, x_{1 k_1'}, \ldots, x_{n1}, \ldots, x_{n k_n'}, x_{1k_1}, \ldots, x_{nk_n}). $$
We recognize $p_x$ as the result of three shuffles of the entries of $x$.  First move all caesuras of $x$ in front of non-caesuras in $sh(C, N)$ steps, without changing the order of the caesuras.  Then put the caesuras in order in $sh(C)$ steps, as before.  Finally, the non-caesuras of $x$ are put in the order $(x_{1k_1}, \ldots, x_{nk_n})$, which essentially amounts to viewing the non-caesuras as they occur in $x$, one for each value $\ell$, as a permutation  $f_x$ of $(1, \ldots, n)$.  
\begin{prop}\label{15.9} The prism sign associated to $x$ is given by $$p(x) = \tau(p_x) = (-1)^{sh(C, N)} (-1)^{sh(C)} \tau(f_x).$$
\end{prop}
\begin{proof}An automorphism of a simplex given by a permutation of vertices has a geometric degree, which is simultaneously the parity sign of the permutation and the  orientation sign  obtained by comparing two lists of tangent vectors arising from two vertex orderings.  The proposition  follows from the fact that the orientation $o(Prism) o(Base)$ of the domain coincides with the orientation given by the list of tangent vectors arising from permutation $p_x$.\\

Specifically, we are naming simplex vertices by unit basis vectors in an affine space.  A pair of vertices then determines a tangent vector $uv = v - u$.  Notice $uv + vw = uw$. The tangent orientation associated to a vertex ordering $(v_1, v_2, \ldots, v_{n+k)}$ is  given by the ordered list of tangent vectors $$(v_1v_2, v_2v_3, \ldots, v_{n+k-1}v_{n+k}).$$   The tangent orientation does not change if a vector early in the list is changed  by adding a later vector.  For example, by adding to $v_i v_{i+1}$ some consecutive following vectors, we can replace $v_iv_{i+1}$ by $v_iv_j$ for any $j > i$.\\

Now we look at the tangent vectors in the order prescribed by adjacent vertices in permutation $p_x$.  We replace each adjacent pair of caesura vertices of form $x_{ik_i'}x_{j1}$ by $x_{ik_i'}x_{i k_i}$.  A non-caesura vertex $x_{ik_i}$ follows any caesura vertex $x_{j1}$. The new list of tangent vectors obviously corresponds to the $o(Prism)o(Base) $ orientation of the domain, since the first blocks now name in order the orientations of the non-trivial factors $\Delta^{k_i-1}$ of the prism and the last block projects to the standard orientation of the base $\Delta^{n-1}$.
\end{proof}

\begin{rem} \label{15.10}  We comment on the terms in the formula for $p(x)$ in Proposition \ref{15.9}.  Of course $(-1)^{sh(C)} = c(x)$ is the sign occurring previously in the isomorphism between $\cS_*^{bf}(n)$ and $\cS_*^{ms}(n)$.  Specifically, $sh(C)$ counts the number of moves needed to put the caesura values  of $x$ in non-decreasing order, without changing the order of cesuras of the same value.\\

There are interesting alternate interpretations of the parity of $sh(C, N)$.  The caesuras can be organized in blocks, the $\ell^{th}$ block consisting of those preceding the first non-caesura if $\ell = 1$, and between the $(\ell-1)^{th}$ and $\ell^{th}$ non-caesura if $\ell > 1$.  A block may be empty, but it still receives a number as a block.  Then the parity of $sh(C, N)$ is the same as the parity of the total number of caesuras in even numbered blocks.\\

Another interpretation of the parity sign $(-1)^{sh(C, N)}$ is as the parity sign $\delta(x) \in \{\pm 1\}$ of the number of subscripts $1 \leq j \leq n+k$ where the signs $\gamma(j)$ in front of boundary terms $dx = \sum \gamma(j)d_j x$  in the two complexes $\cS_*^{aj}(n)$ and $\cS_*^{bf}(n)\  ${\it differ}. Thus we have $p(x) = c(x) \delta(x) \tau(f_x)$. $\qed$
 \end{rem}
\begin{exam}\label{15.11} Consider the surjection $x = (2,1,2,3,4,2,3,1,5,4, 1,2)$.  We found the caesura signs $\gamma(j)$ in the complex $\cS_*^{bf}(5)$ in Example \ref{13.1} which were $(+2|, -1|, +2|, -3|, +4|, -2|, 3, +1|, 5,4,1,2).$  These agree with the alternating signs in $\cS_*^{aj}(5)$, except for the final caesura $+1 |$.  Thus $\delta(x) = -1$. We also see $sh(C, N) = 1$, since only the final caesura entry 1 needs to be moved in front of a single non-caesura entry 3.\\

The caesura shuffle for $x$ is given by $(2,1,2, 3,4,2, 1) \mapsto (1,1,2,2,2,3,4)$, with $sh(C) = 8$.  The caesura blocks are $(2,1,2,3,4,2)$ and $(1)$. So again we see $sh(C,N) = 1$. The non-caesuras permutation is $f_x = (3,5,4,1,2)$, with $\tau(f_x) = -1$.  \\

We have $Prism(x) = (2, 8, 11) \times (1,3,6,12) \times (4, 7) \times (5, 10) \times (9)$.  The $\Sigma_{12}$ permutation  $p_x = (2, 8, 1, 3, 6, 4, 5, 11, 12, 7, 10, 9)$, obtained by writing the 7 caesura entries in the prism in order, followed by the 5 non-caesura entries.  One can check $\tau(p_x) = +1.$\\

Summarizing,  $c(x) = (-1)^8$, $(-1)^{sh(C, N)} = (-1)^1 = \delta(x) $, $\tau(f_x) = -1$.      Finally, $p(x) = \tau(p_x) = c(x) \delta(x) \tau(f_x) = (+1)(-1)(-1)  = +1.$ $\qed$
\end{exam}

We can combine the above results to construct an equivariant chain map isomorphism between the other pair of  surjection complexes $\phi \colon \cS_*^{aj}(n) \to \cS_*^{bf}(n)$.  
\begin{prop}\label{15.12} The map $\phi \colon \cS_*^{aj}(n) \to \cS_*^{bf}(n)$ given by $$\phi(x) = p(x)c(x) x = (-1)^{sh(C, N)}\tau(f_x)x  = \delta(x) \tau(f_x)x   $$  is an equivariant chain map commuting with contractions, hence is the standard procedure chain map constructed from the basis of the domain and the contraction of the range.
\end{prop}

\newpage

\section{The Chain Maps    $ N_*(E\Sigma_n) \rightleftarrows \cS_*^{aj}(n), \cS_*^{bf}(n), \cS_*^{ms}(n)$ } 
The complex $N_*(E\Sigma_n)$ provides quite a large acyclic resolution of the trivial $\Sigma_n$ module $\FF$.  The three surjection complexes provide  rather small resolutions, perhaps minimal in some interesting sense.  The comparisons we make in this section are rather similar to the comparisons  $M_* \leftrightarrows N_*(EC_n)$ we made in Examples 6.8 and 6.11 of Part I between the minimal model and the MacLane model  for the cyclic group.  There are also interesting connections with cases of the Eilenberg-Zilber map and deformations of the Alexander-Whitney map that we will bring out in Subsection 16.3.\\

Following Part I, we will use contractions of the ranges and $\FF[\Sigma_n]$-bases of the domains to construct equivariant chain maps representing each arrow in the title line of this section.  From Proposition 6.3 and Proposition 6.4 of Part I, for any of the surjection complexes the compositions  $\cS_*(n) \to N_*(E\Sigma_n) \to \cS_*(n)$   will be identity maps, regardless of the choices of bases.  In the $\cS_*^{bf}(n)$  case, our maps are the same as those found by Berger-Fresse [3], [5].  In particular, we prove the Berger-Fresse formulas {\it are} chain maps and result from the standard procedure.   

\subsection{The Maps  $ N_*(E\Sigma_n) \to \cS_*^{aj}(n), \cS_*^{bf}(n), \cS_*^{ms}(n)$}
The contractions $\hat{h}$ of $\cS_*^{aj}(n)$ from Proposition \ref{12.1} and $h$ of $\cS_*^{bf}(n)$ and $\cS_*^{ms}(n)$  from Proposition \ref{13.2} and Section 14, along with the  $\FF[\Sigma_n]$ basis  in degree $k$ of the MacLane model $N_*(E\Sigma_n)$ given by $\{(e, g_1, \ldots, g_k)\}$, yield by the standard procedure of Definition 6.0 of Part I  equivariant chain maps $tr \colon N_*(E\Sigma_n) \to \cS_*^{aj}(n)$,  $TR \colon N_*(E\Sigma_n) \to \cS_*^{bf}(n)$, and $TR \colon N_*(E\Sigma_n) \to \cS_*^{ms}(n)$.  We steal the name $TR$ from Berger-Fresse [3], where they called their map $TR$ `Table Reduction'.  In all cases we can exploit Proposition 6.3 and Proposition 6.4 of Part I, which bring out that $tr$ and $TR$ are the unique equivariant maps commuting with contractions.  Specifically, $tr$ and $TR$  are  equivariant and are  defined recursively on  basis elements  in degree $k$ by $$tr (e, g_1, \ldots, g_k) = \hat{h}\  tr(g_1, \ldots, g_k) = (s - is\hat{r} + \ldots \pm i^{n-2} s \hat{r}^{n-2}) tr(g_1, \ldots, g_k)$$
$$TR (e, g_1, \ldots, g_k) = h\  TR(g_1, \ldots, g_k) = (s + isr + \ldots + i^{n-2} s r^{n-2}) TR(g_1, \ldots, g_k).$$
Then we extend these formulas on basis generators to other elements by equivariance and linearity.  Although the formulas for $TR$ on basis elements are identical in the $\cS_*^{bf}(n)$ and $\cS_*^{ms}(n)$ cases, there are signs in the $\Sigma_n$ action on $\cS_*^{ms}(n)$ that make that case a little more complicated.  As usual, we have shifted the burden from coming up with  formulas and proving they are equivariant chain maps  to beginning with  recursively defined equivariant chain maps and then finding formulas.\\

We can write out the fully iterated formula for $TR$ in the $\cS_*^{bf}(n)$ case as follows:

\begin{prop}\label{16.1} The equivariant chain map $TR\colon N_*(E\Sigma_n) \to \cS_*^{bf}(n)$ is described by:
\begin{itemize}
\item In degree 0,  $TR(g_0) =  g_0.$
\item In degree 1, $$ TR(g_0, g_1) = g_0\  TR(e, g_0^{-1}g_1) = g_0\  h( TR(g_0^{-1} g_1)) = g_0\ h(g_0^{-1} g_1).$$
\item In degree 2, $$TR(g_0, g_1, g_2) =  g_0\  TR( e, g_0^{-1}g_1, g_0^{-1}g_2) $$    $$=  g_0\  h\ (TR(g_0^{-1}g_1, g_0^{-1}g_2)) =  g_0\ h (g_0^{-1} g_1\  h( g_1^{-1}g_2)).$$
\item In degree $k$,  $TR(g_0, g_1, \ldots, g_k) =$
$$ g_0\  h (g_0^{-1}g_1\ h(g_1^{-1} g_2\ h( \ldots h(  g_{k-2}^{-1}g_{k-1}\  h (g_{k-1}^{-1}g_k))..).$$
\end{itemize}
\end{prop}

The $\Sigma_n$ actions here simply mean post-compose surjection generators with permutations.  There are no signs in either the $\Sigma_n$ action or the $h$ evaluations.  Thus the formula reveals that $TR(g_0, g_1, \ldots, g_k))$ will be a positive coefficient sum of surjection generators.\\

The algorithm is certainly programmable. Nonetheless, this is a clunky formula.    There are  many $\Sigma_n$ evaluations in the formula. There are $k$ evaluations of $h$, so potentially a sum of $(n-1)^k$ surjection generators in the end.  But many terms are zero because  evaluations of $i^{\ell} s r ^{\ell}$ are zero on surjections that do not contain $1,2, \ldots, \ell$ as singletons. Many other terms will be degenerate.\\

The algorithm of Proposition \ref{16.1}  also applies to the other two surjection complexes.  However, there are quite a few signs, in the $\hat{h}$ formula for $\cS_*^{aj}(n)$ and in the $\Sigma_n$ action on both complexes.  Therefore it is harder to translate 16.1 to a closed formula in the cases $\cS_*^{aj}(n)$ and $\cS_*^{ms}(n)$.\\

We will now {\it derive}\footnote{Shamelessly benefitting from the fact that we already know the answer from [3].} the direct description Berger-Fresse gave for the map $TR \colon N_*(E\Sigma_n) \to \cS_*^{bf}(n)$, which is more of a closed formula than that given by Proposition \ref{16.1}.  Again, our main point is not just to verify  that the Berger-Fresse map $TR$ is an equivariant chain map, but rather to show that their $TR$ formula agrees with  the canonical equivariant chain map constructed by the standard procedure, and that it commutes with contractions.\\

The following notation will be useful.  If $X = (g_0, g_1, \ldots, g_k)  \in N_k(E\Sigma_n)$ and $1 \leq \ell \leq n-1$, we denote by $X<12\ldots \ell -1> \in N_k(E\Sigma_{n -\ell+1})$. the result of deleting all entries $1,2, \ldots, \ell -1$ from the $k+1$ permutations $g_j$  that form $X$.  So if $\ell = 1$, this is just $X$.  Otherwise  we interpret $\Sigma_{n -\ell +1}$ as permutations of $\{\ell, \ldots, n\}$. We first note that in degree 1, $$TR(e, g) =  h(g) =  (s + isr + \ldots + i^{n-2} s r\ ^{n-2}) (g). $$  Since $i^{\ell-1} s r\ ^{\ell -1}(g) = (12 \ldots \ell g<12\ldots \ell -1>)$, we have

$$ TR(e, g)  =   \displaystyle  \sum_{1 \leq \ell \leq n-1}(12\ldots \ell \ g<1,2,\ldots,\ell -1>).$$

Note that the indexing set of this last sum of surjections can be thought of as being over ordered  partitions $a_0 + a_1= n+1$ with $0 <  a_0 = \ell < n$ and hence $a_1 > 1$.  The initial sequence entries $12\ldots \ell$ are the first $a_0 = \ell$ entries of the identity permutation $e = (12\ldots n)$, and the final sequence entries $g<1,2,\ldots,\ell -1>$ are the $ a_1 = n+1 - \ell$ entries of $g$ that remain after $1, 2, \ldots ,\ell -1$ are removed from $g$. \\

Next we bring in the equivariance and look at $TR(g_0, g_1) = g_0 TR(e, g_0^{-1}g_1)$. A brief computation shows that $TR(g_0, g_1)$ is a sum of surjections in $\cS_*^{bf}(n)$ of degree 1 described as follows.  The sum is  over ordered  partitions $a_0 + a_1 = n+1$ with $0 < a_0 < n$. The corresponding surjections are  described as the first $a_0$ entries of $g_0$, followed by the $a_1$ entries of $g_1$ that remain after the first $a_0 - 1$ entries of $g_0$ are removed.\\

To continue the search  for a pattern, look at $TR(e, g_1, g_2) = h\ TR(g_1, g_2) $ and then at $TR(g_0, g_1, g_2) = g_0\ TR(e, g_0^{-1} g_1, g_0^{-1} g_2)$.  After some computations, one might rediscover the Berger-Fresse procedure for the map $TR$.
\begin{prop}\label{16.2} View $X= (g_0, g_1, \ldots, g_k) \in N_k(E\Sigma_n)$ as a table, with rows the permutations $g_j$.  The image $$TR(X) = \sum_{partitions\ a} x_a \in \cS_k^{bf}(n)$$ is computed as a sum of surjections indexed by positive partitions $a_0 + a_1 + \ldots + a_k = n +k$ with $a_k > 1 $.  Given a partition $a$, the corresponding surjection summand $x_a$ of $TR(X)$ is described in sequence form as follows.  Begin  the sequence with the first $a_0$ entries of $g_0$.  Remove the first $a_0 -1$ of these entries, that is, all but the last one, from the remaining $g_j, j \geq 1$.  Continue the sequence with the first remaining $a_1$ entries of $g_1$, and then remove all but the last of those entries from what remains of the $g_j, j \geq 2$.   And so on....  
\end{prop}
Note the sequence $x_a$ ends with all the entries of $g_k$ that remain after the first $k$ steps.  Berger-Fresse describe each summand surjection $x_a$ in terms of separate  rows  of length $a_j$,  which are initial segments of  remnants of the permutation  rows $g_j$ of the table $X$.  The final entries of all but the last of these rows are the caesuras of the surjection $x_a$, in the language of Section 13.  This observation makes it obvious that $TR(X)$ is a sum of distinct surjections $x_a$, and  the row form of a surjection $x$ determines whether $x $ is an $x_a$, and for which partition $a$.  Also note that if $a_k = 1$ then the described sequence is degenerate, and note that  $a_k > 1$ implies $a_j< n$ for $j < k$.
\begin{proof} What we will prove is that the formula on the right side of the equation for $TR(X)$  in Proposition \ref{16.2} is equivariant and commutes with contractions. Then we employ Proposition 6.3 of Part I.  The equivariance is  clear from the description of the procedure.  The reason this is easy is because of the simple post-composition interpretation of the $\Sigma_n$ action on $\cS_*^{bf}(n)$, with no signs.\\

We then need to prove $TR (e, g_1, \ldots, g_k) = h\  TR(g_1, \ldots, g_k).$ We have by definition of the contraction $h$  $$ h\  TR(g_1, \ldots, g_k) = (s + isr + \ldots + i^{n-2} s r^{n-2}) TR(g_1, \ldots, g_k).$$ Set $Y = (g_1, \dots, g_k) \in N_{k-1}(E\Sigma_n)$.  By definition of the $TR$ procedure $$TR (e, g_1, \ldots, g_k) = \sum_\ell (12\ldots \ell\ TR(Y<12 \ldots \ell-1>)),$$  where recall $Y<12\ldots \ell-1>$ means  entries $1,2 \ldots, \ell -1$ are removed from the permutations that form $Y$.  Proposition \ref{16.2} is then immediate from the following:
\end{proof}
\begin{lem}\label{16.3} For $Y \in N_{k-1}(E\Sigma_n)$ and $1 \leq \ell \leq n-1$, it holds that $$i^{\ell-1} s\  r\ ^{\ell -1} TR(Y) = (12\ldots \ell\ TR(Y<12 \ldots \ell-1>)) \in \cS_*^{bf}(n).$$  
\end{lem}
\begin{proof}
On the right side,  the $TR$ application yields a sum of surjections indexed by  partitions of $n + k  -\ell $ into $k$ summands. All  integer entries of the resulting surjections are now from the set $\{\ell,  \ldots, n\}$.  Then $12\ \ldots \ell$ is placed in front of each surjection, putting the expression back in $\cS_k^{bf}(n)$.\\

On the left side, $TR(Y)$ is a sum of surjections in $\cS_{k-1}^{bf}(n)$, with index set  partitions $a_1 + \ldots + a_k = n+k-1$.  Applying  $i^{\ell -1} s\ r\ ^{\ell-1}$ kills all partition terms other than those that yield surjections that contain  $1, 2, \ldots \ell -1$ as singletons. If $s_j$  of these singletons in such a surjection occur as an entry of the  remnant of the $g_j$ row of $Y$ then the partition term $a_j > s_j$.   If we reduce all such terms $a_j$ by $s_j$, we get a partition of $n+k  - \ell$ into $k$ summands.\\

Going the other direction, given a partition  $a'$, say $a'_1 + \ldots + a'_k = n+k  - \ell $, form the corresponding surjection $x_{a'}$  summand of  $TR(Y<12\ldots \ell-1>)$.  We want to construct a partition of $n + k -1$ into $k$ summands that yields a surjection containing $1, 2, \ldots, \ell - 1$ as singletons.  In the original table $Y$, place a separating bar in all except the last permutation row, just after the permutation entry corresponding to the last entry of the $a'$ surjection on that row.  If $s_1$ of the desired singletons occur before the separating bar on the first row, increase the partition term $a'_1$ by $s_1$.   Then remove those singleton entries from all lower rows in the $Y$ table.  If $s_2$ of the remaining desired singletons occur before the separating bar on the second row, increase the partition term $a'_2$ by $s_2$.  Remove those singleton entries from all lower rows in the $Y$ table and continue the process in this manner.  On the last row there is no separating bar.  If $s_k$ desired singletons remain on the last row, increase $a'_k$ by $s_k$. We have now constructed a partition of $n+k-1$ into $k$ summands for which the corresponding surjection contains $1,2, \ldots, \ell -1$ as singletons.  The constructions of this paragraph and the preceding paragraph are inverses of each other.\\

The two paragraphs above describe the correspondence between surjections $x_a$ and $x_{a'}$ in terms of partitions $a$ and $a'$.  But it is easy to directly see the correspondence in terms of the Berger-Fresse row form of surjections.  Given  $x_a$ in row form in which $1,2, \ldots, \ell-1$ occur as singletons, simply remove those singletons to form $x_{a'}$.  These singletons occur before the caesura entries on the rows of $x_a$, or on the last row.  Given  $x_{a'}$ in row form,  insert the singletons $1,2, \ldots, \ell -1$ in front of the caesura entries of $x_{a'}$ in the positions they {\it first} occur in front of an $x_{a'}$ caesura in the corresponding row of the original $Y$ table, or on the last row.
\end{proof}

\begin{rem}\label{16.4}The statements of Proposition \ref{16.2} and Lemma \ref{16.3} have direct analogues for the maps $tr\colon N_*(E\Sigma_n) \to \cS_*^{aj}(n)$ and $TR \colon N_*(E\Sigma_n)  \to \cS_*^{ms}(n)$.  But now the surjection summands $x_a$, which turn out to be the same as the surjection summands of $TR(X)$, will be accompanied by  signs  $\{\pm 1\}$.  Rather than struggle with these signs at the level of tables of permutations, we will exploit the isomorphisms $\phi \colon \cS_*^{aj}(n) \simeq \cS_*^{bf}(n)$ and $\phi \colon \cS_*^{ms}(n) \simeq \cS_*^{bf}(n)$ of the previous section.  These isomorphisms are equivariant and commute with the contractions.   Thus from Proposition 6.3 and Proposition 6.5(i) of Part I, composing these  isomorphisms with $TR$  will give the equivariant chain maps constructed by the standard procedure between the MacLane model and the other surjection complexes.
\end{rem}

 \begin{prop}\label{16.5} The equivariant chain maps $tr \colon N_*(E\Sigma_n) \to \cS_*^{aj}(n)$ and $TR \colon N_*(E\Sigma_n) \to \cS_*^{ms}(n)$ commuting with contractions are given by $$tr(X) = \phi  TR(X) = \sum_{partitions \   a} p(x_a) x_a \in \cS_*^{aj}(n)$$ $$ TR(X) = \phi TR(X) =  \sum_{partitions \  a} c(x_a) x_a \in \cS_*^{ms}(n),$$ where the partitions $a$ and the surjections $x_a$ are as in Proposition \ref{16.2}, and where the signs $p(x_a)$ and $c(x_a)$ are as in Lemma \ref{15.6} and Lemma \ref{15.3}. $\qed$
 \end{prop}
 
 \begin{rem}\label{16.6} It is relatively easy to give a direct proof of the second formula here, following the steps in the proof of Proposition \ref{16.3}.  The first step is equivariance, which follows quickly from Lemma \ref{15.3}.  The second inductive step is also pretty easy, since the coefficients $c(x_a)$ defined in terms of caesura shuffles behave very simply when singletons are dropped from surjection generators.  The coefficients $p(x_a)$ in the first formula in Proposition \ref{16.5} are much harder to deal with directly. $\qed$
\end{rem}

\subsection{The Maps  $ \cS_*^{aj}(n), \cS_*^{bf}(n), \cS_*^{ms}(n) \to N_*(E\Sigma_n)$}
The contraction $h$ of $N_*(E\Sigma_n)$ defined by $h(x) = (e, x)$, along with chosen $\Sigma_n$ bases of the surjection complexes, yield by the standard procedure of Definition 6.0 of Part I equivariant chain maps $pr \colon \cS_*^{aj}(n) \to N_*(E\Sigma_n)$, $PR \colon \cS_*^{bf}(n) \to N_*(E\Sigma_n)$, and $PR \colon \cS_*^{ms}(n) \to N_*(E\Sigma_n)$. We  compute these maps in this section. The initials  $PR$ refer to `prisms'.  The  map $PR \colon \cS_*^{bf}(n) \to N_*(E\Sigma_n)$ agrees with the map called $TC$ in the Berger-Fresse paper describing a prismatic decomposition of the Barratt-Eccles operad [5].  In the previous subsection we copied the Berger-Fresse terminology $TR$ for a map they called `table reduction'.  But we don't know what their initials $TC$ were supposed to refer to, so we have chosen to more directly reference prisms.\\

We have already chosen preferred bases of the three surjection complexes  in Subsection 15.1.  In degree $k$,  bases consist of surjections $b = (b_1, b_2, \ldots, b_{n+k})$ such that the initial entries of the values $1, 2, \ldots, n$ in $b$ occur in that order.  From Proposition \ref{15.1} these basis elements are in the kernel, hence also in the image, of the contractions of the three surjection complexes.  The maps  $PR$ do not commute with contractions.  If they did, by Proposition 6.3 of Part I the compositions $PR \circ TR$  would be identities, which is impossible.\\

In degree 0 the canonical maps defined on permutations $g \in \Sigma_n$ are  the equivariant extensions of  $e \mapsto e$.\footnote{So this is $pr(g) = \tau(g) g$ in $\cS_*^{aj}(n)$ and $PR(g) = g$ in the other two complexes.}  One can begin computing recursively the canonical maps in higher degrees, but it is obscure what is happening.  So we will give a somewhat lengthy discussion  of prisms associated to generators of the surjection complexes and then show that the standard procedure chain maps $pr$ and $PR$ are defined in terms of the Eilenberg-Zilber triangulations of these prisms. The prism connection was found by Berger-Fresse  [5], although in our opinion prisms are more naturally associated to the McClure-Smith surjection complex.  Our approach provides an alternate proof that the formula $TC$ of Berger-Fresse defines a chain map, but in our view it is more interesting in general to identify chain maps and equivariant chain maps arising elsewhere with maps constructed by the standard procedure.\\

{\bf Review of Prisms Associated to Surjection Complex Generators.}  Given a surjection generator, say $x\colon  \Delta^{n-1+k} \to \Delta^{n-1}$, for $1 \leq \ell \leq  n$ let $k_\ell$ denote the number of vertices of the domain of $x$ above vertex $\ell$ of the base.  Then $n+k = \sum k_\ell.$ and $k = \sum (k_\ell -1)$. The inverse images of vertices of the base are simplex faces of the domain, $x^{-1}(\ell) \simeq \Delta^{k_\ell -1}$.  There is a canonical isomorphism because the vertices of any face of the domain inherit an order from the order of the domain vertices.  The full domain simplex is the join of these faces above vertices of the base.\\

The inverse image under $x$ of the barycenter of the base is  isomorphic to a prism $\prod_\ell x^{-1}(\ell) \simeq \prod_\ell \Delta^{k_\ell -1}$ of dimension $k$, which we call $Prism(x)$, discussed in detail in Section 14.   Geometrically, the vertices of $Prism(x)$ are barycenters of $(n-1)$-dimensional faces of the domain that map isomorphically to the base.  Thus they are named by vertex sequences $ v = (v_1, v_2, \ldots, v_n)$ of the domain, with each $v_\ell$ a vertex of the face $x^{-1}(\ell) \simeq \Delta^{k_\ell -1}$.  Vertices of the prism are also named by  sequences $i =  (i_1, i_2, \ldots, i_n)$, where $1 \leq i_\ell \leq k_\ell $.  The corresponding vertices $v_\ell$ of the domain are the $i_\ell$-th vertices of the faces $x^{-1}(\ell)$ in the inherited order.\\

The vertices of the product prism form a poset, with $(v_1, v_2, \ldots, v_n) \leq (w_1, w_2, \ldots, w_n)$ if $v_\ell \leq w_\ell$ for all $\ell$.  Equivalently, in the other vertex notation, $(i_1, i_2, \ldots, i_n) \leq  (j_1, j_2, \ldots, j_n)$ if $i_\ell \leq j_\ell$ for all $\ell$.  The prism is then triangulated  as the poset triangulation of $\prod_\ell \Delta^{k_\ell -1}$, which we will also refer to as the Eilenberg-Zilber triangulation.\\

There are a few ways to name the maximal  $k$-dimensional simplices of the triangulated prism. They can be  named by sequences of vertices using the second vertex notation $$(1, 1, \ldots, 1) < \ldots (i_1, i_2, \ldots, i_n)  \ldots < (k_1, k_2, \ldots, k_n).$$ There are  $k$ steps, where at each step one of the coordinates increases by one and the others remain unchanged.  Thus  $k$-simplices correspond to integral edge paths of length $k$ in a box $\prod_\ell [1,2, \ldots, k_\ell]$, as in the discussion of the functorial Eilenberg-Zilber map  in Section 8 of Part I.\\

These edge paths can also be named by a single sequence $ j = (j_1, j_2, \ldots, j_k)$ consisting of $(k_\ell - 1)$ $\ell$'s, $1 \leq \ell \leq n$. The index $j_m$ names the coordinate direction in the box that is increased by one at the $m^{th}$ step.  Geometrically the $k$-simplices  of $Prism(x)$  are embedded in the domain simplex $\Delta^{n-1+k}$ of $x$ as the join, or hull, of their corresponding prism vertices, which are barycenters of faces of the domain. \\

Note $Prism(x) = \prod \Delta^{k_\ell -1}$ is canonically oriented as a product of oriented manifolds.  Each $k$-simplex of the triangulated prism is thus oriented in two ways, first as a codimension 0 submanifold of the oriented prism and secondly by the ordering of the $k$-simplex vertices.  The Eilenberg-Zilber triangulation map includes the signs given by comparing the two simplex orientations.  As discussed in Section 8 of Part I, the comparison sign can be viewed as the sign of a shuffle permutation or as the parity sign of the area of a collection of unit rectangles spanning the union of the path $(j_1, j_2, \ldots, j_k)$ and the path $(1,  . .  ,1, 2, . .  ,2, . . ,n, . . , n)$ with $(k_\ell - 1)$  $\ell$'s, $1 \leq \ell \leq n$.  $\qed$\\

{\bf The Permutation Associated to a Prism Vertex.} Since the vertices of the domain of $x$ are ordered, each  prism vertex yields a permutation.  To the vertex $v = (v_1, v_2, \ldots, v_n)$ of $Prism(x)$, with $x(v_\ell) = \ell$,  we associate a permutation $\gamma = \gamma_v  \in \Sigma_n$ by rewriting the $v_\ell$ in the order they occur in the domain of $x$.  Thus $v_{\gamma(1)} < v_{\gamma(2)} < \ldots < v_{\gamma(n)}$. Another way to view the permutation $\gamma_v$ is to circle the entries $x(v_\ell)$ in $x = (x(1), x(2), \ldots, x(n+k)$).  These entries will be the numbers $1,2, \ldots , n$ in the order given by the permutation $\gamma_v$.\\

For example, if $x = (2,1,2,3,4,2,3,1,5,4,1,2)$ and we use for $v$ the single occurrence of value  5 and the second occurrence of each value 1,2,3,4, then $v = (8, 3, 7, 10, 9)$ and the permutation $\gamma_v = (23154)$.  $\qed$\\

{\bf Definition of the Chain Maps $\bf {PR}$ and $\bf {pr}$.} We regard the permutation $\gamma_v$ as a vertex of the  simplical set $E\Sigma_n$.  Note the simplicial set $E\Sigma_n$ is just a kind of singular simplicial set canonically associated to a simplex with vertices the elements of $\Sigma_n$.\footnote{These vertices are not ordered, but there is a basepoint $e$.  A singular simplex here means any {\it simplicial} map from the standard ordered simplex to the simplex with vertices $\Sigma_n$.}  By `convexity', the assignment $v \mapsto \gamma_v$ extends to a  map of simplicial sets $ \gamma_x \colon Prism(x) \to E\Sigma_n$. Below we will also  use the name $\gamma_x$ to denote the induced maps on normalized chains and on subcomplexes of the normalized chain complexes.  There is some risk to this because $\gamma_x$ of a non-degenerate prism simplex can be degenerate and these become 0 in the normalized chain complex.  But context makes this pretty harmless.\\ 

We use the notation $$[Prism(x)] = EZ\big(\bigotimes_\ell\  [\Delta^{k_\ell - 1}] \big) \in N_*\big(\prod_\ell \Delta^{k_\ell -1} \big) = N_*(Prism(x)),$$ in order to write the oriented manifold $Prism(x)$ as a sum of signed ordered simplices of dimension $k$.  We now have linear  maps  from each of the three surjection complexes to $N_*(E\Sigma_n)$, defined on generators by $x \mapsto \gamma_x[Prism(x)]$.  

\begin{prop}\label{16.7}The maps $$PR\colon \cS_*^{ms}(n) \to N_*(E\Sigma_n)\ \ given\ by\ \ PR(x) =  \gamma_x[Prism(x)]$$  $$PR \colon \cS^{bf}_*(n) \to N_*(E \Sigma_n)\ \ given\ by\ \ PR(x) = c(x) \gamma_x [Prism(x)]$$ $$pr \colon \cS_*^{aj}(n) \to N_*(E \Sigma_n)\ \ given\ by\ \ pr(x) = p(x) \gamma_x [Prism(x)]$$ are  the equivariant chain maps arising from the standard procedures.  Here,  $c(x), p(x) \in \{\pm1\}$ are the signs studied in Section 15.
\end{prop}
\begin{proof}We only need to prove the first statement.  The other two will follow quickly by pre-composing with isomorphisms between surjection complexes established in Section 15.\\

The first statement is rather easy.  The Eilenberg-Zilber triangulation is a chain map.  The boundary operator in $\cS_*^{ms}(n)$ is the prism boundary formula. The equivariance also follows easily from the $\Sigma_n$ action in $\cS_*^{ms}$, which is given in terms of degree signs of maps that permute factors of a prism.  The Eilenberg-Zilber formula also commutes with these signs.\\

One might think it could still be non-trivial to argue that this equivariant chain map is indeed  the map given by the standard procedure.  It certainly does not commute with contractions.  But, for a basis generator $b \in \cS_*^{ms}(n)$, the first entries of $1,2, \ldots, n$ in $b$ occur in that order.  Therefore the initial vertex of every maximal dimension simplex of $\gamma_b([Prism(b)])$ is $(1,2, \ldots, n) = e \in \Sigma_n$, the identity permutation.  By the definition of $PR$ in terms of the Eilenberg-Zilber triangulation, $PR(b)$ is a signed sum of such simplices, which are in the image of the contraction of $N_*(E\Sigma_n)$.  Therefore the very easy uniqueness result Proposition 6.2 of Part I implies that $PR$ is necessarily the standard procedure equivariant chain map.\\

The argument in the paragraph above is also one way to see why the pre-composition of $PR$ with an isomorphism between surjection complexes also yields a standard procedure map.  This is not automatic because $PR$ does not commute with contractions, and in general compositions of standard procedure maps need not be standard procedure maps.  This was discussed  following Proposition 6.5 of Part I.  But from Proposition 6.5(iii), the fact that the maps between surjection complexes take basis generators to basis generators (up to sign) also implies the compositions with $PR$ are standard procedure maps.
\end{proof}

\begin{rem}\label{16.8}From Section 15, the sign $c(x) \in \{\pm 1\}$ in the formula for $PR \colon \cS_*^{bf}(n) \to N_*(E \Sigma_n)$ is determined by the caesuras of $x$. Suppose the caesuras of $x$ regarded as a generator of $\cS_*^{bf}(n)$ are in order $(c_1, c_2, \ldots, c_k)$.   Then $c(x) \in \{\pm 1\}$ is the parity sign of  the shuffle permutation in $\Sigma_k$ that rewrites the $c_j$  in the order $(1,  . .  ,1, 2, . . , 2, . . , n, . .  ,n)$, preserving the order in $x$ of  $c_j$'s of the same value.\footnote{If surjection $x$ contains singletons, not all integers between 1 and $n$ occur as caesuras.}  This sign $c(x)$ is also  the Eilenberg-Zilber orientation sign of the maximal dimension  $k$-simplex of $Prism(x)$ corresponding to the edge path with the name $(c_1, c_2, \ldots, c_k)$. Its vertices are found by beginning with the permutation given by the first occurrences of $1,2, \ldots, n$ in $x$, then at the $j^{th}$ step, $1 \leq j \leq k$,  take the next highest occurrence of the $c_j$ value. Berger-Fresse call this simplex the {\it fundamental simplex} of $Prism(x)$.\\

We call the simplex corresponding to $(1,..  ,1, 2,.. , 2, .. , n,.. , n)$ the {\it base simplex} of the prism. A base simplex always has orientation sign $+1$ because an orientation sequence of its tangent vectors is seen to coincide with an orientation sequence of tangent vectors of $Prism(x)$.  The relation between vertex sequences and orientation signs in the Eilenberg-Zilber triangulation is part of the discussion of $EZ$ map in Example 8.8 of Part I. $\qed$
\end{rem}
\begin{exam}\label{16.9} If $x = (2,1,2,3,4,2,3,1,5,4,1,2)$ then the caesura sequence is $(2,1,2,3,4,2,1)$.  We calculate $\gamma_x$ of the fundamental simplex. The initial vertex  is $(2,1,3, 4, 5)$. Then take the second occurrence of 2, giving $(1,2,3, 4, 5)$.  Then move the 1 coordinate, then the 2 coordinate again, then the 3 coordinate, etc. The full fundamental 7-simplex in $N_*(E\Sigma_n)$  is $$\big[ (2,1,3, 4,5), (1,2,3,4,5), (2,3, 4, 1, 5), (3, 4,2,1, 5), $$  $$(4,2,3,1,5), (2,3,1,5,4), ( 3,1,5,4,2), (3,5,4,1,2)\big].$$The shuffle sign moving the caesura sequence to $(1,1,2,2,2,3,4)$,  is $c(x) = (-1)^8 = +1.$   Fundamental simplices are always non-degenerate in $N_*(E\Sigma_n)$. The base simplex vertices, corresponding to the sequence $(1,1,2,2,2,3,4)$, are found by the same procedure.  In this example $\gamma_x$ maps the base simplex to a degenerate simplex.  The fourth and fifth permutations coincide.$\qed$
\end{exam}
{\bf Two Additional Properties of $PR$.} We will have occasion to use the following result of Berger-Fresse in Part III when we compare the Barratt-Eccles and surjection operads.
 \begin{prop}\label{16.10}  For a generator $x \in \cS_k^{bf}(n)$, if $X = (g_0, \ldots, g_k) $ is the fundamental  simplex summand of $PR(x)$ then $TR(X) = x$.  If $Z$ is any other maximal dimension simplex summand of $PR(x)$ then $TR(Z) = 0.$
\end{prop}
\begin{proof} This can certainly be proved by a lot of direct computation, thinking about $TR$ as a sum over  partitions.  However, we will use more finesse.  There is nothing to prove if $k = 0$.    By equivariance, we can assume $x$ is a basis generator of the surjection  complex.  Then $g_0 = e$ and each maximal dimension simplex has form $X = (e, Y) = h_{\Sigma_n}(Y)$, where $Y = (g_1, \ldots, g_k)$. Then $TR(X) = h(TR(Y))$.  If $k = 1$ then $x$ is given by inserting one additional later entry $j$ in the identity permutation $e$.  Then $X = (e, g_1)$ where $g_1$ is given by shifting the entry $j$ in $e$ to the right.  Of course a $\Delta^1$ prism is its own fundamental simplex.  By an easy  computation that is essentially a part of Proposition \ref{15.1}, $TR(X) = h(g_1) = \sum_{\ell = 0}^{n-2} i^\ell s r^\ell(g_1)= x$.\\

Assume the proposition in degrees less than $k$.  Any maximal dimension simplex $X$ is named by a sequence $(i_1,  i_2, \ldots, i_k)$, where the $1 \leq i_\ell \leq n$ are the caesura values of $x$ in some order.  The fundamental simplex corresponds to the ordered sequence $(c_1, c_2, \ldots, c_k)$ of caesuras of $x$. Since $x$ is a basis generator, $i_1 \geq c_1$.  Let $x_{j_1} $ denote the first caesura of $x$ with value $i_1$.  Then $Y$ is a maximal dimension  simplex  of $PR(d_{j_1}x)$, corresponding to the sequence $(i_2, \ldots, i_k)$.  Hence by induction   $TR(Y) = d_{j_1} x\ or\ 0$.  If $i_1 > c_1$ then  $hTR(Y) = 0$ in either case,  by Proposition \ref{15.1}. Hence $TR(X) = hTR(Y) = 0.$  If $i_1 = c_1$, then again by 15.1, $hd_{j_1}x = x$.  But now $X$ is  fundamental if and only if $Y$ is fundamental.   If fundamental, then $TR(Y) = d_{j_1}x$ and $TR(X) = h TR(Y) = x$.  If not fundamental, then $TR(Y) = 0$ and $TR(X) = 0$.
\end{proof}

\begin{prop} \label{16.11} The map $PR \colon \cS_*^{ms}(n) \to N_*(E\Sigma_n)$ is an associative coalgebra map, that is, it commutes with the standard procedure diagonals. 
\end{prop}
\begin{proof}The easiest proof is to look at the appropriate diagram of equivariant chain maps and show both directions around the diagram take basis generators of $\cS_*^{ms}(n)$ to elements in the image of the contraction of $N_*(E\Sigma_n) \otimes N_*(E\Sigma_n)$.  Alternatively, one can use the explicit formulas of Subsection 7.1 and Proposition \ref{8.11} of Part I for the diagonals, and Proposition \ref{16.7} above for $PR$.  
\end{proof}

Another proof that $PR$ is a coalgebra map is given in [23]. The map $TR \colon N_*(E\Sigma_n) \to \cS_*^{bf}(n)$ is not a coalgebra map.  For example, consider  $Z = ((123), (132), (312))$, the non-fundamental 2-dimensional simplex of $Prism(12312)$. Then $TR(Z) = 0$, so $\delta TR(Z) = 0$.   But $\delta(Z)$ contains the summand  $((123), (132)) \otimes ((132), (312)) $.  Applying  $TR \otimes TR$ yields  $((1232) \otimes (1312))   \not= 0$.

\subsection{*Further Results on the Maps $ TR$ and $ PR$*}
In this final subsection of Section 16 we  want to record some other interesting observations concerning the maps $TR, PR$ made by Berger-Fresse in [5].   The discussions and examples in this subsection are not central to our overall goals, but we found the issues involved to be quite interesting.  The details are rather intricate and lengthy. There would be no loss of continuity skipping to Section 17 that begins Part III.\\

We look at three issues.  We investigate some geometric cell complexes underlying the algebraic chain maps $\cS_*^{ms}(n) \xrightarrow{PR} N_*(E\Sigma_n) \xrightarrow{TR} \cS_*^{ms}(n)$.  Of course Berger-Fresse used their complex $\cS_*^{bf}(n)$, but the McClure-Smith complex seems more natural.  We find  interesting comparisons between the cell structures underlying $PR$ and $TR$ and cell structures underlying the maps $M_* \xrightarrow{\phi} N_*(EC) \xrightarrow{\pi} M_*$ for the minimal model and MacLane model for the cyclic group from Section 6 of Part I. We also relate the maps $PR$ and $TR$ to the Eilenberg-Zilber and Alexander-Whitney maps for prisms, that is, products of simplices, $\otimes N_*(\Delta^{m_\ell-1}) \xrightarrow{EZ} N_*(\prod \Delta^{m_\ell-1}) \xrightarrow{AW} \otimes N_*(\Delta^{m_\ell-1}).$ \\

The next several remarks, propositions,  and examples,  are related to the very final statements of [5].      However, we find their discussion rather brief.

\begin{rem}\label{16.12} Every  $k$-simplex  $Z \in E\Sigma_n$ belongs to the image of some prism, in fact, many prisms.  For example, say $Z = (g_0, g_1, \ldots, g_k)$, a sequence of permutations.  Regard the concatenation $z = g_0g_1 \ldots g_k$ as a surjection in $\cS_*^{bf}(n)$, of degree  $n(k+1)$. If the first entry of $g_{i+1}$ equals the last entry of $g_i$, remove one of those entries so that the resulting surjection is non-degenerate.  Then one sees $Z$ as the simplex summand of $PR(z)$ corresponding roughly to taking the first entries of all $j$-values of $z$, then the second entries of all $j$-values, and so on.  A simple modification is needed in cases where the initial concatenation has duplicated entries. $\qed$
\end{rem}
We want to be more precise about a certain issue.  First,  the chain complex $\cS_*^{ms}(n)$ is the chain complex associated to an explicit geometric cell complex with open cells the interiors of prisms,  as explained in Remark \ref{14.4}.  How are the cells of that complex related to images of prisms that cover the geometric realization $|E\Sigma_n|$ by the maps $\gamma_x \colon Prism(x) \to E\Sigma_n$ of Subsection 16.2?\\

 {\bf The Images of Prisms of $\bf \cS_*^{ms}(n)$ in $\bf E\Sigma_n$.} The cells of the prism cell structure are named by the generators $x$ of $\cS_*^{ms}(n)$, and we think of the open cells as interiors of  the $Prism(x)$'s.   These cells are canonically oriented.  But some of the codimension one prism boundary faces become degenerate algebraically  in the surjection complex.  So the topology of the resulting cell complex  is defined by identifying common boundary faces of prisms, and also by collapsing some codimension one boundary faces of $Prism(x)$ to codimension two boundary faces by deleting one of a pair of equal adjacent entries in degenerate boundary prism faces.  The interiors of cells are interiors of prisms, and the closures are obtained as just described.  In Remark \ref{14.4} we gave a global description of this cell complex as the inverse image of the barycenter under the tautological map $\chi \colon |SS(\Delta^{n-1})| \to \Delta^{n-1}$, where $|SS(\Delta^{n-1})|$ is the geometric realization of the singular complex consisting of simplicial maps from simplices to $\Delta^{n-1}$.  The boundary operator in the associated prism chain complex is the McClure-Smith boundary.  The Eilenberg-Zilber triangulation of prisms induces the structure of an $s\Delta$-complex\footnote{$s\Delta$-complexes are defined in an appendix in Hatcher's topology text [12].  $\Delta$-complexes are $CW$ complexes whose characteristic cells are standard simplices and whose attaching maps are strict order preserving simplicial maps on boundary faces.   $s\Delta$-complexes allow weakly order preserving attaching maps on boundary faces, hence allow collapsing simplices to lower dimension simplices.}  on this cell complex for $\cS_*^{ms}(n)$.\\

It is much more unclear if the collection of images  in $|E\Sigma_n|$ of the closed cells $Prism(x)$ has a clean global geometric interpretation.   Examples show that an intersection can contain simplices interior to both prisms.  It is also possible for boundary intersections to consist of less than full boundary prism faces, and for general intersections to contain some interior simplices of one of the prisms and some partial boundary faces of the other. \\

Here is our precise statement about how prisms intersect.  The picture is geometrically pretty clear in the case that both maps $\gamma_x, \gamma_y$ are injective on the respective prisms $Prism(x),\ Prism(y)$.  In general these maps can be non-injective on vertices and many simplices in the prisms can map to degenerate simplices.  The shape of the images of prism cells and the intersection picture in $|E\Sigma_n|$ then becomes more obscure.

\begin{prop}\label{16.13}   Let $V(x,y)$ denote the collections of vertices in $E\Sigma_n$ in the images of both prisms.  Then let $Hull_x(V(x,y))$ and $Hull_y(V(x,y))$ denote all the simplices in the images of the Eilenberg-Zilber triangulation of the respective prisms whose vertices belong to $V(x,y)$.  Then $$\gamma_x(Prism(x)) \cap \gamma_y(Prism(y)) = Hull_x(V(x,y)) \cap Hull_y(V(x,y)).$$
\end{prop}

\begin{exam}\label{16.14} Let $x = (123143)$ and $y = (121343)$.  Both prisms are isomorphic to a rectangle $\Delta^1 \times \Delta^0 \times \Delta^1 \times \Delta^0$.  There are three vertices in common in the two images, $(1234), (1243), (2143)$.  The fourth vertices are distinct, $(2314)$ and $(2134)$.  The boundary intersection consists of two intervals from each rectangular prism, but both hulls contain the interior triangle $((1234), (1243), (2143))$. \\

If $x' = (1231431)$ and $y' = (1213431)$ then both prisms are isomorphic to $\Delta^2 \times \Delta^0 \times \Delta^1 \times \Delta^0$.  Five of the six vertices of the image prisms are shared.  The boundary intersection consists of one triangular face, one rectangular face, and one triangle in the rectangular faces $x \subset x'$ and $y \subset y'$.  There is also an interior tetrahedron in common. \\

The intersection of the prisms of $x$ and $y'$ consists of an interior triangle in the rectangle $Prism(x)$, which is part of a boundary rectangle of $Prism(y')$.  $\qed$
\end{exam} 

We generally think of $E\Sigma_n$ as a simplicial set.  But it has a geometric realization, which is an $s\Delta$-complex.  The prisms of the cell structure underlying $\cS_*^{ms}(n)$, which collapse somewhat on boundaries, map to unions of simplices in the geometric realization $|E\Sigma_n|$. These images, which exhibit much additional collapsing, are contractible cells.  Each such contains a unique fundamental simplex of maximal dimension.  These cells can  intersect in interior simplices, but the intersections are contractible subcomplexes.  So we have some kind of variant of a `cell complex', structure on $|E\Sigma_n|$, with cells parametrized by fundamental simplices of prisms.  The symmetric group acts freely on the  cell structures of both $\cS_*^{ms}(n)$ and $|E\Sigma_n|$,   and the equivariant boundary operators agree on the chain complexes underlying the two collapsed prism cell structures.  Thus these algebraic chain complexes are isomorphic, although the geometry of the cells and their intersections in $|E\Sigma_n|$ is complicated. $\qed$\\

{\bf A Surprising Connection Between $\bf {TR}$ and $\bf {AW}$.} We now take up another topic.  The Berger-Fresse remarks at the end of [5] hint at a relation between $TR$ and the Alexander-Whitney map $AW$ that we find quite surprising and  interesting.  We find it slightly more convenient to use the McClure-Smith version of the surjection complex in our next proposition.
\begin{prop}\label{16.15}  Let $x = ( x_1, \ldots, x_{n+m}) \in \cS_*^{ms}(n)$ be a generator.  Say $Prism(x) = \prod_{\ell=1}^n \Delta^{m_\ell- 1}$, $\sum m_\ell = m$.  Then there is a commutative diagram of chain maps
$$\begin{CD} \bigotimes N_*(\Delta^{m_\ell-1}) & \xrightarrow{EZ} &N_*(\prod \Delta^{m_\ell-1}) & \xrightarrow{AW_x}& \bigotimes N_*(\Delta^{m_\ell-1})
 \\ \downarrow \gamma_x & & \downarrow \gamma_x & &\downarrow \gamma_x\\
\cS_*^{ms}(n) & \xrightarrow{PR} & N_*(E\Sigma_n) & \xrightarrow{TR}& \cS_*^{ms}(n)
\end{CD}$$
The map $AW_x$ is canonically chain homotopic to $AW$.  On $Image(EZ)$, $AW = AW_x$ and the canonical chain homotopy is 0.\\

If the fundamental simplex of $Prism(x)$ coincides with the base simplex, then $AW_x = AW$.
\end{prop}

\begin{proof} We will clarify the vertical maps $\gamma_x$ below. Before we begin the proof we point out that if $x'$ is a non-degenerate boundary face of $x$ of any codimension then there is an obvious  commutative  diagram that maps the diagram for $x'$ to the diagram for $x$.  Also, in the diagram for $x$ there is no $\Sigma_n$ action on the top row.  But there is a commutative diagram mapping the diagram for $x$ to the diagram for $gx$.  At each node the connecting map is an action of $g$, hence an isomorphism.  For the tensor complexes in the top rows, there are Koszul signs in the $g$ maps that permute tensor factors.  There are no signs in the $g$ maps between normalized chain complexes of products of simplices that permute simplex factors.\\

Suppose $\tilde{Y}$ is a $k$-simplex of $Prism(x)$, with the $EZ$ triangulation. 
Then $\tilde{Y}$ corresponds to a $(k+1) \times n$ matrix.  The rows are vertices of $Prism(x)$ whose entries are named by tuples consisting of one vertex from each prism factor simplex.  We record the entry in the $i^{th}$ row and $\ell^{th}$ column as an integer between 1 and $m_\ell$. That is, we name the vertices of the prism factors $\Delta^{m_\ell - 1}$ beginning with 1 rather than 0. The entries reading down a column are non-decreasing.  The columns are $k$-simplices of the factors $\Delta^{m_\ell - 1}$.  These individual  column simplices are typically degenerate.\\

The vertical arrows $\gamma_x$ exploit the total order on the set of factor vertices of $Prism(x)$.  The center arrow assigns to each row of $\tilde{Y}$ a permutation  in $\Sigma_n$, as part of the $PR$ map. The left and right arrows take a tensor product of faces of simplices to the element of $\cS_*^{ms}(n)$ obtained as the subsequence of $x$ given by the vertices of those faces.  The maps $\gamma_x$ are chain maps and the  left square commutes because we are using the McClure-Smith surjection complex, which is especially tuned to prisms.   There are hidden caesura shuffle signs in the $TR$ map for $\cS_*^{ms}$, which would disappear if we were to use $\cS_*^{bf}$, but then different signs would  appear in the $PR$ map and one would need to put signs in front of  the left and right $\gamma_x$ map as well.\\

We will now define $AW_x$.  We view $Y = \gamma_x(\tilde{Y})$ as a table of permutations.  In the McClure-Smith complex $TR(Y) = \sum c(y_a)y_a$ where the $y_a$ are surjections parametrized by partitions $a _0 + \ldots + a_k = n+k$, and the signs $c(y_a)$ are as in Proposition \ref{16.5}.  But $\tilde{Y}$ contains more information than the table of permutations $Y$.  Specifically, an entry in  the $i^{th}$ row of $\tilde{Y}$ is in a  column $\ell$, which is a value   of $x$, and the entry itself records the position of that value in $x$.  That is, the first value of $\ell$, or the second value of $\ell$, etc,   Thus, we can enhance the table of permutations $Y = \gamma_x(\tilde{Y})$ by putting a subscript on each permutation entry value that records the relative position in $x$ of that entry value.  We call the enhanced table of permutations $\widehat{Y}$.\\

Then to each partition $a$ associated to the table reduction of $Y = \gamma_x(\tilde{Y})$ we calculate the surjection summand $c(y_a)y_a$ of $TR(Y)$, following the method of Proposition \ref{16.2}.  The entries of each $y_a$ inherit subscripts from $\widehat{Y}$.  We denote by $\hat{y}_a$   these enhanced surjection terms and set $TR(\widehat{Y}) = \sum c(y_a)\hat{y}_a.$ Using the subscripts of repeated values of the entries of $\hat{y}_a$, we associate a tensor of faces of simplex factors, $\tilde{y}_a  \in \bigotimes N_*(\Delta^{m_\ell-1})$, and then observe $ y_a = \gamma_x(\tilde{y}_a)$.  (See Example \ref{16.17} below for specific illustrations of this procedure.)  We then set $$AW_x( \tilde{Y} ) = \sum_a c(y_a)\tilde{y}_a .$$
Clearly the right square in the diagram in Proposition \ref{16.15} commutes,  $\gamma_x \circ AW_x( \tilde{Y}) = \sum c(y_a)y_a = TR \circ \gamma_x (\tilde{Y})$.  The tensor factors of $\tilde{y}_a$ might or might not be degenerate, and then after that the image of the  full tensor $\tilde{y}_a$ under $\gamma_x$ might or might not be degenerate.

\begin{rem}\label{16.16} Early on in [5] Berger and Fresse remark without much discussion that given  surjection generators $x, y$, if $\gamma_yPrism(y)  \subset \gamma_xPrism(x)$ then $y$ is a subsequence of $x$, that is, a face of $x$.  The converse is obvious.  But a stronger statement than theirs follows easily from Proposition \ref{16.10} and  the construction of $AW_x$.  If the fundamental simplex of $\gamma_yPrism(y)$ belongs to $\gamma_xPrism(x)$ then $y$ is a face of $x$.  Namely, if the fundamental simplex is $Y = \gamma_x(\tilde{Y})$, then $TR(Y) = \pm y$ and by construction $AW_x(\tilde{Y})$ identifies a tensor that names a face of $x$ coinciding with $y$. $\qed$
\end{rem}

{\bf Proof that $\bf AW_x$ is a Chain Map.}  Both $EZ$ and $PR$ are injective chain maps, hence if the center $\gamma_x$ map is injective then so is the left and right $\gamma_x$ map.  Commutativity of the diagram then implies $AW_x$ is a chain map.  One can insert new singleton entries in any $x$, forming some $x'$, and arrange that the center map is injective on vertices, which implies it is injective on all simplices.  Thus $AW_{x'}$ is a chain map.\\

Adding singletons only changes the prism by adding $\Delta^0$ factors. The image of the top row of the diagram in the bottom row only sees subcomplexes generated by $x$ and its faces.  The top row of the diagram for the new enlarged $x'$ is isomorphic to the original top row for $x$. Looking carefully at the relevant table reductions, one sees that $AW_x$ and $AW_{x'}$ are also  identified.  Thus $AW_x$ is a chain map for any $x$. $\qed$ \\ 

{\bf The Chain Homotopy Between $\bf AW$ and $\bf AW_x$.}  The Alexander-Whitney map $AW$ is another map in the top row of the diagram with the same domain and range.  Both maps are essentially identity maps in degree 0.  Following a remark at the beginning of Section 9 of Part I, there is a canonical chain homotopy $H_x$ from $AW$ to $AW_x$ given in terms of the contraction $h_\otimes$ of the range. On generators $\tilde{Y}$ of the domain of degree 0, $H_x = 0$.  In higher degrees $$H_x(\tilde{Y}) = h_{\otimes} (AW_x(\tilde{Y}) - AW(\tilde{Y}) - H_x(d\tilde{Y})).$$ Then by induction $dH_x(\tilde{Y})+ H_xd(\tilde{Y}) = AW_x(\tilde{Y}) - AW(\tilde{Y})$.  We do not know a closed formula for $H_x$, but the map $AW_x$ that lifts $TR$ can be viewed as some kind of  twisting of the Alexander-Whitney map $AW$, related to the order of the caesuras of $x$.  One has of course $AW \circ $EZ$ = Id$. $\qed$\\

{\bf Comparison of $\bf AW$ and $\bf AW_x$ on $\bf Image(EZ)$.} We will now explain why  $AW_x = AW$ on $Image(EZ).$ It suffices to show this for the single top dimensional tensor generator of $\bigotimes N_*(\Delta^{m_\ell-1}) $, since other tensor generators correspond to top dimensional generators of prisms associated to faces of $x$. The top dimensional tensor maps by $\gamma_x$ to $x \in \cS_*^{ms}(n)$.\\

For maximal dimension simplex generators $\tilde{X}$ of the triangulated domain $Prism(x)$, all summands of the multidiagonal formula $AW(\tilde{X})$ are zero unless $\tilde{X}$ is the base simplex.  In that case, exactly one multidiagonal $AW$ summand is non-zero, and returns as value  the top dimensional tensor.   This is just one way of showing that $AW \circ EZ = Id$.\\

On the other hand, $AW_x$ is defined so that the diagram in Proposition \ref{16.15} commutes.   From Proposition \ref{16.10}  all partition summands of $TR \circ  \gamma_x(\tilde{X})$ are zero unless $\tilde{X}$ is the fundamental simplex of $Prism(x)$.  In that case,   exactly one $TR$ summand is non-zero.  The shuffle sign $c(x)$ occurs twice.  First, it is the $EZ$ triangulation sign associated to the fundamental simplex, hence  $ c(x) \gamma_x(\tilde{X}) = c(x)X$ occurs as a summand of  $PR(x)$.  Secondly, the shuffle sign occurs in the formula $TR(X) = c(x)x$.  The two signs cancel,  which proves $\gamma_x AW_x(\tilde{X}) = x$.  This forces $AW_x(\tilde{X})$ to also  be the top dimensional tensor. $\qed$\\

For any $x$ and for $\tilde{Y}$ of any degree there is a fairly natural bijection between the multidiagonal partitions that determine summands of $AW(\tilde{Y})$ and the partitions that determine summands of $TR(Y)$, or equivalently $AW_x(\tilde{Y})$.  But the relation between this bijection and the two formulas seems obscure.   In general not even the number of non-zero summands in the two cases, which typically are very few,  need agree.  Even when the number of non-zero summands does coincide, one does not expect $AW(\tilde{Y})$ to equal $AW_x(\tilde{Y})$, unless the base simplex and fundamental simplex of $Prism(x)$ coincide. 

\begin{exam}\label{16.17}{\bf Comparing $\bf AW$ and $\bf AW_x$ on Specific Simplices.}  Let $x = (3123413)$, so as an abstract prism  $$Prism(x) \simeq \Delta^1 \times \Delta^0 \times \Delta^2 \times \Delta^0 \simeq \{1,2\} \times \{1\} \times \{1,2,3\} \times \{1\}.$$  There are two triangular boundary prism faces corresponding to the prisms of $d_2x = (323413)$ and $d_6x = (312343)$.  There are three rectangular faces on the boundary corresponding to the prisms of $d_1x = (123413), d_4x = (312413), d_7x = (312341).$ Base simplices equal fundamental simplices except for $d_4x$ and $d_7x$, and of course for $x$ itself.\\

In degree 0 both maps in the top row of the diagram in Proposition \ref{16.15} are essentially identities, as are the two maps in the bottom row. There are twelve 1-simplices in $Prism(x).$ Nine of them are edges of the boundary prisms of $x$, which are themselves simple prisms. Calculation of all the maps in the diagram is rather simple for these.  The compositions in the two rows of the diagram are identities and $AW = AW_x$ for $\Delta^1$ prisms.\\

The other three 1-simplices are diagonals of the rectangular boundary face prisms.  We make some calculations for the diagonal of $Prism(d_1x) = Prism(123413)$, which in $Prism(x) = Prism(3123413)$ is named by  \\

$
\tilde{Y} =
\begin{bmatrix}
1 & 1 & 2 & 1 \\
2 & 1 & 3& 1 
 \end{bmatrix}
\quad
\gamma_x(\tilde{Y}) = Y =
\begin{bmatrix}
1 & 2 & 3 & 4 \\
2 & 4 & 1 & 3 
 \end{bmatrix}
\quad
\widehat{Y} = 
\begin{bmatrix}
1_1&2_1&3_2&4_1\\
2_1&4_1&1_2&3_3
\end{bmatrix}
$\\

We see $TR(\widehat{Y}) = 1_12_13_24_13_3 + 1_12_14_11_23_3$ and then
$$AW(\tilde{Y}) = 1 \otimes 1 \otimes 23 \otimes 1 + 12 \otimes 1 \otimes 3 \otimes 1 = AW_x(\tilde{Y}), \  H_x(\tilde{Y}) = 0.$$
This example is simple because the base simplex coincides with the fundamental simplex for $d_1x = 123413$.\\

More interesting is the diagonal of $Prism(d_4x) = Prism(312413)$.  In $Prism(x) = Prism(3123413)$ this diagonal is named by\\

$
\tilde{Y'} =
\begin{bmatrix}
1 & 1 & 1 & 1 \\
2 & 1 & 3& 1 
 \end{bmatrix}
\quad
\gamma_x(\tilde{Y'}) = Y' =
\begin{bmatrix}
3 & 1 & 2 & 4 \\
2 & 4 & 1 & 3 
 \end{bmatrix}
\quad
\widehat{Y'} = 
\begin{bmatrix}
3_1&1_1&2_1&4_1\\
 2_1&4_1&1_2&3_3
\end{bmatrix}
$\\

We see $TR(\widehat{Y'}) = 3_12_14_11_23_3  +3_11_12_14_11_2$ and $$AW_x(\tilde{Y'}) = 2 \otimes 1 \otimes 13 \otimes 1 + 12 \otimes 1 \otimes 1 \otimes 1.$$  We also see $$AW(\tilde{Y'}) = 1 \otimes 1 \otimes 13 \otimes 1 + 12 \otimes 1 \otimes 3 \otimes 1$$   $$H_x (\tilde{Y'}) = h_\otimes (AW_x(\tilde{Y'}) - AW(\tilde{Y'})) = 12 \otimes 1 \otimes 13 \otimes 1.  $$

So $H_x(\tilde{Y'})$ is the tensor form of the rectangular prism face of $Prism(x)$ corresponding to $d_4x$ and $AW_x(\tilde{Y'}) - AW(\tilde{Y'})$ is its boundary.\\

The fundamental simplex of $d_4x$ is named in $Prism(x) = Prism(3123413)$  by\\

$
\tilde{Z} =
\begin{bmatrix}
1 & 1 & 1 & 1 \\
1 & 1 & 3 & 1 \\
2 & 1 & 3 & 1 
 \end{bmatrix}
\quad
\gamma_x(\tilde{Z}) = Z = 
\begin{bmatrix}
3 & 1 &2 &4 \\
1&2&4&3\\
2&4&1&3
\end{bmatrix} 
\quad
\widehat{Z} = 
\begin{bmatrix}
3_1 & 1_1 &2_1 &4_1 \\
1_1&2_1&4_1&3_3\\
2_1&4_1&1_2&3_3
\end{bmatrix}
$\\

We see $AW(\tilde{Z}) = 0$ and $TR(\widehat{Z}) = (-1)3_11_12_14_11_23_3$.  Then $AW_x(\tilde{Z}) = (-1)12 \otimes 1 \otimes 13 \otimes 1$.  The sign is from the single non-zero term $z_a = 312413$, with $c(z_a) = -1$.  We remind that it is $c(z_a)\tilde{Z}$ that is a summand of the $EZ$ triangulation of $Prism(x)$, so $AW_x(c(x_a))\tilde{Z}) = 12 \otimes 1 \otimes 13 \otimes 1$.\\

  The base simplex of $d_4x$  is\\

$
\tilde{Z'} =
\begin{bmatrix}
1 & 1 & 1 & 1 \\
2 & 1 & 1 & 1 \\
2 & 1 & 3 & 1 
 \end{bmatrix}
\quad
\gamma_x(\tilde{Z'}) = Z' = 
\begin{bmatrix}
3 & 1 &2 &4 \\
3&2&4&1\\
2&4&1&3
\end{bmatrix}
\quad
\widehat{Z'} = 
\begin{bmatrix}
3_1 & 1_1 &2_1 &4_1 \\
3_1&2_1&4_1&1_2\\
2_1&4_1&1_2&3_3
\end{bmatrix}
$\\

We see $AW(\tilde{Z'}) = 12 \otimes 1 \otimes 13 \otimes 1$ and $TR(\widehat{Z'}) = 0$, $AW_x(\tilde{Z'}) = 0$.\\

The computations for $Prism(d_7x) $ and its diagonal are similar to those for $Prism(d_4x)$.\\

We will do one more computation for an interior 2-simplex of $Prism(x)$.\\

$
\tilde{Z''} =
\begin{bmatrix}
1 & 1 & 1 & 1 \\
1 & 1 & 2 & 1 \\
2 & 1 & 3& 1 
 \end{bmatrix}
\quad
\gamma_x(\tilde{Z''}) = Z'' = 
\begin{bmatrix}
3 & 1 &2 &4 \\
1&2&3&4\\
2&4&1&3
\end{bmatrix}
\quad
\widehat{Z''} = 
\begin{bmatrix}
3_1 & 1_1 &2_1 &4_1 \\
1_1&2_1&3_2&4_1\\
2_1&4_1&1_2&3_3
\end{bmatrix}
$\\

Then $AW(\tilde{Z''}) = 1 \otimes 1 \otimes 123 \otimes 1$.  Also $TR(\widehat{Z''}) = 3_11_12_13_24_13_3 + 3_11_12_14_11_23_3$, so $AW_x(\tilde{Z''}) = 1 \otimes 1 \otimes 123 \otimes 1 + 12 \otimes 1 \otimes 13 \otimes 1$.  Then $$ H_x (\tilde{Z''}) = h_\otimes(AW_x(\tilde{Z''}) - AW(\tilde{Z''}) - H_x(d\tilde{Z''})).$$
We have already computed $H_x$ on the three boundary edges of $ \tilde{Z''}$.  The only non-zero result is $H_x(\tilde{Y'}) = 12 \otimes 1 \otimes 13 \otimes 1$. Thus $H_x(\tilde{Z''})= h_\otimes (0) = 0.$ This example shows that for interior simplices of prisms the relation between $AW$ and $AW_x$ is pretty obscure if the base simplex and fundamental simplex of $Prism(x)$ do not coincide.  $\qed$

\end{exam}
{\bf The Case When Base and Fundamental Simplices  Coincide.} It can happen, although it is somewhat rare, that the fundamental simplex and base simplex of $Prism(x)$ do coincide.  Examples are $x = (1234321)$ and $x = (151262343214)$. The caesuras must occur in non-decreasing order.  In such a case note that the same situation holds for all prism faces of $Prism(x)$.  For these $x$ it turns out that $AW_x(\tilde{Y}) = AW(\tilde{Y})$ for all $\tilde{Y}$.  The surprising point is that this holds for {\it every} simplex of $Prism(x)$, not just for the fundamental simplex which coincides with the base simplex.\\

By composing with a permutation in $\Sigma_n$, there is no loss of generality assuming the caesura entry values of  $\tilde{Y}$ are $1, \ldots, \ell$ and the singleton entry values occur in the order $\ell+1, \ldots, n$. It is a rather amazing fact that there is any connection at all between table reduction and the Alexander-Whitney map. We sketch a proof of the last statement of Proposition \ref{16.15}.\\

We form  a bijection between the non-zero summands of $AW(\tilde{Y})$ and $AW_x(\tilde{Y})$. Given the $(k+1) \times n$ matrix $\tilde{Y}$, a non-zero summand of $AW(\tilde{Y})$ is named by a sequence of integers $1 \leq i_1\leq \ldots \leq  i_k \leq \ell$ where the $i_j$ name the columns where one drops down one row in a path connecting the upper left entry to the lower right entry of $\tilde{Y}$.  When one drops down a row, the entry in that column must increase.  We associate to such a path a summand of $AW_x(\tilde{Y})$.  These can be named by partitions $a_0 +\ldots + a_k = n+k$.  But they can also be named by simply indicating the last (caesura) entries in the  first $k$ rows of the  table form of the enhanced surjection generators $\hat{y}_a$.  These are computed from the enhanced permutation matrix $\widehat{Y}$, using the partition $a$ and the method of Proposition \ref{16.2}.  All the signs $c(y_a) = +1$ because the caesuras of $x$ and all its faces occur in increasing order.  It is also an exercise using this hypothesis that the tensors $\tilde{y}_a$ occurring in the formula for $AW_x(\tilde{Y})$ that arise from non-zero $AW(\tilde{Y})$ summands are non-zero.\\

Going the other direction, given a partition $a$ that deternmines a non-zero summand of $AW_x(\tilde{Y})$, we take the $i_j$ to be the caesura entry values of the associated enhanced surjection generator $\hat{y}_a$. Again, the hypothesis implies the resulting $AW(\tilde{Y}) $ summand is non-zero. This completes our very long discussion of Proposition \ref{16.15}.\\ 
\end{proof}

\begin{exam}\label{16.18} Consider $x = (151262343214)$.  We look at the following 3-simplex in $Prism(x)$.\\
 
$
\tilde{Y} =
\begin{bmatrix}
1 & 2 & 1 & 1 & 1 & 1\\
2 & 2 & 1 & 1  & 1 & 1\\
2 & 3 & 1 & 1 & 1 & 1 \\
3 & 3 & 2 & 2 & 1 & 1
 \end{bmatrix}
\quad
\widehat{Y} = 
\begin{bmatrix}
1_1 & 5_1 &6_1 &2_2  & 3_1 & 4_1\\
5_1&1_2&6_1&2_2 & 3_1 & 4_1\\
5_1&1_2&6_1&3_1 & 4_1 & 2_3\\
5_1 & 6_1 & 3_2 & 2_3 & 1_3 & 4_2
\end{bmatrix}
$\\

There are two sequences of column moves that produce non-zero $AW(\tilde{Y})$ terms, namely $123$ and $124$.  These  yield the tensor terms $$12 \otimes 23 \otimes 12 \otimes 2 \otimes 1 \otimes 1\ \ \rm{and}\ \ 12 \otimes 23 \otimes 1 \otimes 12 \otimes 1 \otimes 1.$$
On the $\widehat{Y}$ side, we compute the summands of $TR(\widehat{Y})$ with caesura entries $1_1,2_2,3_1$ and $1_1,2_2, 4_1$, respectively, in the first three rows.  The results are
$$(1_15_11_26_12_23_13_22_34_2)\ \ \rm{and}\ \ (1_15_11_26_12_23_14_12_34_2),$$ which yield the same tensors as the $AW(\tilde{Y})$ calculation.  Note the first of these tensors maps by $\gamma_x$ to a degenerate surjection element $(151623324)$, while the second maps to $TR \gamma_x(\tilde{Y}) = (151623424)$. $\qed$
\end{exam}

 \begin{rem}\label{16.19} We now take up another interesting aspect of the pair of maps $PR$ and $TR$.  We know for very simple reasons the fact that $TR$ commutes with contractions implies $TR \circ PR = Id$.  This is part of Proposition 6.3 of Part I. The dichotomy from Proposition \ref{16.10} that $TR(X) = x$ for one summand of $PR(x)$ and $TR(Y) = 0$ for all other summands of $PR(x)$  is analogous to the situation for the two maps $\phi \colon M_*(n) \to N_*(EC_n)$ and $\pi \colon N_*(EC_n) \to M_*(n)$ with $\pi \circ  \phi = Id$, for the minimal model and MacLane model of the cyclic group, from Section 6 of Part I.  In that case the dichotomy was forced because there were no signs and thus no possible cancellation in the double sum formula for $\pi \phi(y) = y$.  There are signs in the map $PR$, so possible cancellations could have occurred.\\

  It might seem somewhat mysterious why these two pairs of maps share such similar properties.  But there are strong geometric similarities.  The minimal complex $M_*(n)$  is the chain complex of a cell complex, with $n$ cells $T^i y_k$  in each dimension $k$ with known descriptions as triangulated geometric cells. In fact, in even dimensions the cells have the form $T^ie_k * S^{2k-1},\ 0 \leq k, \ 0 \leq i \leq n-1$, which is a $2k$-disk\footnote{If $k = 0$ the sphere $S^{2k-1}$  is empty, but join with the empty set is the identity operation.}, and in odd dimensions the cells have the form of a join with an interval $[T^ie_k, T^{i+1}e_k]* S^{2k-1}$, which is a $2k+1$ disk.  The $T^ie_k$ are vertices. The spheres $S^{2k-1}, \ 1 \leq k,$ are triangulated as simplicial sets as ordered  iterated joins of $k$ circles, with the $j^{th}$  circle  for $ 0 \leq j \leq k-1 $ triangulated with $n$ verticies $T^ie_j$ and $n$ edges $[T^ie_j, T^{i+1}e_j]$. These triangulated cells are analogous to the triangulated prisms that map to $N_*(E\Sigma_n)$. \\
  
  We pursue further similarities between the maps $PR \colon \cS_*(n) \to N_*(E\Sigma_n)$ and $\phi \colon M_*(n) \to N_*(EC_n)$.\\

{\bf A Lensmatic Decomposition of $\bf{EC_n}$.} The explicit map $\phi \colon M_*(n) \to N_*(EC_n)$ from Section 6 of Part I  can  be interpreted as a `lensmatic decomposition' of $EC_n$ analogous to the Berger-Fresse interpretation of the map $PR \colon \cS_*^{bf}(n) \to N_*(E\Sigma_n)$ as a `prismatic decomposition' of $E\Sigma_n$ [5]. Precisely, each vertex $T^ie_j$ of a triangulated cell of $M_*(n)$ maps to the vertex $T^i$ of $EC_n$.  By convexity, the map on vertices extends to the simplices of the triangulated cells, and then to normalized chain complexes of the triangulated cells. The result yields  the map $\phi \colon M_*(n) \to N_*(EC_n)$ given on generators $y_{2k} = e_k * S^{2k-1}$ of $M_*(n)$ and $y_{2k+1} = [e_k, Te_k] * S^{2k-1}$ by the formulas in Proposition 6.13 of Part I. These generators  $y_j \in M_*(n)$ are interpreted as the sums of the maximal dimension simplices of the corresponding triangulated cells.  It turns out there are no orientation signs needed as there are in the $EZ$ map for prisms. This is because the natural orientations of the simplices as ordered  iterated joins of points and intervals agrees with the full cell orientations.  Some of the simplices of the triangulated cells map to degenerate simplices in $EC_n$.  These match the degeneracies seen  in the recursive formulas $\phi(y_\ell) = h_C\phi(dy_\ell)$ from Section 6.\\

Every simplex of $EC_n$ is in the image of various triangulated cells $y_\ell.$  For example, consider $\sigma = (T^{a_0} T^{a_1} \ldots T^{a_k})$, with $a_{j+1} \not= a_j$.  Form the concatenation $$ (T^{a_0} T^{a_0+1}T^{a_1}T^{a_1+1} \ldots T^{a_k}T^{a_k+1}).$$ If all $a_j \not= a_{j-1}+1$ this will be the image of a codimension 1 simplex in the triangulated  $y_{2k+2}$ if $a_0 \not= 1$ and the image of  a codimension 0 simplex  in  the triangulated   $y_{2k+1}$ if $a_0 = 1$.  If some $a_j = a_{j-1}+1$, delete pairs $T^{a_j}T^{a_j+1}$  from the concatenation.  Call the resulting non-degenerate simplex  $\hat{\sigma}$. Then $\sigma$ is a face of $\hat{\sigma}$.  This construction is analogous to a previous argument  that all simplices of $E\Sigma_n$ belong to prisms associated to surjection generators. The intersection of images of $M_*(n)$ cells in $|EC(n)|$ seems complicated.$\qed$
\end{rem}
\begin{rem}\label{16.20}{\bf Fundamental Simplices.} There are also analogues of fundamental simplices in the $M_*(n)$ case.  The proof of Proposition \ref{16.10} above shows that the fundamental simplices of the $Prism(x)$'s can be defined by induction and equivariance, using the fact that $TR$ commutes with contractions.  Back in Section 6 of Part I, we identified the analogous fundamental simplices $\sigma_{2k} = (1, T^{n-1}, 1, \ldots, T^{n-1}, 1)$ and $\sigma_{2k+1} = (1,T, \ldots, 1,T)$ associated to the $y_{2k}$ and $y_{2k+1}$ by directly analyzing the formula for the retraction $\pi \colon N_*(EC_n) \to M_*(n)$.  We showed $\pi(\sigma_{2k}) = y_{2k}$ and $\pi(\sigma_{2k+1}) = y_{2k+1}$ and $\pi$ vanishes on all other simplex summands of the $\phi(y_j)$. This is analogous to a proof of Proposition \ref{16.10} that would directly use the sum over partitions formula for the retraction $TR \colon N_*(E\Sigma_n) \to S_*^{bf}(n)$.  But we can also give an inductive proof in the $\pi \colon N_*(EC_n) \to M_*(n)$ case, using the same method we used for Proposition \ref{16.10} above.  \\

Note $\sigma_{2k} =  h_C(T^{n-1}\sigma_{2k-1})$ and $\sigma_{2k+1} = h_C(T\sigma_{2k})$.  The main point now is that $\pi$ is equivariant and commutes with contractions.  The inductions go
$$\pi( \sigma_{2k}) = \pi h_C(T^{n-1}\sigma_{2k-1}) = h_M \pi(T^{n-1} \sigma_{2k-1} )  = h_M(T^{n-1} y_{2k-1}) = y_{2k}$$
and
$$\pi( \sigma_{2k+1}) = \pi h_C(T\sigma_{2k}) = h_M \pi(T \sigma_{2k} )  = h_M(T y_{2k}) = y_{2k+1}. $$
Non-fundamental simplex summands of the $\phi(y_j)$ also pair up in adjacent dimensions and the same inductions imply $\pi$ vanishes on these. Again, this is similar to the non-fundamental simplex part of Proposition \ref{16.10}. $\qed$
\end{rem}

\newpage
\addcontentsline{toc}{section}{PART III: The Barratt-Eccles and Surjection Operads}

\section*{  III: The Barratt-Eccles and Surjection Operads}

\section {Functorial Chain Maps $ \cS_*^{bf}(n) \otimes N_*(X) \rightarrow N_*(X)^{\otimes n} $ }

 In this first section of Part III of our paper we return to our main program and recover the Berger-Fresse equivariant natural transformation chain map of [3], $$\Phi \colon \cS_*^{bf}(n) \otimes N_*(X) \rightarrow N_*(X)^{\otimes n}, $$ as a special case of the explicit equivariant functorial standard procedure construction discussed in Section 8 of Part I.  We will denote the standard procedure map by $\gamma$, and the goal will be to prove $\Phi = \gamma$ after we give the Berger-Fresse formula for $\Phi$.\\
 
 We can replace $\cS_*^{bf}(n)$ by $\cS_*^{aj}(n)$ or $\cS_*^{ms}(n)$ by composing with the isomorphisms between surjection complexes constructed in  Section 15, since up to sign  these isomorphisms preserve the chosen basis elements.  It is somewhat trickier but true that  the composition $$\Phi \circ (TR \otimes Id) \colon N_*(E\Sigma_n) \otimes N_*(X) \to \cS_*^{bf}(n) \otimes N_*(X) \to N_*(X) ^{\otimes n}$$ is also the standard procedure map.  In this section, we deal only with the chain complex structures.  In Sections 19 and 20  below we return to the discussion and bring in the operad structures of the  Barratt-Eccles and surjection operads,  following a review of some basic concepts about operads in Section 18. Subsection 18.3 carefully treats the $End$ and $CoEnd$ operads for chain complexes and the related Eilenberg-Zilber operad of functorial chain  operations $\cZ(n) = HOM_{func}(N_*( - ), N_*( - )^{\otimes n})$.\\
 
We choose the $\FF[\Sigma_n]$ basis of $\cS_k^{bf}(n)$ consisting of $b =( b_1, b_2, \ldots, b_{n+k})$ so that the initial occurrences of $\{1,2, \ldots, n\}$ occur in that order.  With $e = 1 \in  \FF[\Sigma_n] = \cS^{bf}_0(n)$, the map $ \cS_0^{bf}(n) \otimes N_*(X) \to N_*(X)^{\otimes n}$ is the equivariant extension of the $n$-fold Alexander-Whitney multidiagonal map $\delta^{(n)} \colon \{e\} \otimes N_*(X) = N_*(X) \to N_*(X)^{\otimes n}$, in both the Berger-Fresse $\Phi$ and functorial standard procedure $\gamma$ contexts.  The left action of $g \in \Sigma_n$ on $N_*(X)^{\otimes n}$ is given by $$g(a_1 \otimes \ldots \otimes a_n ) = (-1)^{k(a,g)}(a_{g^{-1}(1)} \otimes \ldots \otimes a_{g^{-1}(n)}),$$ where $(-1)^k$ is a Koszul sign, as explained in Section 3 of Part I.

\subsection{The Berger-Fresse Map}
 
For each universal simplex generator $\Delta^m \in N_m(\Delta^m)$ and each surjection generator $x \in \cS_k^{bf}(n)$ with $k > 0$ we will define $\Phi (x \otimes \Delta^m) \in (N_*(\Delta^m)^{\otimes n})_{k+m}$, and prove that this agrees with the  acyclic model standard procedure  map of Remark 8.1 of Part I.  The inductive formula for the standard procedure map on surjection basis elements is $$\gamma (b \otimes \Delta^m) = h_{\otimes^n} (\gamma( db \otimes \Delta^m + (-1)^k b \otimes d \Delta^m)),$$ where $h_{\otimes^n}$ is the preferred contraction of $N_*(\Delta^m)^ {\otimes n}$.  It is important to pay attention to functoriality of $\gamma$ and interpret  boundary summands   $\gamma(b \otimes d\Delta^m)$ as lying in the images of various minimal carrier face maps $N_*(\Delta^{m-1}) ^{\otimes n} \to N_*(\Delta^m)^{\otimes n}$.\\

Throughout Part III the simplices $\Delta^m$ are unrelated to generators of surjection complexes.  Therefore we will name vertices  beginning with 0, $\Delta^m = (0,1, \ldots, m)$, as that seems notationally cleaner in the current context.\\
 
We will now define the Berger-Fresse map $\Phi(x \otimes \Delta^m)$ in all degrees. This will require introducing quite a bit of notation.   We have the sum of monomial tensors $\delta^{(n+k)}(\Delta^m) \in N_*(\Delta^m)^{\otimes{(n+k)}}$, where $\delta^{(n+k)}$ is the Alexander-Whitney multidiagonal.   We will use the notation $M \in \delta^{(n+k)}(\Delta^m)$ to refer to the monomial summands.  These have the form $M = M_1 \otimes M_2 \otimes \ldots \otimes M_{n+k}$, where the blocks $M_j = (m_{j-1}, \ldots, m_j)$ are faces of $\Delta^m$ given by strings of consecutive vertices.   For each such multidiagonal term $M$ we have $0 = m_0 \leq m_1 \leq \ldots \leq  m_{n+k} = m$. These strings partition the interval $[0, m]$ into subintervals $[m_{j-1}, m_j]$.  It is allowed that some $m_{j-1} = m_j$ in the multidiagonal formula, that is, the faces $M_j$ can be single vertices.\\

For a generator  $x =( x_1, x_2, \ldots, x_{n+k}) \in \cS_k^{bf}(n)$ and $1 \leq \ell \leq n$,  say the value $\ell$ occurs $k_\ell$ times in $x$.  Thus $\sum k_\ell = n+k$.  Denote by $M_{\ell1} ,M_{\ell2}, \ldots, M_{\ell k_\ell}$ the blocks of $M$, in order, corresponding to those  $M_j$ with $x_j = \ell$.
Denote by  $F_\ell(M) = F_\ell(M, x) = M_{\ell1}M_{\ell 2}\ldots M_{\ell k_\ell}$ the face of $\Delta^m$ named by the subset of vertices, in order, of all the blocks $M_j$ of $M$ with $x_j = \ell.$  If $x$ is degenerate one of the faces $F_i(M)$ is degenerate for every $M$.\footnote{Such a face will also be degenerate if all the blocks of $M$ between two blocks associated to  consecutive occurrences of $\ell$ in $x$ are single vertices.  But there is no real harm in including these faces in formulas.  Such faces become 0 in a normalized chain complex.} We point out that the face $F_\ell(M)$ is the join  of its component subfaces $M_{\ell i}$.\\
  
{\bf The Berger-Fresse Formula.} The definition is, for $M \in \delta^{(n+k)}(\Delta^m)$, $$\Phi (x \otimes \Delta^m) = \sum_M \epsilon (M) F_1(M) \otimes F_2(M)\otimes  \ldots  \otimes F_n(M) \in (N_*(\Delta^m) ^{\otimes n})_{m+k},$$ where $\epsilon(M) = \epsilon (M,x) \in \{\pm 1\} $ is a sign that we will  now define.  The sign $\epsilon(M)$ will be the product of two signs $sh(M) pos(M)$, both also depending on $x$. \\

 Note the block $M_{\ell k_\ell}$ corresponds to the non-caesura entry of $x$ of value $\ell$.  All other blocks $M_{\ell i}$ correspond to caesura entries of $x$ of value $\ell$.  Assign `lengths' $||M_{\ell i} ||$ to the blocks as follows.  If $M_{\ell i}$ is a caesura block, $||M_{\ell i} ||$ is  the number of vertices in the block. This is one more than the geometric dimension of the corresponding face simplex.  For non-caesura blocks  $|| M_{\ell k_\ell}|| $ is one less than the number of vertices, which is the same as the geometric dimension of the corresponding face.  Note that the geometric dimension of an amalgamated face $F_\ell(M)$ is exactly the sum of the `lengths' of its component blocks, since $F_\ell(M)$ is the join of its component faces $M_{\ell i}$.\\

For any monomial $M$ we can rearrange by a shuffle permutation  the blocks $M_1 M_2 \ldots M_{n+k}$ in the order $$M_{11}M_{12}\ldots M_{1k_1} M_{21}\ldots M_{2k_2} \ldots M_{n1}\ldots M_{nk_n}.$$   Define the {\it shuffle sign} $sh(M) \in \{ \pm 1 \}$ to be the `Koszul sign' associated to this shuffle permutation of blocks, using the product of two  `lengths' $|| M_{ij} ||$ to determine a sign when one block is passed by another block. \\

An alternate view of $sh(M)$ is obtained by discarding the final entries of all non-caesura blocks  and replacing all other entries of blocks $M_j$ by $x_j$.  This produces a string of $(m+k)$ numbers $(x_1 \ldots x_1\ x_2 \ldots x_2\ \ldots x_{n+k} \ldots x_{n+k})$, each of which is in the interval $(1,2 \ldots, n)$.  Then $sh(M)$ is the parity sign of the shuffle permutation that rewrites the $x_j$-string in the order $(1 \ldots 1\ 2 \ldots 2\ \ldots n \ldots n)$, keeping all entries of the same value $x_j$ in their original order. This parity sign is determined by the parity count of the number of pairs in the $x_j$ string consisting of an entry $x_j$ and an entry $x_i$ with $i < j$ and $x_j < x_i$. \\

The other sign $pos(M)$ in the definition $\epsilon(M) = sh(M) pos(M)$ is called the {\it position sign} by Berger-Fresse.  It is given by $$pos(M) = \prod_{caesuras\ x_j} (-1)^{m_j},$$ where the caesura blocks are $(m_{j-1}, \ldots,  m_j)$.  The 
position sign is determined by the parity count of all pairs consisting  of a final entry of a caesura block and an earlier entry that is not a final entry of any block.  This  invariant interpretation will be useful  later when dealing  with faces of $\Delta^m$ and functoriality.\\

We will give an example to illustrate the above conventions.
\begin{exam}\label{17.1} Let $n = 3,\ k = 4,\ m = 5$.  Consider $x = (1213213)$ and multidiagonal term $M = M_1M_2 \ldots M_7 =(0 | 01|12|2|234|45|5)$.  Then $$F_1 = M_1M_3M_6 = (0\ 12\ 45),\ F_2 = M_2M_5 = (01\ 234),\ F_3 = M_4M_7 = (2\ 5).$$
The caesura blocks $M_1, M_2, M_3, M_4$ have `lengths'  1,2,2,1 respectively.  The non-caesura blocks $M_5, M_6, M_7$ have `lengths' 2,1,0 respectively.  The shuffle permutation $M_1M_2M_3M_4M_5M_6M_7 \mapsto M_1M_3M_6M_2M_5M_4M_7$ has sign $sh(M) = (-1)^4(-1)^5(-1)^2 = -1$ since $M_3$ moves past $M_2$, then $M_6$ moves past $M_2, M_4, M_5$, then $M_5$ moves past $M_4$.  Alternatively, the $x_j$ sequence with final entries of non-caesura blocks removed is $(1 2 2 1 1 3 2 2 1)$, with corresponding shuffle sign $(-)^{11} = -1$.  Finally the position sign is seen to be $pos(M) = (-1)^0(-1)^1(-1)^2(-1)^2 = -1$.\\

The term of degree 9  in $N_*(\Delta^5)^{\otimes 3}$ corresponding to $x$ and $M$ is thus $$\epsilon(M)F_1(M) \otimes F_2(M) \otimes F_3(M) = (-1)(-1) (01245)\otimes (01234) \otimes (25)\ \  \qed$$ 
\end{exam}

\begin{rem}\label{17.2}{\bf The Step Diagram Picture.}  It is appropriate to review the step-diagram picture of the data consisting of a generator $x \in \cS_*(n)$ and a monomial $M$.  For the study of  Steenrod operations for an odd prime p, the number $(p-1)/2$ is typically called $m$.  Therefore, in this remark we will change the name of our basic simplex in the Berger-Fresse formula to $\Delta^q$.\\

For a fixed surjection generator $x$, we place closed  intervals $M_j$ in a box $[1, n] \times [0, q]$ with $n$ rows and $q+1$ columns. All intervals $M_j = [m_{j-1}, m_j]$ with $x_j = \ell$ are placed on row $\ell$, with endpoint coordinates  $(\ell, m_{j-1})$ and $(\ell, m_j)$.  These intervals are single points if $m_{j-1} = m_j$.  It is preferred to include in diagrams only those $M$ so that the faces $F_\ell(M, x)$ are non-degenerate.  The picture itself certainly determines $M$, but does not quite determine $x$ if there are columns containing more than one singleton interval.  The $x$ entries corresponding to singleton intervals in a column, which means adjacent singleton intervals in $M$, could be permuted without changing the diagram.  One could include additional decoration in the diagrams to clarify the order in $x$ in which multiple singletons occur in  columns.  But note the faces $F_\ell(M, x)$ are independent of these permutations of entries of $x$, which is probably a good thing.\\

{\bf *The Special Case of Diagrams That Completely Fill a Box.*} The results of the next few paragraphs will be used later to prove that for an odd prime $p$, the Steenrod operation $P^0 = Id.$ In the non-degenerate cases for $x$ of degree $k$, the number of integer coordinate points in the box that are covered by intervals is $n + k + q$.  The box itself has $n(q+1)$ integral points.   So with $n$ and $q$ fixed, all points can be covered non-degenerately if $n+k+q = nq + n$, that is, if $k =  q(n-1)$. In this case $\Phi(x \otimes \Delta^q)$ will be either 0 or $\pm (\Delta^q)^{\otimes n} \in (N_*(\Delta^q)) ^{\otimes n}$.\\

The only pairs $(M, x)$ that produce a non-degenerate term are $$M = (0| \ldots | 0 | 0 1| |1| \ldots |1| |12|  \ldots |(q-1) q| |q| \ldots |q)$$ with  $n$ occurrences each of $0,1, \ldots, q$, and $x$ formed by concatenating $q+1$ permutations $g_i$ of $\{1, \ldots, n\}$, $0 \leq i \leq q$,   with $g_i(n) = g_{i+1}(1)$, and removing one of each adjacent repeated entry in the concatenation of the $g_i$.\\

In this case we can calculate the signs $sh(M, x)$ and $pos(M, x)$ using the descriptions of these signs given above. The non-caesura entries of $x$ are the last $n$ entries, corresponding to the final permutation $g_q$. Inspection of $M$ shows that the values of $m_i$ corresponding to final entries of caesura blocks consist of $n-1$ values each of $1,2, \ldots, (q-1)$.  Thus $pos(M,x) = (-1)^{(n-1)q(q-1)/2}$.\\

Thus if $n$ is odd the position signs  $pos(M, x) = +1$. If $n$ is even, the position signs are $pos(M, x) = (-1)^{q(q-1)/2}$.\\

We make use of the second description given above for $sh(M,x)$. We need the shuffle sign for putting the concatenation of the first $m$ permutations $g_i \in\Sigma_n$ in the order $(1 \ldots 1\ 2 \ldots 2 \ldots n \dots n)$.  Here we recall that the entries of $M$ are also labeled by entries of $x$. We can first just put the entries of each $g_i$, $i < q$, in the order $(12 \ldots n) = Id_n$.  These moves do not change the order of $M$ entries labeled with the same $x$ value. The parity sign for this first step  is $\prod_{i < q} \tau(g_i)$.  Then we want to rearrange the entries of $q$ copies $(Id_n\ Id_n\ \ldots Id_n)$ in non-decreasing order.  Move the $1$'s to the front, then the $2$'s to follow the $1$'s, and so on. A simple count gives the parity sign $(-1)^{(n(n-1)/2)(q(q-1)/2)}$.  Thus $sh(M, x) = \prod_{i < q} \tau(g_i) (-1)^{(n(n-1)/2)(q(q-1)/2)}$.\\

For example, take $q = 2$ and $x = (14523 1245 1234 )$, thus permutations $g_0 = (14523), g_1 = (31245), g_2 = (51234)$.  So $n = 5$, $\tau(g_0) = +1, \tau(g_1) = +1$.  Labeling $M$ with $x$-entries produces $M_x = (145233124551234)$.  Remove the final 5 entries, corresponding to the final entries of non-caesura blocks of $M$, yielding $(1452331245)$.  We get $pos(M,x) = +1$ since $n$ is odd,  and $sh(M, x) = +1$.\\

We make two further observations about special diagrams related to evaluations of certain cochain operations that we take up in Part IV. If $k > q(n-1)$  then all terms in the Berger-Fresse formula are degenerate.  If $k = 0$ then $x$ is a permutation in $\Sigma_n$, $M = (0,m_1,\ldots, m_{n-1}, q)$, and the diagram consists of a single interval on each row $\ell = x_j$ with column coordinates $[m_{j-1}, m_j]$.  In the $k = 0$ case, the Berger-Fresse formula  essentially just records, as $x \in \Sigma_n$ varies,  the $\Sigma_n$-orbit of the  $n$-fold Alexander-Whitney diagonal of a $q$-simplex.     $\qed$
\end{rem}
We now take up the main result of this section.  We revert to calling our basic simplex $\Delta^m$. The following proposition is one of the most difficult results of our paper.
\begin{prop} \label{17.3}The Berger-Fresse map, for $M \in \delta^{(n+k)}(\Delta^m)$, $$\Phi (x \otimes \Delta^m) = \sum_M \epsilon (M) F_1(M) \otimes F_2(M)\otimes  \ldots  \otimes F_n(M) \in (N_*(\Delta^m) ^{\otimes n})_{m+k},$$ agrees with the standard  procedure chain map $\gamma \colon \cS_*^{bf}(n) \otimes N_*(X) \to N_*(X)^ {\otimes n}$.
\end{prop}
{\bf Discussion.} The proof that follows is an induction, comparing the Berger-Fresse formula with the standard procedure chain map in each new degree.  An early step is that the Berger-Fresse map is equivariant.  After that, the result would follow from Step 2 below, and the uniqueness result Proposition 8.2 of Part I.   But this would require first establishing that the Berger-Fresse formula is a chain map, which is no easy feat.  We believe that approach is  less informative than the proof we give. By identifying the somewhat mysterious Berger-Fresse map with the standard procedure map, our proof shows the Berger-Fresse formula is a chain map.  Also, the steps of our proof bring out some rather appealing structure that  shows up in other contexts.  This structure is part of the general discussion towards the end of  preview Subsection 2.3 of Part I concerning  chain maps of form $\phi(y) = \sum \pm \cT y$. Here the terms $\cT(x \otimes \Delta^m)$ are the tensors $F(M, x) = F_1(M, x) \otimes \ldots \otimes F_n(M, x)$, where the parameter set consists of the monomials $M \in \delta^{(n+k)}(\Delta^m)$.\\

Many of the tensors $F(M, x)$ will be degenerate, hence 0, but that does not affect the form of the answer.  In fact, a little computation in low degrees, beginning with the $n$-fold multidiagonals for simplices, suggests that the answer should be a formula exactly  like that in the Proposition, with unknown signs.  It turns out that we can actually prove by induction such a formula holds for the standard procedure chain map $\gamma$.  The key is to carefully prove by induction the claims in Steps 2 through 6 below, for the standard procedure map, with no reference to the Berger-Fresse formula.  So that is very similar to the discussion of the Eilenberg-Zilber map $EZ$ back in Sections 6 and 8 of Part I.  The arguments are rather pleasant, but since we eventually do need the signs we have chosen only to present an inductive proof that compares the Berger-Fresse formula with the standard procedure method.

\begin{proof} 

{\bf Step 0}.  It is an observation that if $k= 0$ the Berger-Fresse map $\Phi$ and the functorial procedure map $\gamma$ on $\cS^{bf}_0(n) \otimes \Delta^m = \FF[\Sigma_n] \otimes \Delta^m$  both agree with the equivariant extension of the $n$-fold multidiagonal $ \delta^{(n)} (\Delta^m)$. Permutations $g \in \cS_0^{bf}$ have no caesuras, so $pos(M,g) = 1$ for all monomials $M$ and $sh(M,g)$ is the Koszul sign of a permutation of tensors.  If $m = 0$ and $k > 0$ then $(N_*(\Delta^0)^{\otimes n})_k = 0$, hence both the Berger-Fresse map and the standard procedure map are trivial. $\qed$ \\

{\bf Step 1.}  {\bf The Berger-Fresse Map is Equivariant.}  Given a surjection generator $x$ and a permutation $g \in \Sigma_n$, we need to show $g \Phi(x \otimes \Delta^m) = \Phi(gx \otimes \Delta^m)$.  From the definition, for each monomial $M$,  $$g \bigotimes_{\ell = 1}^n F_\ell (M,x) = (-1)^k \bigotimes_{\ell = 1}^n F_{g^{-1}\ell}(M, x) = (-1)^k \bigotimes_{\ell = 1}^n F_\ell(M, gx).$$
Namely,  the face $F_\ell(M, gx)$ coincides with the face $F_{g^{-1}\ell}(M, x)$ since these faces just amalgamate blocks $M_j$ of $M$ with $g(x_j) = \ell$ and $x_j = g^{-1}\ell$ respectively.  The sign $(-1)^k$ is the Koszul sign of the permutation $g$ acting on the $n$-tensor. \\

The caesuras in the two generators $x$ and  $gx$ are in the same positions, so the position signs agree $pos(M,x) = pos(M, gx)$.  Finally we need to compare the shuffle signs $sh(M,x)$ and $ sh(M, gx)$ with the  sign associated to $g$ acting on the $n$-tensor.  The relation we want is $(-1)^k sh(M, x) = sh(M, gx)$.\\

After the shuffle of the $M_j$ using $x$, the blocks corresponding to $j$ with $x_j = \ell$ are adjacent and form the face $F_\ell(M, x)$.  The shuffle sign $sh(M,x)$ is determined using the `lengths' of the blocks $M_j$. The key now is that the sum of the `lengths' of the blocks forming $F_\ell(M, x)$  is the dimension of the face.  Now permute these faces in the tensor by $g$, including the true Koszul sign $(-1)^k$ which depends on the dimension of the faces.  The result is the same as the signed shuffle of the blocks $M_j$ that uses their individual  `lengths'  to form the faces $F_\ell(M, gx)$. That is, one obtains the same sign moving an entire face block across another face block using dimensions,  as one obtains moving the separate subfaces across subfaces  using `lengths'.\\

Perhaps this equivariance argument for the shuffle sign is easier to follow using the second description of the shuffle sign, in terms of sequences $(x_1 \text{'s} \  x_2  \text{'s} \ldots\ x_{n+k} \text{'s})$ and  $(x_{g1} \text{'s} \  x_{g2} \text{'s} \ldots\ x_{g(n+k)} \text{'s})$ of length $m+k$  obtained from $M$ by deleting final entries of caesura blocks and replacing remaining entries of $M_j$ by $x_j$ or $x_{gj}$.  Shuffle permute the first sequence to $(1 \text{'s} \ 2 \text{'s}\ \ldots n \text{'s})$, with parity sign $sh(M, x)$.   Then shuffle permute this string to the order $(g^{-1}(1) \text{'s}\ g^{-1}(2) \text{'s} \ \ldots g^{-1}(n) \text{'s})$, with parity sign $(-1)^k$.  The composed shuffle is conjugate as a permutation of $m+k$ objects to the shuffle that moves $(x_{g1} \text{'s} \  x_{g2}\text{'s} \ldots x_{g(n+k)}\text{'s})$ to  $(1 \text{'s} \ 2 \text{'s} \ldots n \text{'s})$, with sign $sh(M, gx)$.  $\qed$\\

{\bf Preview of the Induction.} Consider a basis generator $b$ of degree $k > 0$, which from Subsection15.1 of Part II means the initial entries in $b$ of $1,2, \ldots n$ occur in that order. Assume the proposition is proved in surjection degrees less than $k$.  By Step 0 we may assume $m > 0$. By induction and Step 1, it suffices to prove $$\Phi(b \otimes \Delta^m) = h_{\otimes^n} \Phi (db \otimes \Delta^m + (-1)^k b \otimes d \Delta^m)$$ for basis generators, where $h_{\otimes^n}$ is the contraction of $N_*(\Delta^m) ^{\otimes n}$. In the next five steps we will prove by induction that  this formula is correct if we ignore signs. In Steps 2, 3, and 4 we show that most summands of the boundary terms $db$ and $d\Delta^m$ contribute 0. In Steps 5 and 6 we compare $h_{\otimes^n}$ applied to the remaining boundary terms with the formula for $\Phi(b \otimes \Delta^m)$.   In the last step we deal with the signs.\\

{\bf Review of the Contraction of $\bf N_*(\Delta^m)^{\otimes n}$.} We remind that $$h_{\otimes^n} = h \otimes Id^{\otimes (n-1)} + \rho \otimes h \otimes Id^{\otimes (n-2)} + \ldots +  \rho^{\otimes (n-1)} \otimes h,$$ where $h$ is the contraction of $N_*(\Delta^m)$ and $\rho(j) = 0$ is the basepoint map of $N_*(\Delta^m)$.  Thus evaluating an $h_{\otimes^n} $ term could result in as many as $n$ summands and this will play a role in the proof evaluating $h_{\otimes^n}( b \otimes d \Delta^m)$.\\

Note  $Im(h_{\otimes^n}) = Ker(h_{\otimes^n})$ is spanned by all `clean' tensors $(0)  \otimes  \ldots \otimes (0) \otimes (0a...) \ldots$, where $a > 0$, since these are clearly in $Ker(h_{\otimes^n})$.  This point was emphasized in Subsection 5.3 of Part I, where the preferred contraction $h_{\otimes^n}$ of tensor products was introduced. Also, if $a' > 0$ $$h_{\otimes^n} [(0)\text{'s}.. \otimes (a'...)  \otimes .. (0)\text{'s} .. \otimes (0a...) \ldots] = [ (0)\text{'s}.. \otimes (0a'...)  \otimes .. (0)\text{'s} .. \otimes (0a...) \ldots]$$
This formula  is correct even if the $(a' ...)$ term is a singleton and more $h_{\otimes^n}$ summands must be evaluated. These additional summands will be 0.\\

{\bf Step 2.} {\bf Clean Surjection Generators.}  This is a major step. Suppose $c = (1, \ldots, \ell, \ldots, \ell, \ldots)$ is a clean surjection generator, as in Subsection 15.1.   So $\ell$ is the first caesura and the smaller entries are singletons.  In Proposition \ref{15.1} it was proved that the clean surjections span the image of the contraction $h_{\cS}$.  Then for any monomial $M$ it is easy to see that either the tensor $F(M, c)$ is degenerate or  $$F(M, c) = F_1(M) \otimes F_2(M) \ldots  \otimes F_n(M)  = (0) \otimes \ldots (0) \otimes (0a...) \otimes \ldots,$$ where the $(0a...)$ term occurs before or at the $\ell^{th}$ position.  So $\Phi$ maps $Im(h_{\cS})$ to $Im(h_{\otimes^n})$. Thus $\Phi(c \otimes \Delta^m) \in Im(h_{\otimes^n})$, and $h_{\otimes^n}\Phi(c \otimes \Delta^m) = 0.$\\ 

{\bf Remark.} A purely inductive proof of the claim in Step 2 for $\gamma$ rather than $\Phi$, namely that if $c$ is clean at $\ell$ then each tensor summand of $\gamma(c \otimes \Delta^m)$ is clean at or before $\ell$, procedes by writing $c = gb$, where $b$ is a basic generator clean at $\ell$ and $g$ fixes $1,2, \ldots, \ell$.  Then one studies by induction $g \gamma(b \otimes \Delta^m)$, also making use of Step 3 below.  The main point, which gets a bit delicate with $\ell$ and $\ell -1$, is that permutations in general do not preserve $Im(h_{\otimes})$ but if a permutation fixes small values then it will preserve the collection of tensors that are  clean at roughly  these same small values.  Purists can fill in the details. $\qed$\\

{\bf Step 3.} {\bf Relevant Boundary Terms From $\bf db \otimes \Delta^m$.}  If $b$ is a basis generator then the only terms in $h_{\otimes^n}\gamma(db \otimes \Delta^m)$ that can be non-zero are the terms arising from the boundary term  $d_\ell b$ obtained by deleting the first caesura $\ell = b_\ell$ of $b$. By induction, we know $\Phi = \gamma$ in the lower dimension.\\

The initial entries of $1,2, \ldots, n$ in $b$ occur in that order.  If $j < \ell$ then $d_jb = 0$ since $j$ is a singleton.  If $j > \ell$ and  $b_j = \ell$ then $d_jb$ is either degenerate or is also  a basis generator.  In the basis generator case $\gamma(d_jb \otimes \Delta^m)$ is in $Im(h_{\otimes^n}) = Ker(h_{\otimes^n})$, so $ h_{\otimes^n} \gamma(d_j b, \Delta^m) = 0$.\\

If $j > \ell$ and $b_j> \ell$   then $d_jb = (1,2,\ldots, \ell, \ldots, \ell, \ldots)$, with $b_j$ removed, is either  degenerate or a clean generator of one lower degree.  In the clean generator case, as observed in Step 2, $\Phi(d_jb \otimes \Delta^m) \in Im(h_{\otimes^n}) = Ker(h_{\otimes^n}) $ $\qed$ \\ 

{\bf Step 4.} {\bf Relevant Boundary Terms From $\bf b \otimes d\Delta^m$.} The only terms in $h_{\otimes^n} \gamma(b \otimes d\Delta^m)$ that can be non-zero are the terms arising from the boundary face $d_0 \Delta^m = (12 \ldots m)$.\\

Consider a face $d_j\Delta^m = (0, ... \hat{j} ... ,m)$, with $j > 0$.  We must pay attention to functoriality.  From the characterization of clean tensors preceding Step 2, the map $N_*(\Delta^{m-1})^{\otimes n} \to N_*(\Delta^m) ^{\otimes n}$ induced by the face inclusion carries $Im(h_{\otimes^n})$ to $Im(h_{\otimes^n})$, since $0 \mapsto 0$ and any $c \mapsto c$ or $c+1$.  In fact, such a face inclusion commutes with contractions. Then $\gamma(b \otimes \Delta^{m-1})$, which is in $Im(h_{\otimes^n})$ for $\Delta^{m-1}$, is carried by the face inclusion to $\gamma(b \otimes d_j \Delta^m)$, which is in $Im(h_{\otimes^n})$ for $\Delta^m$. Hence $h_{\otimes^n}$ vanishes on this element.  $\qed$ \\

{\bf Step 5.} {\bf Matching Summands of $\bf \gamma(b \otimes \Delta^m)$ and $\bf h_\otimes \gamma d(b \otimes \Delta^m)$.}  This is the main step.  We have observed that induction implies that if $b$ is a basis generator then the standard procedure map $\gamma$ satisfies 
$$\gamma(b \otimes \Delta^m) = h_{\otimes^n}( \gamma(d_\ell b \otimes \Delta^m + b \otimes d_0 \Delta^m)) = h_{\otimes^n}(\Phi(d_\ell b \otimes \Delta^m + b \otimes d_0 \Delta^m)).$$  We must show this last term coincides with $$\Phi(b \otimes \Delta^m)  = \sum_M \pm  F_1(M,b) \otimes \ldots \otimes F_n(M, b).$$ 
The monomials $M \in \delta^{(n+k)} (\Delta^m)$ divide into two types for each possibility for  the first caesura $b_\ell$ of $b$.  If the first caesura is $b_1 = 1$ then, ignoring signs, the sum of the terms in $\gamma(b \otimes \Delta^m)$ corresponding to monomials of the form $M = (0\vert 0...)$ coincides with $h_{\otimes^n} \gamma(d_1b \otimes \Delta^m)$.  The sum of the terms in $\gamma(b \otimes \Delta^m)$ corresponding to monomials of the form $M = (01...)$ coincides with $h_{\otimes^n} \gamma(b \otimes d_0\Delta^m)$.\\

If the first caesura of $b$ is $b_\ell$  then $b$ must begin $b = (12...\ell, \ell+1,...\ell ...)$, with $b_\ell = \ell$ and the entries before the $\ell$ being singletons.   The sum of the terms in $\gamma(b \otimes \Delta^m)$ corresponding to monomials of the form $M = (0\vert \ldots | 0 | 0...)$ with at least $\ell$ singleton 0's coincides with $h_{\otimes^n} \gamma(d_\ell b \otimes \Delta^m)$.  The sum of the terms in $\gamma(b \otimes \Delta^m)$ corresponding to monomials  $M$  with fewer than $\ell$ singleton 0's coincides with $h_{\otimes^n} \gamma(b \otimes d_0\Delta^m)$. $\qed$ \\

{\bf Step 6.} {\bf Proof of Step 5.}  We discuss the proof of the statements in Step 5.  Assume the first caesura is $b_1 = 1$. If $M' = (0...) \in \delta^{(n+k-1)}\Delta^m$ is obtained from $M = (0 | 0...)$ by dropping the first singleton 0 then the summand of $\gamma(b \otimes \Delta^m)$ corresonding to $M$ equals the summand of $h_{\otimes^n} \gamma(d_1b \otimes \Delta^m)$ corresponding to $M'$. In fact, $$F(M', d_1b) = (a'.. )\otimes (0) \otimes  \ldots  \otimes (0)\otimes (0a...) \dots $$ is obtained from $F(M, b)$ by dropping the first 0.  Here $a'$ is the first entry of the monomial block corresponding to the second occurrence of 1 in $b$.  Then applying $h_{\otimes^n}$ puts that 0 back.  Note only the first summand of $h_{\otimes^n}$ acts nontrivially.\\

It is possible that the summand of $\gamma(b \otimes \Delta^m)$ corresponding to $M$ is degenerate.  This can happen in a few ways.  But then the term obtained from $F(M', d_1b)$ by dropping the first 0 from $M$ and evaluating on $d_1b$ is also degenerate, or it is of the form  $(0) \otimes \ldots \otimes (0) \otimes (0a...)...$.  Either way applying $h_{\otimes^n}$ gives 0.\\

A simple example of the interesting case is $b = (1,2,1,3)$ and $M = (0 | 0 | 0 | 012)$. Then $F(M', d_1b) = (0) \otimes (0) \otimes (012)$, which is in the image of $h_{\otimes^n}$.\\

Continuing to assume the first caesura of $b$ is $b_1$, consider a monomial of the second type $M = (01...| ...)$ and  set $M'' = (1... |...)$ by dropping the first (and only) 0 from $M$.  Then the tensor summand of $\gamma(b \otimes d_0 \Delta^m)$ corresponding to monomial $M''$ is obtained from the tensor $F(M, b)$ by dropping the first (and only) 0.  That $M''$ tensor begins $(1...a...) \otimes ...$, where the $a$ is the first entry in the block of $M$ coming from the second occurrence of 1 in $b$.  Applying $h_{\otimes^n}$ gives back $F(M, b)$, the $n$-tensor of $\gamma(b \otimes \Delta^m)$ corresponding to $M$.  Again only one summand of $h_{\otimes^n}$ acts non-trivially.\\

Moving on, assume the first caesura of $b$ is $b_\ell = \ell$.  Then $b = (12...\ell...\ell ...)$, with  singletons before the first $\ell$.    Consider a first type monomial $M = (0 | 0 | 0...)$ with at least $\ell$ singleton 0's.  Again obtain $M' = (0 | 0..)$ by dropping the $\ell^{th}$ 0. Then the tensor $F(M', d_\ell b)$ is obtained from the tensor $F(M, b)$ by dropping the $\ell^{th}$ 0.  Again this makes sense even if $F(M, b)$ is degenerate.  Applying $h_{\otimes^n}$ formally puts that 0 back, the result of which will be 0 exactly if $F(m, b)$ is degenerate.\\

Next consider the second type monomials $M \in \delta^{(n+k)}(\Delta^m)$.  These have the form $M = (0 | 0 |.. | 01..| ...)$ where there are fewer than $\ell$ singleton 0's.  Form  $M'' = (1 | 1 ... |1..| ..) \in \delta^{(n+k-1)}(d_0\Delta^m)$ by deleting the final occurring 0 and changing all earlier (singleton) 0's to 1's. Now something new happens.  Let $j+1$, with $1 \leq j \leq \ell$, denote the total number of  0 and 1 entries in $M$, in or before the $\ell^{th}$ block of $M$.  We notice that there are $j-1$ singleton $(1)$ tensor factors  in $F(M'', b)$, the summand of $\gamma(b \otimes d_0\Delta^m)$ corresponding to $M''$. In fact, the tensor $F(M'', b)$ is obtained from the tensor $F(M, b)$ by deleting the final occurrence of 0, and changing any previous singleton $(0)$ tensor factors to $(1)$'s. Also there are $j $ different second type $M$'s yielding this same $M''$.  Namely, the first $j+1$ entries of $M$ (not blocks) could consist of $j'$ 0's and $(j +1 - j')$ 1's, where $1 \leq  j' \leq j$.  \\

At the same time, evaluating $h_{\otimes^n}$ on the $n$-tensor summand of $\gamma(b \otimes d_0 \Delta^m) $ corresponding to $M''$ yields a sum of $j$ terms.  It can happen that $F(M, b)$ is degenerate, but not before the $\ell^{th}$ tensor factor.  Then $F(M'', b)$ is also degenerate, so all summands of $h_{\otimes^n} F(M'', b)$ are degenerate.\\

For example, if $\ell = 5$, $M = (0 | 0 | 0 | 012 |2| ... )$  and $M'' = (1 |1 | 1 | 12| 2 | ...)$, then $j = 4$ and the term in $\gamma(b   \otimes d_0  \Delta^m)$ corresponding to $M''$ is $ (1) \otimes (1) \otimes (1) \otimes (12) \otimes (2a...) ...$, where $a$ is the first entry of the second block of $M$ corresponding to a $b$ value $\ell = 5$.  Evaluating $h_\otimes$ gives exactly the sum of the 4 terms in $\gamma(b \otimes \Delta^m)$ corresponding to all the $M$ terms yielding the same $M''$.   If $a = 2$, these are all degenerate.  $\qed$ \\

{\bf Step 7.} {\bf Reconciliation of Signs.} Finally we need to deal with the shuffle signs and the position signs in the Berger-Fresse map in all the steps above.  The contraction $h_{\otimes^n}$ introduces no signs.  The standard procedure recursively determines all associated signs. Again we use induction.  Both $\gamma$ and $\Phi$ are equivariant, so we only need to deal with basis elements $b$.  Assuming the signs for $\gamma$ and $\Phi$ agree in  lower degrees, we need to see why the Berger-Fresse sign of terms in $\Phi(d_\ell b \otimes  \Delta^m)$ and $\Phi((-1)^{|b|}b \otimes d_0 \Delta^m)$ arising from  monomials $M'$ and $M''$ agree with the Berger-Fresse signs of terms in $\Phi(b \otimes \Delta^m)$ arising from monomials $M$ that correspond to $M'$ and $M''$ by the process of the proof.\\

This turns out to be not too hard.  The process in the proof drops a 0 from $M$ monomials to produce $M'$ monomials, or in the second type  case drops a 0 and changes other 0's to 1's to form $M''$ monomials. The shuffles of blocks that occur in forming the $n$-tensors leave the first block in a monomial alone and just shuffle later blocks that follow the first block.  Later blocks do not change `lengths'. Thus the shuffle signs  $sh(M')$ or $sh(M'')$ and $sh(M)$ agree in all cases.\\

The position signs $pos(M')$ and $pos(M)$ also agree in terms arising from $\gamma(d_\ell b \otimes \Delta^m)$ .  This is because dropping the $\ell^{th}$ singleton 0 from $M$ to form $M'$ does decrease by one the number of caesura blocks when $d_\ell b$ is applied to $M'$, so a sign $(-1)^0 = 1$ is `lost' comparing $pos(M')$ and $pos(M)$.  But the process does not change the final vertices in other caesura blocks of $M$ and $M'$.\\

For example, if $b = (1213213)$ and $M = (0 | 01 | 12 | 2 | 2 | 234 | 45)$ then the caesura blocks of  $M$ are $(0), (01), (12), (2)$ and the caesura blocks of $M' = ( 01 | 12 | 2 | 2 | 234 | 45) $ with respect to $d_1b = (213213)$ are $(01), (12), (2)$.  Examples where the first caesura of $b$ is $b_\ell = \ell$, with $\ell > 1$ behave similarly, since the $M$ of first type begin with $\ell$ singleton 0's, only the last of which is a caesura block.\\

The position signs $pos(M'')$ and $pos(M)$ arising from $\Phi(b \otimes d_0 \Delta^m)$ and monomials $M$ of the second type do not necessarily agree.   There are the same total number of blocks, namely $n+k$,  and number of caesura blocks, namely $k = deg(b)$, when $b \otimes \Delta^m$ is applied to $M$  as when $b \otimes d_0 \Delta^m$ is applied to $M''$, since  blocks $M_j$ and $M''_j$ both correspond to  the entries $b_j$ of $b$.  Monomial $M''$ is obtained from $M$ by dropping the  0 from a block $(01..)$ and changing previous singleton 0's in $M$ to 1's.   But the position sign is obtained by counting how many times entries that are not final entries of blocks  occur in front of  final entries of caesura blocks. Since the dropped 0 occurs in front of all final caesura block entries, the count is reduced by $k$. \\

But also recall that there is a sign $(-1)^k b \otimes d_0 \Delta^m$ in the boundary $d(b \otimes \Delta^m)$, where $deg(b) = k$. Since the contraction $h_{\otimes^n}$ introduces no signs, one sees that indeed the signs in $h_{\otimes^n} \Phi((-1)^{|b|}b \otimes d_0 \Delta^m)$ agree exactly with the signs in the terms in $\Phi(b \otimes \Delta^m)$ arising from monomials $M$ of the second type. This completes our long discussion of the proof of Proposition \ref{17.3}.
\end{proof}
\begin{rem}\label{17.4} {\bf The Map  $\Phi$ is Injective.} Notice that given  an $n$-tensor summand $\pm F_1 \otimes \ldots \otimes F_n \in N_*(\Delta^m)^{\otimes n}$ of $\Phi(x\otimes \Delta^m)$, one can unravel the $n$-tensor to see exactly which $M \in \delta^{(n+k)}(\Delta^m)$ it came from.  Namely, just write in non-decreasing order all the vertices of $\Delta^m$ occurring in all the faces $F_j$.  Then put separating bars between repeated entries.  For example, from Example \ref{17.1}, consider the 3-tensor $(01245) \otimes (01234) \otimes (25)$, with $x = (1213213)$. Unravel to $M = (0 | 01 | 12 | 2 |  234 | 45| 5) $.  Then $F(M,x)$ is the named 3-tensor.\\

In general, assuming the tensor arises from $M$ and a basis generator $b$,  it is  immediate to read off from the single tensor summand  the first caesura of $b$ and whether the monomial $M$ is of first type or second type.  If $m$ is not too small $b$ itself can be recovered from relatively few tensor summands of $\Phi(b \otimes \Delta^m)$.  One can pursue this and prove that the adjoint $Ad(\Phi)\colon \cS_*(n) \to HOM(N_*(\Delta^m), N_*(\Delta^m)^{\otimes n})$ is injective. In fact, in degree $k$ it suffices to take $m = n+k$ and a single $M = (01|12| \ldots | (m-1)m)$.  Then the $M$-coordinate alone of $\Phi(x \otimes \Delta^m)$ is sufficient to see that $Ad(\Phi)$ is  injective on the $\FF$ vector space $\cS_*(n)$. $\qed$
\end {rem} 

\subsection{Compositions $ A_*(n)\otimes N_*(X) \to \cS_*^{bf}(n)\otimes N_*(X) \to N_*(X)^{\otimes n}$}
 We next want to use Proposition \ref{17.3} to find the canonical functorial procedure maps $A_*(n) \otimes N_*(X) \to N_*(X)^{\otimes n}$ for $A_*(n) = \cS_*^{aj}(n), \cS_*^{ms}(n), N_*(E\Sigma_n)$, and $M_*(n), N_*(EC_n)$.  The hard one is $N_*(E\Sigma_n)$, the others will follow easily from two lemmas. The results are not unexpected, but seem to require some arguments.  The issue is that compositions of standard procedure chain maps need not be standard procedure chain maps.  Therefore, when such a composition does turn out to be a standard procedure map, some proof must be given.
\begin{prop}\label{17.5} The functorial standard procedure map $N_*(E\Sigma_n) \otimes N_*(X)$  $\to N_*(X)^{\otimes n}$ is the composition $$\gamma \circ (TR \otimes Id) \colon N_*(E\Sigma_n) \otimes N_*(X) \to \cS_*^{bf}(n) \otimes N_*(X) \to N_*(X)^{\otimes n}.$$
\end{prop}
We postpone the proof.  But the first step will be an application of the following general result.
\begin{lem}\label{17.6} Suppose $\phi \colon A_* \to B_*$ is a standard procedure equivariant chain map, with $A_*$  free over $\FF[G]$,  $h_B^2 = 0$, and $ h_B \circ  \iota_B = 0$.  Then  the map $\phi \otimes Id \colon A_* \otimes N_*(X) \to B_* \otimes N_*(X)$ is a standard procedure equivariant  functorial chain map.
\end{lem}
\begin{proof}  The argument we give here is basically a special case of a functorial version of Proposition 6.11 of Part I, which dealt with tensor products of standard procedure chain maps.  That argument can be reviewed to motivate the following lines. The lemma is also a consequence of the uniqueness result Proposition \ref{8.2} of Section 8 of Part I, but establishing the hypothesis of Proposition \ref{8.2} is not so different from the direct argument here. \\

In total degree 0 there is nothing to prove.  In positive degrees it suffices to show  for basis generators $a \in A_*$ that $$(\phi \otimes Id)(a \otimes \Delta^m) = \phi(a) \otimes \Delta^m  =  h_\otimes [ \phi(da) \otimes \Delta^m + (-1)^{|a|} \phi(a) \otimes d\Delta^m],$$ where $h_\otimes =  h_B \otimes Id_\Delta  + \rho_B \otimes h_\Delta $. Recall $\rho_B = \iota_B \epsilon_B$. If $m = 0$ and $|a| > 0$  the claim is obvious.\\

Suppose $m > 0$. If $|a| = 0$ the assumptions imply $\phi(a) = \iota_B \epsilon_A(a)$, hence $h_B\phi(a) = 0$.  Also  $\rho_B \phi(a) = \phi(a)$  and $h_\Delta(d\Delta^m) = \Delta^m$. \\

 If $|a| > 0$ then  $\phi(a) \otimes \Delta^m =  h_B \phi(da) \otimes \Delta^m$ and $h_\Delta(\Delta^m) = 0$. Also $\phi(a) \otimes d\Delta^m  \in Ker(h_\otimes),$ since $\phi(a) \in Im(h_B) = Ker(h_B)$. Thus the sum of the  four terms on the right side of the desired equation reduces to $\phi(a) \otimes \Delta^m$ in all cases.
\end{proof}
\begin{lem}\label{17.7} Suppose $\phi \colon A_*  \to B_*$ is a standard procedure equivariant chain map, with $A_*$ free over $\FF[G]$ and $B_*$ free over $\Sigma_n \supseteq G$, such that for $G$-basis elements $a \in A_*$ it holds that $\phi(a) = \sum \epsilon_i b_i \in B_*$ for constants $\epsilon_i \in \FF$ and $\Sigma_n$-basis elements $b_i \in B_*$.  If  $\Phi \colon B_* \otimes N_*(X) \to N_*(X)^{\otimes n}$ is a functorial standard procedure map then $\Phi \circ (\phi \otimes Id) \colon A_* \otimes N_*(X) \to B_* \otimes N_*(X) \to N_*(X)^{\otimes n}$ is also a functorial standard procedure map.
\end{lem} 

\begin{proof} This is essentially a functorial version of the composition situation discussed in Proposition 6.5(iii) of Part I.  It suffices to consider generators $a \otimes \Delta^m$. The point is, the images $\phi(a)$ are $\FF$-linear combinations of  basis elements of $B_*$, hence $$\Phi  (\phi(a) \otimes \Delta^m) = h_{\otimes^n} \Phi d_\otimes (\phi(a) \otimes \Delta^m) = h_\otimes \Phi (\phi \otimes Id) d_\otimes (a \otimes \Delta^m).$$  
\end{proof}
\begin{prop}\label{17.8} The following compositions $$\cS_*^{aj}(n) \otimes N_*(X)  \to \cS_*^{bf}(n) \otimes N_*(X)  \to N_*(X)^{\otimes n}$$
 $$\cS_*^{ms}(n) \otimes N_*(X)  \to \cS_*^{bf}(n) \otimes N_*(X)  \to N_*(X)^{\otimes n}$$
$$N_*(EC_n) \otimes N_*(X)  \to N_*(E\Sigma_n) \otimes N_*(X) \to \cS_*^{bf}(n) \otimes N_*(X) \to N_*(X)^ {\otimes n}$$
$$M_*(n) \otimes N_*(X)  \to N_*(E\Sigma_n) \otimes N_*(X) \to \cS_*^{bf}(n) \otimes N_*(X) \to N_*(X)^ {\otimes n}$$
are  standard procedure functorial maps.
\end{prop}

\begin{proof} The first two claims are immediate from Lemma \ref{17.7} and results from Section 15 where we compared surjection complexes.  The third and fourth claims follows from the calculation of $M_*(n) \to N_*(EC_n) \subset N_*(E\Sigma_n)$ in Example 6.13 of Part I, and Lemma \ref{17.7} and Proposition \ref{17.5} above.
\end{proof}
We now return to the proof of Proposition \ref{17.5}.
\begin{proof} This result is somewhat tricky to prove.  The first step is to observe that Lemma \ref{17.6} implies $TR \otimes Id$ is the standard procedure functorial chain map between its domain and range.\\ 

Proposition \ref{17.5} concerns two functorial  equivariant chain maps $N_*(E\Sigma_n) \otimes N_*(X) \to N_*(X)^{\otimes n}$.  Each is determined by a collection of maps $N_*(E\Sigma_n) \otimes \Delta^m \to N_*(\Delta^m)^{\otimes n}$ for fundamental classes of simplices $\Delta^m$. The two maps agree in degree 0. We need the functorial version of our key uniqueness result, Proposition 8.2 of Part I.  We will restate and prove a simple case relevant for us here.  
\begin{prop}\label{17.9}  Suppose $B_*$ is free over $\FF[\Sigma_n]$.  Consider an equivariant functorial chain map $\psi\colon B_* \otimes N_*(X) \to N_*(X)^{\otimes n}$ with the property that the images of basis generators $\psi(b \otimes \Delta^m)$ all belong to $Im(h_{\otimes^n}) = Ker(h_{\otimes^n}) \subset N_*(
\Delta^m)$.  Then $\psi$ is the standard procedure equivariant functorial chain map $\Phi$ that agrees with $\psi$ in degree 0.
\end{prop}
\begin{proof} The proof is an induction.  The claim about the functorial maps follows from computations with the acyclic model $\Delta^m$.  We assume $\psi$ and $\Phi$ agree up to some degree and consider basis elements  in the next degree
$$\Phi(b \otimes \Delta^m) = h_{\otimes^n} \Phi d_\otimes (b \otimes \Delta^m) = h_{\otimes^n} \psi d_\otimes (b \otimes \Delta^m) = $$ $$  h_{\otimes^n}  d_{\otimes ^n} \psi (b \otimes \Delta^m) = \psi(b \otimes \Delta^m) -  d_{\otimes^n} h_{\otimes^n}   \psi (b \otimes \Delta^m) =  \psi (b \otimes \Delta^m).$$
\end{proof}
To complete the proof of Proposition \ref{17.5} we need to show that 
$$\psi = \Phi \circ (TR \otimes Id) \colon N_*(E\Sigma_n) \otimes N_*(X) \to \cS_*^{bf}(n) \otimes N_*(X) \to N_*(X)^{\otimes n}$$
satisfies the hypothesis of Proposition 17.9.  The troublesome issue is that for basis elements $(e, x) \in N_*(E\Sigma_n)$, some summands of $TR(e, x)$ may not be basis elements of $\cS_*^{bf}(n)$.\\

Recall from Section 16 of Part I that the map $TR$ commutes with contractions. Therefore $TR(e, x)$ belongs to the image of the contraction of $\cS_*^{bf}(n)$, which we have characterized as the span of clean surjection generators.  More specifically, the summands of $TR(e, x)$ are parametrized by partitions.  Consider a partition $ a_0 + \ldots + a_k = n+k$ with first summand $a_0  = \ell$.  The corresponding surjection generator $x_a$ then has the form $x_a= (1, 2,  \ldots, \ell, ..., \ell \ldots )$, where $\ell$ is the first caesura entry.  Such $x_a$ are clean  surjection generators. Therefore, $\Phi(TR(e, x) \otimes \Delta^m) \in Im(h_{\otimes^n})$ follows from Step 2 of the proof of Proposition \ref{17.3}.
\end{proof}

{\bf *Extension of the Berger-Fresse Map to Multisimplicial Sets.*}  We conclude this subsection with some remarks about the $k$-fold multisimplicial category that relate our methods to results of Medina-Mardones, Pizzi, and Salvatore in [23].  In Remark \ref{14.6} we discussed multidiagonals for the surjection complexes $\delta \colon \cS_*(n) \to \cS_*(n)^{\otimes k}$.  Combining with the Berger-Fresse maps $\Phi$  of Proposition \ref{17.3}, or its equivalent for other surjection complexes, there are composition equivariant chain maps, functorial on the $k$-fold product category $\bold{\Delta}^k$:
$$\cS_*(n) \otimes [N_*(\Delta^{m_1}) \otimes \ldots \otimes N_*(\Delta^{m_k})] \xrightarrow{\delta \otimes Id} \cS_*(n)^{\otimes k} \otimes [N_*(\Delta^{m_1}) \otimes \ldots \otimes N_*(\Delta^{m_k})]$$  $$\simeq \bigotimes_i(\cS_*(n) \otimes N_*(\Delta^{m_i})) \xrightarrow{\otimes \Phi_i} \bigotimes_i [ N_*(\Delta^{m_i})^{\otimes n}] \simeq [ (N_*(\Delta^{m_1}) \otimes \ldots \otimes N_*(\Delta^{m_k})] ^{\otimes n}.$$
 Hence we obtain functorial maps $\Phi_n \colon \cS_*(n) \otimes C_*(X) \to C_*(X)^{\otimes n}$ for $k$-fold multisimplical sets $X$, where the normalized cellular chain complex $C_*(X)$ for multisimplicial sets is defined in  Subsection 8.5 of Part I.  Generators $b = x \otimes \bigotimes_i \Delta^{m_i}$ map to elements in the image of the target contraction $h$ for prisms.  Thus,  by a variant of Proposition \ref{8.2},  these maps $\Phi_n$ coincide with equivariant standard procedure functorial maps with $\gamma(b) = h \gamma (db)$.\\
 
 In Subsection 20.2 of Part III below we will prove that the adjoints of the Berger-Fresse maps $\cS_*(n) \otimes N_*(X) \to N_*(X)^{\otimes n}$ for {\it simplicial sets} $X$ define an operad morphism $\cS_*(n) \to HOM_{func}(N_*( -), N_*( - )^{\otimes n})$ from the surjection operad to the Eilenberg-Zilber operad.  That argument extends to the adjoints of the maps just described for {\it multisimplicial sets} $X$, that is $\cS_*(n) \to HOM_{func}(C_*( - ), C_*( - )^{\otimes n}).$  The extension does not start from scratch but uses the composition formula above  that relates the multisimplicial set case to the  simplicial set case, and then uses extensions of various uniqueness theorems for standard procedure chain maps.\\
 
 In [23] it is shown that the $C_*(X)$ for multisimplicial sets are coalgebras over a certain $E_\infty$-operad different from the surjection operad.  We do not know how that operad is related to the surjection operad $\cS$ or the Barratt-Eccles operad $\cE$ with components $N_*(E\Sigma_n)$.  We study the two operads $\cS, \cE$ in Sections 19 and 20 below, along with an operad morphism $\cE \to \cS$ defined by the maps $TR \colon N_*(E\Sigma_n) \to \cS_*(n)$ of Subsection 16.1.  In any event, as mentioned in the paragraph above, our methods in Subsection 20.2 do extend to prove the cellular chain complexes $C_*(X)$ for multisimplicial sets are $E_\infty$-coalgebras over the $E_\infty$-operads $\cS$ and $\cE$, parallel to results of [23]. $\qed$
 
\subsection {*A Proof that the Steenrod Operation $ P^0 = Id$*}
 In Part IV we will exploit the $C_p$-equivariant functorial standard procedure map $\gamma \colon M_*(p)  \otimes N_*(X) \to N_*(X)^{\otimes p}$ to define and study the odd prime $p$ Steenrod operations  $P^j$ of degree $2j(p-1)$.  When $p = 2$ the fact that the degree 0 operation $Sq^0 = Id$ is obvious from the definition of the Steenrod operations in terms of the cochain $\cup_i$'s. In the language of Remark \ref{7.2}, there is only one diagram that completely fills a box of size $2 \times (q+1)$ and begins in the upper left corner.  In the odd prime case, somewhat more is required to define the reduced powers $P^j$ and to prove $P^0 = Id$. Here is the result needed to prove $P^0 = Id$, in terms of certain even dimensional generators $y_{2k} \in M_{2k}(p)$ and fundamental classes $\Delta^q \in N_q(\Delta^q)$.
\begin{prop}\label{17.10} Set $m = (p-1)/2$ and  $c_{q,p} = (-1)^{(q(q-1)/2)(p(p-1)/2)}(m!)^q$. Then\footnote{For various reasons the degree 0 cohomology operation $P^0$ on cocycles $\alpha \in N^q(X), q \leq 0$,  is defined as a certain multiple of the dual operation $\gamma(y_{|q|(p-1} \otimes \alpha^{\otimes p}) \in N^q(X)$.  With $\FF_p$ coefficients that multiple is $(-1)^{|q|p(p-1)/2}\  c_{|q|, p}^{-1}$. Then $P^0 = Id$ follows from bullet point 9 of Proposition \ref{3.1} that evaluates $<\alpha^{\otimes p}, (\Delta^{|q|})^{\otimes p} > = (-1)^{|q|p(p-1)/2}<\alpha, \Delta^{|q|}>$.} $$\gamma(y_{q(p-1)} \otimes \Delta^q)  = c_{q,p}(\Delta^q)^{\otimes p} \in N_*(\Delta^q)^{\otimes p}.$$
Modulo $p$, if $q = 2\ell$ is even, $c_{q,p} \equiv (-1)^\ell \in \FF_p$.  If $q = 2\ell +1$ is odd, $c_{q,p} \equiv (-1)^\ell m! \in \FF_p$.
\end{prop}
\begin{proof}

The main step of the proof  is the computation in Remark \ref{17.2} of $\gamma(x \otimes \Delta^q)$ for certain surjection generators $x$ of degree $2k = q(p-1)$, where $\gamma$ is the Berger-Fresse map. We combine here various results proved in Section 6 of Part I, Section 16 of Part II, and this Section 17 of Part III.\\

From Propositions \ref {17.8} and \ref{17.5}, the map $\gamma$ is a composition of three maps
$$M_*(p) \otimes N_*(X) \to N_*(EC_p) \otimes N_*(X) \to N_*(E\Sigma_p) \otimes N_*(X) \to N_*(X)^{\otimes p}.$$
All chain complexes have $\FF_p$ coefficients. The first map is $\phi \otimes Id$, where in Section 6 we computed the standard procedure map $\phi \colon M_*(p) \to N_*(EC_p)$.  In particular, $$\phi(y_{2k}) = \sum (1, T^{b_1-1}, T^{b_1}, \ldots, T^{b_k-1}, T^{b_k }) \in N_*(EC_p),$$ where the sum is taken over all $b_j \in \{ 1, 2, \ldots, p\}$. Of course degenerate summands can be removed. Then we compose with the map $i_* \otimes Id$, where $i_* \colon N_*(EC_p) \to N_*(E\Sigma_p)$ is induced by $T \mapsto t = (23\ldots p 1) \in \Sigma_p$.  The third map is $\Phi \circ (TR \otimes Id)$.
The standard procedure table reduction map $TR\colon N_{2k}(E\Sigma_p) \to S_{2k}^{bf}(p)$ was computed in Section 16 of Part II, and the Berger-Fresse map $\Phi \colon S_{2k}^{bf}(p) \otimes (\Delta^q) \to (N_*(\Delta^q)^{\otimes p})_{2k+p}$ was defined at the beginning of Subsection 17.1, prior to the proof of Proposition \ref{17.3} which asserts that the Berger-Fresse map is the standard procedure map.\\

With $2k = q(p-1)$ this last map $\Phi(x \otimes \Delta^q) = \sum_M pos(x)sh(x)F(M, x)$ was computed in Remark \ref{17.2}.  The only non-zero terms arise from the monomial $$M = (0| \ldots | 0 | 0 1| |1| \ldots |1| |12|  \ldots |(q-1) q| |q| \ldots |q)$$ with  $p$ occurrences each of $0,1, \ldots, q$, and $x \in \cS_*^{bf}(p)$ formed by concatenating $q+1$ permutations $g_i$ of $\{1, \ldots, p\}$, $0 \leq i \leq q$,   with $g_i(p) = g_{i+1}(1)$, and removing one of each adjacent repeated entry.\\

The table reductions $TR(e, t^{b_1-1}, t^{b_1}, \ldots, t^{b_k-1}, t^{b_k })$, where $k = q(p-1)/2 = qm$ and $b_i \in \{1,2, \ldots, p\}$, are sums of many surjection generators  $x_a$.  These generators are parametrized by partitions $a_0 + a_1 + \ldots + a_{q(p-1)} = qp$. The non-caesura entries of the surjection generators $x$ described in the paragraph above are the $p$ entries of the final permutation $g_q$.  Thus any table reduction surjection generator $x_a$ that can be such an $x$ must have $p$ distinct entries in the last row, hence $a_i = 1$ for $i < q(p-1)$ and $g_q = t^{b_k}$\\

Next, for which sequences of $b_j$'s does this particular $TR$ summand $x_a$  have the form of $q+1$ concatenated permutations?  Note $t^b = (b+1, b+2, \ldots, b)$. The first $p$ entries of $x_a$, namely, $(1,b_1, b_1+1, \ldots b_m, b_m+1)$ must be named by one of the $m!$ permutations of $m$ adjacent pairs $(2,3), (4,5), \ldots, ((p-1),p)$. As a permutation  $ g_0 \in \Sigma_p$, this is an even permutation. Then $b_m+1$ followed by the next $p-1$ entries must be named by one of the $m!$ permutations of $m$ adjacent pairs from the cyclically ordered set $(1,2,\ldots, p)$ with $b_m+1$ removed.  Such a permutation $g_1 \in \Sigma_p$ is also an even permutation\\

The process continues, and we see that the allowable $b_j$ sequences are named by $q$ permutations of $m$ pairs, each of which corresponds to an even permutation $g_i \in \Sigma_p$.  There are $(m!)^q$ such collections of permutations.    Since $p$ is odd, from Remark \ref{17.2} the position sign of all the resulting Berger-Fresse terms is $+1$.  Since all parity signs $\tau(g_i) = +1$, again from Remark {17.2} the shuffle signs are all $(-1)^{(q(q-1)/2)(p(p-1)/2)}$.\\

These calculations complete the proof of the first statement in Proposition \ref{17.10}, which we point out is a counting formula with $\ZZ$ coefficients.  Working modulo $p$, by Wilson's Theorem $(p-1)! \equiv -1$, hence $(m!)^2 \equiv (-1)(-1)^m \equiv (-1)^{(m-1)}.$ The modulo $p$ calculations of the $c_{q,p}$ follow easily.
\end{proof}
\begin{exam}\label{17.11} Take $p = 5,\  m = 2,\ q = 2$. Then the monomial is $M =  (0|0|0|0|01|1|1|1|12|2|2|2|2)$. There are $4 = (2!)^2$ relevant $TR$ summands.  The $b$ sequence $(b_1, b_2, b_3, b_4)$ must begin with $(b_1, b_2) = (2,4)\ or\ (4,2)$.  In the first case, $(b_3, b_4) = (1,3)\ or\ (3,1)$. In the second case $(b_3, b_4) = (1,4)\ or\ (4,1)$.\\

The corresponding surjection generators are  $$x_a = (1234512345123), (1234534123451), (1452312451234), (1452345123451).$$ The position signs and shuffle signs are all $+1$.  Thus $\Phi(y_8 \otimes \Delta^2) = 4(\Delta^2)^{\otimes 5}$.  $\qed$
\end{exam}

\newpage
\section{Preliminaries on Operads}
 A {\it symmetric operad} is a collection of objects $P(n)$, $n \geq 1$, in a  symmetric monoidal category,\footnote{Symmetric monoidal categories were discussed in Section 3 of Part I. 
 For us, the only examples will be operads in categories of chain complexes.} together with structure maps $$\cO_P \colon P(r) \otimes P(s_1) \otimes \ldots \otimes P(s_r) \to P(s_1+\ldots + s_r)$$ that satisfy a number of conditions.  There should be a {\it left} group action of $\Sigma_r$ on each $P(r)$.  In all cases of interest to us,  the component $P(1)$ is a unit object for the product in the category.   The maps $\cO_P$ should satisfy a {\it unit axiom}, an {\it associativity axiom}, and an {\it equivariance axiom}.  The unit axiom is usually pretty obvious and deals with cases where $r = 1$ or all of the $s_i = 1.$  In both cases corresponding $\cO_P$ maps are identities.\\

In the sections following this one we introduce a modified version of our recursive standard procedure for constructing equivariant chain maps,   $$\cO_B \colon  B_*(r) \otimes B_*(s_1) \otimes \ldots \otimes B_*(s_r) \to B_*(s)\ \ \ \ \ \ s = s_1+\ldots+s_r, $$  using  preferred bases of the domains over products of symmetric groups and  preferred contractions of the ranges.  We assume the collection of chain complexes $B_*(n)$  are free over $\Sigma_n$, satisfy $B_0(n) = \FF[\Sigma_n]$, and have bases in the image of  contractions satisfying  $H_n^2 = 0$. From Proposition 6.3 and Proposition 6.4 of Part I, such complexes $B_*(n)$ are canonically direct summands of  $N_*(E\Sigma_n)$.  The point of our new recursive construction is that these maps $\cO_B$, as the $s_i$ vary, can be fit together to form maps  that `almost' form an operad.  \\

There is a difference between the new construction  and our previous examples of  equivariant recursive standard procedure constructions.  Namely,  the product of symmetric groups $\widehat{G}$ that acts freely on the domain sides embeds as a subset, but not a subgroup, of $\Sigma_s$.  We will recursively define the structure map $\cO_B$ on basis elements $b$ of the domain in the usual way by $\cO_B(b) = H_s \cO_B (db)$, where $H_s$ is the contraction of the range. But then in order to satisfy the operad equivariance axiom for the structure maps, to be discussed below,   we need to extend $\cO_B$ by a  twisted form of equivariance on product elements $\hat{g}b$.  We then need to prove our map is a chain map and satisfies the operad equivariance axiom for arbitrary product elements $\hat{g}x$.\\

The operad associativity axiom seems to require additional assumptions about the $B_*(n)$. The troublesome issue  for associativity is that compositions of standard procedure chain maps need not be standard procedure chain maps, as discussed back in Proposition 6.5 of Part I.   But the  additional assumptions easily hold for the $N_*(E\Sigma_n)$ and less easily for the surjection complexes $\cS_*(n)$. Therefore we obtain by a new method operad structure maps for the complexes $N_*(E\Sigma_n)$ and for the $\cS_*(n)$\\

Using uniqueness theorems extending those of Sections 6 and 8 of Part I to the twisted equivariant context, we prove our operad structures agree with previous constructions of Barratt-Eccles and surjection operads.  By this method  we eliminate quite a lot of what the original authors referred to as `tedious computations', which were sometimes not even included.  But there is no free lunch. We try to be careful with details, of which there are  many.\\

But first we need to record  our conventions and some preliminary facts about operads in general.  There is nothing new in this section, which we have included  for completeness.  In the following sections we do present some completely different operadic constructions that seem rather surprising, motivated, and attractive.

\subsection{Operad Strucure Maps and Associativity}
We  need to set out some conventions concerning the axioms for  operads.  We also need to develop carefully what is known as the {\it symmetric group operad} in the category of sets.  In Subsection 18.3 we develop carefully the $End$ and $CoEnd$ operads in the category of chain complexes and the functorial  Eilenberg-Zilber  operad for simplicial sets.
\begin{rem}\label{18.1}{\bf Associativity.}  The associativity axiom states that a certain big diagram commutes.   That diagram is built given sums $s = \sum_{i=1}^r s_i$ and $s_i = \sum_{j=1}^{r_i} s_{ij}$. The equal sign on the left side of the diagram below refers to the canonical isomorphism that permutes factors in the symmetric monoidal category.  For example, in categories of chain complexes, there will be some Koszul signs.

$$\begin{CD} \displaystyle P(r) \otimes \bigotimes_{i=1}^r \bigl[P(r_i) \bigotimes_{j=1}^{r_i}P(s_{ij}) \bigr] & \xrightarrow{Id\otimes \bigotimes\cO_P} & \displaystyle P(r) \otimes \bigotimes_{i=1}^r P(s_i)  \xrightarrow{\cO_P} & P(s)\\
\vert \vert \ \sigma&  \boxed{18.1} & & \vert \vert\\
\displaystyle \bigl[P(r) \otimes \bigotimes_{i=1}^r P(r_i)\bigr] \otimes \bigotimes_{j=1}^{r_i}P(s_{ij}) & \xrightarrow{\cO_P \otimes Id} & \displaystyle P(\Sigma r_i) \otimes \bigotimes_{j=1}^{r_i} P(s_{ij})  \xrightarrow {\cO_P}  & P(s)
\end{CD}$$
\end{rem}
The equivariance axiom is somewhat more complicated and important and requires some preliminary constructions to even state.\\

Before that, we acknowledge that we will only be considering a few special  operads and we ignore all issues involving a $P(0)$. We also ignore the unit axiom concerning $P(1)$, which is trivial in all our examples. Perhaps more relevant is the fact that most treatments work with right actions of the symmetric groups.  But one of our main examples is (any of) the surjection operads $\cS$, with components $\cS_*(n)$, in a symmetric monoidal category of chain complexes.  The symmetric group $\Sigma_n$ acts most naturally on the left of surjection generators $x \colon \{1,\ldots, n+k\} \to \{1, \ldots, n\}$ by post-composition, sometimes with signs depending on which surjection complex is meant.  The $\Sigma_n$ actions on the various surjection complexes were studied in Part II. Another crucial example for our purposes is the Barratt-Eccles operad $\cE$, with components $N_*(E\Sigma_n)$. Here the symmetric group could equally well act on the left or the right, but because of the importance of our comparisons $\cS_*(n) \xrightarrow{PR} N_*(E\Sigma_n) \xrightarrow{TR} \cS_*(n)$ we want the left action.\\

Another key example for us is the Eilenberg-Zilber operad  of natural transformations of functors $\cZ(n) = HOM_{func}(N_*( - )), N_*( - )^{\otimes p})$ or a cochain version $HOM_{func}(N^*( - )^{\otimes p}, N^*( - )$, of normalized chain or cochain complexes of simplicial sets.   We view a cochain complex $N^*(X)$ as an {\it algebra} over an $E_\infty$ operad $\cP(n)$, such as $ N_*(E\Sigma_n)$ or $ \cS_*(n)$, via  an operad morphism $\cP(n) \to End(n) = HOM(N^*(X)^{\otimes n}, N^*(X))$.  The images of these maps lie in the functorial subcomplexes of the $HOM$ complexes. But we regard the cochain algebra structure to be an {\it emergent} phenomenon. The underlying phenomenon is the structure of a chain complex $N_*(X)$ as a {\it coalgebra}  via  $\cP(n) \to CoEnd(n) = HOM(N_*(X), N_*(X)^{\otimes n})$.  One obtains the $End(n)$ algebra structure  as an image of the $CoEnd(n)$ coalgebra structure, by composing $\cP(n) \to CoEnd(n)$, which was the topic of Section 17,  with  duality chain maps $CoEnd(n) \to End(n)$ from Proposition 3.1  of Part I.   Diagonal maps of chains precede cup products of cochains.

\subsection {The Symmetric Group Operad and Equivariance}
Before  developing the equivariance axiom for symmetric operads we will first review an operad $\Sigma$ in the category of sets called the {\it symmetric group operad}.  We have $\Sigma(r) = \Sigma_r$, the permutation group as a set.  The operad structure map, which is not a group homomorphism, $$\cO_\Sigma \colon \Sigma_r \times \Sigma_{s_1} \times \ldots \times \Sigma{s_r} \to \Sigma_s, \ \ \ s = s_1+ \ldots + s_r,$$ is defined as follows. Given $u \in \Sigma_r$ and  $v_i \in \Sigma_{s_i}$ we define a modified kind of block permutation $u_*(v_1, \ldots, v_r)  = \cO_\Sigma(u; v_i) \in \Sigma_s$.   We partition the integral interval $[1, s]$ into blocks $$B_1 = [1, s_1], B_2 = [s_1+1 , s_1 + s_2], \ldots, B_r = [s_1 + \ldots + s_{r-1} +1 , s_1 + \ldots + s_r].$$ We regard each $v_i$ as a permutation of the $i^{th}$ block $B_i$ by `identifying' that ordered block in the obvious way with $[1, \ldots, s_i]$. Then $(v_1, \ldots, v_r)$ acts on $[1,s]$ with each $v_i$ acting   on  $B_i$.  We then permute the rearranged blocks $v_iB_i$ according to how $u \in \Sigma_r$ acts on the left of the ordered set of $r$ blocks.\\

We could equally well first permute the original blocks by $u$ to $[B_{u(1)}, \ldots, B_{u(r)}]$, then apply the $v_i$ to the new locations of the $B_i$.  We can write the resulting permutation as a concatenation  of blocks $$\cO_\Sigma(u; v_i) = [v_{u(1)}B_{u(1)}, \ldots, v_{u(r)}B_{u(r)}] \in \Sigma_s,$$ which expresses a permutation in $\Sigma_s$ in function form as a list of the integers in the interval $[1, s]$ in some order.\\

Note that if all $s_i = 1$ then $\cO_\Sigma(u; v_i) = u$. We abbreviate $ u_*(Id_{s_1}, \ldots, Id_{s_r}) = u_*(s_1, \ldots, s_r)$, when the $v_i$ are all identity permutations.  \\

It helps to pretend the $s_i$ are distinct.  It is not quite enough to write $v_i \in \Sigma_{s_i}$, if $s_i$ occurs more than once in the ordered set $\{s_j\}$, it is necessary to keep track of which block of size $s_i$ is associated with which permutation $v_i$.  $\qed$

\begin{lem}\label{18.2} We have a composition identity in the permutation group $\Sigma_s = \Sigma_{s_1 + \ldots + s_r}$. $$\cO_\Sigma(u; v_i)  = (v_1 \oplus \ldots \oplus v_r) \circ u_*(s_1, \ldots, s_r) $$  $$ = u_*(s_1, \ldots s_r) \circ  (v_{u1} \oplus \ldots \oplus v_{ur}).$$
\end{lem}
\begin{proof} The direct sum notation for permutations means a permutation of an ordered union of disjoint sets.  So in the first formula, $v_1$ acts on the initial block $B_1$ of size $s_1$, wherever it occurs in the output of $u_*(s_1, \ldots, s_r)$, $v_2$ acts on the second initial block $B_2$ of size $s_2$, and so on.  This is  the second description of the permutation operad action above.  Although strictly speaking $v_i$ is in $\Sigma_{s_i}$, we routinely also denote by $v_i$ the `same' permutation of appropriate other blocks of $s_i$ consecutive integers.\\

Then in the second formula,  $v_{u1}$ acts on an {\it initial} block of size $s_{u1}$, that is, $[1, \ldots, s_{u1}]$, $v_{u2}$ acts on the following block of size $s_{u2}$, and so on.  This is followed by the permutation $u_*(s_1, \ldots, s_r)$.  So that is a little tricky.\\ 

In working with permutations that move blocks of integers in $[1,s]$ it is best to line up blocks of the same size in describing the permutations. In our situation, denote by $B_1, B_2, \ldots, B_r$ the consecutive blocks of sizes $s_1, s_2, \ldots, s_r$.  Denote by $B_1', B_2', \ldots, B_r'$ the consecutive blocks of sizes $s_{u1}, s_{u2}, \ldots, s_{ur}$.  So $|B_j'| = |B_{uj}| = s_{uj}$. Note $[1, s] = [B_1B_2 \ldots B_r] = [B_1' B_2' \ldots B_r']$. Then the composed permutation $(\oplus v_i) \circ u_*(s_1, \ldots, s_r) \in \Sigma_s$ can be viewed in function form by reading the display of blocks below from the top row to the middle row to the bottom row.
$$\begin{CD}
[B_1'\ \ \ B_2'\ \ \ \ldots\ \ \ B_r']\\
[B_{u1}\ \ B_{u2}\ \ \ldots\ \ B_{ur}] \\
[v_{u1}B_{u1}\ v_{u2}B_{u2}\ \ldots\ v_{ur}B_{ur}] 
\end{CD}$$
We point out that if $u_*$ moves a block $B'$ to a block $B$ of the same size then $u_*$ will move $B'$ scrambled by a permutation to $B$ scrambled by the `same' permutation. We can therefore display the composed permutation $u_*(s_1, s_2, \ldots, s_r) \circ (\oplus v_{ui})$ by the diagram below, which proves the lemma.
$$\begin{CD}
[B_1'\ \ \ B_2'\ \ \ \ldots\ \ \ B_r']\\
[v_{u1}B_1'\ v_{u2}B_2'\ \ldots\ v_{ur}B_r'] \\
[v_{u1}B_{u1}\ v_{u2}B_{u2}\ \ldots\ v_{ur}B_{ur}] 
\end{CD}$$
\end{proof}
{\bf Example.} Let $u = (231), v_1 = (21), v_2 = (3124), v_3 = (321)$.  Then we first form a permutation in $\Sigma_9$ by applying the $v_i$ to the initial blocks of size 2,4,3 in that order.  This yields the permutation $(21\ 5346\ 987) \in \Sigma_9$. (Recall we name a permutation by writing  $(12 \ldots s)$ in some order.)  Then we permute the new first, second, and third blocks by $u$, yielding the permutation $$u_*(v_1, v_2, v_3) = u(v_1B_1, v_2B_2, v_3B_3) = (v_2B_2, v_3B_3, v_1B_1) = (5346\ 987\ 21) \in \Sigma_9.$$

We have $u_*(2,4,3) = (3456\ 789\ 12)$.   The first and second formulas in Lemma \ref{18.2} give respectively $$  (21\ 5346\ 987) \circ (3456\ 789\ 12) = (5346\ 987\ 21)  =   (3456\ 789\ 12)  \circ(3124\ 765\ 98). \qed$$ 

\begin{rem}\label{18.3}{\bf*Connection Between $\cO_{\Sigma}$ and Semidirect Products.*} Given a group $G$ and subgroups $H, K \subset G$ where $K$ normalizes $H$ and $K \cap H = \{ e \}$, there is a semidirect product subgroup $HK = KH \subset G$.  Elements can be written uniquely either as products $hk$ or $k'h'$.  Thus the subgroup can be identified setwise with either $H \times K$ or $K \times H$, with  appropriate twisted products.\\

Suppose in the discussion of $\cO_{\Sigma}$ the $s_i$ coincide, say $s_i = t, 1 \leq i \leq r$.  Then there is a semidirect product subgroup $ HK = (\Sigma_t \times \cdots \times \Sigma_t)  (\Sigma_r )   \subset  \Sigma_{rt}$, involving $r$ successive blocks of length $t$, whose elements are identified with elements of  $(\Sigma_t \times \cdots \times \Sigma_t) \times (\Sigma_r)$, with a twisted product.  We write elements of the product of $\Sigma_t$'s subgroup as $(\oplus v_i) \in (\Sigma_t \times \cdots \times \Sigma_t) \subset \Sigma_{rt}$. These act with the $v_i$ permuting elements of the successive blocks $B_i$ of length $t$ in $[1, \ldots rt]$. The $\Sigma_r$ subgroup permutes those blocks, $u(B_1, \ldots, B_r) = (B_{u1}, \ldots, B_{ur})$, which is the permutation $u_* = u_*(t, \ldots, t)$. The subgroup $\Sigma_r$ normalizes the product subgroup by $u_* \circ (\oplus v_i) \circ u_*^{-1} = (\oplus v_{u^{-1} i}) \in \Sigma_{rt}$, which is the standard formula for a left action of a permutation group on a product set.   The twisted product on $(\Sigma_t \times \cdots \times \Sigma_t) \times  (\Sigma_r )$ in $\Sigma_{rt}$ is then
$$((\oplus v_i), u_*) \circ  ((\oplus v_i'), u_*') =  ((\oplus v_i) (u_* (\oplus v_i') u_*^{-1}), u_* u_*') = ((\oplus v_i v'_{u^{-1} i}), u_* u_*').$$

By Lemma \ref{18.2} the tuple $(u; v_1, \ldots, v_r)$ maps by $\cO_{\Sigma}$ to the permutation $ (\oplus v_i) \circ u_* = u_* \circ (\oplus v_{ui})$. Because the left term is simpler, we prefer to identify the semidirect product setwise with $(\Sigma_t \times \cdots \times \Sigma_t) \times (\Sigma_r)$.  The equivariance axiom for $\Sigma$ to be studied next then identifies $\cO_{\Sigma}$ with the subgroup inclusion of the semidirect product. To be precise, the equivariance axiom will say  $$  \cO_{\Sigma}(u; v_i) \cO_{\Sigma}(u'; v'_i)  = \cO_{\Sigma} (u u'; v_i v'_{u^{-1}i}).$$ But the permutation $\oplus ( v_i v'_{u^{-1}i}) = (\oplus v_i) (u_* (\oplus v_i') u_*^{-1})$.
In particular, the operad structure map $\cO_\Sigma$ in this case is identified with the subgroup inclusion $HK \subset \Sigma_{rt}$, if the domain $K \times H$ of $\cO_\Sigma$ is (benignly) identified with $H \times K$. This observation is of some importance for establishing the Cartan formula and the Adem relations for Steenrod operations.$\qed$
\end{rem}

\begin{rem}\label{18.4} {\bf Equivariance.} We will next write down the general equivariance axiom for operads, as it is often stated. One can give two separate formulas or combine them into a single formula. For $g \in \Sigma_r$, $u \in P(r)$, and $h_i \in \Sigma_{s_i}$, $v_i \in P(s_i)$, $s = s_1 + \ldots + s_r$, the axiom states $$\cO_P(g  u; v_1, \ldots, v_r) = g_*(s_1, \ldots, s_r)  \cO_P(u; v_{g1}, \ldots, v_{gr}) \in P(s)$$
$$\cO_P(u; h_1v_1, \ldots, h_rv_r) = (h_1 \oplus \ldots \oplus h_r)  \cO_P(u; v_1, \ldots, v_r) \in P(s)$$
$$  \cO_P(gu; h_1v_1, \ldots, h_rv_r) = \cO_\Sigma(g; h_1, \ldots, h_r) \cO_P(u; v_{g1}, \ldots, v_{gr}) \in P(s).$$

But there are some issues with what these statements mean, and with the consistency of these properties, that many articles about operads seem to ignore.  The first formula really means that a certain diagram commutes.
$$\begin{CD} \displaystyle P(r)  \otimes  \bigotimes_{i=1}^{r}P(s_{gi})   & \xrightarrow{g\otimes \tau(g)} & \displaystyle P(r) \otimes  \bigotimes_{i=1}^{r}P(s_i) \\
\downarrow \cO_P& \boxed{18.4}  & \downarrow \cO_P \\
\displaystyle  P(s)& \xrightarrow{g_*(s_1, \ldots, s_r) } & P(s)
\end{CD}$$

The diagram includes the canonical symmetric monoidal category permutation isomorphism $\tau(g) \colon \otimes_i P(s_{gi})  \simeq \otimes_i  P(s_i) $, where we think of $\Sigma_r$ acting on the {\it left} of the direct sum (coproduct) of all the $\bigotimes_{j=1}^{r}P(s_j)$, the sum  taken over all ordered partitions  $s = s_1 + \ldots + s_r$.  In particular,  for the left group action on a tensor product of chain complexes,  $\tau(g)$ acts as $g^{-1}$ on  subscripts that name the {\it ordered position}\footnote{Thus, in the top row of $\boxed{18.4}$, the point is $\tau(h)(\otimes_i v_{gi}) = \pm \otimes_i v_{gh^{-1}i}$, since $i$ names the ordered position. See Subsection 3.3 for a discussion of left group actions on tensor products.}  of a tensor term factor in a multi-tensor, along with Koszul signs in order that the permutation isomorphism defines a chain map.  But at least in chain complexes a tensor product of elements $u \otimes \bigotimes v_i$ makes sense.  In the categories Sets and Vector Spaces, the symmetric monoidal product is even more elementary. But who knows what $\bigotimes P(s_i)$ and $\tau(g)$ mean in abstract symmetric  monoidal categories.\\

The second formula in the equivariant axiom is somewhat more straightforward to express as a commutative diagram.
$$\begin{CD} \displaystyle P(r)  \otimes  \bigotimes_{i=1}^{r}P(s_i)   & \xrightarrow{Id \otimes \bigotimes_{i=1}^r h_i} & \displaystyle P(r) \otimes  \bigotimes_{i=1}^{r}P(s_i) \\
\downarrow \cO_P& \boxed{18.4'}  & \downarrow \cO_P \\
\displaystyle  P(s)& \xrightarrow{h_1 \oplus \ldots \oplus h_r } & P(s)
\end{CD}$$

Note that the permutation  identity in Lemma \ref{18.2} implies we get the same answer computing $\cO_P(gu; h_1v_1, \ldots,h_rv_r)$ by using the first two equivariance rules in either order.  This then explains the equivalence of the first two formulas with the third.\\

There is no problem simply composing $h$'s in each $\Sigma_{s_i}$ in the second formula in Remark \ref{18.4}.  That is, consistency of the second formula for $(h_ih'_i)v_i$ and $h_i(h'_i v_i)$ is immediate. But there is an issue composing $g$'s in $\Sigma_r$ in the first formula.  We will show below that the first formula is always consistent with the left group action on $P(r)$.  That is,  $\cO_P((hg) u; v_1, \ldots, v_r) = \cO(h (gu); v_1, \ldots, v_r)$ must hold. But simple iterated computation of the second term yields a different looking formula.  At the same time, we will show the general operad equivariance axiom does hold for the permutation operad $\Sigma$.   We found all of this less obvious than we expected. The key for both issues is the following.\footnote{In an earlier version of this paper we thought Lemma 18.5 was much harder than it really is.}

\begin{lem}\label{18.5} For $ h, g \in \Sigma_r$ and positive  $s_1, \ldots, s_r$ with $s = s_1 + \ldots + s_r$, it holds that $(hg)_*(s_1, \ldots, s_r) = h_*(s_1, \ldots, s_r) \circ g_*(s_{h1}, \ldots, s_{hr}) \in \Sigma_s.$
\end{lem}
\end{rem}
Before proving the lemma, we address the two points raised just above.  First we will use  Lemma \ref{18.5}  to establish the first operad equivariance axiom for $\Sigma$. We have from Lemma \ref{18.5} and Lemma \ref{18.2}, $$\cO_\Sigma(gu, v_1, \ldots, v_r) = (v_1 \oplus \ldots \oplus v_r) (gu)_*(s_1, \ldots, s_r)$$  $$ = (v_1 \oplus \ldots \oplus v_r)g_*(s_1, \ldots, s_r)u_*(s_{g1}, \ldots, s_{gr}) $$ $$  =  g_*(s_1, \ldots, s_r) (v_{g1} \oplus \ldots \oplus v_{gr}) u_*(s_{g1}, \ldots, s_{gr}) $$ $$ = g_*(s_1, \ldots, s_r) \cO_\Sigma (u; v_{g1}, \ldots, v_{gr}).$$
Here, $g, u \in \Sigma_r$ and $v_j \in \Sigma_{s_j}$.  The second equivariance axiom for $\Sigma$ is trivial.\\

Next, for the consistency of the first operad equivariance formula in Remark \ref{18.4}, we use  the  formula three times and then use Lemma \ref{18.5}.  We also use $\tau(g^{-1})(\otimes v_{hi}) = \otimes v_{hgi}$.   Suppressing the categorical $\otimes$ symbols in the subtleties of Diagram $\boxed{18.4}$, we write
$$\cO_P(h (gu); v_1, \ldots, v_r) = h_*(s_1, \ldots, s_r) \cO_P(gu; v_{h1}, \ldots, v_{hr}) $$  $$ = h_*(s_1, \ldots , s_r) g_*(s_{h1}, \ldots, s_{hr}) \cO_P ( u; v_{hg1}, \ldots, v_{hgr}).$$ 
$$ \cO_P((hg)u; v_1, \ldots, v_r) = (hg)_*(s_1, \ldots, s_r) \cO_P(u; v_{hg1}, \ldots, v_{hgr}).$$
Instead of referring to elements $u, v_i$, one can nicely prove the consistency of the equivariance axiom for general $P$ by concatenating two Diagrams $\boxed{18.4}$.\\

 {\bf Proof of Lemma \ref{18.5}.}  We use notation from the proof of Lemma \ref{18.2}. We denote successive initial blocks of size $s_i$ by $B_i$,  successive initial blocks of size $s_{hi}$ by $B_i'$, and successive initial blocks of size $s_{hgi}$ by $B_i''$.  Then we claim  the composed permutation  $h_*(s_1, \ldots, s_r) \circ g_*(s_{h1}, \ldots, s_{hr})$ is given by reading down from the top row to the middle row to the bottom row in the following diagram.  
$$\begin{CD}
[B_1''\ \ \ B_2''\ \ \ \ldots\ \ \ B_r']\\
[B_{g1}'\ \ B_{g2}'\ \ \ldots\ \ B_{gr}'] \\
[B_{hg1}\ B_{hg2}\ \ldots\ B_{hgr}] 
\end{CD}$$ 
Each column consists of blocks of the same size. But this diagram from top to bottom obviously defines $(hg)_*(s_1, s_2, \ldots, s_r)$.\\

We  explain why  the middle and bottom rows of the first diagram define $h_*(s_1, s_2, \ldots, s_r) $. The first view of $h_*(s_1, s_2, \ldots, s_r) $ would be the diagram
$$\begin{CD}
[B_1'\ \ \ B_2'\ \ \ \ldots\ \ \ B_r']\\
[B_{h1}\ \ B_{h2}\ \ \ldots\ \ B_{hr}] \\
 \end{CD}$$
 But blocks of the same size are lined up here, and the middle and bottom rows of the first diagram are exactly the same block pairs as in the second diagram permuted by $g$.  So the permutations of the integers in $[1, s]$ are the same, hence the middle and bottom row of the first diagram also define $h_*(s_1, s_2, \ldots, s_r)$.  $\qed$
  
\begin{rem}\label{18.6}{\bf Associativity Axiom for $\Sigma.$} To conclude this subsection we will establish the associativity axiom for $\Sigma$.  From Diagram $\boxed{18.1}$, the axiom amounts to a somewhat formally complicated identity in a permutation group $\Sigma_s$.  It seems best to just think about the block permutation form of the permutations arising from the top and bottom rows of Diagram $\boxed{18.1}$.  Divide the interval $[1,s]$ first into $r$ blocks $B_i$ of length $s_i$.  Then subdivide each block $B_i$  into $r_i$ subblocks $B_{ij}$ of length $s_{ij}$.\\

On the top row of Diagram $\boxed{18.1}$, first simultaneously permute the elements of block $B_i$ by applying separate permutations in  $\Sigma_{s_{ij}}$ to the subblocks $B_{ij}$, then permute those subblocks  by a permutation in $\Sigma_{r_i}$.  Finally permute the resulting rearranged blocks $B_i$ by a permutation in $\Sigma_r$.\\

Going down and then across the bottom row, first rearrange the subblocks $B_{ij}$ of each $B_i$ by permutations in $\Sigma_{r_i}$.  Then permute the rearranged blocks $B_i$ by a permutation in $\Sigma_r$.  Finally apply permutations in $\Sigma_{s_{ij}}$ to each block $B_{ij}$ in its new location.  The two described permutations in $\Sigma_s$ are seen to be the same. $\qed$
\end{rem}
\subsection{The End and CoEnd Operads}
 In this subsection  we will establish the operad properties of the $End(A_*)$ and $CoEnd(A_*)$ chain complex operads of multlinear maps defined for chain complexes $A_*$ by
 $$ End(A_*)(n) = HOM(A_*^{\otimes n}, A_*) $$ $$CoEnd(A_*)(n) = HOM(A_*, A_*^{\otimes n}).$$
 For chain complexes and cochain complexes of simplicial sets these operads  are central ingredients in the approach to cochain operations underlying Steenrod cohomology operations that we take up in Part IV.  It is actually the Eilenberg-Zilber operad $\cZ$ with components the chain complex of functorial graded linear maps $\cZ(n) = HOM_{func}(N_*( - ), N_*( - )^{\otimes n})$ for simplicial sets that is more important. We will suppress the $HOM_{func}$ notation in our discussion, but every step in the basic discussion of $CoEnd$ and $End$ operads for simplicial sets is natural, so they make sense if we interpret $ HOM(N_*(X), N_*(X)^{\otimes n})$ and $HOM(N^*(X)^{\otimes n}, N^*(X))$ directly for fixed $X$ or as ``natural transformations of functors". Of course the reason that the functorial operads are tractible is because of the functorial representability of bases of $N_q(X)$ by contractible models $\Delta^q \to X$.\\

We remind that cochain complexes  $A^* = HOM(A_*, \FF)$ are  chain complexes, negatively graded if $A_*$ is positively graded.  So in discussing $End$ and $CoEnd$ operads in general, there is no need to distinguish chain and cochain complexes.  The notation seems awkward if asterisks are used to denote both grading indices and the dual complex construction.  So in the general discussion to follow we will simply denote chain complexes by a single letter $A$ and denote the dual cochain complex by $A^*$.\\

We recall from bullet point 9 of Proposition \ref{3.1} in Section 3 of Part I that there are `duality chain maps' $Hom(A, B) \to Hom(B^*, A^*)$ and $(A^*)^{\otimes n} \to (A^{\otimes n})^*$,  and then a composition
$$ HOM(A, A^{\otimes n}) \to  HOM((A^{\otimes n})^*, A^*)  \to HOM((A^*)^{\otimes n}, A^*).$$
Koszul type signs are needed in the definition of these maps in order that they are chain maps.\\

This last composition is thus a chain map $CoEnd(A)(n) \to End(A^*)(n)$, which will be seen to define an operad morphism.  When the $A$ are restricted to  $N_*(X)$,  this morphism passes to the natural transformations of functors $HOM_{func}$, that is, to the Eilenberg-Zilber operads.\\
 
\begin{rem} \label{18.7}{\bf The End and CoEnd Operad Structure Maps.} Operad structure maps  $$\cO^* \colon End(r) \otimes \bigotimes_i End(s_i) \to End(s),\ s = \sum s_i$$ 
 $$\cO_* \colon CoEnd(r) \otimes \bigotimes_i CoEnd(s_i) \to CoEnd(s),\ s = \sum s_i,$$
  will  be compositions of functions.  Recall  from bullet point 5 of Proposition 3.1  of Part I that we have chain maps $ \otimes v_i \mapsto \underline \otimes v_i,$  $$\bigotimes_{1 \leq i \leq r} HOM(A^{\otimes s_i}, A) \xrightarrow{\underline{\otimes}} HOM(A^{\otimes s}, A^ {\otimes r}),$$
 $$\bigotimes_{1 \leq i \leq r} HOM(A, A^{\otimes s_i}) \xrightarrow{\underline{\otimes}} HOM(A^{\otimes r}, A^ {\otimes s}),  $$
  where signs are needed when evaluating the linear maps $\underline \otimes v_i$ on  tensors.\\

From bullet point 4 of Proposition 3.1, composition of functions is a chain map, but one must be careful with the order of the composition.  The correct chain map order, with $d(\alpha \circ \beta) = d\alpha \circ \beta + (-1)^{|\alpha|} \alpha \circ d\beta$,   is
$$HOM(C_*, D_*) \otimes HOM(B_*, C_*) \to HOM(B_*, D_*),\ \alpha \otimes \beta \mapsto \alpha \circ \beta.$$
 
 \begin{defn}\label{18.8}Given $u \colon A^{\otimes r} \to A$ and $v_i  \colon A^{\otimes s_i} \to  A$,  the operad structure map value  for the $End$ operad, $\cO^*(u\otimes \bigotimes_i  v_i) \colon A^{\otimes s} \to A,$ with $ s = \sum s_i$, is defined  to be the composition,
 $$ \cO^* = u \circ \underline \otimes v_i  \colon End(r) \otimes \bigotimes End (s_i) \to End(s).$$ 

Given $u \colon A \to A^{\otimes r}$ and $v_i  \colon A \to  A^{\otimes s_i}$,  the operad structure map value  for the $CoEnd$ operad, $\cO_*(u\otimes \bigotimes_i  v_i) \colon A \to A^{\otimes s}$ with $ s = \sum s_i$, is defined  to be the composition $ \cO_* =  ( \underline \otimes v_i \circ u) \circ \tau$, where $\tau(u \otimes  \bigotimes_i v_i ) = (-1)^{|u| | \otimes_i v_i |}  \bigotimes_i v_i \otimes u$, with the  Koszul sign, is the permutation of tensor factors. Thus $\cO_* = \cO_*'  \tau$,
$$ CoEnd(r) \otimes \bigotimes_i CoEnd(s_i) \xrightarrow{\tau} \bigotimes_i CoEnd(s_i) \otimes CoEnd(r) \xrightarrow{\cO_*'} CoEnd(s),$$
where $\cO_*'(\bigotimes_i v_i \otimes u) = \underline \otimes v_i \circ u$.
\end{defn}
Other authors have noticed that it would be sometimes advantageous to change the order of tensor factors in the basic definition of operad structure maps $\cO_P$, as it was in Remark \ref{18.3} involving semidirect products and the symmetric group operad. But it seemed like this would cause us more trouble than it was worth.  For example, $CoEnd$ gets better, but $End$ gets worse. You can also mess around with right vs left group actions, or composing functions in the other order.  But it is like a game of whack-a-mole. $\qed$
\end{rem}

{\bf Proof of Associativity of the End(A) and CoEnd(A) Operads.} To prove the associativity axiom for the $End(A)$ operad, look at  Diagram $\boxed{18.1}$ with $P(n) = HOM(A^{\otimes n}, A)$.  We take $u \in HOM(A^{\otimes r}, A)$, $u_i \in HOM(A^{\otimes r_i}, A)$, and $v_{ij} \in HOM(A^{\otimes s_{ij}} , A)$.  The equality of the two ways around the diagram essentially amounts to associativity of a composition of three chain maps $\alpha \circ \beta \circ \gamma $.  Here, $\alpha = u,\  \beta = \underline{\otimes}_i u_i$, and   $ \gamma = \underline{\otimes}_i \underline {\otimes}_j v_{ij}.$\\

The composition $(\alpha \circ \beta) \circ \gamma$ goes across the bottom of Diagram $\boxed{18.1}$.  Across the top of the diagram one sees the composition $u \circ \underline{\otimes}_i (u_i \circ \underline{\otimes}_j v_{ij})$.  From bullet point 6 of Proposition \ref{3.1} of Part I, the tensor $ \underline{\otimes}_i (u_i \circ \underline{\otimes}_j v_{ij})= \pm \underline{\otimes}_i u_i \circ \underline{\otimes}_i \underline {\otimes}_j v_{ij}.$ Then across the top one sees the composition $ \pm \alpha \circ (\beta \circ \gamma)$.  What is the sign?  It is the Koszul sign that shuffles various $u_k$ across various $\underline{\otimes}_j v_{\ell j}.$ But this is the same as the Koszul sign of the shuffle in the vertical canonical isomorphism on the left of Diagram $\boxed{18.1}$, so the diagram commutes.\\

Of course, the sophisticated way to handle signs is to just say ``Chain complexes form a closed symmetric monoidal category, and in such a category $End(A)$ and $CoEnd(A)$ are always operads because all diagrams of  certain types commute...".   We did observe  back in Section 3 that the Koszul sign for  a permutation of tensors does not depend on the specific iterations of basic symmetric monoidal braiding isomorphisms $\tau : B \otimes C \simeq C \otimes B$ and associativities $(B_* \otimes C_*) \otimes D_* \simeq B_* \otimes (C_* \otimes D_*)$  that one uses to accomplish the permutation. In fact, this is essentially just the argument that the sign of a permutation is well defined.  This is part of the proof that chain complexes  do form a symmetric monoidal category, which is quite easy.\\   

However, in {\it closed} symmetric monoidal categories, one must deal with diagrams in which internal hom complexes $HOM(A, B)$ and  compositions of maps between them and evaluation maps  $HOM(A, B) \otimes A \to B$ also appear.  That is more subtle. We do point out that keeping track of signs in a specific category like chain complexes, where $\otimes$ and $HOM$ are fairly simple concepts, is easier than an extended discussion of all the `coherence' involved in an abstract discussion of general closed symmetric monoidal categories.  It is  bullet points 4, 5, and  6 of Proposition 3.1 that are needed for the $End$ and $CoEnd$ operads. Although not completely trivial, it is still elementary to prove the formulas in these bullet points by  direct computations evaluating both sides of the formulas. That was our recommendation back in Section 3, rather than trying to master too much of the general theory of commuting diagrams in closed symmetric monoidal categories. \\

The proof of associativity for the $CoEnd$ operad is similar.  Start with $u \in HOM(A, A^{\otimes r})$, $u_i \in HOM(A, A^{\otimes r_i})$, and $v_{ij} \in HOM(A, A^{\otimes s_{ij}})$, and let $\alpha, \beta, \gamma$ be the same maps as above. But before taking some compositions, permute the top row of Diagram $\boxed{18.1}$ by moving the $u_i$ across the $\underline{\otimes}_j v_{ij}$ and then moving $u$ across everything.  Permute the bottom row of Diagram $\boxed{18.1}$ by moving $u$ across $\underline{\otimes}_i u_i$, then moving $\underline{\otimes}_i u_i \otimes u$ across the $\otimes_i \otimes_j v_{ij}$.  The Koszul signs of these two permutations of tensors in the two rows are the same. Then argue that commutativity of Diagram $\boxed{18.1}$ is  just the associativity $\gamma \circ (\beta \circ \alpha) = (\gamma \circ \beta) \circ \alpha$, along with checking that a sign from bullet point 6 of Proposition 3.1 and the  Koszul sign on the left side of the diagram are the same.  $\qed$\\

{\bf Proof of Equivariance for the CoEnd(A) Operad.} To make sense of the equivariance axiom for  $CoEnd(A)$ we first need to review the left $\Sigma_n$ action on $CoEnd(n)$.  This will be post-composition $$gu \colon A \xrightarrow{u} A^{\otimes n} \xrightarrow{g} A^{\otimes n}.$$  But what is the map $g$ here?  It must be the {\it left chain map} action on the $n$-fold tensor product, $g(x_1 \otimes \ldots \otimes x_n) = \pm (x_{g^{-1}1} \otimes \ldots \otimes x_{g^{-1}r})$, which includes a Koszul sign.  Left action so that  $(hg)u = h(gu),$ chain map so that $d(gu) = g du \in HOM(A, A^{\otimes n})$.\\

The second equivariance property in Remark \ref{18.4} is  fairly easy. The general statement is $$\cO_P(u; h_1v_1, \ldots, h_rv_r) = (h_1 \oplus \ldots \oplus h_r)  \cO_P(u; v_1, \ldots, v_r) \in P(s).$$  We point out that the direct sum of permutations $\oplus h_i \in \Sigma_s$ regarded as an operator on $A^{\otimes s}$ is the $\underline{\otimes}$ tensor product of the separate morphisms $h_i \colon A^{\otimes s_i} \to A^{\otimes s_i}$, with no signs. With $P(n) = HOM(A, A^{\otimes n})$, the formula we want is $\underline{\otimes}_i h_iv_i \circ u = (h_1 \oplus \ldots \oplus h_r)\  \underline{\otimes}_i v_i \circ u.$  But actions of permutations on $HOM(A, A^{\otimes n}) $ are just compositions of morphisms with degree 0 chain maps.  Thus our desired formula becomes $$\underline{\otimes}_i (h_i \circ v_i) \circ u = (h_1 \oplus \ldots \oplus h_r) \circ (\underline{\otimes_i} v_i \circ u).$$
But from bullet point 6 of Proposition \ref{3.1} of Part I we have $\underline{\otimes}_i (h_i \circ v_i) = \underline{\otimes}_i h_i \circ \underline{\otimes}_i v_i = (h_1 \oplus \ldots \oplus h_r ) \circ \underline{\otimes_i} v_i$, with no signs, which proves what we want.\\

The first equivariance property of Remark \ref{18.4} requires a closer look.  In the notation of Diagram $\boxed{18.4}$, we  need to prove the diagram below commutes.  Strictly speaking there are preliminary signs caused by the reversal of order of tensors in the basic operad structure map $\cO = \cO' \tau$.  But $|gu| |\underline \otimes_i v_i| = |u| |\underline \otimes_i v_{gi}| $, so those signs are harmless. 
$$\begin{CD} A & \xrightarrow{u} & A^{\otimes r} & \xrightarrow{g}&  A^{\otimes r} \\
 &  & \downarrow { (-1)^k\underline {\otimes} v_{gi}}& & \downarrow {\underline {\otimes v_i }}\\
 & A^{\otimes s} =\ & \otimes_i A ^{\otimes s_{gi} }& \xrightarrow{g_*(s_1, \ldots, s_r)} & \otimes_i A^{\otimes s_i} &\  = A^{\otimes s}
\end{CD}$$
The sign $(-1)^k$ is the Koszul sign of $v_1 \otimes \ldots \otimes v_r \mapsto v_{g1} \otimes \ldots \otimes v_{gr}$, which is part of the inverse of the top arrow in Diagram $\boxed {18.4}$. The maps $u$ and $v_i$ can take $\FF$-basis elements to sums of basic tensors, but linearity of all the operations takes care of that. We can work with single tensors.\\

Let us follow a tensor in the upper right that we will abbreviate $(gu)(x) = (x_1 x_2 \ldots x_r)$, where $x, x_i \in A$.  The desired commutativity can be written $$(-1)^k \underline  \otimes v_{gi} \circ g^{-1}(x_1 \ldots x_r) =    (g_*(s_1, \ldots, s_r))^{-1}\underline \otimes v_i (x_1 \ldots  x_r).$$ 
Both domain and range in the bottom row of the digram equal $A^{\otimes s}$, but with tensors organized differently in blocks. The key is to carefully keep track of evaluation signs and the {\it left}  group chain map actions of $g^{-1}$ on $A^{\otimes r}$ and $g_*(s_1, \ldots, s_r)^{-1}$ on $A^{\otimes s}$. The {\it inverse} of  permutation group elements are applied to position subscripts of basic tensors, and $(g^{-1})^{-1} = g$ and $(g_*^{-1})^{-1} = g_*$.\\

On the left in the desired equation the tensor  moves two times, each move with further Koszul signs. $$ (x_1x_2\ldots x_r) \mapsto (x_{g1} x_{g2} \ldots x_{gr})  \mapsto  (-1)^k(v_{g1} x_{g1})  \ldots (v_{gr} x_{gr}).$$
The sign in the first move is from the $g^{-1}$ action. In the second move the sign is from bullet point 5 of Proposition 3.1.\\

On the right the tensor also moves two times, with additional  Koszul signs from bullet point 5 of Proposition \ref{3.1} in the first move and from the $g_*^{-1}$ action in the second move. $$(x_1 \ldots x_r) \mapsto (v_1x_1 \ldots v_r x_r) \mapsto g_*^{-1}(v_1x_1 \ldots v_r x_r).$$

The block permutation $g_*(s_1, \ldots, s_r)$ rearranges concatenated blocks $B_1\ldots B_r$ of tensors of respective lengths $s_1\ldots s_r$ in the order $B_{g1}\ldots B_{gr}$.  Applying  $g_*(s_1, \ldots, s_r)$ to blocks of the $s$ tensor $(v_ 1x_1)  \ldots (v_r x_r)$  yields $(v_{g1} x_{g1})  \ldots (v_{gr} x_{gr})$, with a Koszul sign. This agrees with the two moves on the left, up to sign.\\

But because combined Koszul signs of a sequence of permuted tensors do not depend on the choice of the separate permutations, we have indeed proved the first equivariance axiom of Remark \ref{18.4}.   $\qed$

\begin{exam}\label{18.9} Suppose $r = 3, g = (231)$.  We will write $|x_i|$ and $|v_i|$ for the degrees of these elements.  The exponent $k$ in the argument for the permutation $v_1 v_2 v_3 \mapsto  v_2v_3v_1 $ is $k = |v_1||v_2| + |v_3||v_1|$.  The exponent for the Koszul sign of $x_1 x_2 x_3 \mapsto x_2 x_3 x_1$ is $|x_1|| x_2| + |x_3|| x_1|$.  The exponent for the evaluation $(v_1v_2v_3)(x_1x_2x_3) \mapsto (v_1x_1)(v_2x_2)(v_3x_3)$ is $|x_2||v_3| + |x_2||v_1| + |x_3||v_1|$. The exponent for the evaluation $(v_2v_3v_1)(x_2x_3x_1) \mapsto (v_2x_2)(v_3x_3)(v_1x_1)$ is $|x_2||v_3| + |x_2| |v_1| + |x_3||v_1|$. Finally, the exponent for the $g_*$ evaluation $(v_ 1x_1)(v_2 x_2) (v_3x_3) \mapsto(v_{g1} x_{g1}) (v_{g2}x_{g2} )(v_{g3} x_{g3})$ is $(|v_2| +|x_2|)(|v_1| + |x_1|) + (|v_3| + |x_3|)(|v_1| + |x_1|).$\\

The Koszul sign part of the proof is then the mod 2 computation $$ \big(|x_1||v_2| + |x_1||v_3| + |x_2||v_3|\big) + \big((|v_2|+|x_2|)(|v_1| + |x_1|) + (|v_3| + |x_3|)(|v_1| + |x_1|)\big)$$  $$ \equiv \big(|v_1||v_2| + |v_3||v_1|\big) + \big(|x_1|| x_2| + |x_3|| x_1|\big) + \big(|x_2||v_3| + |x_2| |v_1| + |x_3||v_1|\big).$$  
$\qed$
 \end{exam}
 
\begin{rem}\label{18.10} We will leave the proofs of the first and second equivariance properties for the $End$ operad as  exercises.  It is more of the same, combining the definition of the operad structure maps with manipulation of tensors,  permutations, compositions, and evaluations.\\

An easier result is the claim made earlier that the duality map $CoEnd(A) \to End(A^*)$ is an operad morphism.  The duality map in each arity is a composition of two maps and the proof amounts to showing the following diagram commutes. To simplify notation we have eliminated the $\otimes$ symbol  in powers of complexes.
$$\begin{CD} HOM(A, A^r) \otimes \bigotimes_i HOM(A, A^{s_i}) & \xrightarrow{(-1)^k \underline{\otimes}v_i \circ u} & HOM(A, A^s)  \\
  \downarrow{u^* \otimes \bigotimes_i v_i^*} &  & \downarrow{(\cdot)^*} \\
  HOM((A^r)^*, A^*) \otimes \bigotimes_i HOM((A^{s_i})^*, A^*) & \xrightarrow{(u^* \circ \iota) \circ \underline{\otimes}v_i^*} & HOM((A^s)^*, A^*)\\
\downarrow {\circ \iota} \otimes \bigotimes_i {\circ \iota}& & \downarrow{\circ \iota}\\
 HOM((A^*)^r, A^*) \otimes \bigotimes_i HOM((A^*)^{s_i}, A^*) & \xrightarrow{(u^*\circ \iota) \circ \underline{\otimes}(v_i^* \circ \iota)} & HOM((A^*)^s, A^*)\\
\end{CD}$$
Each $\iota$ is a map $(B^*)^n \to (B^n)^*$ for some $B$ and $n$, from Proposition \ref{3.1} of Part I. The integer $k$ in the top square is $|u| |\underline{\otimes}_i v_i|$.  That square commutes because of the  formula for compositions $(\alpha \circ \beta)^* = (-1)^{|\alpha| |\beta|} \beta^* \circ \alpha^*$.  The commutativity of the bottom square is routine. $\qed$

 \end{rem}
\newpage 
 \section {The Barratt-Eccles  Operad}

In this section we will first define, in rather general situations, candidates for operad structure maps using a modified version of the standard procedure of Part I of our paper that constructed equivariant chain maps.  Then we will specialize to $N_*(E\Sigma_n)$ and in the following section to $\cS_*(n)$.\\
 
Let $B_*(n)$ be arbitrary free $\FF[\Sigma_n]$   complexes with $B_0(n) = \FF[\Sigma_n]$ and with the obvious augmentation and basepoint.  The ground ring $\FF$ can be any commutative ring.  Assume contractions $H_n$ satisfying $H_n^2 = 0$ and $H_n \circ \iota_B = 0$. We also assume $\Sigma_n$-bases in the image of $ H_n$.   For the $N_*(E\Sigma_n)$ we will use as basis the tuples of permutations with first entry the identity.  For the $\cS_*(n)$ we will use as basis the surjections so that the first entries of $1,2,\ldots, n$ occur in that order.

\subsection{Candidates for Operad Structure Maps}

  \begin{rem}\label{19.1} We will  define operad structure map candidates  $$\cO_B \colon B_*(r) \otimes B_*(s_1) \otimes \ldots   \otimes B_*(s_r) \to B_*(s),$$ when the $B_*(n)$ satisfy the conditions in the paragraph above. The domain is free over $\widehat{G} = \Sigma_r \times \Sigma_{s_1} \times \ldots \times \Sigma_{s_r}$, with the  basis $\{b = (b_0 \otimes b_1 \otimes \ldots \otimes b_r)\}$ where each $b_j$ is a basis element of the corresponding $B_*(n)$ factor.  So this is just the product of separate $\Sigma_n$-bases of the tensor factors $B_*(n)$.  Product basis elements $b$ are in the image of the tensor product contractions. An $\FF$-basis is given by all products $\hat{g}b,\  \hat{g} = (g; h_1, \ldots, h_r) \in \Sigma_r \times \Sigma_{s_1} \times \ldots \times \Sigma_{s_r}$.\\ 

 In degree 0,  $\cO_B$ will be the operad structure map $\cO_\Sigma$ of the permutation group operad $\Sigma$. On basis generators of higher degree we  define recursively $\cO_B(b) = H_s \cO_B(db)$, where $H_s$ is the contraction of $B_*(s)$.  But for this to make sense we need to clarify the equivariance formula, since $db$ will not usually be a sum of basis elements.\\
 
 Let us abbreviate $$Id \otimes \tau(g^{-1}) = \tau_g \colon B_*(r) \otimes B_*(s_1) \otimes \ldots   \otimes B_*(s_r) \to B_*(r) \otimes B_*(s_{g1}) \otimes \ldots   \otimes B_*(s_{gr}).$$  So $\tau_{gg'} = \tau_{g' }\tau_g.$ It will be key that $\tau_g$ is a chain map and that it is equivariant for $\Sigma_r \times \Sigma_{s_1} \times \ldots \times \Sigma_{s_r} \to \Sigma_r \times \Sigma_{s_{g1}} \times \ldots \times \Sigma_{s_{gr}}.$  Then  we define $$\cO_B(\hat{g}b) =  \cO_\Sigma(\hat{g}) \cO_B(\tau_g b)$$ $$= \cO_\Sigma(g; h_1, \ldots, h_r)\cO_B(\pm b_0\otimes b_{g1} \otimes \ldots \otimes b_{gr}).$$
 Here we are playing close attention to Remark \ref{18.4} and Diagram $\boxed{18.4}$.\\
 
 We call  the formula $\cO_B(\hat{g}b) =  \cO_\Sigma(\hat{g}) \cO_B(\tau_g b)$ {\it twisted equivariance}.\\
 
 Note the argument $\tau_gb$ of the last $\cO_B$ evaluation is up to sign a basis element of a different tensor product domain.  So we are really defining the maps $\cO_B$ simultaneously by induction on $\FF$-basis elements for all positive ordered partitions $s = s_1 + \ldots + s_r$.  Then we  extend all these maps $\FF$-linearly.\\

Of course the definition of $\cO_B(\hat{g} b)$ is just a special case of the equivariance axiom for operads.  We need to extend that to all products $\hat{g}x$.  This turns out to be a consequence of permutation identities from Lemma \ref{18.2} and Lemma \ref{18.5}.  We will  then use the general equivariance result to prove that $\cO_B$ is a chain map. $\qed$
\end{rem}
\begin{prop}\label{19.2} Let $\hat{g} = (g; h_1, \ldots, h_r)$ and $x = (x_0 \otimes x_1 \otimes \ldots \otimes x_r).$\\

(i). It holds that $\cO_B(\hat{g} x) = \cO_\Sigma(\hat{g})\cO_B ( \tau_g x).$\\

(ii). It also holds that $d \cO_B(x) = \cO_B (d x)$.
\end{prop}
\begin{proof} We assume the $x_j$ are $\FF$-basis elements, not sums of such.  Then there is a unique formula $x = \hat{g}'b$, where $b = (b_0 \otimes b_1 \otimes \ldots \otimes b_r)$  is a basis element and $\hat{g}' = (g'; h_1', \ldots, h_r') \in \widehat{G}$.   We have on the left side of (i) $$\cO_B(\hat{g}x) = \cO_B (\hat{g} \hat{g}' b) = \cO_\Sigma(gg'; h_1h'_1, \ldots, h_rh'_r) \cO_B( \tau_{gg'} b).$$
To compare with the right side of (i), we first observe $$ \tau_g x = \tau_g(\hat{g}'b)  = (g'; h'_{g1}, \ldots, h'_{gr}) \tau_gb.$$
We compute $$\cO_\Sigma(\hat{g})\cO_B(\tau_g x)  = \cO_\Sigma(g; h_1, \ldots, h_r)\cO_\Sigma(g'; h'_{g1}, \ldots, h'_{gr})\cO_B(\tau_{g'} \tau_g b).$$
From the equivariance axiom of Remark \ref{18.4} for the operad $\Sigma$, which is a consequence of  Lemmas \ref{18.2} and \ref{18.5},  we have $$\cO_\Sigma(gg'; h_1h'_1, \ldots, h_rh'_r) =     \cO_\Sigma(g; h_1, \ldots, h_r)\cO_\Sigma(g'; h'_{g1}, \ldots, h'_{gr}),$$ which proves (i) since $\tau_{gg'} = \tau_{g'} \tau_g$.\\
  
To prove (ii), we work by induction.  In degree 0 there is nothing to prove, although it is relevant that $\cO_B = \cO_\Sigma$ does commute with augmentations and base points in that degree. Inductively we first compute for basis elements, using the full boundary property in one lower dimension. $$d \cO_B(b) = d H_s \cO_B (db) = \cO_B (db) - H_s d\cO_B(db) = \cO_B(db) - H_s 0 = \cO_B(db).$$
If $|b| = 1$, the second equality uses $\rho \cO_B(db) = \cO_B \rho(db) = \cO_B(0) = 0$, where $\rho$ is the basepoint.\\
 
For the general element $x = \hat{g}b$ we make use of the fact that the boundary operators satisfy the ordinary group equivariance.  Thus $$d \cO_B (\hat{g} b) = d(\cO_\Sigma(\hat{g})\cO_B (\tau_g b)) = \cO_\Sigma(\hat{g})\cO_B (d\tau_g b).$$
On the other side, $$\cO_B (d \hat{g} b) = \cO_B (\hat{g} db) = \cO_\Sigma(\hat{g})\cO_B(\tau_g db).$$  
Since $\tau_g$ is a chain map, (ii) is proved.
\end{proof}
{\bf Associativity.} Finally we want to ask when do the $\cO_B$ satisfy the associativity axiom of Diagram $\boxed{18.1}?$   The first observation is that in degree 0, the commutativity of Diagram $\boxed{18.1}$ is exactly the associativity axiom for the symmetric group operad $\Sigma$, discussed in Remark \ref{18.7}.\\

To go further, we review and extend the uniqueness result of Proposition 6.2 of Part I. That result asserted that an equivariant map $\psi_0 \colon B_0 \to C_0$ has a unique equivariant chain map extension to $\psi \colon B_* \to C_*$ satisfying $\psi(b) \in Im(h_C) = Ker(h_C)$ on basis elements, namely the standard procedure map $\phi$.  The proof was an easy induction, $$\phi(b) = h_C\phi(db) =  h_C \psi(db) = h_C d \psi(b) = \psi(b) - dh_C\psi(b) = \psi(b) - 0.$$ Then equivariance forces $\psi(gb) = \phi(gb)$. The proof of this uniqueness result extends by the same induction to the twisted $\widehat{G}$-equivariant situation $$\displaystyle { \cO_B \colon  B_*(r) \otimes B_*(s_1) \otimes \ldots   \otimes B_*(s_r) \to B_*(s)}$$ of Proposition \ref{19.2}, since on basis elements $\cO_B(b) = H_s \cO_B(db)$.
\begin{prop} \label{19.3}The standard procedure map $\cO_B$ of Remark \ref {19.1} and Proposition \ref{19.2} is the unique twisted $\widehat{G}$-equivariant chain map $\psi$ that extends $\cO_\Sigma$ in degree 0 and takes basis elements of the domain to elements in $Im(H_s) = Ker(H_s) \subset B_*(s)$.   $\qed$ 
\end{prop}
\begin{proof} The inductive argument above proves $\psi = \cO_B$ on all basis elements.  If $b$ is a basis element then so is $\tau_g(b)$, hence the two twisted equivariance formulas give $\cO_B(\hat{g}b) = \cO_\Sigma(\hat{g})\cO_B(\tau_g(b) = \cO_\Sigma(\hat{g})\psi(\tau_g( b) = \psi (\hat{g} b)$.
\end{proof}
Proving the associativity axiom, that is, commutativity of Diagram $\boxed{18.1}$, involves some more complicated versions of twisted equivariance maps $\Phi$ than occur in the basic situation $\cO_B$ of Proposition \ref{19.2}.  There are two basic ingredients in all versions.  First, the product of symmetric groups $\widehat{G}$ that acts freely on the domain of $\Phi$, which will be a tensor product, embeds non-homomorphically in a symmetric group or product of symmetric groups  that acts on the range of $\Phi$, by a map $\cO$ that is a variant or extension of the structure map $\cO_\Sigma$ for the permutation group operad.  Secondly, the group $\widehat{G}$ will also permute tensor factors of domain basis elements, say by some operation $b \mapsto \tau_{\hat{g}} b$ with $\tau_{\hat{g}}\tau_{\hat{h}} = \tau_{\widehat{hg}}$.\\

In degree 0, $\Phi_0  = \cO_\Sigma$. Then the map $\Phi$ will always be  defined in two steps.  On basis elements of the domain, $\Phi(b) = H \Phi (db)$, where $H$ is the contraction of the range.  Then $\Phi(\hat{g} b) = \cO(\hat{g}) \Phi(\tau_{\hat{g}}b)$.  The occurrence of $\tau_{\hat{g}}$ in this preliminary equivariance formula means $\Phi$ is really  being defined by induction not on a single tensor product but on a direct sum of tensor products related by permutations of factors.\\

The proofs of Proposition \ref{19.2}(i),(ii)  extend routinely, once the map $\cO$ generalizing $\cO_\Sigma$ is clarified. In words, the preliminary equivariance in the definition extends to full equivariance $\Phi(\hat{g} x) = \cO(\hat{g}) \Phi(\tau_{\hat{g}}x)$, where $x = \hat{g}'b$. Also $\Phi$ is a chain map.  The uniqueness result Proposition \ref{19.3}  also extends quite routinely to the more complicated twisted contexts.\\  

We will give some examples related to Diagram $\boxed{18.1}$ in the case of interest to us where $P(n) = B_*(n)$.  We will see that the first map $\Phi' = Id \otimes \bigotimes \cO_B$ in the top row is a standard procedure twisted equivariant map for the inclusion $$\cO' = Id \times \prod_i \cO_\Sigma \colon \widehat{G}' \to \Sigma_r \times \prod_i \Sigma_{si}, \ where\ \widehat{G}' = \Sigma_r \times \prod_i [\Sigma_{ri} \times \prod_j \Sigma_{sij}].$$

We will also see that the first map $\Phi'' = \cO_B \otimes \bigotimes Id$ in the bottom row is a standard procedure twisted equivariant map for the inclusion $$\cO'' = \cO_\Sigma  \times \prod_{i,j} Id \colon \widehat{G}'' \to \Sigma_{\oplus ri} \times \prod_{i,j} \Sigma_{sij}, \ where\ \widehat{G}'' = [\Sigma_r \times \prod_i \Sigma_{ri}] \times \prod_{i,j} \Sigma_{sij}.$$
The second map in both rows are standard twisted equivariant maps $\cO_B$ from Proposition \ref{19.2}.\\

We have mentioned that in degree 0 in our case Diagram $\boxed{18.1}$ commutes, the result being the one map $\widehat{G}' \to \Sigma_s$ produced by the associativity axiom for the symmetric group operad.  That is, one gets the same map, say $\cO$,  by following $\cO'$ by the symmetric group operad map $\cO_\Sigma \colon \Sigma_r \times \prod_i \Sigma_{si} \to \Sigma_s$ or by first applying the reordering of factors  isomorphism $\widehat{G}' \simeq \widehat{G}''$ and then following $\cO''$ by $\cO_\Sigma \colon  \Sigma_{\oplus ri} \times \prod_{i,j} \Sigma_{sij} \to \Sigma_s$.  Thus there is a well-defined twisted equivariant standard procedure map from the top left to the lower right of Diagram $\boxed{18.1}$, twisted with respect to the non-homomorphic embedding $\cO$.  We would like to prove that both compositions around Diagram $\boxed{18.1}$ coincide with this standard map.  It is just a check using $\Sigma$ operad associativity that  both compositions satisfy the same twisted equivariance. 

 \begin{prop}\label{19.4} (i). The first  map $$\Phi' = Id\otimes \bigotimes\cO_B\colon \displaystyle B_*(r) \otimes \bigotimes_{i=1}^r \bigl[B_*(r_i) \otimes \bigotimes_{j=1}^{r_i}B_*(s_{ij}) \bigr]  \to  \displaystyle B_*(r) \otimes \bigotimes_{i=1}^r B_*(s_i) $$ on the top row of Diagram $\boxed{18.1}$ is a standard procedure chain map that satisfies the basic twisted equivariance formulas as in Proposition \ref{19.2} separately on each tensor factor $\cO_B \colon [B_*(r_i) \otimes \bigotimes_{j=1}^{r_i}B_*(s_{ij}) ] \to B_*(s_i)$, and ordinary equivariance on the first $Id$ factor. \\

(ii). The first map $$\Phi'' = \cO_B \otimes \bigotimes Id \colon \bigl[B_*(r) \otimes \bigotimes_{i=1}^r B_*(r_i)\bigr] \otimes \bigotimes_{j=1}^{r_i}B_*(s_{ij}) \to  B_*(\Sigma r_i) \otimes \bigotimes_{j=1}^{r_i} B_*(s_{ij})$$ on the bottom row of Diagram $\boxed{18.1}$ is a standard procedure chain map that satisfies the basic twisted equivariance formula on the first tensor factor $ \cO_B \colon  B_*(r) \otimes \bigotimes_{i=1}^r B_*(r_i)  \to B_*(\Sigma r_i)$, and ordinary equivariance on all the $Id$ factors. \\

(iii). The vertical isomorphism $$\sigma \colon B_*(r) \otimes \bigotimes_{i=1}^r \bigr[ B_*(r_i) \otimes \bigotimes_{j=1}^{r_i}B_*(s_{ij})\bigl] \to \bigl[ B_*(r) \otimes \bigotimes_{i=1}^r  B_*(r_i) \bigr]\otimes \bigotimes_{j=1}^{r_i}B_*(s_{ij}) $$ on the left side of Diagram $\boxed{18.1}$ that permutes tensor factors is a standard procedure chain map that satisfies ordinary equivariance with respect to an isomorphism between products of symmetric groups that permutes factors.
\end{prop}
\begin{proof} All three statements are modified cases of untwisted versions of tensor products of chain maps that we studied in Proposition 6.11(i), (ii) of Part I.   We offered two proofs of the various parts of that result.  One proof was an inductive proof for basis elements, beginning in degree 0, that used explicit formulas for contractions of tensor products.  Two formulas are equal if they agree on basis elements and satisfy the same equivariance. The other proof was simpler and made use of the uniqueness result Proposition 6.2 of Part I.  The key is that tensor products of elements in the image of contractions are in the image of the tensor product contraction. Both proofs carry over to all parts of the present proposition, with an extended Proposition \ref{19.3} replacing Proposition 6.2.  We leave the relatively easy details as exercises.
\end{proof}
 It follows immediately from the extended Proposition \ref{19.3}  that the collection of maps $\cO_B$ define an operad structure on the components $B_*(n)$ if the two routes around Diagram $\boxed{18.1}$, which we have named  $\cO_B \circ \Phi'$ and $\cO_B \circ \Phi'' \circ \sigma$, have the property that both composed maps send basis elements of the domain  to elements in $Im(H_s)$.  Going back to Proposition 6.5 of Part I, we know that compositions of standard procedure maps need not be standard procedure maps, but there are various hypotheses that imply compositions are standard procedure maps.  However, the real issue for establishing associativity for the candidate operad structure maps $\cO_B$ is not so much the extension of Proposition 6.11 and Proposition 6.5 to twisted equivariant situations, which is pleasant enough and one way to look at it,  but rather the following more direct criterion for associativity of the $\cO_B$.

\begin{prop}\label{19.5} Suppose the maps $\cO_B$ have the property that  for all $c \in B_*(r) \otimes \bigotimes B_*(s_i)$ that are tensor products of elements in the image of contractions of the factors, it holds that $\cO_B(c) \in Im(H_s) \subset B_*(s)$.   Then  both compositions naming the two routes around Diagram $\boxed{18.1} $ do send basis elements to elements in $Im(H_s)$, and therefore the $B_*(n)$ form an operad.
\end{prop}
\begin{proof} Starting with a big tensor product of basis elements in either row of Diagram $\boxed{18.1}$,  the first map $\Phi'$ or $\Phi''$ will carry that tensor to a tensor of elements in the image of contractions.    Therefore by the assumed criterion applying another $\cO_B$ map  produces an element in $Im(H_s)$.
\end{proof}

\subsection{The Barratt-Eccles Operad Structure Maps}

We first observe that in the case $B_*(n) = N_*(E\Sigma_n)$ the criterion of Proposition \ref{19.5} trivially holds.  Namely, for each factor $N_*(E\Sigma_r)$ or $N_*(E\Sigma_{s_i})$  separately of a tensor product, $Im(H)$ is exactly the same thing as the $\FF$-span of basis elements, because of the contraction formula $H(X) = (e, X)$ for a tuple of permutations $X$.  Applying $\cO_B$ to a tensor product of basis elements is by definition in the image of $H_s$. Therefore, we have proved
\begin{prop}\label{19.6} The complexes  $N_*(E\Sigma_n)$  form an operad, using the standard twisted equivariant procedure chain maps of Remark \ref{19.1} and Proposition \ref{19.2} as structure maps.\ \ \ $\qed$
\end{prop}
In the case $B_*(n) = \cS_*(n)$, for any of the surjection complexes, there are elements in $Im(H_n)$ that are not $\FF$-linear combinations of basis elements.  For example, for $\cS_*^{bf}(n)$, from Section 13 of Part II we have $H_4(14324) = (s + isr + i^2 s r^2) (14324) = (124324) + (123434)$, but $(124324)$  is not a basis element.  We will need to work harder to establish the hypothesis of Proposition \ref{19.5}.\\

{\bf Reconciliation with the Symmetric Monoidal Functor View.}  We will next reconcile our approach to the Barratt-Eccles operad with the standard approach found in the literature.  Among other things, this provides a closed formula for the operad maps, something the recursive approach does not do immediately.\\

Barratt-Eccles in [2] first defined a symmetric operad $W$ in the category of simplicial sets with components $W(n) = E\Sigma_n$, the contractible MacLane models for the symmetric group.  The operad structure maps are induced from the symmetric operad $\Sigma$ of sets. The degree 0 simplices  (vertices)  of the $E\Sigma_n$ are elements of the symmetric groups $\Sigma_n$.  So in degree 0 the structure map $\cO_W \colon W(r) \times W(s_1) \times \ldots \times W(s_r) \to W(s)$ is just the $\Sigma$ operad structure map $\cO_\Sigma$. We point out that $$W(r) \times W(s_1) \times \ldots \times W(s_r) = E\Sigma_r \times E\Sigma_{s_1} \times \ldots \times E\Sigma_{s_r} =  E(\Sigma_r \times \Sigma_{s_1} \times \ldots \times \Sigma_{s_r}).$$
In degree $k$, a `simplex' in  $W(n)_k$ is just a $(k+1)$-tuple of permutations in $\Sigma_n$.   `Simplices' in products of simplicial sets  are tuples of `simplices' in the factors.  We thus also view a $k$-simplex of the simplicial set product as a $(k+1)$-tuple of vertices of the MacLane model for the product group. We define $\cO_W$ to be $\cO_\Sigma$ on each  vertex of the $(k+1)$-tuple of product vertices. It is routine to see that $\cO_W$ commutes with face and degeneracy maps, which just delete or repeat vertices.   Since $\Sigma$ is an operad in sets, it is immediate that $W$ is an operad in simplicial sets.    The operad structure map $\cO_W$ can then be viewed as the inclusion of MacLane models induced by the inclusion of groups  ({\it not} a homomorphism) $\cO_\Sigma \colon \Sigma_r \times \Sigma_{s_1} \times \ldots \times \Sigma_{s_r} \to \Sigma_s$.

\begin{rem}\label{19.7}Next we apply the normalized chain complex functor, which is a symmetric monoidal functor, and define the components of the Barratt-Eccles operad in the category of chain complexes to be $\cE(n) = N_*(E\Sigma_n)$.  The term {\it symmetric monoidal functor} refers to properties of the functorial Eilenberg-Zilber maps $EZ \colon N_*(X) \otimes N_*(Y) \to N_*(X \times Y)$.   Recall that any set theoretic inclusion of groups $G \to G'$ that corresponds identity elements induces a map of normalized chain complexes  $N_*(EG) \to N_*(EG')$ that commutes with the contractions $h_G, h_{G'}$.  This fact was discussed in Example 5.3 of Part I.  Combined with the Eilenberg-Zilber map $EZ$, we then obtain alternate candidates  $ N_*(\cO_W)\circ EZ$  for the operad structure maps  of the Barratt-Eccles symmetric operad in the category of chain complexes,     $$N_*(E\Sigma_r) \otimes N_*(E\Sigma_{s_1}) \otimes \ldots   \otimes N_*(E\Sigma_{s_r}) \xrightarrow{EZ} N_*(E( \Sigma_r \times \Sigma_{s_1} \times \ldots \times \Sigma_{s_r}))$$ $$ \xrightarrow{N_*(\cO_{W})} N_*(E \Sigma_{s_1 + \ldots + s_r}). $$
Of course it is necessary to prove these structure maps of chain complexes do satisfy the operad axioms. There are  easy proofs that simply show the maps $ N_*(\cO_W)\circ EZ$  coincide with the standard procedure twisted equivariant chain maps $\cO_B$  that we  have already proved define an operad. 
\end{rem}
\begin{prop} \label{19.8}   For $B_*(n) = N_*(E\Sigma_n)$, the  maps $ N_*(\cO_W)\circ EZ$ coincide with the standard twisted equivariant procedure  chain maps $\cO_B$ between domain and range from Remark \ref{19.1} and  Proposition \ref{19.2}, using the preferred basis of the domain over  products of symmetric groups and the preferred contraction of the range.
\end{prop}
\begin{proof} The maps $\cO_B$ and $ N_*(\cO_W)\circ EZ$ agree in degree 0.  Also, as explained during the discussion of Subsection 6.4 of Part I, the map $N_*(\cO_W) \circ EZ$ takes basis elements of the domain to elements in the image of the contraction of the range. This holds because up to signs the $EZ$ map takes basis elements to sums of basis elements, which are in the image of the contraction, and then $N_*(\cO_W)$ commutes with contractions.\\

We next prove the twisted equivariance formula for $ N_*(\cO_W)\circ EZ$.  Then we simply observe that the uniqueness result Proposition \ref{19.3} applies.\\

The twisted equivariance property we need is $$N_*(\cO_W) EZ (\hat{g} b) = \cO_\Sigma(\hat{g}) N_*(\cO_W) EZ (\tau_g b).$$ Here $b$ is a basis element in  $N_*(E\Sigma_r) \otimes \bigotimes N_*(E\Sigma_{s_i})$.  Also $\hat{g} = (g; h_i) \in \Sigma_r \times \prod \Sigma_{s_i}$ and $\tau_gb = (Id \otimes \tau(g^{-1}))(b) \in N_*(E\Sigma_r) \otimes \bigotimes N_*(E\Sigma_{s_{gi}})$. The $EZ$ map is ordinary equivariant for the product of symmetric groups actions, so $EZ (\hat{g} b) = \hat{g} EZ (b)$.\\

$EZ(b)$ is a sum of (signed)  basis elements $(u; v_i) \in N_*(E(\Sigma_r \times \prod \Sigma_{s_i}))$. The equivariance property of $N_*(\cO_W)$ is $$N_*(\cO_W)( \hat{g} (u; v_i) )) =  \cO_\Sigma(\hat{g}) N_*(\cO_W)(u; v_{gi}) = \cO_\Sigma(\hat{g}) N_*(\cO_W)(Id \times \tau(g^{-1}))(u; v_i).$$ 

The commutativity property of $EZ$ from Subsection 6.4 implies that the following diagram commutes.
$$\begin{CD} N_*(E\Sigma_r) \otimes \bigotimes N_*(E\Sigma_{s_i})) & \xrightarrow {EZ} & N_*(E\Sigma_r) \times \prod E\Sigma_{s_i}) &  & \\
 \downarrow \tau_g = Id \otimes \tau(g^{-1}) &  & \downarrow Id \times \tau(g^{-1}) & &  \\
N_*(E\Sigma_r) \otimes \bigotimes N_*(E\Sigma_{s_{g_i}})) & \xrightarrow {EZ} & N_*(E\Sigma_r) \times \prod E\Sigma_{s_{g_i}}) & \xrightarrow{N_*(\cO_W)} & N_*(E\Sigma_n) \\
\end{CD}$$
We then compute as desired 
$$N_*(\cO_W) EZ(\hat{g}b) = N_*(\cO_W)(\hat{g} EZ (b)) $$  $$ =  \cO_\Sigma(\hat{g}) N_*(\cO_W)(Id \times \tau(g^{-1})(EZ (b))  = \cO_\Sigma(\hat{g}) N_*(\cO_W) EZ(\tau_g b).$$ 
The delicate issue here is about signs.   There is no sign in the center direct product vertical map $Id \times \tau(g^{-1})$. There is a  Koszul sign in the left-hand  tensor product vertical map $\tau_g$. But the products of simplices in the prisms involved in the two $EZ$ maps are permuted when passing from $b$ to $\tau_g b$. The $EZ$ maps require orientation signs on each prism, which explains why the diagram commutes.\\

Here are some alternate  arguments.  Once the equivariance property is clarified, instead of using the uniqueness result Proposition \ref{19.3}, we could view $N_*(\cO_W) \circ EZ$ as a composition of standard procedure chain maps. We could then use the fact that $EZ$ takes basis elements  to sums of basis elements, along with Proposition \ref{6.5}(iii) of Part I, to deal with the composition.   Also, we observed in Remark \ref{19.7} that $N_*(\cO_W)$ commutes with contractions.   Proposition \ref{6.5}(ii) of Part I would then also deal with the composition.
\end{proof}

 \begin{rem}\label{19.9} The simplicial set map switching factors $ E(\Sigma_r) \times E(\prod \Sigma_{s_i}) \simeq E( \prod \Sigma_{s_i} )\times E(\Sigma_r)$ lifts to an isomorphism of normalized chain complexes $ N_*(E(\Sigma_r) \times E(\prod \Sigma_{s_i})) \simeq N_*(E( \prod \Sigma_{s_i} )\times E(\Sigma_r))$, with no signs.  
Then the $EZ$ maps will commute with the Koszul sign map $\tau \colon N_*(E\Sigma_r) \otimes \bigotimes N_*(E\Sigma_{s_i}) \simeq \bigotimes N_*(E\Sigma_{s_i}) \otimes N_*(E\Sigma_r)$.\\

Suppose all $s_i = t$.  There is a natural free action of the semidirect product $\prod \Sigma_t  \rtimes \Sigma_r$ on the right hand tensor product $\bigotimes N_*(E\Sigma_t) \otimes N_*(E\Sigma_r)$,  just as there is on $N_*(E(\prod \Sigma_t \times \Sigma_r) = N_*(E (\prod \Sigma_t \rtimes \Sigma_r)$. Extending the discussion in Remark \ref{18.3}, by means of the Koszul sign identification $\tau$ we can  identify the operad structure map $\cO_\cE$ in this restricted case with an equivariant map for the semidirect product actions and the semidirect product subgroup inclusion, 
$$   \bigotimes N_*(E\Sigma_t ) \otimes N_*(E\Sigma_r) \xrightarrow{EZ} N_*(E (\prod \Sigma_t \rtimes \Sigma_r)) \subset N_*(\Sigma_{rt}).$$  
This statement can be proved directly for our twisted equivariant standard procedure construction of $\cO_{\cE}$, as well as for the $N_*(\cO_W) \circ EZ$ construction. $\qed$
\end{rem}

\begin{rem}\label{19.10} We make some comments that relate our $\cO_B$ constructions more directly to the symmetric monoidal functor viewpoint. Since $EZ$ is ordinary equivariant and $N_*(\cO_W)$ is twisted equivariant, the composition $ N_*(\cO_W) \circ EZ$  is twisted equivariant.   All that remains is the associativity axiom. The original point of view is that associativity follows from symmetric monoidal functor properties of the $EZ$ maps and the symmetric operad properties of $W$.  We studied the $EZ$ maps for MacLane models extensively in Sections 6  of Part I, in the context of our standard procedures for defining chain maps using bases of the domain and contractions of the range.  In particular, basic properties of the $EZ$ maps for MacLane models, such as full associativity and commutativity, drop out directly from the uniqueness result Proposition 6.2 of Part I.  The explicit  closed formula is not needed.\\

The operad properties of $W$ are trivial consequences of the operad properties of the symmetric group operad $\Sigma$. Then $EZ$ {\it injectively} maps Diagram $\boxed{18.1}$ for our $\cO_B$ maps  to $N_*$ applied to Diagram $\boxed{18.1}$ for the $\cO_W$ maps.  Squares connecting the two diagrams  certainly commute by the easy result Proposition \ref{19.3}. The terminal target in all rows of both diagrams is $N_*(E\Sigma_s)$.  The full commutativity and associativity properties of $EZ$ can be used to show the full three dimensional diagram commutes, which implies the operad associativity of our $\cO_B$.\\

This last argument is just a special case of the argument that symmetric monoidal functors in general produce operads in one category from operads in another category. But we prefer our direct proof  that the $N_*(E\Sigma_n)$ form an operad,  using the twisted equivariant procedure, followed by the uniqueness theorem identification of the twisted procedure maps $\cO_B$ with the maps $N_*(\cO_W) \circ EZ$  $\qed$
\end{rem}

\newpage
  
\section{The Surjection Operad}

In this section we show that the candidate structure maps $\cO_\cS$ of Proposition \ref{19.2} for the various surjection complexes $\cS_*(n)$ satisfy the associativity axiom, and hence the surjection complexes form operads.  Of course since the surjection complexes are all isomorphic, we can work with any one of them.  The signs are simplest for the Berger-Fresse complexes.\\

We will actually prove more, namely we will prove the diagram below commutes.  Each map, including the tensor product of $TR$ maps, is a standard procedure map.
 $$\begin{CD}  N_*(E\Sigma_r) & \otimes & N_*(E\Sigma_{s_1})& \otimes & \cdots &\otimes & N_*(E\Sigma_{s_r})  & \xrightarrow{\cO_{\cE} } &  N_*(E\Sigma_s)  \\
\downarrow TR&\otimes & \downarrow TR  & \otimes&\cdots & \otimes & \downarrow TR & & \downarrow TR \\
 \cS_*^{bf}(r) & \otimes & \cS_*^{bf}(s_1)&  \otimes &\cdots & \otimes &  \cS_*^{bf}(s_r) & \xrightarrow{\cO_{\cS} } & \cS_*^{bf}(s).
\end{CD}$$
The $TR$ are surjective so associativity for the Barratt-Eccles operad then provides another proof of associativity for the surjection operad.  But the diagram also says that the surjection operad is a quotient operad of the Barratt-Eccles operad, that is, $\cT \cR \colon \cE \to \cS$ is a surjective  operad morphism.\\

We will also show that the adjoints of the Berger-Fresse functorial standard procedure maps $\phi \colon \cS_*^{bf}(n) \otimes N_*(X) \to N_*(X)^{\otimes n}$ studied in Section 17 fit together to define an operad map  from the surjection operad to the $CoEnd$ operad,  $\Phi \colon \cS \to CoEnd$, where $CoEnd(n) = HOM_{func}(N_*( - ), N_*( - )^{\otimes n}).$\\

This final section is the longest and the most difficult section of the paper.

\subsection {The Surjection  Operad Structure Maps}

 \begin{rem}\label{20.1} {\bf Notation Needed to Define $\cO_{\cS}$.} In order to study the candidate structure maps $\cO_{\cS}$ of Proposition \ref{19.2} we need quite a bit of notation.  Given a tuple of integers $y$, denote by $y[t]$ the result of adding $t$ to each entry of $y$.  A {\it $k$-division} of $y$ is obtained by repeating $k-1$ not necessarily distinct entries of $y$ and inserting dividers between the repeated entries.  For example, if $y = (1,2,3,2,4)$ then one 3-division of $y[2] = (3,4,5,4,6)$ is given by $(3,4|4,5,4,6|6)$.  We refer to the $k$ delineated subtuples of a $k$-division as the {\it division subtuples.}  In the example, these are $(3,4), (4,5,4,6), (6)$. The tuple consisting of a single integer $(n)$ has a unique $k$-division $(n\ |\ldots |\ n)$.  Any $y$ has a unique $1$-division, namely $y$ itself.  Obviously $k$-divisions of $y$ and $y[t]$ are equivalent concepts.\\

Given a generator $x = (x_1, x_2, \ldots, x_{r+k}) \in \cS_*(r)$, with $k_i$ occurrences of the values $i$, and given generators $y_i \in \cS_*(s_i), 1 \leq i \leq r$ along with a $k_i$-division $\cD_i$ of each $y_i[t_i] $, where  $t_i = s_1+ \ldots +s_{i-1}$, we define $\cD x = (\cD_1, \ldots, \cD_r) x \in \cS_*(s),\ s = s_1 + \ldots + s_r$, to be the result of replacing the successive occurrences of the $i$-entries of $x$ by the successive division subtuples of the $y_i[t_i] $.  In all discussions in this section, the integer $k_i$ will always denote the number of $i$'s occurring in a generator $x$ that is part of the discussion.\\

We also want to associate a $\pm$ sign to each term $\cD x$.  The signs depend on which surjection complex we use, but the terms $\cD x$ do not. The sign will be related to a shuffle permutation of caesura {\it positions} of the caesuras of $x$ and the $y_i$.  The values of entries of $\cD x$ are what they are, but the positions of the caesuras of $\cD x$ only depend on the positions of the caesuras of $x$ and the positions of the caesuras of the $y_i$ with respect to the $k_i$-divisions of the $y_i$.  There are different ways to describe the relevant shuffle of caesuras and a sign, but here is one.\\

The number of caesuras of $\cD x$ is the degree, which is $|x|+ \sum |y_i|$.  The entries of the $y_i$ that are final entries of division subtuples that are repeated in a following subtuple generate caesura entries of $\cD x$ because they are repeated later in $\cD x$.  These correspond to the caesura entries of $x$.  The caesura entries of the $y_i$ in  division subtuples that come before the final entry of subtuples correspond bijectively  to the caesuras of the $y_i$.  These are then moved to positions in $\cD x$ that precede the positions of the corresponding final subtuple entry.  These non-final $y_i$-caesuras  also generate caesura entries of $\cD x$.  Define $sh(C_x, C_\cD)$ to be the number of pairs consisting of an  $x$-caesura of $\cD x$ and a $y$-caesura of $\cD x$ that precedes it.  The sign associated to  $\cD x$ in the surjection  complex $\cS_*^{bf}(s)$ will be $(-1)^{sh(C_x, C_\cD)}$.\\

 Note that this discussion simplifies somewhat if $y_j = (1) \in \cS_0(1)$ for all $j \not = i$ and a $k_i$-division is chosen for $y_i \in \cS_*(s_i)$.  For $j \not=i$ there is still an implicit $k_j$-division $(1|1| \ldots | 1)$ of $y_j$.  These simple cases are all that is needed in the {\it partial compositions} treatment of operads that we will discuss further below.
\end{rem}

\begin{exam}\label{20.2} Take $x = (1,2,1,3,2)$, and $y_1 = (1231),\  y_2 = (1,2,1,4,3)$ and $ y_3 = (1,2,1).$ Consider the $k_1 = 2$-division    $\cD_1y_1[0] = (1, 2 | 2, 3, 1),$ the $k_2 = 2$-division   $\cD_2y_2[3] = (4, 5, 4, 7 | 7, 6),$ and the $k_3=1$-division $ \cD_3 y_3[7] = (8,9,8).$ Then
$$\cD x = (\cD_1, \cD_2, \cD_3)x = (\underline{1, 2}, \underline{\underline {4, 5,4,7}}, \underline {2,3,1} , \underline{\underline{\underline {8,9,8}}}, \underline{\underline {7,6}}).$$
The two $x$-caesura entries of $\cD x$ are the 2 and 7, corresponding to final entries of $y_i[t_i]$ subtuples that replace caesuras of $x$. The three $y$-caesura entries are the $\underline{1}, \underline{\underline {4}}$, and $\underline{\underline{\underline{8}}}$.  The order of these caesuras in $\cD x$ is $ \underline{1}\ 2\ \underline{\underline {4}}\ 7\ \underline{\underline{\underline{8}}}$. Thus $sh(C_x, C_\cD) = 1+2$ and the sign associated to $\cD x$ is $-1$.\\

For a partial composition example, take $x = (1,2,1,3,2), y_1 = (1), y_3 = (1)$, $y_2 = (1,2,1,4,3)$ with the 2-division $\cD_2y_2[1] = (2,3,2,5 | 5,4).$ We also use the $k_1 = 2$-division $(1|1)$ of $y_1$.  Then $\cD x = (\underline{1}, \underline{\underline{2,3,2,5}}, \underline{1}, \underline{\underline{\underline{6}}}, \underline{\underline{5,4)}}).$ The $x$-caesuras are $1, 5$ and the $y$-caesura is $2$, in the order $1 \underline{\underline 2} 5$.  The sign associated to $\cD x$ is $-1$. $\qed$
\end{exam}
 
{\bf The Berger-Fresse Formula for $\cO_{\cS}$}. We now give a formula for the surjection operad map for $x \in \cS_*^{bf}(r), y_i \in \cS_*^{bf}(s_i)\ s = \sum s_i$. The proposition below is technically the most difficult result of our paper. The result was stated by Berger-Fresse  [3], [4], but details of a proof were not given.  The result is equivalent to formulas that can be found in the seminal work of McClure-Smith [19], [20], which is also technically difficult.  Our approach is different, and we regard our details as providing replication of important results of McClure and Smith and of Berger and Fresse about cochain operations and operads.  We do not claim our proof is simple but we do believe we offer a more structured viewpoint.

\begin{prop} \label {20.3}(i). With the $\cD_i$ indicating $k_i$-divisions of $y_i[t_i],$ as above, it holds that with suitable signs the equivariant twisted standard procedure map $\cO_\cS$ of Proposition \ref{19.2} coincides with a map 
$$\Phi(x; y_1, \ldots, y_r) = \sum_{all\ \cD = (\cD_1, \ldots, \cD_r) } \pm\  \cD x \in \cS_*(s).$$ 
(ii). For the surjection complex $\cS_*^{bf}(s)$ the sign of the $\cD x$ term   is $(-1)^{sh(C_x, C_\cD)}$.
\end{prop}
The signs are recursively determined by the standard procedure process.  We state the proposition in two parts because we will first prove by induction that a formula (i) holds for all the surjection complexes. Then we will verify the signs given in (ii) are correct for the Berger-Fresse complex. The steps in the proof very much resemble the steps in the proof of Proposition \ref{17.3} that identified the functorial standard procedure map $\cS_*^{bf}(n) \otimes N_*(X) \to N_*(X)^{\otimes n}$ with a Berger-Fresse map.  As in that case, the result here provides an example of the general method described towards the end of preview Section 2.3 of Part I for studying standard procedure maps of form $\phi(x) = \sum \pm \cT x$.\\

{\bf Step 0}. {\bf Degree 0.}  In degree 0, $x$ and the $y_i$ are permutations, all $k_i = 1$, and the map $\Phi = \cO_\cS$ coincides with the operad structure map $\cO_\Sigma$ of Section 18.1.  Since permutations have no caesuras, all signs are $+1$. $\qed$\\ 

{\bf Step 1}. {\bf Equivariance.} The map $\Phi$ is twisted equivariant.  It is best to look at the separate formulas in Remark \ref{18.4}, where the  equivariant property for operad structure maps was defined. The second equivariance formula $\Phi(x; h_iy_i) = (\oplus h_i) \Phi(x; y_i)$ is trivial, on each $\cD x$ summand, because one is just permuting the entries of each $y_i[t_i]$ and their division subtuples, which can be done either before or after replacing occurrences  of $i$ in $x$ by division subtuples of $y_i[t_i]$.  The signs in part (ii) of Proposition \ref{20.3} also match because the relative positions of $y_i$ and $h_iy_i$-caesuras, compared to $x$-caesuras, coincide  in corresponding $\cD$ terms.\\

The first equivariance formula $\Phi(gx; y_i) = g_*(s_i)\Phi(x; y_{gi})$ is also seen by matching $\cD gx$ summands on the left with $\cD x$ summands on the right.   On the left, choosing a $k_{g^{-1}(i)}$-division of each $y_i[t_i]$ is equivalent to choosing on the right a $k_i$-division of each $y_{gi}[t_i']$, where the entries of the $y_{gi}[t_i']$ and their division subtuples are from successive intervals of length $s_{gi}$ ending at $s_{g1} + \ldots + s_{gi}$.  Use these to form a summand of $\Phi(x; y_{gi})$ in $\cS_*^{bf}(s)$, then act by the permutation $g_*(s_i)$. That permutation $g_*(s_i) = (B_{g1}B_{g2}\ldots B_{gr})$   takes  successive intervals of length $s_{gi}$ to the blocks $B_{gi}$, where $(1,2, \ldots, s) = (B_1, B_2, \ldots, B_r)$ with each block $B_i$ having length $s_i$.   The result is the element of $\cS_*^{bf}(s)$ obtained by replacing occurrences  of $i$ in $gx$ by division subtuples of $y_i[t_i]$, which is a summand of $\Phi(gx; y_i)$.\\

An advantage of the complexes $\cS_*^{bf}(n)$ is that the permutation group $\Sigma_n$ acts with no signs. Therefore the signs in part (ii) of the formula for $\Phi$ also match because the relative positions of $x$ and $y_{gi}$-caesuras match those of $gx$ and $y_i$-caesuras in corresponding $\cD$ summands.\\

It requires some diligence to parse this last argument because of the subtlety of the permutation $g_*(s_i)$.   Surprisingly it does not really seem  easier to follow in the case of partial compositions with a single non-trivial $y_i$. \ \ \ \ $\qed$
  
\begin{exam}\label{20.4} Take $r = 3$, $x = (1,2,1,3,2)$, $y_1 = (1,2,3,2,4,1)$, $y_2 = (1,2,1)$, and $y_3 = (1,2,1,3,1)$.  Take $g = (231) \in \Sigma_3$, so $gx = (2,3,2,1,3)$. To build a typical summand of $\Phi(x; y_2, y_3, y_1)$\footnote{We remind that it is the {\it order} the $y_i$'s appear in $\Phi(x; y)$, not their subscripts, that determines which $y_i$ subtuples replace which $x$ entries.}, consider 2-divisions $(1 | 121)$ of $y_2[0]$ and $(3,4,3,5 | 5,3)$ of $y_3[2]$.  Since 3 is a singleton in $x$, we use the trivial division $y_1[5] = (6,7,8, 7,9,6)$.\\

The associated summand of $\Phi(x; y_2, y_3, y_1)$ is 
$$(\underline{1}, \underline{\underline{3,4,3,5}}, \underline{1,2,1}, \underline{\underline{\underline{6,7,8,7,9, 6}}}, \underline{\underline{5, 3}}).$$
Act on this by $g_*(4,2,3) = (5,6,\ 7,8,9,\ 1,2,3,4)$. The result is
$$ (\underline{\underline{5}}, \underline{\underline{\underline{7,8,7,9}}}, \underline{\underline{5,6,5}}, \underline{1,2,3,2,4,1}, \underline{\underline{\underline{9,7}}})$$
which is the associated term of $\Phi(gx; y_1,y_2,y_3)$, using the 1-division of $y_1[0]$ and the chosen 2-divisions of  $y_2[4]$ and $y_3[6]$.  It is easy to observe  the similar structure and relative positions of $x$-caesuras and $y$-caesuras in the two $\cD$ terms, not just in this example but in all cases.   $\ \ \ \ \qed$
\end{exam}

\begin{rem}\label{20.5} {\bf The Image of $\bf H_s$.} Recall from Subsection 15.1 of Part II that to say $x$ is a clean generator means in degree 0 that $x$ is the identity permutation $Id_r = (1,2, \ldots,r)$ and in higher degrees $x = (1,2, \ldots, \ell, \ldots, \ell, \ldots)$, where $x_\ell = \ell$ is the first caesura.  We will use the terminology {\it $x$ is clean at $\ell$} to describe such generators. 
From Proposition 15.1 of Part II, $Im(H_s)$ is exactly the $\FF$-span of the clean generators, where $H_s$ is the contraction of $\cS_*(s)$.\\

 Of great importance is the following fact about the formula for $\Phi$ in Proposition \ref{20.3}(i). If $x$ and the $y_i$ are clean surjection generators then $\Phi(x; y_1, \ldots, y_r)$ is a sum of clean generators.  This will be an easy consequence of Lemma \ref{20.6} below.    The fact that $\Phi = \cO_{\cS}$ maps tensors of clean generators to sums of clean generators is the key to proving the associativity axiom for the surjection operad because that fact is exactly the criterion of Proposition \ref{19.5}.  The same fact is used to prove the surjection operad is a quotient of the Barratt-Eccles operad. The signs in Proposition \ref{20.3}(ii) are not needed for the associativity or for the comparison with the Barratt-Eccles operad.  $\qed$ \\
\end{rem}

{\bf Step 2.} {\bf A Key Lemma.}  The map $\Phi$ coincides with $\cO_\Sigma$ in degree 0, so cleanliness is no issue.  In higher degrees we have the following different cases, each of which is proved by simply reviewing the definitions. Recall $t_i = s_1 + \ldots + s_{i-1}$. Let $\cD = (\cD_1, \ldots, \cD_r)$ denote fixed $k_i$-divisions of the $y_i$.

\begin{lem}\label{20.6}(i). If $x = Id_r$ and if $n$ is least so that $y_n \not= Id_{s_n}$, with $y_n$ clean at $m$, then $\cD x$ is clean at $t_n + m$.\\

(ii). If all $y_i = Id_{s_i}$ and $x$ is clean at $\ell$ then $\cD x$ is clean at $t_\ell + j$, where $j$ is the length of the first division subtuple of $y_\ell = Id_{s_\ell}.$\\

(iii). If $x$ is clean at $\ell$ and if $n$ is least so that $y_n \not= Id_{s_n}$, with $y_n$ clean at $m$, let $j$ denote the length of the first division subtuple of $ y_\ell$.  Then there are cases.   If $n < \ell$  then $\cD x$ is clean at $t_n + m$. If $n > \ell$ then $\cD x$ is clean at $t_\ell + j$.  If $n = \ell$ and  $j \leq m$ then $\cD x$ is clean at $t_\ell + j$. If $n = \ell$ and $j > m$ then $\cD x$ is clean at $t_\ell + m$.
\end{lem}

It is perhaps surprising that for $i > n$ the generators $y_i$ can be arbitrary and the conclusions still hold.  It is an obvious consequence of the lemma that if $x$ and the $y_i$ are all clean, for example basis generators, then $\Phi(x; y_i)$ is clean.   $\ \ \ \qed$\\

 We interrupt the proof of Proposition \ref{20.3} to explain some important things. First, from Step 0 and Step 1 above the map $\Phi$ agrees with $\cO_\Sigma$ in degree 0 and is twisted equivariant.  It is also possible to directly prove that $\Phi$ with the signs given in 20.3(ii) is a chain map.  Therefore, it follows from Lemma \ref{20.6} and the twisted equivariant uniqueness result Proposition \ref {19.3} that $\Phi$ is the standard procedure map.  However, it is not  easy to directly prove that $\Phi$ is a chain map.  This was accomplished by McClure-Smith in [19] for their surjection complex, with different signs, by essentially brute force computation.  We believe our proof of Proposition \ref{20.3} below, although long and complicated in its own way, offers a more structured viewpoint.  The steps in the proof are interesting by themselves and bring together quite a few things. \\ 

{\bf  $\bf \cS_*$ is an Operad and  $\bf \cE_* \to  \cS_*$ is an Operad Morphism.}  We will  explain here how Lemma \ref{20.6} and Proposition \ref{20.3}(i) imply that the surjection complex operad structure maps from Proposition \ref{19.2} satisfy all the operad axioms, especially the associativity axiom, that is, the commutativity of Diagram $\boxed{18.1}$.  In particular we do not need to know the signs in Proposition \ref{20.3}(ii).\\

The first parts of the argument are the same as the corresponding discussion for the Barratt-Eccles operad in the previous section, and is a universal argument that applies to any collection of complexes $B_*(n)$ to which Proposition \ref{19.2} applies.   The remaining issue then is why do the two maps around Diagram $\boxed{18.1}$ for the surjection complexes, which are both compositions of  standard procedure maps, agree with the standard procedure twisted equivariant map from upper left to lower right?  The reason is because  Lemma \ref{20.6} and Proposition \ref{20.3}(i), establish the criterion of Proposition \ref{19.5} for the surjection complexes.  The composed standard procedure maps in Diagram $\boxed{18.1}$ take basis generators to elements in $Im(H_s) \subset \cS_*(s)$.  Thus, assuming Proposition \ref{20.3}(i) we have proved 

\begin{prop}\label{20.7} The surjection complexes $ \cS_*(n)$  form an operad, using the standard twisted equivariant procedure chain maps of Proposition \ref{19.2} as structure maps.\ \ \ $\qed$
\end{prop}
We  also explain how the same reasoning  implies the following.
 \begin{prop}\label{20.8} The diagram below commutes.
 $$\begin{CD}  N_*(E\Sigma_r) & \otimes & N_*(E\Sigma_{s_1})& \otimes & \cdots &\otimes & N_*(E\Sigma_{s_r})  & \xrightarrow{\cO_{\cE} } &  N_*(E\Sigma_s)  \\
\downarrow TR&\otimes & \downarrow TR  & \otimes&\cdots & \otimes & \downarrow TR & & \downarrow TR\\
 \cS_*(r) & \otimes & \cS_*(s_1)&  \otimes &\cdots & \otimes &  \cS_*(s_r) & \xrightarrow{\cO_{\cS} } & \cS_*(s).
\end{CD}$$
Thus the surjection operad is a quotient operad of the Barratt-Eccles operad.
\end{prop}
\begin{proof}The proof is the same argument for all  versions of the surjection operad and the table reduction map.  We have shown that the map $\cO_\cE$ in the top row is a twisted equivariant standard procedure chain map.  The vertical ordinary $\Sigma_s$-equivariant  map $TR$ on the right commutes with contractions.  Therefore by the twisted equivariant version of Proposition 6.5(ii), the composition across the top and down is the twisted equivariant standard procedure map.\\

The tensor product of $TR$ maps on the left is the standard procedure map by Proposition 6.11 of Part I, so in going down and across the bottom of the diagram we again need to deal with a composition of standard procedure maps. The vertical map is an ordinary equivariant map and the $\cO_\cS$ map is twisted equivariant.  So the composition is twisted equivariant.\\

We will again use the basic uniqueness result Proposition \ref{19.3} for twisted equivariant chain maps. The key point is that for a basis generator $(e, X)$ of any $N_*(E\Sigma_n)$, the summands of $TR(e, X)$ are parametrized by partitions, as described in Section 16.1.  Consider a partition $ a_0 + \ldots + a_k = n+k$ with $a_0  = \ell$.  The corresponding surjection summand $\pm x_a \in \cS_*(n)$ then has the form $x_a= (1, 2,  \ldots, \ell, ..., \ell, \ldots )$, where $\ell$ is the first caesura entry.  Such $x_a$ are clean  surjection generators, hence in the image of the contraction $H_n$.  A tensor product of such generators is in the image of the tensor product contraction.  Then, by  Lemma \ref{20.6} and Proposition \ref{20.3}(i), applying $\cO_\cS$ to such a tensor lands in $Im(H_s)$, hence the uniqueness result Proposition \ref{19.3} applies.\\

Note this argument is very similar to the proof of Proposition \ref{17.5} that the composition $N_*(E\Sigma_n) \otimes N_*(X)  \to \cS_*(n) \otimes N_*(X) \to N_*(X)^{\otimes n}$ is the functorial standard procedure map.
 \end{proof}

\begin{rem}\label{20.9} Proposition \ref{20.8} provides an indirect proof that the surjection operad  structure maps satisfy the associativity axiom, although the key idea used is Proposition \ref{20.3}(i), whch is  the same as in the direct proof in Proposition \ref{20.7}.  We have not yet discussed the signs in the operad structure map of Proposition \ref{20.3}(ii), but the signs are provided, in a somewhat hidden form, by the Barratt-Eccles operad.\\

Given basis surjection generators $x, y_i \in \cS_*^{bf}(n)$, let $X, Y_i \in N_*(E\Sigma_n)$ denote the corresponding fundamental simplices in the Barratt-Eccles operad, as discussed in Remark 16.8 of Part II.  Then from Proposition 16.10 of Part II,  $TR(X) = x$ and $TR(Y_i) = y_i$.  The $\cO_\cE$ map is the composition of the map $EZ$, which of course has explicit signs, followed by the non-homomorphic inclusion $N_*(E(\Sigma_r \times \prod \Sigma_{si})) \subset N_*(E\Sigma_s)$, which has no signs.  Thus $$\cO_\cS(x; y_i) = TR  \circ N_*(\cO_W) \circ EZ(X \otimes \bigotimes Y_i) \in \cS_*^{bf}(s)$$
is a formula that contains the signs.  It is possible, but not easy, to keep track of the $EZ$ signs and the $TR$ formula and reconcile the implicit signs in this formula for $\cO_\cS$  with the signs in Proposition \ref{20.3}(ii).  In fact, Berger and Fresse seem to have carried this out in Section 1.5 of their paper [3].  But we prefer to establish the signs as part of our inductive proof of Proposition \ref{20.3}.  That proof will be long, involving many cases, but keeping track of signs is a relatively small part of the total work involved.\\

It might seem that the implicit signs in this last formula for $\cO_{\cS}(x; y_i)$ involving $EZ$ is not sensitive to which surjection complex is meant.  But from Proposition 16.5 of Part II, for the complexes $\cS_*^{ms}(n)$ and $\cS_*^{aj}(n)$, the vertical $TR$ maps do have signs attached to  summands $x_a$ of an image $TR(e, X)$. These are the same signs that occur in the isomorphisms between surjection complexes in Section 15.  The strategy of using the uniqueness result Proposition \ref{19.3} to prove certain diagrams commute obviously also works to prove that the isomorphisms between surjection complexes are operad isomorphisms. The signs in the isomorphisms between surjection complexes could be used along with Proposition \ref{20.3}(ii) to determine the signs in the operad structure maps for $\cS_*^{ms}(n)$ and $\cS_*^{aj}(n)$.  $\qed$
\end{rem}
{\bf Resumption of the Proof of Proposition \ref{20.3}.}  Suppose $x$ and all the $y$'s are basis surjection generators. That is, the first occurrences of $1,2, \ldots$ occur in that order. Basis generators in positive degrees are clean at the first caesura.  We attack Proposition \ref{20.3} by induction. Thus in the first unknown degree the standard procedure map on basis generators can be written $$\cO_\cS(x; y_1, \ldots, y_r) = H_s \Phi \Big[ (dx; y_1, \ldots, y_r) + \sum_i \pm (x; y_1, \ldots, dy_i, \ldots, y_r)\Big],$$ where $H_s$ is the contraction of $\cS_*^{bf}(s)$.\\

{\bf Step 3}. {\bf Relevant Boundary Terms of $d(x; y_1, \ldots, y_r)$.}  If basis generator $x \not= Id_r \in \cS_*^{bf}(r)$  has first caesura at $\ell$ then,  just as in Step 3 of the proof of Proposition \ref{17.3}, the only boundary term of $dx$ that is not clean or degenerate is $d_\ell x$.   The same holds true for  the $dy_i$ for those $y_i \not= Id_{si}$, which let us assume have first caesuras  at $m_i$.   Since $H_s$ vanishes on clean generators, a conclusion from Lemma \ref{20.6} in one lower degree is that only two boundary terms are relevant.  To be precise
$$ \cO_\cS(x; y_1, \ldots, y_r) = H_s \Phi (d_\ell x; y_1, \ldots, y_r) + \sum_iH_s \Phi(\pm (x; y_1, \ldots, d_{m_i} y_i, \ldots, y_r))$$ $$= H_s \Phi (d_\ell x; y_1, \ldots, y_r) + (-1)^{|x|} H_s \Phi (x; y_1, \ldots, d_m y_n, \ldots, y_r),$$

where $n$ is least with $y_n \not= Id_{s_n}$ and $m$ is the first caesura of $y_n$. The point of the last equality is that for $i > n$ the boundary term $d_{m_i} y_i$ may not be clean, but by Lemma \ref{20.6} that doesn't matter.  If $x = Id_r$ then only the second summand occurs.  If all $y_i = Id_{s_i}$ then only the first summand occurs. The  sign is $(-1)^{|x|}$ because $|y_i| = 0$ for $i < n$.   $\qed$

\begin{rem}\label{20.10} {\bf Where We are Headed.} By induction we have formulas for the two $\Phi$ evaluations in one lower degree. To prove Proposition \ref{20.3}(i) we will apply $H_s$ to the two boundary terms and match the result with the asserted value of $\Phi(x; y_1, \ldots, y_r)$, up to signs.    Signs for the standard procedure map are forced recursively. We will also deal with signs inductively, hence prove cases of Proposition \ref{20.3}(ii) at the same time that we prove cases of part (i).  Note that for the Berger-Fresse surjection complex the contraction $H_s$ introduces no signs.\\

There are some additional cases where all $\cD$ summands of the boundary term $\Phi(d_\ell x, y_1, \ldots, y_r)$ or $\Phi(x; y_1, \ldots, d_my_n, \ldots, y_r)$ are clean. We are assuming here that $x$ and the $y_i$ are basis generators, $n$ is least with $y_n \not = Id_{sn}$, the first caesura of $x$ is $\ell$, and the first caesura of $y_n$ is $m$. We will denote by $\cD''$  choices of $k_i$-divisions of $y_i$ for $i \not= \ell$ and a $(k_\ell -1)$-division of $y_\ell$.  We denote by $\cD'$  choices of $k_i$-divisions of $y_i$ for $i \not= n$ and a $k_n$-division of $d_m y_n$.\\

If $ n < \ell$ then for each relevant division $\cD'' $ the summand $\cD'' d_\ell x$ is a clean generator, since   $y_n[t_n]$ is inserted early into $d_\ell x$.  Thus in this case we want $\Phi(x; y_i) = H_s \Phi (x; y_1, \ldots, d_my_n, \ldots, y_r)$, which can be written $$\sum_{\cD} \pm \cD x  = (-1)^{|x|}H_s(\sum_{\cD'} \pm\cD' x).$$  If $\ell < n $ then for each relevant division $\cD'$ the summand $\cD' x$ will be clean at some shifted value $t_\ell+\ell$.  Thus in this case we want $\Phi(x; y_i) = H_s \Phi(d_\ell x; y_i)$, which can be written $$\sum_{\cD} \pm \cD x =  H_s(\sum_{\cD''} \pm \cD'' d_\ell x).$$
\end{rem}

\begin{rem}\label{20.11}{\bf Review of Calculations of $\bf H_s$.} We need a digression to explain how $H_s$ will be computed on the two lower degree terms.  We recall that for the Berger-Fresse complex $\cS_*^{bf}(s)$ the contraction\footnote{We apologize for the double meaning of $s$ in this section, the first term in the contraction formula, $s(z) = (1,z)$, and the arity of a surjection complex. Hopefully this will not cause trouble.} $H_s = \sum_{j=0}^{s-2} i^j s r^j$ is described in Section 13.  If $z$ is a surjection generator $i^q s r^q (z) = 0$ unless $1,2, \ldots, q$ are singletons in $z$, in which case $i^q s r^q (z)$ is described by removing those singletons from $z$ and putting the sequence $12 \ldots q(q+1)$ at the front.  Thus if the first entry of $z$ after those singletons are removed is $q+1$ then still $i^q s r^q (z) = 0$. \\

We will encounter two types of $H_s$ evaluations.  All cases of the evaluations of summands $H_s(\cD' x)$ of $H_s \Phi(x; y_1, \ldots, d_my_n, \ldots,  y_r) $ will be variants of the results from Proposition 15.1(iv),(v) of Part II that for basis generators $y$ with first caesura $m$, the formula  $y = H dy$ consists of a single non-zero term $y = i^{m-1}sr^{m-1} d_my$. More generally, if $z $ is such a $y$ followed by a shifted clean generator with larger entries, then the same conclusion holds.  That is, $z = H dz$ consists of a single non-zero term $z = i^{m-1}sr^{m-1} d_mz$.\\

In the case $d_m y = (1,2, \ldots, m-1, m+1, ... , m, ...)$ there is either no repeated entry or the first repeated entry is another $m$ or it is some $m+q$, following a  string $(m+1, m+2, ...)$ that increases except possibly straddling a singleton $m$. From this, calculation of $H d y$ is clear. The same argument works for the more general $z$.\\

In our variants $H_s (\cD' x)$, the entries of $d_m y_n$ are shifted by some amount $t$, a string of consecutive singletons $1,2, \ldots t$ will occur in front, and larger entries forming a shifted clean generator will be adjoined after the shifted $d_my_n$ sequence. But the corresponding $\cD x$ will be a basis generator followed by a shifted clean generator, with $\cD' x = d_{t+m} \cD x$. Only the single $H_s$ term $i^{t+m-1} s r^{t+m-1} \cD' x = \cD x$ will be non-zero.\\ 

All cases of the evaluations of summands $H_s(\cD'' d_\ell x)$ of $H_s \Phi(d_\ell x; y_1, \ldots, y_r)$ evaluations will be variants of $H_s z$, where $z  = ( larger, 1, 2, \ldots, j,\ larger, j, ...)$ or $z = (larger, 1,2, \ldots, j,\ larger,...)$. Here `larger' means entries greater than $j$.  In both cases, $1,\ldots, j-1$ are singletons.  In the second case, $j$ is a singleton but  removing $1, \ldots, j$ will result in a shifted clean generator beginning with $j+1$.   In both cases  $H_s z$ will consist of a sum of $j$ terms $\sum_{q =0}^{j-1} i^q s r^q z$. In our variants $H_s( \cD'' d_\ell x)$  these entries of $z$ are shifted by some amount $t$ and  a string of consecutive singletons will occur in front. But the $H_s (\cD'' d_\ell x)$ evaluation will still consist of a similar sum of $j$ terms.  $\qed$
\end{rem}
{\bf Step 4}. {\bf Description of  Six Cases.} At the end of this subsection we prove various cases of Proposition \ref{20.3}.  In the various cases, we will  write some sentences then give examples.  The examples are supposed to help understand the sentences. It is all rather lengthy, but not hard once you get the hang of it.  The structure is actually rather attractive, and quite similar to the structure in the proof of Proposition \ref{17.3} identifying the functorial map $\cS_*^{bf}(n) \otimes N_*(X) \to N_*(X)^{\otimes n}$ with a Berger-Fresse map.\\ 

We divide the analysis of tensor products of basis generators $x \otimes \bigotimes_i y_i$ of positive degree into six cases.  There are  two cases, $(a)$ where $x = Id_r$ and some $y_n \not=Id_{sn}$, and $(b)$ where all $y_i = Id_{si}$ and $x \not= Id_r$. In the remaining four cases, say $x$ has first caesura $\ell$ and $n$ is least with $y_n \not = Id_{sn}$, with first caesura $m$. Two of these cases are $(a')$ where $n < \ell$, and $(b')$ where $\ell < n$.  Finally, when $\ell = n$, let $j$ denote the length of the first division subtuple of $y_\ell$.  Then we have the final cases $(a'')$ where $j > m$, and $(b'')$ where $j \leq m$.\\ 

We will formally introduce in Remark \ref{20.12} just below an alternate `partial operad composition map' method, which would simplify the proof  of Proposition \ref{20.3}.  In the partial composition situations a single $y_n \not= (1)$ and all other $y_i = (1)$.  Thus only the $y_n$ has non-trivial $k_n$-divisions.  The other $y_i$ have only the trivial $k_i$-divisions $(1,1, \ldots, 1)$.   In these special cases the relevant $H_s$ evaluations are also simplified.\\

The partial composition approach also relies on Lemma \ref{20.6} in order to justify that certain compositions of standard procedure partial operad chain maps are standard procedure chain maps.  We do not rely solely on the partial composition proof  because it seemed a little surprising to us that the direct proof for the full operad map was essentially no harder than the usually simpler partial operad composition map arguments.\\

Readers who prefer partial compositions are invited to digest Remark \ref{20.12}, and  then think about the six cases and examples for Proposition \ref{20.3} only for partial compositions.  Less timid readers can follow the full discussions of all the cases and examples.$\ \ \ \qed$

\begin{rem} \label{20.12}{\bf Partial Compositions.} 
Of importance in the operad structure map and division mechanism is the case $$\cO_{\cS}(x; 1, \ldots,1, y_i, 1, \ldots, 1),$$ when all but one $y_j = (1)$, the unit element  in $\cS_*(1) = \FF$.  These are usually called {\it partial operad compositions} and abbreviated $\cO_i(x; y).$  It is understood that $y_j = (1)$ for $j \not = i$ and the $k_j =1$-divisions for these $y_j[t_j]$ must be the unique trivial one.  Note $t_j = j-1$ for $j \leq i$, and $t_j = s_i + j-2$  for $j > i$.

\begin{exam} \label{20.13} Take $x = (1,2,1,3,2)$ and $y = y_2 = (1,2,1)$. There are three $k_2 = 2$-divisions of $y_2[1] = (2,3,2)$, namely $(2 | 2,3,2), (2,3 | 3,2)$ and $(2,3,2  | 2)$.  Then $y_1[t_1] = (1)$ and $y_3[t_3] = (4)$.  We obtain$$\cO_2(x; y) = (1, \underline 2 ,1,4, \underline{2,3,2}) - (1, \underline {2,3}, 1, 4,  \underline {3,2}) - (1, \underline{2,3,2},1,4, \underline 2).$$
The $y$-caesura is the second 2 in the first summand and the first 2 in the other two summands.  The $x$-caesuras are the other two caesuras  in all summands. These precede the $y$-caesura in the first summand and straddle the $y$-caesura in the last two summands,  which accounts for the signs.  $\ \ \  \qed$
\end{exam}

The partial composition operations $\cO_i(x; y) \in P(s+r-1)$ are defined for any operad $P$, with $x \in P(r)$ and $y = y_i \in P(s)$.    It is understood that there are really $r-1$ other $y_j = (1) \in P(1),\ j \not= i$.  It has long been understood that such partial operations satisfying certain axioms determine full operad compositions.  Often results about operads can be proved more easily by focusing on the partial compositions.  All we might use of that method is a formula equating the full operad operation to a composition of partial operations.  That identity, which can be proved in any operad  by repeated use of the associativity axiom and unit axiom  in special cases, is the following composition representation of $\cO_P(x; y_1 \ldots, y_r) \in P(s)$, where $x \in P(r),\ y_i \in P(s_i), \ s = \sum s_i$: 
$$ P(r) \otimes P(s_1) \otimes \ldots P(s_r) \xrightarrow{\cO_1 \otimes Id} P(s_1+r-1) \otimes P(s_2) \otimes \ldots \otimes P(s_r) \xrightarrow{\cO_{s_1+1} \otimes Id }$$ $$P(s_1+s_2+r-2) \otimes P(s_3) \otimes \ldots \otimes P(s_r) \xrightarrow{\cO_{s_1+s_2+1} \otimes Id} \ldots  \to P(s).$$
In the case of the surjection complexes it is easy to connect this composition with Proposition \ref{20.3}.  One first uses the $k_1$-divisions $\cD_1$ to replace the 1's in $x$ by $k_1$-divisions of $y_1$.  Then one replaces the original 2's in $x$, which have now become $(s_1+1)$'s, by the $k_2$-divisions $\cD_2$ of $y_2[s_1]$, and so on.\\

The signs in the full operad map from Proposition \ref{20.3}(ii) associated to divisions $\cD x$ are just the products of the signs in the composition steps for the indvidual $\cD_i$.  This is clear since after the $j^{th}$ step of the composition, each $y_i[t_i]$-caesura, $i < j$, still precedes the same number of $x$-caesuras, and new $y_j[t_j]$-caesuras only precede $x$-caesuras of (original) value at least $j$. \\

It is somewhat easier to prove some of the results and handle examples in the six cases of Proposition \ref{20.3} for partial operad composition maps than it is for the full operad composition map.  We have proved enough, especially Lemma \ref{20.6}, to  argue that  the composition of partial operation standard procedure maps is itself a standard procedure map, which is not automatic.  So that is interesting, and does provide an alternate proof of Proposition \ref{20.3}, including the signs.  But we are able to handle the full standard procedure map of Proposition \ref{20.3} directly, so we do so. $\qed$
\end{rem}
{\bf Details for the Six Cases.} We now continue Step 4 and begin the analysis of the six cases for Proposition \ref{20.3}.  Since we include many examples, this will take several pages.\\

{\bf Case $\bf{(a)}$:} Suppose $x = Id_r$ and suppose $n$ is least so that $y_n \not= Id_{sn}$, with first caesura $m$.  Then from Remark \ref{20.10} we want $ \Phi(x; y_1, \ldots, y_r) =  + H_s \Phi(x;y_1, .., d_m y_n, .., y_r) $.  The sign is $+1$ because $|x| = 0$. Now all $k_i = 1$, and there are only the  unique  1-divisions $\cD x, \cD' x$ on both sides. One observes that $\cD x =  \Phi(x; y_1, \ldots, y_r) $ is a basis generator with first caesura $t+m$, where $t = t_n$. In the partial composition case this is obvious.  Then $t = n-1$ and $$\cD x = (1, \dots, n-1, y_n[n-1], n +s_n, \ldots,  s).$$  Also, $$\cD' x = (1, \dots, n-1, (d_my_n)[n-1], n +s_n, \ldots,  s) = d_{t+m} \cD x$$ and $H_s (\cD' x) =  i^{t+m-1} s r^{t+m-1}\cD' x =  \cD x$.\\

In the general case, the $y_i,\ i < n$, can be higher $Id_{si}$ and the basis generators $y_i,\ i > n$, can be arbitrary.  But the shifts in the insertions guarantee that $\cD x$ is a basis generator. Also $\cD' x =  \Phi(x; y_1, .., d_my_n, .., y_r) = d_{t+m}  \Phi(x; y_1, \ldots, y_r)$.  Again the only the summand  $i^{t+m-1} s r^{t+m-1}$ of $H_s$ evaluates non-trivially and inserts the missing $t+m$ back where it belongs in $\cD x = \Phi(x; y_i)$.\\

In all cases, since $x = Id_r$ has no caesuras, the shuffle counts are 0 on both sides of the equation and the desired signs check out.\\

{\bf Examples $\bf{(a)}$:}  A partial composition example is provided by $x = (1,2,3,4)$, $y_1 = y_3 =y_4 = (1)$, and $y_2 = (1,2,3,4,3)$.  Then $\cD x = (1,2,3,4,5,4, 6,7)$.  Also $d_3 y_2 = (1,2,4,3)$ and $$  \cD' x = \Phi (x; (1), d_3y_2, (1), (1)) =  (1,\ 2, 3, 5,4,\  6, 7) = d_4 \cD x.$$ Applying the single $H_s$ summand $ i^3 s r^3$ yields the desired result.\\

For a more general example, take $x = (1,2,3,4),\ y_1 = (1,2),\ y_2 = (1,2,3,4,3)$.  Then $n = 2,\ m = 3$ and $(d_3 y_2)[2] = (3,4,6,5)$.  We find $$\cD x = \Phi(x; y_i)  = (1, 2, 3,4,5,6,5, y_3[6], y_4[t_4]),$$ where the entries of $y_4[t_4]$ are greater than entries of $y_3[6]$, which has first entry 7. We still see $\cD x$ is a basic generator since the first occurrences of all entries, including those of $y_3[6]$ and $y_4[t_4]$,  occur in increasing order.  Also $$  \cD' x = \Phi (x, y_1, d_3y_2, y_3, y_4) =  (1,2, 3, 4, 6,5, y_3[6], y_4[t_4]) = d_5 \cD x.$$
 Applying the single $H_s$ summand $ i^4 s r^4$ yields the desired result. If $y_2$ had more than two 3 entries, it would have been more obvious that only one summand of $H_s\cD' x$ was non-zero. $\ \ \ \ \ \qed$\\

{\bf Case $\bf{(a')}$:} Suppose basis generator $x$ is clean at $\ell$ and suppose $n$ is least so that $y_n \not= Id_{sn}$, with first caesura $m$.  Assume $n < \ell$.  From Remark \ref{20.10} we want $$ \sum \pm \cD x = \Phi(x; y_1, \ldots, y_r) = $$  $$(-1)^{|x|}H_s \Phi(x; y_1, ..., d_my_n, ... ,y_r) = (-1)^{|x|} \sum H_s (\pm \cD' x).$$  In this case there should be a sign change by $(-1)^{|x|}$ between all pairs of  corresponding $\cD x$ and $\cD' x$ terms. Part of the argument resembles Case (a) above.  Since $n < \ell$ implies $k_n = 1$, a shifted $y_n$ or $d_my_n$ is inserted into $x$ in the part of $x$ that still looks like $Id$. For $i \not= n$ inserting choices of $k_i$-divisions of $y_i[t_i]$ on both sides match up the choices of $\cD x$ and $\cD' x$ terms.\\

Case $(a')$ does become simpler in a partial composition situation, with only $y_n \not= (1)$. Then there is only a single, trivial, pair of divisions $\cD, \cD'$, since $n < \ell$ implies $y_n$ is not divided.  Also, it is not hard to see that $\cD x$ will be a basic surjection generator, which simplifies the $H_s (\cD' x)$ analysis.\\

In the general Case $(a')$, $\cD x$ is not necessarily a basis generator, but it is a basis generator followed by a shifted clean generator of larger values. The first caesura is $t+m$, where $t = t_n$, and  $\cD' x = d_{t+m} \cD x$.  For $i < n$ inserting  $y_i[t_i]$ terms just inserts a sequence of singletons.   For $n \leq i < \ell$, inserting  the shifted basis elements $y_i[t_i]$ terms does not interfere with the claim that the first occurrences of all entries of $\cD x$ up to that point occur in increasing order. Inserting division subtuples of the $y_i[t_i]$ for $\ell \leq i$ appends a shifted clean generator with larger entries.  The $H_s (\cD'x)$ calculation reduces to the single term $i^{t+m-1}s r^{t+m-1} \cD' x = \cD x$, which puts a $t+m$ back where it belongs. \\

The additional entry $t+m$ is a $y$-caesura that comes before all the $x$-caesuras in every summand $\cD x$.  That $y$-caesura is not present in the corresponding $\cD' x$ term in $\Phi(x; y_1, ..., d_my_n, ... ,y_r)$.  The other $y$-caesuras are in the same positions relative to $x$-caesuras on both sides. This confirms the desired sign change, $sh(C_x, C_\cD) = |x| + sh(C_x, C_{\cD'})$ in this case.\\

{\bf Examples $\bf{(a')}$:}  We first give a partial composition example. Take $x = (1,2,3,4,3, 4)$, $y_1 = (1,2,3,1)$ and $y_2 = y_3 = y_4 = (1)$.  Then $\cD x = (1,2,3,1,4,5,6,5,6)$, a basic surjection generator.  Replacing $y_1$ by $d_1y_1$ gives the corresponding $\cD' x = (2,3,1,4,5,6,5,6) = d_1 \cD x$. Then $H_6(\cD' x) = s \cD' x = (1, \cD' x) = \cD x$.  In $\cD' x$ there are two $x$-caesuras and no $y$-caesuras. In $\cD x$ there is the $y$ caesura 1 in front of the two $x$-caesuras 5, 6.  The shuffle counts and signs are in agreement with $|x| = 2$.  \\

As a more general example, take $x = (1,2,3,4,3)$, $y_1 = (1,2,3,1)$, $y_2 = (1,2)$, $y_3 = (1,2,3,4)$ with 2-division $(1,2 |, 2,3,4)$, and $y_4 = (1,2,1)$.  Then $\cD x = (1,2,3,1,4,5,6,7, 10, 11, 10, 7,8,9)$.  We see that $\cD x$ is a basic generator $(1,2,3,1,4,5)$ followed by a shifted clean generator  $ (6,7,10,11,10, 7,8,9)$.\\

The difference from the partial composition case is that with $y_\ell \not= (1)$ a $k_\ell$-division will split the entries of $\cD x$, and larger entries from a shifted $y_j, j > \ell$, can come before the first occurrences of some of the shifted $y_\ell$ entries.  In the specific example, we see an initial 10 and 11, before the first occurrences of 8 and 9.\\

Replacing $y_1$ with $d_1 y_1$ gives $\cD' x = (2,3,1, 4,5, 6,7,10, 11, 10, 7,8,9) = d_1 x$.  Then $H_{11} \cD'  x = s \cD' x = (1, \cD' x) = \cD x$.  There is one $x$-caesura, 7, in $\cD x$ and two $y$-caesuras, 1 and 10, in the order 1,7,10.   In $\cD' x$ there is the $x$-caesura, 7, and only one $y$-caesura, 10.  The shuffle counts are 1 and 0 respectively, confirming the desired sign relation since $|x| = 1$.\\

{\bf Case $\bf{(b)}$:}  Suppose all $y_i = Id_{si}$ and $x$ is clean at $\ell$. Then from Remark \ref{20.10} we want $$\sum \cD x = \Phi(x; y_i) = H_s \Phi (d_\ell x, y_i) = \sum H_s (\cD'' d_\ell x).$$   The $y_i = Id_{si}$ have no caesuras, so in this case all shuffle counts are 0 and the signs are correct.\\

In Case (b) there will not be a bijection between $\cD x$ summands and $\cD'' d_\ell x$ summands.  For each choice of $\cD''$ so that the first division subtuple of $y_\ell$ has length $j$ it will happen that $H_s (\cD'' d_\ell x)$ is a sum of $j$ summands $\cD x$.  We explain this assertion.\\

For $i < \ell$ the substitutions of the 1-divisions $y_i[t_i]$ into $x$ and $d_\ell x$, just introduce identical  strings of initial singletons on both sides, and do not materially affect the comparisons.  If $i > \ell$, inserting $k_i$-divisions of $y_i[t_i]$ on both sides inserts entries larger than the first repeated entry. These do not mess up the $H_s$ evaluation.    \\

So to compare all terms on both sides we only need to look closely at the $k_\ell $-divisions of $y_\ell = Id_{s\ell}$ in $\Phi(x; y_i)$ and the $(k_\ell -1)$-divisions in $H_s \Phi(d_\ell x; y_i)$.  The only relevant $(k_\ell-1)$-division subtuple of $y_\ell = Id_{s\ell}$ for the evaluation $H_s\Phi(d_\ell x; y_i)$ is the first subtuple, say $(1,2, \ldots, j)$, with an appropriate shift of those entry values.  One computes $H_s \Phi(d_\ell x; y_i)$ to be the   sum of terms corresponding to terms in $\Phi(x; y_i)$ with the first two division subtuples of $y_\ell$ being $(1 | 1, .., j), (1,2 | 2, .., j), \ldots, (1, 2, \ldots, j | j)$, also shifted.  Of course $j$ varies unless $k_\ell -1 = 1$ in which case $j = s_\ell$, but the total result is confirmation of $\Phi(x; y_i) = H_s \Phi (d_\ell x, y_i)$.\\

The prototype here for $H_s \Phi (d_\ell x, y_i)$ is the case $$z = (larger, 1, 2, \ldots, j,\ larger, j...)\ or\ (larger, 1,2, \ldots, j,\ larger, ...)$$ discussed in Remark \ref{20.11}, with a shift of entries and initial singletons at the beginning of such $z$.  The second form occurs when $k_\ell  = 2$.  Then it is important for the $H_s$ evaluation that if the entries of the  insertions of the $y_i[t_i]$  for $i < \ell$ and the 1-division  $y_\ell[t_\ell]$ are removed from $\cD'' d_\ell x$, the result is a shifted clean generator.  In  cases where $k_\ell > 2$ it is more obvious that $H_s$ is a sum of $j$ terms because of a  repeated  entry which is a shift of the $j$.\\

The following examples illustrate the essential features.\\

{\bf Examples ${\bf (b)}$:} We first look at a partial composition example. Take $x = (1,2,3,4,2,3,2, 4),\  y_1 = (1),\  y_2 = (1,2,3,4)\ y_3 = (1), y_4 = (1)$.  Then $\ell = 2$, $k_2 = 3$,  and $d_2 x = (1,3,4,2,3,2, 4)$.  The first $k_2-1 =2$-division subtuple of $y_2[1]$ can be $(2), (2,3), (2,3,4), (2,3,4,5)$. In the first case, the 2-division is $(2 | 2,3,4,5)$. Then $H_s \Phi (d_2x, y_i)$ has a summand $H_s \cD'' d_2 x$ given by
$$H_s(1, 6, 7, 2, 6, 2,3,4,5, 7) = (1,2, 6, 7, 2, 6, 2,3,4,5, 7),$$ which is the summand of $\Phi(x; y_i)$ corresponding to the 3-division $(2 | 2 | 2, 3, 4, 5)$ of $y_2[1]$.\\
 
Next, look at the   2-division of $y_2[1]$ given by $(2,3 | 3,4,5)$. Then $\Phi(d_2x, y_i)$ has a summand $(1, 6, 7, 2, 3, 6, 3,4,5, 7)$. The $H_s$ application will involve two non-zero summands $i s r + i^2sr^2$, yielding a sum  $$(1,2, 6, 7, 2,3,6, 3,4,5, 7)+ (1,2,3,6, 7, 3, 6, 3,4,5, 7).$$  These two terms are seen in $\Phi(x; y_i)$, with $y_2[1]$ subdivisions $(2 | 2,3 | 3,4,5)$ and $(2,3 | 3 | 3,4,5)$.  The first two division subtuples amalgamate to $(2,3)$.\\
 
Look at the next 2-division of $y_2[1]$, namely $(2,3,4 | 4,5)$.  Then compute $$H_s(1,6, 7, 2,3,4, 6, 4,5, 7) = (1,2,6, 7, 2,3,4, 6, 4,5, 7) $$ $$+(1,2,3, 6, 7,3,4, 6, 4,5, 7) + (1,2,3, 4, 6, 7, 4, 6, 4,5, 7),$$ giving three more terms from $\Phi(x; y_i)$, corresponding to 2-divisions of $(2,3,4)$.\\

Finally, the 2-division $(2,3,4,5 | 5)$ of $y_2[1]$ will yield $$H_s(1, 6,7, 2,3,4,5, 6, 5, 7),$$ giving four terms of $\Phi(x; y_i)$, corresponding to 2-divisions of $(2,3,4,5)$.\\

If $x$ is replaced by $(1,2,3,4,2,3,5,4)$ then $k_2-1 = 1$ and we only look at the 1-division $(2,3,4,5)$ of $y_2[1]$ in $\Phi(d_2x; y_i)$.  There will now be four summands in  $H_s(1, 6,7, 2,3,4,5, 6, 8,7)$ corresponding to the 2-divisions of $y_2[1] $ in $\Phi(x; y_i)$.  The reason is, if $1,... , 5$ are all removed, the sequence $(6,7,6,8,7)$ is a shifted clean sequence beginning with a 6.\\

For a general example, take $x = (1,2,3,4,2,3,2, 4),\  y_1 = (1,2),\  y_2 = (1,2,3,4),\ y_3 = Id_{s3},\ y_4 = Id_{s4}$.  Then $\ell = 2$, $k_2 = 3$,  and $d_2 x = (1,3,4,2,3,2, 4)$.  The first 2-division subtuple of $y_2[2]$ can be $(3), (3,4), (3,4,5)$, or $(3,4,5,6)$. In the first case, the 2-division is $(3 | 3,4,5,6)$. We get a term of $H_s \Phi(d_2x; y_i)$ given by the summand $i^2 s r^2$ of $H_s$,
$$H_s \cD'' x = H_s(1, 2, Y_3', Y_4', 3, Y_3'', 3,4,5,6, Y_4'') = (1,2, 3, Y_3', Y_4', 3, Y_3'', 3,4,5,6, Y_4''),$$ which is a summand of $\Phi(x; y_i)$ corresponding to the 3-division $(3 | 3 | 3, 4, 5, 6)$ of $y_2[2]$.  The $Y_3$ and $Y_4$ terms involve higher entries and depend on chosen subdivisions of shifts of $y_3 =Id_{s3}$ and $y_4 = Id_{s4}$.\\  

Next, look at the 2-division of $y_2[2]$ given by $(3,4 | 4,5,6)$. Then $\Phi(d_2x, y_i)$ has terms $(1, 2,Y_3', Y_4', 3,4, Y_3'', 4,5,6, Y_4'')$. The $H_s$ application will involve two non-zero summands $i^2 s r ^2+ i^3sr^3$, yielding a sum  $$(1,2, 3,Y_3', Y_4', 3,4,Y_3'', 4,5,6, Y_4'')+ (1,2,3, 4,Y_3', Y_4', 4, Y_3'', 4,5,6, Y_4'').$$  These two terms are seen in $\Phi(x; y_i)$, with $y_2[2]$ subdivisions $(3 | 3,4 | 4,5,6)$ and $(3,4 | 4 | 4,5,6)$.  The first two subtuples amalgamate to $(3,4)$.\\
 
Look at the next 2-division of $y_2[2]$, namely $(3,4,5 | 5,6)$.  Then compute $$H_s(1,2,Y_3', Y_4', 3,4,5, Y_3'', 5,6, Y_4'') = (1,2,3, Y_3', Y_4', 3,4,5, Y_3'', 5,6, Y_4'') $$ $$+(1,2,3, 4, Y_3', Y_4',4,5, Y_3'', 5,6, Y_4'') + (1,2,3, 4, 5, Y_3', Y_4', 5, Y_3'', 5,6, Y_4''),$$ giving three more terms from $\Phi(x; y_i)$, corresponding to 2-divisions of $(3,4,5)$.\\

Finally, the 2-division $(3,4,5,6 | 6)$ of $y_2[1]$ will yield $$H_s(1, 2, Y_3', Y_4', 3,4,5,6, Y_3'', 6, Y_4''),$$ giving four terms of $\Phi(x; y_i)$, corresponding to 2-divisions of $(3,4,5,6)$.\\ 

If $x$ is replaced by $(1,2,3,4,2,3,5,4)$ then $k_2-1 = 1$ and we only look at the 1-division $(3,4,5,6)$ of $y_2[2]$ in $\Phi(d_2x; y_i)$.  There will now be four summands in  $H_s(1, 2, Y_3', Y_4',3,4,5,6, Y_3'', Y_5, Y_4'')$ corresponding to the 2-divisions of $y_2[2] $ in $\Phi(x; y_i)$.  The reason there are not more is, with $1,... , 6$ removed, the sequence $(Y_3', Y_4', Y_3'', Y_5, Y_4'')$ is a shifted clean sequence beginning with a 7. This cuts off other possible non-zero summands of $H_s.$ $\ \ \ \ \qed$\\

{\bf Case $\bf{(b')}$: }Again suppose basis generator $x$ is clean at $\ell$ and suppose $n$ is least so that $y_n \not= Id_{sn}$, with first caesura $m$.  Assume now $ \ell < n$.  Then from Remark \ref{20.10}  we again want $\Phi(x; y_i) = H_s \Phi(d_\ell x; y_i)$.  The argument in part resembles Case (b) above.  Again $y_\ell = Id_{s\ell}$ and we compare $k_\ell $-divisions on the left with $(k_\ell -1)$-divisions on the right. Only the first division subtuples $(1,\ldots, j)$ of $y_\ell$ on the $\Phi(d_\ell x; y_i)$ side are relevant.  The calculation of the summands of $H_s \Phi(d_\ell x; y_i)$ proceeds exactly as in Case (b).\\

Partial compositions in Case $(b')$ are rather trivial since $y_\ell = (1)$, so only $j = 1$ is possible. All relevant $H_s$ evaluations  reduce to single terms.\\

To see that there is no sign change between corresponding terms in the $H_s$ evaluation we need to look at shuffle counts. The only difference between caesuras of $x$ and $d_\ell x$ is the initial $\ell$ caesura of $x$.  Since $ \ell < n$ the only $y$-caesuras in summands on both sides correspond to caesuras of $y_i,$ with $  \ell < i$.  None of these precede the first $\ell$ caesura of $x$. Therefore the shuffle counts of $y$-caesuras preceding  $x$-caesuras or $d_\ell x$-caesuras are the same for each summand on both sides of the equation.\\

{\bf Examples $\bf {(b')}$:} We first give a partial composition example. Take $x = (1,2,3,2,3)$ and $y_1 = y_2 = (1),  y_3 = (1,2,1)$. Then $\ell = 2 < 3 = n$ and $d_2x = (1,3,2,3)$.  There are three  2-divisions of $y_3[2] = (3,4,3)$, namely $(3 | 3,4,3), (3,4 | 4,3)$ and $(3,4,3  | 3)$.   We calculate $$\Phi(x; (1), (1), y_3) = (1,2,3,2,3,4,3) - (1,2,3,4,2,4,3) - (1, 2,3,4,3, 2,3).$$
The $y$ caesura is the second 3 in the first summand, which comes before no $x$-caesuras.  The $y$-caesura is the first 3 in the last two summands, which comes before one $x$-caesura.  This explains the signs.\\

We also calculate $$\Phi(d_2x; (1), (1), y_3) = (1,3,2,3,4,3) - (1,3,4,2,4,3) - (1,3,4,3, 2,3).$$ The $H_4$ calculation reduces to one term, $isr$, for each summand. The shuffle counts and signs also check.\\

We also give a more general example of Case $(b')$.  Let $x = (1,2,3,2,3,2)$ $y_1 = (12), y_2 = (1,2,3), y_3 = (1,2,1,3,4,2,3)$.  Then $ \ell = 2 < 3 = n$ and $d_2x = (1,3,2,3,2)$. For the $d_2 x$ calculation we will focus on just one 2-division of $y_2[2]$, say $(3,4 | 4,5)$ with $j = 2$, and one 2-division of $y_3[5]$, say $(6,7,6,8 | 8,9,7,8)$.\\

The corresponding summand of $\Phi(d_2 x; y_i)$ is $$(1,2,6,7,6,8,3,4,8,9,7,8,4,5)\ with\ sh(C_{d_2x}, C_{\cD''}) = 4.$$ Applying $H_9$ gives two terms $$(1,2,3,6,7,6,8,3,4,8,9,7,8,4,5) + (1,2,3,4, 6,7,6,8,4,8,9,7,8,4,5),$$ also each with $ sh(C_x, C_{\cD}) = 4.$ These are the two terms of $\Phi(x; y_i)$ corresponding to the two 3-divisions $(3 | 3,4 | 5)$ and $(3,4 | 4 | 4,5)$ of $y_2[2]$ and the fixed 2-division of $y_3[5]$.\\

{\bf Final Cases:} Lastly we assume $x$ is clean at $\ell$ and also that $\ell$ is least so that $y_\ell \not= Id_{s\ell}$.  Assume $y_\ell$ is clean at $m$.  Then we want $$\Phi(x, y_i) = H_s \Phi(d_\ell x; y_1, \ldots, y_r) + (-1)^{|x|}H_s \Phi(x; y_1,..., d_my_\ell, ... y_r).$$  In this case, there will be summands from both terms on the right.  Again, the divisions of $y_i$ for $i < \ell$ and $i > \ell$ just come along for the ride. The summands  $\cD x$ of $\Phi(x; y_i)$ will separate according to the first division subtuple of $y_\ell$.     If that subtuple is $(1,2, \dots, j)$ with $j > m$ then the terms come from $H_s \Phi(x; y_1, ..., d_my_\ell, ... ,y_r)$, and the computations are similar to Case $(a')$.  If  $j \leq m$ then the terms come from $H_s \Phi(d_\ell x, y_i)$.  The computations are similar to Case $(b')$. \\

{\bf Case $\bf{(a'')}$: }The terms $H_s \Phi(x; y_1, ..., d_my_\ell, ... ,y_r)$  are the easiest to clarify.  Given a division $\cD$ in which the first $y_\ell$ subtuple is $(1,2, \ldots, m-1, m, m+1, ...)$ of length $j > m$, observe that $\cD x$ is a basis generator followed by a shifted clean generator with larger entries. The first caesura is $t+m$, where $t = t_\ell$. Remove the first $m$ from the first $y_\ell$ subtuple,  to form the first subtuple of a $k_\ell$-division of $d_m y_\ell$.  Leave all other subtuples of the division $\cD_\ell $ alone, and also leave alone all divisions $\cD_i$ of the $y_i$ for $i \not= \ell$. This produces a collection of  divisions $\cD'$ paired up with the described divisions $\cD$.  For these $\cD'$, the first division subtuple of $d_m y_\ell$ has length at least $m$, and $\cD' x = d_{t+m} \cD x$.  Then $H_s \cD' x= i^{t+m-1} s r^{t+m-1}\cD' x = \cD x$. Any divisions $\cD'$  for which the first division subtuple of $d_m y_\ell$ has length less than $m$ will result in a clean term $\cD' x$ on which $H_s$ vanishes.\\

The term $\cD x$ has one new first $y$-caesura $t+m$ that comes before all the $x$-caesuras.  Other caesuras of $\cD' x$ and $\cD x$ are in the same relative positions.  Thus $sh(C_x, C_\cD) = |x| + sh(C_x, C_{\cD'})$, as in Case $(a')$.  This confirms the desired sign relation, involving terms $ (-1)^{|x|}  H_s(\Phi(x; y_1, \ldots, d_m y_\ell, \ldots, y_r)$ as summands of $\Phi(x; y_i)$.\\

{\bf Examples ${\bf (a'')}$:} Take $x = (1,2,3,4,2,3,2)$, $y_1 = (1),\ y_2 = (1,2,3,4,3)$, with $y_3, y_4$ arbitrary basis generators.  Then $\ell = n = 2$, $k_2 = 3$, and $m = 3$.  For a partial composition example, take $y_3 = y_4 = (1)$.  In the partial composition case the division subtuples called $Y_3', Y_3'', Y_4$ below simplify to singletons $(6), (6), (7)$, respectively, which contribute no $y$-caesuras and make the example a little easier to follow.\\ 

We have $d_3y_2[1] = (2,3,5,4)$, and we first consider all $k_2 = 3$-divisions of $d_3y_2[1]$ with first subtuple of length $i \geq m = 3$.   These are the 3-divisions $$(2,3,5 | 5,4 | 4),\ (2,3,5 | 5 | 5,4),\ (2,3,5,4 |4 | 4).$$
For each there will be exactly one non-zero $H_s$ summand $i^3 s r^3$, which are $$H_s(1, 2,3, \underline5, \underline{Y_3'}, Y_4, 5, \underline4, Y_3'', 4) = (1,2,3, \underline{\underline4}, \underline5, \underline{Y_3'}, Y_4, 5, \underline4,Y_3'',  4)$$
$$H_s(1, 2,3, \underline5, \underline{Y_3'}, Y_4, \underline5, Y_3'',  5, 4) = (1,2,3,\underline{\underline4}, \underline5, \underline{Y_3'}, Y_4, \underline5, Y_3'',  5,4)$$
$$H_s(1, 2,3,5, \underline4, \underline{Y_3'}, Y_4, \underline4, Y_3'', 4) = (1,2,3,\underline{\underline4},5,\underline4, \underline{Y_3'}, Y_4, \underline4,Y_3'', 4).$$
These account for all the summands $\cD x$ for which the first subtuple of $ y_2$ has length $j > m =3$.  Summands of $\Phi(x; y_1, d_3y_2, y_3, y_4)$ for which the first division subtuple of $d_3y_2$ has length $i < 3$ will result in clean surjection generators $\cD' x$, so $H_s$ will vanish on those.  For example $(2,3 | 3,5 | 5,4)$ results in $\cD' x = (1,2,3, Y_3', Y_4, 3,5, Y_3'', 5,4).$\\

In these examples with $i \geq m = 3$, we have underlined the $x$-caesuras on both sides.  In the case of the underlined $Y_3'$, the $x$-caesura will be the last entry of that division subtuple.  We have not given specific generators $y_3, y_4$ in this example.  There could be $y$-caesuras in the $Y_3', Y_3'', Y_4$ terms, but these are in the same positions relative to $x$-caesuras, so do not create any differences between the signs associated to the terms on the two sides.\\

$d_3y_2[1]$ has no $y$-caesuras, so there are no further $y$-caesuras on the left.  But $y_2[1]$ has a caesura 4, and we have double underlined that $y$-caesura on the right. Note $|x| = 3$.  On the right the $y$-caesura is in front of the three $x$-caesuras.  This is consistent with the term $(-1)^{|x|} H_s \Phi(x; y_1, \ldots, d_m y_n, \ldots, y_r)$ that is part of the formula for $\Phi(x; y_i)$. \ \ \  $\qed$\\

{\bf Case $\bf{(b'')}$:} Now we discuss $k_i$ divisions $\cD$ of the $y_i$ for which the first division subtuple of $\cD_\ell$ has length $j \leq m$.  Define divisions $\cD''$ of the $y_i$ by leaving all $\cD_i$ alone if $i \not= \ell$ and forming a $(k_\ell -1)$-division $\cD''_\ell$ of $y_\ell$ whose first subdivision tuple is the  amalgamation of the first two division subtuples of $\cD_\ell$.  Later division subtuples of $\cD_\ell$ are left unchanged in $\cD''_\ell$.  Then $H_s( \cD'' d_\ell x)$ will consist of $min\{j, m\}$ summands, including $\cD x.$ Note there are exactly $min\{j, m\}$ such $\cD_\ell $ with the same `amalgamation' $\cD''_\ell$.\\

The shuffle counts $sh(C_x, C_\cD)$ and $sh(C_{d_\ell x}, C_{\cD''})$ are the same because even though $\cD x$ has an additional first $x$-caesura, the $y$-entries that precede it are singletons, not caesuras.\\

{\bf Examples $\bf{(b'')}$: }We continue with the same $x, y_1, y_2, y_3, y_4$ as in Example $(a'')$, including the partial composition example where $y_3 = y_4 = (1)$. Again the terms called $Y_3', Y_3'', Y_4$ below then become singletons $(6), (6), (7)$. We have $d_2 x = (1,3,4,2,3,2)$. There are five $(k_2-1) = 2$-divisions of $y_2[1]$, used to compute $\cD'' d_2x$ summands. These will account for the summands  $\cD x$ for which the first subtuple of $ y_2[1]$ has length $j \leq m = 3$. First $(2 | 2, 3,4,5,4)$, $ (2,3 | 3,4,5,4)$, and $(2,3,4 | 4,5,4)$ are the 2-divisions of $y_2[1]$ with length of the first subtuple 1, 2, or 3.  Applying $H_s$ to the corresponding $\cD'' d_2 x$ terms we get 
$$ H_s(1, \underline{Y_3'}, Y_4, \underline2, Y_3'', 2, 3, \underline{\underline4}, 5, 4) = (1, \underline2, \underline{Y_3'},  Y_4, \underline2, Y_3'', 2, 3, \underline{\underline4}, 5, 4)$$

$$H_s(1, \underline{Y_3'}, Y_4, 2, \underline3, Y_3'', 3, \underline{\underline4}, 5, 4) = (1,\underline2, \underline{Y_3'},  Y_4, 2, \underline3,Y_3'', 3, \underline{\underline4}, 5, 4) + $$ $$(1, 2, \underline3, \underline{Y_3'}, Y_4, \underline3, Y_3'',  3, \underline{\underline4}, 5, 4)$$

$$H_s(1, \underline{Y_3'}, Y_4, 2, 3, \underline4, Y_3'',  \underline{\underline4}, 5, 4) = (1,\underline{2}, \underline{Y_3'}, Y_4, 2, 3, \underline{4}, Y_3'',  \underline{\underline4} , 5, 4) + $$ $$(1, 2, \underline{3}, \underline{Y_3'}, Y_4, 3, \underline{4}, Y_3'',  \underline{\underline4}, 5, 4) + (1, 2, 3, \underline{4}, \underline{Y_3'}, Y_4, \underline{4}, Y_3'', \underline{\underline4}, 5, 4).$$

The number of summands on the right is the number of 2-divisions of $(2)$, $(2,3)$, or $(2,3,4)$.\\

The $y$-caesuras in $Y_3', Y_3''$ and $Y_4$ occur in the same relative position to the $x$-caesuras on both sides. We have double underlined the one $y_2$-caesura, a 4,  in all terms. It appears after the two single underlined $x$-caesuras on the left and after all three $x$-caesuras on the right. So the associated shuffle signs are the same for every term, the $y_2$ caesura is not involved.  \\

Then there are the 2-divisions $ (2,3,4,5 | 5,4)$, and $(2,3,4,5,4 | 4)$ of $y_2[1]$ with length of the first subtuple  4 or 5.  Applying $H_s$ to the corresponding $\cD'' d_2 x$ terms gives three $\cD x$ summands each on the right. In each summand the amalgamation of the first two $\cD x$ $y_2[1]$ division subtuples is $(2,3,4,5)$ or $(2,3,4,5,4)$, subject to the constraint that the first subtuple is of length $j = 1, 2, \rm{or}\ 3$.

$$H_s(1, \underline{Y_3'}, Y_4, 2, 3, \underline{\underline4}, \underline{5}, Y_3'', 5, 4) = (1,\underline{2}, \underline{Y_3'}, Y_4, 2, 3, \underline{\underline4}, \underline5, Y_3'', 5, 4) +$$
$$  (1, 2 ,\underline{3}, \underline{Y_3'}, Y_4, 3, \underline{\underline4}, \underline{5}, Y_3'', 5, 4) + (1, 2, 3, \underline{4}, \underline{Y_3'}, Y_4, \underline{\underline4}, \underline5, Y_3'', 5, 4).$$

$$H_s(1, \underline{Y_3'}, Y_4, 2, 3, \underline{\underline4}, 5, \underline4, Y_3'', 4) = (1,\underline{2}, \underline{Y_3'}, Y_4, 2, 3, \underline{\underline4}, 5, \underline{4}, Y_3'', 4) +$$
$$  (1, 2, \underline{3}, \underline{Y_3'}, Y_4, 3, \underline{\underline4}, 5, \underline{4}, Y_3'', 4) + (1,2, 3, \underline{4}, \underline{Y_3'}, Y_4, \underline{\underline4}, 5, \underline4, Y_3'', 4).$$

In all terms of these last two examples the $y_2$-caesura double underlined 4 precedes the final $x$-caesura, which is a single underlined 5 in the first example and a single underlined 4 in the second example.  Thus in each example the shuffle signs associated to all terms on both sides are the same, as they should be.  But the sign in these last two examples is opposite the sign in the first three examples. The $y_2$-caesura now contributes +1 to all shuffle counts.\\

To repeat the pattern (without signs) in this last somewhat complicated Case $(b'')$, for each $\cD x$ term with first $y_\ell$ subtuple of length $j \leq m$, amalgamate the first two $y_\ell$ subtuples, producing a $\cD'' d_\ell x$ term, and calculate $H_s$.  Each such $H_s$ calculation will be a sum of $min(j, m)$ $\cD x$ terms, namely, the terms with the same amalgamation of the first two $y_\ell$ subtuples.  The signs are what they are, but are the same for a given $\cD'' d_\ell x$ term and all associated $\cD x$ terms.  $\qed$

\subsection{The Operad Morphism $\cS \to \cZ$}
In this final subsection of Part III we prove that the equivariant functorial maps $\Phi \colon \cS_*(n) \otimes N_*(X) \to N_*(X)^{\otimes n}$ studied in Section 17 fit together to give an operad morphism $\phi \colon \cS \to \cZ$ from the surjection operad to the Eilenberg-Zilber operad $\cZ$, which is the functorial $CoEnd$ operad studied in Subsection 18.3.

\begin{prop}\label{20.14} The diagram below commutes. 
$$\begin{CD} 
 \cS_*(r) & \otimes & \cS_*(s_1)&  \otimes &\cdots & \otimes &  \cS_*(s_r) & \xrightarrow{\cO_\cS } & \cS_*(s) \\
\downarrow \phi&\otimes & \downarrow \phi  & \otimes&\cdots & \otimes & \downarrow \phi & & \downarrow \phi\\
\cZ_*(r)& \otimes & \cZ_*(s_1) & \otimes & \cdots & \otimes & \cZ_*(s_r) & \xrightarrow{\cO_\cZ} & \cZ_*(s)
\end{CD}$$
where $\cZ_*(n) = HOM_{func}(N_*( - ), N_*( - )^ {\otimes n})$ and the vertical maps $\phi$ are adjoints of the maps  $\Phi$ as $X$ varies.
\end{prop}

\begin{proof}The vertical maps are equivariant chain maps with respect to appropriate permutation groups and the horizontal maps are twisted equivariant chain maps.  The two compositions around the diagram are thus twisted equivariant chain maps.  In degree 0 both compositions are the twisted equivariant extension of the map that takes identity elements $(e_r; e_{si})$ to the Alexander-Whitney diagonal $\delta^{(s)} \colon N_*(X) \to N_*(X)^{\otimes s}$.\\

It suffices to work with the universal acyclic model $X = \Delta^m$, in fact with the fundamental class $\Delta^m \in N_m(\Delta^m)$,  and take adjoints with respect to the lower right corner, giving an equivalent diagram.
 $$\begin{CD} 
 \cS_*(r) & \otimes & \cS_*(s_1)&  \otimes &\cdots & \otimes &  \cS_*(s_r) &\otimes& N_*(\Delta^m)& \xrightarrow{\cO_\cS \otimes Id } & \cS_*(s) \otimes N_*(\Delta^m) \\
\downarrow \phi&\otimes & \downarrow \phi  & \otimes&\cdots & \otimes & \downarrow \phi & \otimes & \downarrow Id& & \downarrow \Phi\\
\cZ_*(r)& \otimes & \cZ_*(s_1) & \otimes & \cdots & \otimes & \cZ_*(s_r) &\otimes &N_*(\Delta^m) & \xrightarrow{Ad\  \cO_\cZ}  & N_*(\Delta^m)^{\otimes s}.
\end{CD}$$
Now we are in position to use results about contractions and  twisted equivariant standard procedure chain maps.  Specifically we want to use the results from Section 17 and Subsections 18.3 and 20.1 to show that both ways around the diagram take basis elements to elements in the image of the contraction of the range.  Then the uniqueness result Proposition \ref{17.9}, extended routinely to a twisted equivariant version,   implies both maps coincide with the standard twisted equivariant chain map. Thus it suffices to prove the following.
\begin{lem}\label{20.15} Suppose $x \in \cS_*(r)$ and $y_i \in \cS_*(s_i)$ are basis elements.  Then both $$\Phi \circ (\cO_\cS \otimes Id)(x \otimes \otimes_i\ y_i \otimes \Delta^m)$$ and $$Ad\ \cO_\cZ \circ (\phi \otimes   \otimes_i\ \phi \otimes Id) (x \otimes \otimes_i\  y_i \otimes \Delta^m)$$ belong to $Im(h_{\otimes^s}) \subset N_*(\Delta^m)^{\otimes s}$.
\end{lem}

We first go across the top and down in the diagram. In the notation of Section 17 and  the previous subsection, $$\Phi \circ(\cO_\cS \otimes Id)(x \otimes \otimes_i y_i \otimes \Delta^m) =\sum_\cD  \Phi ( \pm \cD x\otimes \Delta^m) = \sum_\cD \sum_{\widehat{M}} \pm F(\widehat{M}, \pm \cD x)$$
where the $\widehat{M}$ are monomial summands of the multidiagonal $\delta^{(s+| \cD x|)}(\Delta^m) $ and $F(\widehat{M}, \cD x)$ are tensors in $N_*(\Delta^m)^{\otimes s}.$ From Lemma \ref{20.6} the $\cD x$ are clean surjection generators.  Then from Step 2 of the proof of Proposition \ref{17.3}, the tensors $F(\widehat{M}, \cD x)$ are in the image of the contraction of $N_*(\Delta^m)^ {\otimes s}.$\\

Going down and across in the diagram requires a closer look.  Recall from Subsection 18.3 that $\cO_\cZ (u \otimes \otimes_i v_i) = \pm \underline{\otimes} v_i \circ u$.  Then $Ad\ \cO_\cZ(u \otimes \otimes_i v_i \otimes \Delta^m) = \pm \underline{\otimes} v_i( u(\Delta^m))$, where $u(\Delta^m) \in  N_*(\Delta^m)^{\otimes r}.$\\

In our case, $u = \phi(x)$ and $v_i = \phi(y_i)$.  So $u(\Delta^m) = \Phi(x \otimes \Delta^m) = \sum_M \pm F(M, x)$, where the $M$ are monomial summands of the multidiagonal $\delta^{(r + |x|)} (\Delta^m)$, and the $F(M,x)$ are $r$-tensors.  Specifically, we can write $F(M, x) = F_1(M, x) \otimes \cdots \otimes F_r(M, x)$, where each $F_i(M, x) \in N_{m_i}(\Delta^m)$ is a face of some dimension, depending on the recipe of Section 17 for turning $M$ and $x$ into the tensor $F(M, x)$.  Then we need to calculate the terms $$\underline{\otimes} \phi(y_i) (\otimes F_i(M, x)) \in \bigotimes N_*(\Delta^m) ^{\otimes s_i} = N_*(\Delta^m)^{\otimes s}$$ as  sums of amalgamated $s_i$-tensors.  In particular, we could identify the face $F_i(M, x)$ with a standard $\Delta^{m_i} \subset \Delta^m$ and use functoriality to compute  $$\Phi(y_i \otimes \Delta^{m_i}) = \sum_{M(i)} \pm F(M(i), y_i)) \in N_*(\Delta^{m_i}) ^{ \otimes {s_i}} \subset N_*(\Delta^m)^{\otimes {s_i}},$$ where the monomials $M(i)$ are monomials in appropriate multidiagonals of the $\Delta^{m_i}$.\\

But the vertices of the faces $F_i(M, x)$ are already named as vertices of $\Delta^m$, so there is no real need to work with the standard simplex $\Delta^{m_i}$.  One can just interpret the monomials $M(i)$ directly in terms of $\Delta^m$ vertices.\\

\begin{exam}\label{20.16} Let $x = (1,2,1,2),\ y_1 = (1,2,1,3,2)\ y_2 = (1,2,3,2,3)$.  Take $m = 9$ and monomial $M = (01234|456|67|789)$.  Calculate the summand  of $\Phi(x, \Delta^9)$ given by  $$F(M, x) = F_1(M,x) \otimes F_2(M, x) = (01234\ 67) \otimes (456\ 789) \in N_*(\Delta^9)^{\otimes 2}.$$   So the face dimensions are $m_1 =  6$ and $m_2 = 5$. We have separated the vertices of the faces to emphasize that the faces have subfaces revealed in the calculation of $F(M, x)$.\\

Take monomials $$M(1) = (0|012|23|346|67)\ \ and\ \ M(2) = (45|56|67|789|9)$$ and form the summand of $\Phi(y_1 \otimes F_1(M,x)) \otimes \Phi(y_2 \otimes F_2(M,x)) \in N_*(\Delta^9)^{\otimes 3} \otimes N_*(\Delta^9)^{\otimes 3}$ produced by the monomials $M(1), M(2)$. This 6-tensor is $$(023) \otimes (012\ 67) \otimes (346) \otimes  (45) \otimes (56\ 789) \otimes (67\ 9).\ \ \ \ \ \qed$$
\end{exam}
The first thing we notice is that this tensor does belong to $Im(h_{\otimes^6})$.  In fact, we can complete the proof of Lemma 20.15.  In complete generality it is easy to see that the tensors $F(M(i), y_i) \in N_*(\Delta^{m_i})^{\otimes s_i}$ are in the image of the contractions, where we interpret the $M(i)$ as monomials in the vertices of the faces.\footnote{So for the tensor associated to $M(2)$, the initial vertex is the 4.}  In fact, this is  nothing but the definition of these tensors as summands of the standard procedure map of Proposition \ref{17.3} applied to a basis surjection generator. In our case we have an amalgamation of $r$ such tensors.  The first one, with $i = 1$, always begins with a 0.  It is possible that this first tensor is just a tensor of $(0)$'s. Then the first vertex of the second face will  be 0 and the second tensor will begin with a 0.  In positive degrees, eventually one sees $(0) \otimes \ldots \otimes (0) \otimes (0a ...) \in Im(h_{\otimes^s})$ as desired.
\end{proof}
\begin{exam}\label{20.17} Let $x = (1,2,3,2),\ y_1 = (1,2,3),\ y_2 = (1,2,3,1),\ y_3 = (1)$. Take $m = 5$ and monomial $M= (0|012|234|45)$. Calculate the summand of $\Phi(x, \Delta^5)$ given by
$$F(M, x) = F_1(M,x) \otimes F_2(M, x) \otimes F_3(M, x) = (0) \otimes (012\ 45) \otimes (234) \in N_*(\Delta^5)^{\otimes 3}.$$
Take monomials $M(1) = (0|0|0)$, $M(2) = (01|1|124|45)$, and $M(3) = (234)$.  Form the summand of $$\Phi(y_1 \otimes F_1(M,x)) \otimes \Phi(y_2 \otimes F_2(M,x)) \otimes \Phi(y_3, \otimes F_3(M, x))$$ produced by the monomials $M(i)$. This 7-tensor is $$(0) \otimes (0) \otimes (0) \otimes (01\ 45) \otimes(1) \otimes (124) \otimes (234).   \qed$$
\end{exam}

\begin{rem}\label{20.18} It is rather surprising perhaps that the arguments above do not require knowing the signs in Proposition \ref{20.3}(ii). The reader is welcome to assume that $\cS_*(n) = \cS_*^{bf}(n)$ is the Berger-Fresse complex, but everything works as stated for any of the surjection complexes. Since the $\phi$ are {\it injections}, and we know those signs, in principal the signs in $\cO_{\cZ}$ determine the signs in $\cO_{\cS}$.  Also since $\phi$ is injective, Proposition \ref{20.14} implies the associativity property for the surjection complexes.\\

These two comments are quite analogous to Remark \ref{20.9}, where we pointed out that the {\it surjections} $TR \colon N_*(E\Sigma_n) \to S_*(n)$ imply the associativity property for the surjection complexes, and also implicitly determine the signs.  Although this remark is quite amusing, we prefer the direct calculation of signs carried out in the proof of Proposition \ref{20.3}.  $\qed$
\end{rem}
{\bf *Matching Summands in Lemma \ref{20.15}.*} Although we have now proved Proposition \ref{20.14}, and thus proved that the maps $\cS_*(n) \to CoEnd(n)$ form an operad morphism, it is not much harder to actually get inside the identity from Lemma \ref{20.15} $$\sum_\cD \sum_{\widehat{M}} F(\widehat{M}, \cD x) = \sum_M \sum_{M(i)} \bigotimes_i F(M(i), y_i)$$ that directly expresses the commutativity of the operad morphism diagram, at least ignoring signs.  The terms in the two sums match up as follows.\\

Beginning with $\widehat{M}$ and a $\cD x$, the term $F(\widehat{M}, \cD x)$ is a tensor in $N_*(\Delta^m)^{\otimes s} = \otimes_i N_*(\Delta^m)^{\otimes s_i}$.  The factors, taken in consecutive blocks of $s_i$-tensors, are the terms $F(M(i), y_i)$, and we need to identify the associated monomials $M$ and $M(i)$.\\

$\cD x$ is a string of entries between 1 and $s$, organized in shifted subblocks $(y_1^{(1)}, y_2^{(1)}[t_2], \ldots, y_i^{(j)}[t_i], \ldots, y_r^{(k_r)}[t_r]),$ where the $y_i^{(j)}$ are the division subtuples of the $y_i$ determined by $\cD = (\cD_1, \ldots, \cD_r).$ We unravel the blocks of $s_i$-tensors to produce monomials, as described in Remark \ref{17.4}.  These monomials are the $M(i)$.\\

The monomial $M$ is determined as follows. The monomial blocks of $\widehat{M}$ can be labelled by entries  of $x$.  Namely, each $\cD x$ entry is an entry of one of the subblocks $y_i^{(j)}[t_i]$, so the corresponding subblock of $\widehat{M}$ receives label $i$.  Amalgamating adjacent subblocks of $\widehat{M}$ with the same $x$ entry label determines the subblocks of monomial $M$, and hence determines $M$.

\begin{exam}\label{20.19}  Take $x = (1,2,1,2)$ and $y_1 = (1,2,1, 3,2)$, $y_2 = (1,2,3,2,3)$, with the 2-divisions $(1,2, 1,3 | 3,2)$ and $(1,2,3 | 3,2,3)$.  So $s_1 + s_2 =3+3 = 6 = s$. Then $$\cD x = (1,2,1,3,\ 4,5,6,\  3,2,\ 6,5,6).$$  Take $m = 9$ and monomial $\widehat{M} = (0 | 01 | 1 | 1 | 123 | 34 | 4 | 4 | 456 | 6 | 678 | 89)$.  Then $$F(\widehat{M}, \cD x) = (0\ 1) \otimes (01\ 456) \otimes (1\ 4) \otimes (123) \otimes (34\ 678) \otimes (4\ 6\ 89) \in N_*(\Delta^9)^{\otimes 6}.$$
Unraveling the product of the first three tensor terms and then the second three tensor terms gives $$M(1) = (0 | 01 | 1 | 14 | 456)\ \  and\ \  M(2) = (123 | 34 | 46 | 678 | 89).$$ The first three tensor terms themselves form $F(M(1), y_1)$, a summand of $\Phi(y_1 \otimes (01 456))$.  Similarly, the second three tensor terms  form $F(M(2), y_2)$, a summand of $\Phi(y_2, (1234 6789))$.\\

The subblocks of $\widehat{M}$ have $x$-labels $(1,1,1,1,2,2,2,1,1,2,2,2)$.  Amalgamating those subblocks, we get the subblocks of $M$ to be  $M = (01 | 1234 |  456 | 6789)$ and $F(M, x) = (01\ 456) \otimes (1234\ 6789)$.\\

Finally, to go the other direction, beginning with $M$ and the $M(i)$,  the tensor $\bigotimes_i F(M(i), y_i) \in N_*(\Delta^m)^{\otimes s}$ can be unraveled to form $\widehat{M}$.  We need to identify $\cD x$, which means identify the lengths of the division subtuples of the $y_i$.  Equivalently, we need to put  labels $1,2, \ldots, r$ on the subblocks of $\widehat{M}$.  The blocks of $M$ also have  labels, with $k_i$ blocks labeled $i$.  Consider the $j^{th}$ block of $M$ labeled $i$, say $M(ij)$.  Note $M(i)$ will have $s_i + |y_i|$ blocks.  Count the number of those blocks that intersect the block $M(ij)$. This is the length of the division subtuple $y_i^{(j)}$.\\

In the example we have $M(11) = (01),\ M(12) = (456),\ M(21) = (1234)$, and $M(22) = (6789)$. We see $M(11)$ intersects four blocks of $M(1) = (0 | 01 | 1 | 14 | 456)$, and $M(12)$ intersects two blocks of $M(1)$.  The division subtuples of $y_1$ thus have lengths 4 and 2, and the 2-division $\cD y_1 = (1,2, 1,3 | 3,2)$.  The blocks $M(21)$  and $M(22)$ intersect three and three blocks of $M(2) = (123 | 34 |46 | 678 | 89)$, respectively.  Thus the 2-division $\cD y_2 = (1,2,3 | 3,2,3)$.  $\qed$\\

{\bf *Final Exam*}. Take $x = (1,2,1,3,2,1,3)$ and $y_1 = (1,2,3,2,1)\ y_2 = (1,2,1,2)$ and $y_3 = (1,2,3,4,2)$. Take $m = 5$ and $M = ( 0 | 01 | 1 | 12 | 23 | 345 | 5)$.  Then $F(M, x) = (0\ 1\ 345) \otimes (01\ 23) \otimes (12\ 5)$.  Take  $M(1) = (0 | 13 | 34 | 4 | 45)$, $M(2) = (0 | 012 | 23 | 3)$, $M(3) = (12 | 2 | 2 | 25 | 5)$.  Find $\cD x$ and $\widehat{M}$ so that  $$F(\widehat{M}, \cD x) = F(M(1), y_1) \otimes F(M(2), y_2) \otimes F(M(3), y_3) \in N_*(\Delta^ 5)^{\otimes 9}.$$
We have proved the operad morphism diagram for $\cS \to \cZ$ commutes, therefore signs of matched up summands will agree.  At various points we have clarified all signs in the separate morphisms of the operad morphism diagram.  Separately calculate the signs associated to the two matched expressions in this exercise for the Berger-Fresse operad $\cS_*^{bf}$.  $\qed$ $\qed$ $\qed$
\end{exam}

\newpage
\addcontentsline{toc}{section}{\bf REFERENCES}
\section*{\bf REFERENCES}

1. M. Adamaszek and J.D.S. Jones, The Symmetric Join Operad, arXiv:1110.2989, 2011.\\

2. M. Barratt and P. Eccles, $\Gamma^+$-Structures - I: A Free Group Structure for Stable Homotopy Theory, Topology, Vol 13, 1974, pp. 25-45.\\

3. C. Berger and B. Fresse, Combinatorial Operad Actions on Cochains, arXiv:math/0109158v2, 2002.\\

4. C. Berger and B. Fresse, Combinatorial Operad Actions on Cochains, Mathematical Proceedings of the Cambridge Philosophical Society, Vol. 137, Issue 1, July 2004, pp. 135-174.\\

5. C. Berger and B. Fresse, Une Decomposition Prismatique de l'Operade de Barratt-Eccles, arXiv:math/0204326v1, 2002.\\

6. C. Berger and B, Fresse, Une Decomposition Prismatique de l'Operade de Barratt-Eccles, C. R. Math.
Acad. Sci. Paris 335(4), 365–370 (2002).\\

7. G. Brumfiel, A. Medina-Mardones, J. Morgan, A cochain level proof of Adem relations in the mod 2 Steenrod algebra, arXiv:2006.09354v2, 2021.\\

8. G. Brumfiel, A. Medina-Mardones, J. Morgan,  A cochain level proof of Adem relations in the mod 2 Steenrod algebra, Journal of Homotopy and Related Structures (2021) 16:517–562 https://doi.org/10.1007/s40062-021-00287-3.\\

9. G. Brumfiel and J. Morgan, The Pontrjagin Dual of 3-Dimensional Spin Bordism, arXiv:1612.02860v2, 2018.\\

10. G. Brumfiel and J. Morgan, The Pontrjagin Dual of 4-Dimensional Spin Bordism, arXiv:1803.08147, 2018.\\

11. G. Brumfiel and J. Morgan,  Quadratic Functions of Cocycles and Pin Structures, arXiv:1808.10484, 2018.\\

12. A. Hatcher, Algebraic Topology, 2001.\\
 
13. R.M. Kaufmann and A.M. Medina-Mardones, Cochain Level May-Steenrod Operations, arXiv:2010.02571v4, 2021.\\

14. L. Lambe and J. Stasheff, Applications of Perturbation Theory to Iterated Fibrations, Manuscripta Mathematica, (1987) Volume :58, page 363-376.\\

15. J. Loday and B. Vallette, Algebraic Operads, Grundlehren der mathematischen Wissenschaften 346.\\
 
16. M. Mandell, $E_\infty$ Algebras and $p$-adic Homotopy Theory, Topology 40 (2001), no. 1, pp. 43-94.\\

17. M. Mandell, Cochains and Homotopy Type, arXiv:math/0311016, 2002.\\

18. M. Mandell, Cochains and Homotopy Type, Publications Mathématique de l'Institut des Hautes etudes Scientifiques, 103 (2006), pp. 213-246.\\

19. J.E.  McClure and J. H. Smith, Multivariable Cochain Operations and Little n-Cubes, arXiv:math/0106024v3, 2002.\\

20. J. E. McClure and J. H. Smith,  Multivariable Cochain Operations and Little n-Cubes, J. Am. Math. Soc. 16(3), 681–704 (2003).\\

21. J.E. McClure and J. H. Smith, Operads and Cosimplicial Objects: An Introduction, arXiv:math/0402117v1, 2004.\\

22. A. M. Medina-Mardones, An effective proof of the Cartan formula: the even prime, J. Pure Appl. Algebra, 224(12):106444, 18, (2020).\\

23. A.M. Medina-Mardones, A. Pizzi, and P. Salvatore, Multisimplicial chains and configuration spaces, arXiv:2012.02060v2, 2023.\\

24. P. Real, Homotopy Perturbation Theory and Associativity, Homology, Homotopy, and Applications, Vol. 2, No.5, 2000, pp 51-88.\\

25. J. Rubio and F. Sergeraert, Algebraic Models for Homotopy Types, Homology, Homotopy and Applications, vol.7(2), 2005, pp.139–160.\\

26. V. A. Smirnov,  Homotopy Theory of Coalgebras,  Mathematics of the USSR-Izvestiya (1986)27(3): 575.\\

27. J. R. Smith, M-Structures Determine Integral Homotopy Type, arXiv:math/9809151, 1998.\\

28. J. R. Smith, Operads and Algebraic Homotopy, arXiv:math/0004003v7, 2000.\\

29  J. R. Smith,  Cellular Coalgebras Over the Barratt-Eccles Operad I, arXiv:1304.6328v4 2013.\\

30. N. E. Steenrod, Cohomology Operations, written and revised by D.B.A.Epstein, Annals of Math Studies vol 50, Princeton University Press, 1962.\\

31. H. Whitney, Moscow 1935: Topology Moving Toward America, Hassler Whitney Collected Papers Volume I, J. Eels, Domingo Toledo editors,  1992.

\end{document}